\documentclass[12pt,amstex]{amsart}
\usepackage{cite}
\usepackage{amsmath,amsthm,amssymb,amsfonts,amscd,verbatim,color}
\usepackage{mathrsfs}
\usepackage{bbm}
\usepackage{indentfirst}
\usepackage{enumerate}
\usepackage[linktocpage,backref=page,colorlinks, citecolor=red, anchorcolor=black, linkcolor=red]{hyperref}
\usepackage[top=35mm, bottom=35mm, left=30mm, right=30mm]{geometry}
\usepackage{xcolor}

\allowdisplaybreaks[4]

\theoremstyle{plain}

\theoremstyle{definition}

\newtheorem{definition}{Definition}[section]
\newtheorem{lemma}{Lemma}[section]
\newtheorem{theorem}{Theorem}[section]

\newtheorem{remark}{Remark}[section]

\numberwithin{equation}{section}
\usepackage{cleveref}

\newcommand{\bs}{\boldsymbol}
\newcommand{\mb}{\mathbb}

\newcommand{\mc}{\mathcal}
\newcommand{\ms}{\mathscr}
\newcommand{\mr}{\mathrm}

\def \pa{\partial}

\begin{document}
\title[]{Almost Global Solutions of Kirchhoff Equation}

\let\thefootnote\relax\footnotetext{Supported by NNSFC (Grant Nos. 12090010, 12090013)}

\author{Jianjun Liu \quad Duohui Xiang}
\address[Jianjun Liu] {College of Mathematics\\ Sichuan University\\ Chengdu 610065, China}
\email{jianjun.liu@scu.edu.cn}

\address[Duohui Xiang] {College of Mathematics\\ Sichuan University\\ Chengdu 610065, China}
\email{duohui.xiang@outlook.com}

\thanks{}

\begin{abstract}
    This paper is concerned with the original Kirchhoff equation 
    $$
\left\{\begin{aligned}
  & \pa_{tt}u-\Big(1+\int_{0}^{\pi}|\pa_xu|^2 dx\Big)\pa_{xx}u=0, 
  \\&u(t,0)=u(t,\pi)=0.
\end{aligned}\right.
$$
We obtain almost global existence and stability of solutions for almost any small initial data of size $\varepsilon$.
In Sobolev spaces, the time of existence and stability is of order $\varepsilon^{-r}$ for arbitrary positive integer $r$.
In Gevrey and analytic spaces, the time is of order $e^{\frac{|\ln\varepsilon|^2}{c\ln|\ln\varepsilon|}}$ with some positive constant $c$.
%
To achieve these, we build rational normal form for infinite dimensional  reversible vector fields without external parameters. We emphasize that for vector fields, the homological equation and the definition of rational normal form are significantly different from those for Hamiltonian  functions.
\end{abstract}

\maketitle
\tableofcontents

\section{Introduction and main results}
Consider the Kirchhoff equation with Dirichlet boundary conditions
\begin{equation}
\label{form11}
\left\{\begin{aligned}
  & \pa_{tt}u-\Big(1+\int_{0}^{\pi}|\pa_xu|^2 dx\Big)\pa_{xx}u=0, \; x\in[0,\pi],
  \\&u(t,0)=u(t,\pi)=0.
\end{aligned}\right.
\end{equation}
%
This equation was firstly introduced by Kirchhoff \cite{K1876} to model the small transversal oscillations of a clamped string. 
The Cauchy problem for \eqref{form11}  is locally well-posed for initial data $\big(u(0,x),\pa_tu(0,x)\big)$ in Sobolev space $H^{\frac{3}{2}}\times H^{\frac{1}{2}}$, seeing \cite{AP96} for example, where for any $s\geq0$, the Sobolev space is defined by
\begin{equation*}
H ^s:=\Big\{u=\sum_{a\in\mb{N}_*}u_a\sqrt{\frac{2}{\pi}}\sin ax\;\big|\;\; u_a\in\mb{R},\; \|u\|^{2}_{s}:=\sum_{a\in\mb{N}_*}a^{2s}|u_a|^{2}<+\infty\Big\}
\end{equation*}
with the positive integer set $\mb{N}_*$.
%
For analytic initial data, Bernstein \cite{B40} proved that the Cauchy problem for \eqref{form11} is globally well-posed. 
After that, the global well-posedness result was extended to a little larger spaces, such as quasi-analytic spaces, seeing \cite{N84} for example.
However, the global well-posedness is still an open problem for general initial data in Sobolev spaces or Gevrey spaces.
Notably, it is not even known for small initial data.
We will investigate almost global existence and stability for equation \eqref{form11} with small initial data in these spaces.
%

Under the standard inner product on $L^{2}([0,\pi])$, the equation \eqref{form11} is written as the system
\begin{equation}
\label{form12}
\left\{\begin{aligned}
  &\pa_{t}u=\;\;\frac{\pa H}{\pa v}=v,
  \\& \pa_{t}v=-\frac{\pa H}{\pa u}=\Big(1+\int_{0}^{\pi}|\pa_xu|^2 dx\Big)\pa_{xx}u,
\end{aligned}\right.
\end{equation}
with the Hamiltonian
\begin{equation}
\label{form13}
H(u,v)=\frac{1}{2}\int_{0}^{\pi}v^2 dx+\frac{1}{2}\int_{0}^{\pi}|\pa_xu|^2 dx+\frac{1}{4}\Big(\int_{0}^{\pi}|\pa_xu|^2 dx\Big)^2.
\end{equation}
For phase space $H^{s+\frac{1}{2}}\times H^{s-\frac{1}{2}}$, introduce the Gaussian measure formally defined by
\begin{equation}
d\mu_g=\frac{e^{-\sum_{a\in\mb{N}_{*}}(a^{2s+3}|u_a|^{2}+a^{2s+1}|v_a|^2)}dudv}{\int_{H^{s+\frac{1}{2}}\times H^{s-\frac{1}{2}}}e^{-\sum_{a\in\mb{N}_{*}}(a^{2s+3}|u_a|^{2}+a^{2s+1}|v_a|^2)}dudv}.
\end{equation}
This measure $\mu_g$ is viewed as the weak limit of finite dimensional Gaussian measures. By Theorem 2.4 in \cite{K19}, the measure $\mu_g$ is countably additive on the space $H^{s+\frac{1}{2}}\times H^{s-\frac{1}{2}}$ with $\mu_g(H^{s+\frac{1}{2}}\times H^{s-\frac{1}{2}})=1$.
Denote the open ball
$$B_{s}(R)=\{(u,v)\in H^{s+\frac{1}{2}}\times H^{s-\frac{1}{2}}\mid \|u\|^2_{s+\frac{1}{2}}+\|v\|^2_{s-\frac{1}{2}}<R^2\}$$ with the radius $R>0$.
Then $0<\mu_g(B_{s}(R))<1$, and thus we further  introduce the Gaussian measure $\mu$ in the unit ball $B_{s}(1)$ by
\begin{equation}
\label{form14}
d\mu=\frac{d\mu_g}{\mu_g(B_{s}(1))}.
\end{equation}
Then for small initial data in Sobolev spaces, we have  the following result.
\begin{theorem}
\label{th11}
For any integer $r\geq4$, there exists $s_{0}$ depending on $r$ such that for any $s\geq s_{0}=O(r^2)$, there exist $0<\varepsilon_{0}\ll 1$ depending on $r,s$ and an open set $\mc{V}_{r,s}\subset B_{s}(\varepsilon_0)$ such that for any $0<\varepsilon\leq\varepsilon_{0}$, if the initial datum $\big(u(0,x),v(0,x)\big)\in\mc{V}_{r,s}\cap B_{s}(\varepsilon)$,
then for any $|t|\leq\varepsilon^{-r}$,  the solution $u(t,x)$ of the Kirchhoff equation \eqref{form11} exists and satisfies
\begin{equation}
\label{form15}
\|u(t,x)\|_{s+\frac{1}{2}}^2+\|v(t,x)\|_{s-\frac{1}{2}}^2\leq4\varepsilon^2,
\end{equation}
\begin{equation}
\label{form16}
\sup_{a\in\mb{N}_*}a^{2s}|I_{a}(t)-I_{a}(0)|\leq\varepsilon^{3},
\end{equation}
where the action $I_a=\frac{a|u_a|^{2}+a^{-1}|v_a|^2}{2}$.
Moreover, the open set $\mc{V}_{r,s}$ is asymptotically of full measure, i.e., \begin{equation}
\label{form17}
\mu\Big(\varepsilon (u,v)\in\mc{V}_{r,s}\Big)\geq1-\varepsilon^{\frac{1}{14}}.
\end{equation}
\end{theorem}
Remark that apart from the existence of solutions, the theorem also contains two stability etimates, seeing \eqref{form15} for Sobolev norm and \eqref{form16} for every Fourier mode.
Moreover, we could assume that the actions $\{I_{a}\}_{a\in\mb{N}_*}$ are independent and uniformly distributed in an infinite cube $\prod_{a\in\mb{N}_{*}}(0, |a|^{-2s-2})$, and then use the product measure instead of the Gaussian measure \eqref{form14}. 

Next, introduce the Hilbert space 
\begin{equation}
\label{form18}
\mc{G}_{\rho,\theta,s}:=\Big\{u=\sum_{a\in\mb{N}_*}u_a\sqrt{\frac{2}{\pi}}\sin ax\;\big|\;\; u_a\in\mb{R},\;  \|u\|_{\rho,\theta,s}^{2}:=\sum_{a\in\mb{Z}}e^{2\rho a^{\theta}}a^{2s}|u_a|^{2}<+\infty\Big\},
\end{equation}
where $\rho>0$, $0<\theta\leq1$ and $s\geq0$. Taking 
$0<\theta<1$, it is the Gevrey space; and taking 
$\theta=1$, it is the analytic space.
Then we consider the Kirchhoff equation \eqref{form11} in the phase space $\mc{G}_{\rho,\theta,\frac{3}{2}}\times \mc{G}_{\rho,\theta,\frac{1}{2}}$.
 Denote the open ball
$$B_{\rho,\theta}(R)=\{(u,v)\in \mc{G}_{\rho,\theta,\frac{3}{2}}\times \mc{G}_{\rho,\theta,\frac{1}{2}}\mid \|u\|^2_{\rho,\theta,\frac{3}{2}}+\|v\|^2_{\rho,\theta,\frac{1}{2}}<R^2\}.$$
In a manner akin to \eqref{form14}, we formally define the Gaussian measure in the unit ball $B_{\rho,\theta}(1)$:
\begin{equation}
\label{form19}
d\mu=\frac{e^{-\sum_{a\in\mb{N}_{*}}e^{2\rho a^{\theta}}(a^{5}|u_a|^{2}+a^3|v_a|^2)}dudv}{\int_{B_{\rho,\theta}(1)}e^{-\sum_{a\in\mb{N}_{*}}e^{2\rho a^{\theta}}(a^{5}|u_a|^{2}+a^3|v_a|^2)}dudv}.
\end{equation}
Then for small initial data in Gevrey and analytic spaces, we have  the following  result.
\begin{theorem}
\label{th12}
Fix $\rho>0$ and $0<\theta\leq1$. There exist $0<\varepsilon_{0}\ll 1$ depending on $\rho,\theta$ and an open set $\mc{V}_{\rho,\theta}\subset B_{\rho,\theta}(\varepsilon_0)$ such that for any $0<\varepsilon\leq\varepsilon_{0}$, if the initial datum $u(0,x)\in\mc{V}_{\rho,\theta}\cap B_{\rho,\theta}(\varepsilon)$,
then for any
\begin{equation}
\label{form110}
|t|\leq\varepsilon^{-\frac{|\ln\varepsilon|}{15800(1+\frac{2}{\theta})\ln|\ln\varepsilon|}},
\end{equation}
 the solution $u(t,x)$ of the Kirchhoff equation \eqref{form11} exists and satisfies
\begin{equation}
\label{form111}
\|u(t,x)\|_{\rho,\theta,\frac{3}{2}}^2+\|v(t,x)\|_{\rho,\theta,\frac{1}{2}}^2\leq4\varepsilon^2,
\end{equation}
\begin{equation}
\label{form112}
\sup_{a\in\mb{N}_*}e^{2\rho|a|^{\theta}}a^{2}|I_{a}(t)-I_{a}(0)|\leq\varepsilon^{3}.
\end{equation}
Moreover, the open set $\mc{V}_{\rho,\theta}$ is asymptotically of full measure, i.e., \begin{equation}
\label{form113}
\mu\Big(\varepsilon (u,v)\in\mc{V}_{\rho,\theta}\Big)\geq1-\varepsilon^{\frac{1}{15}}.
\end{equation}
\end{theorem} 
Compared to Sobolev spaces, the time of existence and stability in \Cref{th12} is longer. As indicated in \eqref{form110},  the larger the value of $\theta$, the longer the time. It reflects an increase in stability time with higher regularity.
%
For analytic spaces, while the existence of global solutions has been previously established, this theorem is still meaningful as it provides the stability results \eqref{form111} and \eqref{form112}.
Remark that both \Cref{th11} and \Cref{th12} focus on  small initial data. Hence, the conclusions still hold true for the generalized Kirchhoff equation
$$\left\{\begin{aligned}
  & \pa_{tt}u-\phi\Big(\int_{0}^{\pi}|\pa_xu|^2 dx\Big)\pa_{xx}u=0, \; x\in[0,\pi],
  \\&u(t,0)=u(t,\pi)=0,
\end{aligned}\right.$$
where $\phi$ is a real analytic function on a neighborhood of the origin satisfying $\phi(0)>0$ and
$\phi'(0)\neq0$. Especially, taking $\phi(y)=1+y$, it is the Kirchhoff equation \eqref{form11}.

As previously noted, the Kirchhoff equation \eqref{form11} constitutes a Hamiltonian system. For Hamiltonian partial differential equations, Birkhoff normal form is typically employed to investigate the long time stability of solutions. 
For instance, Bambusi and Gr\'{e}bert \cite{BG06} proved an abstract Birkhoff normal form theorem adapted to a wide class of Hamiltonian partial differential equations. These equations often contain external parameters to guarantee non-resonant conditions, such as the mass in wave equations and the potential in Schr\"{o}dinger equations. Consequently, for nonlinear wave equations and nonlinear Schr\"{o}dinger equations in Sobolev spaces $H^s$ with sufficiently large $s$, they demonstrated arbitrarily polynomially long time stability, i.e., the stability time is of order $\varepsilon^{-r}$ for any positive integer $r$. 
There are numerous other results on polynomially long time stability for equations in Sobolev spaces with external parameters.
For equations with bounded nonlinear vector fields, see
\cite{B96,B03,G07,BDGS07,B08,GIP09,Z10,CLY16,BFG20,BFM24} for example. For equations with unbounded nonlinear vector fields, refer to \cite{D09a,D12,YZ14,D15,BD18,Z20,BMM24}.
%
%
Moreover, for equations in Gevrey or analytic spaces with external parameters, there are sub-exponentially long time stability results, i.e., the stability time is of order $\varepsilon^{-|\ln\varepsilon|^{\beta}}$ for some $0<\beta<1$. See \cite{FG13,CLSY18,BMP20,CCMW22,CMWZ22} for example.

For equations without external parameters, their linear frequencies typically fail to meet the high-order non-resonant conditions. An effective method is to extract parameters from nonlinear integrable terms by amplitude-frequency modulation, in which the amplitudes of initial data are used as parameters. 
For one dimensional nonlinear Schr\"{o}dinger equations, Kuksin and P\"{o}schel \cite{KP96} employed four-order normal form terms as part of the unperturbed Hamiltonian to modulate the resonant linear frequencies, thereby ensuring the non-resonant conditions in KAM iteration. Through this nonlinear modulation,  Bambusi \cite{B99b} demonstrated exponentially long time stability for a particular set of initial data in Sobolev space $H^1$, and Bourgain \cite{B00} proved arbitrarily polynomially long time stability for most small initial data in Sobolev spaces $H^s$ with large $s$. 
More recently, Bernier, Faou and Gr\'{e}bert \cite{BFG20b,BG21} used four-order and six-order normal form terms to modulate the resonant linear frequencies, and proposed rational normal form method to research more Hamiltonian partial differential equations, including one dimensional nonlinear Schr\"{o}dinger and Schr\"{o}dinger-Poisson equations, generalized KdV and Benjamin-Ono equations. This approach results in more twisted frequencies with respect to parameters.
In \cite{LX24}, we built rational normal form with exact global control of small divisors. As an application to one dimensional nonlinear Schr\"{o}dinger equations in Gevrey spaces, we obtained a stability time scale of order $\varepsilon^{-|\ln\varepsilon|^{\beta}}$ for any $0<\beta<1$. In \cite{BCGW24}, Bernier, Camps, Gr\'{e}bert and Wang achieved the stability time scale of order $\varepsilon^{-\frac{|\ln\varepsilon|}{\ln|\ln\varepsilon|}}$ for nonlinear Schr\"{o}dinger-Poisson equations, where this time scale is conjectured to be optimal.
Moreover, see\cite{LX24JMAA} for derivative nonlinear Schr\"{o}dinger equations, and see \cite{BC24} for  nonlinear Schr\"{o}dinger equations on non-rectangular flat tori.

We emphasize that the Kirchhoff equation \eqref{form11} is a quasi-linear partial differential equation without external parameters. Compared with KdV and Benjamin-Ono and derivative nonlinear Schr\"{o}dinger equations studied in \cite{BG21,LX24JMAA}, the Kirchhoff equation  exhibits a higher level of unboundedness. This leads to the failure of energy estimates, as noted in \cite{BH20}.
Actually, Baldi and Haus \cite{BH20,BH21,BH22} investigated the Kirchhoff equation on the $d$-dimensional torus $\mb{T}^d$. 
They eliminated the off-diagonal unbounded terms prior performing the normal form process, which is based on the method from \cite{D09a,D12,D15} of constructing a normal form for quasi-linear Klein-Gordon equations.
Concretely, the  $d$-dimensional Kirchhoff equation is transformed into
\begin{equation}
\label{formdK1}
\pa_t\left(\begin{array}{c}\eta
\\\bar{\eta}\end{array}\right)=(1+\mc{P}(\eta,\bar{\eta}))\left(\begin{array}{c}-{\rm i}|D_x|\eta \\{\rm i}|D_x|\bar{\eta}\end{array}\right)+R_{\geq3}(\eta,\bar{\eta}),
\end{equation}
where the term $\mc{P}$ is a small scalar multiplicative factor independent of the space variable $x$,
the notation $|D_x|$ is the Fourier multiplier $e^{{\rm i}k\cdot x}\mapsto |k|e^{{\rm i}k\cdot x}$ for any $k\in\mb{Z}^d$, and the remainder term $R_{\geq3}$ is a bounded vector field of order at least three.
Then after two steps of normal form procedure, 
the system \eqref{formdK1} becomes 
\begin{equation}
\label{formdK}
\pa_t\left(\begin{array}{c}w
\\\bar{w}\end{array}\right)=(1+\mc{P}(w,\bar{w}))\left(\begin{array}{c}-{\rm i}|D_x|w \\{\rm i}|D_x|\bar{w}\end{array}\right)+Z_3(w,\bar{w})+K_5(w,\bar{w})+R_{\geq7}(w,\bar{w}),
\end{equation}
where $Z_3$ is a cubic resonant vector field and gives no contribution to energy estimates,  $K_5$ is a quintic resonant vector field but provides a non-zero contribution to energy estimates, and $R_{\geq7}$ is a remainder term of order at least seven.
As a result, they obatined a lower bound on the existence time of order $\varepsilon^{-4}$ for all small initial data and $\varepsilon^{-6}$ for small initial data satisfying a suitable non-resonant condition. 
It is natural to ask for a longer time scale of existence and stability.
However, this is a challenging problem due to the resonant linear frequencies and the absence of external parameters. 
%
%
Recently, based on the structure of $K_5$ in \eqref{formdK}, the existence of chaotic-like motions was demonstrated in \cite{BGGH25}. 
%
For other equations with unbounded nonlinear vector fields and lacking external parameters , there are some results of long time stability featuring time scales of order $\varepsilon^{-4}$ or shorter, seeing \cite{T04,YZ16,BFP23,FLX24} for example.

For the one dimensional Kirchhoff equation \eqref{form11}, luckily, the cubic vector field $Z_3$ in \eqref{formdK} is integrable of the form
$${\rm i}\sum_{a\in\mb{N}_*}I_a(z_a\pa_{z_a}-\bar{z}_a\pa_{\bar{z}_a})$$
with the action $I_a=z_a\bar{z}_a$.
This provides a chance to perform rational normal form process as mentioned earlier. We emphasize that the previous rational normal form research has only focused on Hamiltonian systems. However, the transformed system \eqref{formdK} is no longer a Hamiltonian system.
%
%
%
Instead, we find that the transformed system retains the structure of a reversible vector field. 
Concretely, for the transformed vector field of   \eqref{form11}, the homogeneous polynomial vector field of order $2l+1$ is of the form 
\begin{equation}
\label{form120}
X=\sum_{a\in\mb{N}_*}\big(X^{(z_a)}\pa_{z_a}+ \overline{X^{(z_a)}}\pa_{\bar{z}_a}\big)
\end{equation}
with the $z_a$-component
\begin{equation}
\label{form121}
X^{(z_a)}=z_a\sum_{(\bs{b},\bs{c},\bs{d})\in\mb{N}_*^{l}}\tilde{X}^{(z_a)}_{(\bs{b},\bs{c},\bs{d})}z^2_{\bs{b}}\bar{z}^2_{\bs{c}}I_{\bs{d}}+\bar{z}_{a}\sum_{(\bs{b},\bs{c},\bs{d})\in\mb{N}_*^{l}}\tilde{X}^{(\bar{z}_a)}_{(\bs{b},\bs{c},\bs{d})}z^2_{\bs{b}}\bar{z}^2_{\bs{c}}I_{\bs{d}},
\end{equation}
where $z_{\bs{b}}=z_{b_1}\cdots z_{b_p}$, $\bar{z}_{\bs{c}}=\bar{z}_{b_1}\cdots \bar{z}_{b_q}$, $I_{\bs{d}}=I_{d_1}\cdots I_{d_k}$ with $p+q+k=l$. Especially, the coefficients $\tilde{X}^{(z_a)}_{(\bs{b},\bs{c},\bs{d})},\tilde{X}^{(\bar{z}_a)}_{(\bs{b},\bs{c},\bs{d})}$ are purely imaginary, which implies the vector field \eqref{form120} is reversible. 
In order to obtain a longer time scale of existence and stability,  we build rational normal form for infinite dimensional reversible vector fields. In the following, we elaborate on some significant differences between reversible vector fields and Hamiltonian functions in the context of rational normal form.

{\bf{(i) Solving homological equation.}} Due to the distinct operations of commutators and Poisson brackets, vector fields and Hamiltonian functions present a significant difference in solving homological equations, i.e., solving homological equations for vector fields generates a new class of terms that do not arise when solving the homological equations for Hamiltonian functions. The new terms are non-integrable and even of lower order. 

Precisely, the integrable polynomial vector field $Z_3+Z_5$ is employed to handle the resonant rational vector field $K_{2l+1}$ of higher order. 
As usual, we wish to find a rational vector field $\chi$ such that 
\begin{equation}
\label{form127}
[Z_3+Z_5,\chi]+K_{2l+1}=Z_{2l+1},
\end{equation}
where $Z_{2l+1}$ is an integrable rational vector field of order $2l+1$. 
However, for reversible vector fields,
we encounter the following homological equation
\begin{equation}
\label{form136}
[Z_3+Z_5,\chi]+K_{2l+1}=Z_{2l+1}+\tilde{Z}_{2l+1}+\tilde{Z}_{2l-1}
\end{equation}
with non-integrable rational vector fields $\tilde{Z}_{2l+1}$ and $\tilde{Z}_{2l-1}$,  where 
%
 $\tilde{Z}_{2l+1}$ is of the same order as $K_{2l+1}$, and $\tilde{Z}_{2l-1}$ is even two orders lower than $K_{2l+1}$.
This indicates that a lower order non-integrable term is generated when using $Z_3+Z_5$ to eliminate higher order rational vector fields.
More precisely, the $z_a$-component of $\tilde{Z}_{2l+1}$ consists of the terms  
\begin{equation}
\label{form130-4-22}
z_aI_dDI_a[\chi]\quad\text{and}\quad z_aI_dDI_d[\chi],
\end{equation}
and the $z_a$-component of $\tilde{Z}_{2l-1}$ is of the form
\begin{equation}
\label{form130}
z_aDI_a[\chi],
\end{equation}
where  
\begin{equation}
\label{form131}
DI_a[\chi]:=\bar{z}_a\chi^{(z_a)}+z_a\overline{\chi^{(z_a)}}.
\end{equation}
Notably, the term $DI_a[\chi]$ is the sum of two conjugate functions and does not only depend on the actions, which means $\tilde{Z}_{2l+1}$ and $\tilde{Z}_{2l-1}$ are non-integrable. 
Therefore,  a new definition and a different process of rational normal form are required. Moreover, during the iterative process, the lower order non-integrable terms also lead to  a new difficulty of estimating the number of small divisors in rational vector fields.

{\bf{(ii)  The definition of rational normal form.}} 
Due to the generated terms $\tilde{Z}_{2l+1}$ and $\tilde{Z}_{2l-1}$ in the homological equation \eqref{form136}, the definition of rational normal form differs from that in previous papers, i.e., it is not merely composed of integrable vector fields.
Precisely, consider resonant rational functions consisted of monomials of the form 
\begin{equation}
\label{form122}
\tilde{X}^{(z_a)}_{(\bs{b},\bs{c},\bs{d},\bs{h})}\frac{z_az^2_{\bs{b}}\bar{z}^2_{\bs{c}}I_{\bs{d}}}{\prod\limits_{m=1}^{\beta}\Omega_{\bs{h}_{m}}(I)}\quad\text{and}\quad \tilde{X}^{(\bar{z}_a)}_{(\bs{b},\bs{c},\bs{d},\bs{h})}\frac{\bar{z}_az^2_{\bs{b}}\bar{z}^2_{\bs{c}}I_{\bs{d}}}{\prod\limits_{m=1}^{\beta}\Omega_{\bs{h}_{m}}(I)}
\end{equation}
with $(\bs{b},\bs{c},\bs{d})\in\mb{N}_*^{\alpha}$ and $0\leq\beta<\alpha$, where $\bs{h}_{m}$ is an integer vector, and $\Omega_{\bs{h}_{m}}$ is a small divisor for each $m=1,\cdots,\beta$. 
%
Then we define the rational normal form by the coefficient conditions 
$$\tilde{X}^{(z_a)}_{(\bs{b},\bs{c},\bs{d},\bs{h})}=\tilde{X}^{(z_a)}_{(\bs{c},\bs{b},\bs{d},\bs{h})}\quad\text{and}\quad\tilde{X}^{(\bar{z}_a)}_{(\bs{b},\bs{c},\bs{d},\bs{h})}=\tilde{X}^{(\bar{z}_a)}_{(\bs{c},\bs{b},\bs{d},\bs{h})}.$$
Notice that in the homological equation \eqref{form136}, the terms $\tilde{Z}_{2l+1}$ and $\tilde{Z}_{2l-1}$  could be regarded as  rational normal forms.
Importantly, the rational normal forms have no effect on the action $I_a$ for any $a\in\mb{N}_*$.
Nevertheless, these terms still cause a difference in obtaining the long time stability result, i.e., although the final vector field is bounded, we must perform energy estimation rather than  direct vector field estimation. 

{\bf{(iii)  The control of small divisors.}} To construct rational vector fields, we propose an  exact global control condition of small divisors and prove its preservation in normal form process. In addition, the generated term $\tilde{Z}_{2l-1}$ in \eqref{form136} complicates the control of the number of small divisors in rational vector fields. We will illustrate these in detail below.

The small divisor $\Omega_{\bs{h}_{m}}(I)$ in \eqref{form122} is controlled by $\kappa_{\bs{h}_{m}}$, which could be roughly viewed as the smallest index of $\bs{h}_{m}$.  
Introduce the general Hilbert space 
$$\mr{h}_{\mr{w}}:=\{u=\{u_a\}_{a\in\mb{N}_*}\mid |u|_{\mr{w}}^{2}:=\sum_{a\in\mb{N}_*}\mr{w}_a^{2}|u_a|^{2}<+\infty\},$$
where the weight $\mr{w}=\{\mr{w}_a\}_{a\in\mb{N}_*}$ is  a real sequence satisfying $1\leq\mr{w}_a\leq\mr{w}_{a'}$ for $a\leq a'$.
Then it very roughly holds that
\begin{equation*}
|\Omega_{\bs{h}_{m}}(I)|>\mr{w}_{\kappa_{\bs{h}_{m}}}^{-1}\quad\text{and}\quad |z_a|<\mr{w}_a^{-1}.
\end{equation*}
In \cite{LX24}, two global conditions were established to exactly control small divisors in rational Hamiltonian functions, and these two  conditions could not be separately kept in the normal form process.
%
Instead, for reversible rational vector fields, we propose the following condition to globally control small divisors in rational monomials \eqref{form122}:
\begin{equation}
\label{form125}
\prod_{m=1}^{\beta}\mr{w}_{\kappa_{\bs{h}_{m}}}\leq\Big(\prod_{m=1}^{p}\mr{w}_{b_m}\Big)\Big(\prod_{m=1}^{q}\mr{w}_{c_m}\Big)\Big(\prod_{m=1}^{k}\mr{w}_{d_m}\Big).
\end{equation}
Specially, in Sobolev spaces, the weight ${\rm w}_a=a^{s}$ and the control condition is
$$\prod_{m=1}^{\beta}\kappa_{\bs{h}_{m}}\leq\Big(\prod_{m=1}^{p}b_m\Big)\Big(\prod_{m=1}^{q}c_m\Big)\Big(\prod_{m=1}^{k}d_m\Big);$$
and in Gevrey or analytic spaces, the weight ${\rm w}_a=e^{\rho a^{\theta}}$ and the control condition is
$$\sum_{m=1}^{\beta}\kappa^{\theta}_{\bs{h}_{m}}\leq\sum_{m=1}^{p}b_m^{\theta}+\sum_{m=1}^{q}c_m^{\theta}+\sum_{m=1}^{k}d_m^{\theta}.$$ 
The condition \eqref{form125} is exactly adequate to bound rational vector fields. Moreover, it could be solely preserved for the commutator of two rational vector fields.

In addition, it is also crucial to control the number of small divisors in rational vector fields. The number significantly impacts the  time scale of existence and stability. 
%
%
As mentioned above, after solving the homological equation \eqref{form136}, a lower order rational normal form is generated. Even worse, this term newly introduces an additional small divisor. This brings a difficulty to estimate the number of small divisors in the iterative process.
To address this issue, we propose a novel approach, i.e., we categorize the order of rational vector fields based on iterative steps, and then control the upper bound of the number of small divisors in each segment.

{\bf{(iv)  The rational normal form process.}} The first two steps of rational normal form process differ significantly from those in the previous papers. 
Typically, the cubic integrable vector field $Z_3$ is used to deal with the quintic resonant term, and after one step of rational normal form process, the integrable polynomial vector field $Z_3+Z_5$ is introduced to handle higher order non-integrable vector fields.
However, a quintic non-integrable term $\tilde{Z}_5$ is newly generated after solving the homological equation.  Despite $\tilde{Z}_5$ being a rational normal form term,  it must be  eliminated to ensure the subsequent iterative process.  To achieve this, we divide the process into two sub-steps to construct the solution of the homological equation in the first step. 
Moreover,  if we address the septic vector field via the homological equation \eqref{form136}, then the quintic rational normal form term $\tilde{Z}_5$ will re-emerge. Thus we still solely use the cubic integrable vector field $Z_3$ to eliminate the septic vector field in the second step. We will illustrate these
in detail below.

The homological equation only associated with $Z_3$ is of the form
\begin{equation}
\label{form129}
[Z_3,\chi]+K_{2l+1}=Z_{2l+1}+\tilde{Z}_{2l+1},
\end{equation}
where $\tilde{Z}_{2l+1}$ is a newly generated rational vector field of the same order as $K_{2l+1}$. It is different from the homological equation \eqref{form136} associated with $Z_3+Z_5$.
%
%
%
%
In the first step, the homological equation \eqref{form129} is 
\begin{equation}
\label{form128}
[Z_3,\chi_3]+K_{5}=Z_{5}+\tilde{Z}_{5}.
\end{equation}
%
%
To eliminate $\tilde{Z}_5$, we must solve another homological equation
\begin{equation}
\label{form133}
[Z_3,M_3]+\tilde{Z}_5=\tilde{\tilde{Z}}_5,
\end{equation}
in which another term $\tilde{\tilde{Z}}_{5}$ is regenerated. This seems to be an endless process.
%
%
Fortunately, we find that the solution $M_3$ satisfies
\begin{equation}
\label{form134}
DI_a[M_3]=0,
\end{equation}
which implies $\tilde{\tilde{Z}}_{5}=0$.
In fact, we have 
$$M_3=\sum_{a\in\mb{N}_*}(M^{(z_a)}\pa_{z_a}+\overline{M^{(z_a)}}\pa_{\bar{z}_a})$$ 
with its $z_a$-component of the form
\begin{equation}
\label{form135}
M^{(z_a)}=z_a(M^a-\overline{M^a}),
\end{equation}
which is the difference of two conjugate functions multiplied by $z_a$.
%
In view of the structure \eqref{form131}, the solution $M_3$  satisfies \eqref{form134}.
Combining \eqref{form128} and \eqref{form133} with $\tilde{\tilde{Z}}_5=0$, it is equivalent to seek a  modified solution $\chi_3+M_3$ satisfying the homological equation
\begin{equation}
\label{form132}
[Z_3,\chi_3+M_3]+K_5=Z_5.
\end{equation}
Remark that two additional small divisors are newly introduced into the denominators of $M_3$. 
Therefore, the method of constructing the modified solutions is not universally applicable to eliminating higher order rational normal form terms, i.e., for $l\geq3$, we could not solve the following homological equation
$$[Z_3+Z_5,\chi+M]+K_{2l+1}=Z_{2l+1}\quad\text{or}\quad[Z_3,\chi+M]+K_{2l+1}=Z_{2l+1}$$
instead of \eqref{form136} or \eqref{form129}.
%

%
%
%
Moreover, we need make sure that the solution of the homological equation \eqref{form129} for $K_7$ meets the control condition \eqref{form125}.
Luckily, we find that after the first step of rational normal form process,  the septic vector field $K_7$ satisfies a stronger condition than the control condition \eqref{form125}, namely
\begin{equation}
\label{form138}
\prod_{m=1}^{\beta}\mr{w}_{\kappa_{\bs{h}_{m}}}\leq\frac{\Big(\prod_{m=1}^{p}\mr{w}_{b_m}\Big)\Big(\prod_{m=1}^{q}\mr{w}_{c_m}\Big)\Big(\prod_{m=1}^{k}\mr{w}_{d_m}\Big)}{\mr{w}_{\mu_1(\bs{b},\bs{c},\bs{d})}}.
\end{equation}
With this help, we show that after solving the homological equation \eqref{form129} for $K_7$, the control condition \eqref{form125} is satisfied.
In fact, it is precisely sufficient to perform two steps of rational normal form process by only using the integrable polynomial vector field $Z_3$ to solve homological equations.

We emphasize that the above four key points of rational normal form are not only valid for the Kirchhoff equation \eqref{form11}, but also applicable to more general  reversible systems. 
In addition, we also mention some results concerning the existence of periodic and quasi-periodic solutions for Kirchhoff equations with an external force or a Fourier multiplier, seeing \cite{Baldi09,M17,CM18,CG22,CG24} for example. We will further investigate the Kirchhoff equation \eqref{form11} via KAM theory in our next paper.

 Now we lay out an outline:
 \Cref{sec2}-\Cref{sec6} focus on proving the long time stability for the Kirchhoff equation \eqref{form11} in Sobolev spaces. We emphasize that for all
constants in the normal form process, the dependence on the iterative step $r$ is concretely calculated.  In \Cref{sec7}, the process is extended to Gevrey and analytic spaces. By taking $r$ appropriately large depending on $\varepsilon$, we achieve a longer time scale of existence and stability.
\\\indent \Cref{sec2} contains notations and resonant normal form theorem. In \Cref{sec21}, we define homogeneous polynomial vector fields based on the structure of the transformed Kirchhoff equation \eqref{form218-12-1}. We distinguish reversible and anti-reversible vector fields, with the latter used to characterize solutions of homological equations. In addition, the estimates of polynomial vector fields and their commutator are also provided.
In \Cref{sec22}, the system \eqref{form218-12-1} is associated with a vector field \eqref{form218}. Then for this vector field, we prove the resonant normal form theorem, seeing \Cref{th31}. In the theorem, the expression of the cubic integrable polynomial vector field $Z_3$ and the quintic resonant polynomial vector field $K_5$ are explicitly given.
\\\indent In \Cref{sec3}, we propose two suitable
small divisor conditions and define rational vector fields. In \Cref{sec31}, we truncate the resonant polynomial vector fields in \Cref{th31}. 
%
Then according to the integrable polynomial vector fields $Z_3$ and $Z_5$, we introduce the non-resonant set $\mc{U}^{N}_{\gamma}$, and demonstrate the preservation of the small divisor conditions under perturbations, seeing \Cref{le41}.  
In \Cref{sec32}, we define rational vector fields with the global condition \eqref{form416} to exactly control the small divisors in Sobolev spaces. Then we estimate rational vector fields, seeing \Cref{le42}, where the number of small divisors and the order of numerators are critical factors in these estimates.
In \Cref{sec33}, we estimate the commutator of reversible and anti-reversible rational vector fields, seeing \Cref{le43}. As a supplement, the increase in the number of small divisors is indicated in \Cref{re42}.
\\\indent In \Cref{sec4}, we firstly define the rational normal form, 
seeing \Cref{def43}. The rational normal form has no effect on the actions, seeing \Cref{re41-12}.
Then we solve three types of homological equations. The first type is related to the integrable polynomial vector field $Z_3+Z_5$, seeing \Cref{le44}. It is used to eliminate non-normal form parts of higher order resonant rational vector fields.
%
The other two types of homological equations are associated solely with the cubic integrable polynomial vector field $Z_3$, seeing \Cref{le51} and \Cref{le52}. They are respectively used to eliminate non-integrable parts of $K_5$ and non-normal form parts of $K_7$. 
%
Moreover, we indicate the increase in the number of small divisors after solving homological equations, seeing \Cref{re43}.
\\\indent \Cref{sec5} presents two rational normal form theorems. In \Cref{sec51}, we eliminate non-integrable parts of quintic terms and non-normal form parts of septic terms by using the integrable polynomial vector field $Z_3$, and thus get \Cref{th51} by two steps of rational normal form process. In \Cref{sec52}, we eliminate higher order resonant truncated rational vector fields by $Z_{3}+Z_{5}$, and thus get \Cref{th52} by arbitrary finite steps of rational normal form process. Especially, in the iterative lemma, a novel approach is  proposed to estimate the number of small divisors in rational
vector fields, seeing \eqref{form596} in \Cref{le53}
\\\indent In \Cref{sec6}, firstly, we estimate the measure of the non-resonant set $\mc{U}^{N}_{\gamma}$, seeing \Cref{le61}. Then combining with the above three normal form theorems, we complete the proof of \Cref{th11}.
\\\indent \Cref{sec7} is devoted to proving \Cref{th12}. In \Cref{sec71}, for Gevrey and analytic spaces, we propose two non-resonant conditions \eqref{form82} and \eqref{form83} instead of \eqref{form49} and \eqref{form410}, and then establish the corresponding rational framework. 
%
In \Cref{sec72}, we give the normal form theorem for Kirchhoff equation \eqref{form11} in Gevrey and analytic spaces, seeing \Cref{th71}. 
%
In \Cref{sec73}, we estimate the measure in \Cref{le71},  and then achieve the sub-exponentially long time stability result by selecting the iterative step $r$ and the truncation parameter $N$.
\section{Resonant normal form}
\label{sec2}
Recall the original system \eqref{form12}, and then we rewrite this system as a discrete infinite dimensional system. 
For any $a\in\mb{N}_*$, let
\begin{equation}
\label{form21}
\psi_a=\frac{a^{\frac{1}{2}}u_a+{\rm i}a^{-\frac{1}{2}}v_a}{\sqrt{2}},\quad \bar{\psi}_a=\frac{a^{\frac{1}{2}}u_a-{\rm i}a^{-\frac{1}{2}}v_a}{\sqrt{2}},
\end{equation}
and then the system \eqref{form12} becomes 
\begin{equation}
\label{form23'}
\left\{\begin{aligned}
  & \pa_t\psi_a=-{\rm i}a\psi_a-{\rm i}aQ(\psi,\bar{\psi})(\psi_a+\bar{\psi}_a),
  \\& \pa_t\bar{\psi}_a={\rm i}a\bar{\psi}_a+{\rm i}aQ(\psi,\bar{\psi})(\psi_a+\bar{\psi}_a),
\end{aligned}\right.
\end{equation}
where $$Q(\psi,\bar{\psi})=\frac{1}{4}\sum_{b\in\mb{N}_*}b\big|\psi_b+\bar{\psi}_b\big|^2=\frac{1}{2}\|u\|_{1}^2.$$ 
%
As a result of Lemma 3.1 in \cite{BH20}, there exists  a nonlinear bounded transformation $(\psi_a,\bar{\psi}_a)_{a\in\mb{N}_*}\mapsto(\eta_a,\bar{\eta}_a)_{a\in\mb{N}_*}$ to remove the unbounded off-diagonal parts of the system \eqref{form23'}. Precisely, make the change of variables
\begin{equation}
\label{form22}
\left(\begin{array}{c} \eta_a \\ \bar{\eta}_a  \end{array}\right)=
\frac{1}{\sqrt{1-\rho^2\big(Q(\psi,\bar{\psi})\big)}}
\left(\begin{array}{cc}1 & \rho\big(Q(\psi,\bar{\psi})\big) \\\rho\big(Q(\psi,\bar{\psi})\big) & 1 \end{array}\right)
\left(\begin{array}{c} \psi_a \\ \bar{\psi}_a  \end{array}\right)
\end{equation}
with $$\rho(x)=\frac{x}{1+x+\sqrt{1+2x}},$$
and then the system \eqref{form23'} becomes
\begin{equation}
\label{form23''}
\left\{\begin{aligned}
  & \pa_t\eta_a=-{\rm i}\sqrt{1+2\varphi\big(Q(\eta,\bar{\eta})\big)}a\eta_a-{\rm i}\frac{\sum_{b\in\mb{N}_*}b^2(\eta_b^2-\bar{\eta}_b^2)}{4\Big(1+2\varphi\big(Q(\eta,\bar{\eta})\big)\Big)}\bar{\eta}_a,
  \\& \pa_t\bar{\eta}_a={\rm i}\sqrt{1+2\varphi\big(Q(\eta,\bar{\eta})\big)}a\bar{\eta}_a-{\rm i}\frac{\sum_{b\in\mb{N}_*}b^2(\eta_b^2-\bar{\eta}_b^2)}{4\Big(1+2\varphi\big(Q(\eta,\bar{\eta})\big)\Big)}\eta_a,
\end{aligned}\right.
\end{equation}
where 
$$Q(\eta,\bar{\eta})=\frac{1}{4}\sum_{b\in\mb{N}_*}b\big|\eta_b+\bar{\eta}_b\big|^2=Q(\psi,\bar{\psi})\sqrt{1+2Q(\psi,\bar{\psi})},$$
and $\varphi$ is the inverse of the real map $x\mapsto x\sqrt{1+2x},\; x\geq0$.

Reparameterizing the time variable by 
\begin{equation}
\label{form23-12}
\frac{d\tau}{dt}=\sqrt{1+2\varphi\big(Q(\eta,\bar{\eta})\big)},
\end{equation}
the system \eqref{form23''} becomes
\begin{equation}
\label{form23}
\left\{\begin{aligned}
  &\pa_{\tau}\eta_a=-{\rm i}a\eta_a-\frac{{\rm i}}{4}\Big(1+2\varphi\big(Q(\eta,\bar{\eta})\big)\Big)^{-\frac{3}{2}}\Big(\sum_{a\in\mb{N}_*}b^2(\eta_b^2-\bar{\eta}_b^2)\Big)\bar{\eta}_a,
  \\& \pa_{\tau}\bar{\eta}_a=\;\,\;{\rm i}a\bar{\eta}_a-\frac{{\rm i}}{4}\Big(1+2\varphi\big(Q(\eta,\bar{\eta})\big)\Big)^{-\frac{3}{2}}\Big(\sum_{a\in\mb{N}_*}b^2(\eta_b^2-\bar{\eta}_b^2)\Big)\eta_a.
\end{aligned}\right.
\end{equation}
%
For convenience, introduce the new coordinate $z=\{z_{a}\}_{a\in\mb{N}_*}$ by letting 
\begin{equation}
\label{form24}
\eta_a=\frac{z_a}{a},
\end{equation}
, and then the system  \eqref{form23} becomes
\begin{equation}
\label{form218-12-1}
\left\{\begin{aligned}
  &\pa_{\tau}z_a=-{\rm i}az_a-\frac{{\rm i}}{4}\Big(1+2\varphi\big(y(z,\bar{z})\big)\Big)^{-\frac{3}{2}}\Big(\sum_{b\in\mb{N}_*}(z_b^2-\bar{z}_b^2)\Big)\bar{z}_a,
  \\& \pa_{\tau}\bar{z}_a=\;\,\;{\rm i}a\bar{z}_a-\frac{{\rm i}}{4}\Big(1+2\varphi\big(y(z,\bar{z})\big)\Big)^{-\frac{3}{2}}\Big(\sum_{b\in\mb{N}_*}(z_b^2-\bar{z}_b^2)\Big)z_a,
\end{aligned}\right.
\end{equation}
where 
$$y(z,\bar{z})=\frac{1}{4}\sum_{b\in\mb{N}_*}\frac{\big|z_b+\bar{z}_b\big|^2}{b}=\frac{1}{4}\sum_{b\in\mb{N}_*}\frac{z_b^2+\bar{z}_b^2+2I_b}{b}.$$

Next in this section, we firstly define homogeneous polynomial vector fields based on the structure of the Kirchhoff equation. Then we prove a resonant normal form theorem.

\subsection{Polynomial vector fields}
\label{sec21}
For any $s\geq0$, introduce the Sobolev space 
$$\ell_{s}^2:=\{z=\{z_a\}_{a\in\mb{N}_*}\in\mb{C}^{\mb{N}_*}\mid \|z\|_{s}:=\big(\sum_{a\in \mb{N}_*}a^{2s}|z_a|^2\big)^{\frac{1}{2}}<+\infty\},$$
and the phase space 
$$\ms{P}_{s}:=\ell_{s}^2\times \ell_{s}^2\ni\bs{z}:= (z,\bar{z})= (\{z_{a}\}_{a\in\mb{N}_*},\{\bar{z}_{a}\}_{a\in\mb{N}_*}).$$
Write $\bar{\bs{z}}:=(\bar{z},z)$, and denote the open ball
$$B_s(\varepsilon):=\{\bs{z}\in\ms{P}_{s}\mid \|\bs{z}\|^2_{s}:=\|z\|^2_{s}+\|\bar{z}\|^2_{s}<\varepsilon^2\}.$$
Consider the vector field of the form 
$$X(\bs{z})=\big(X^{(z)}(\bs{z}), X^{(\bar{z})}(\bs{z})\big)=\big(\{X^{(z_a)}(\bs{z})\}_{a\in\mb{N}_*}, \{X^{(\bar{z}_a)}(\bs{z})\}_{a\in\mb{N}_*}\big)\in\ms{P}_{s},$$ 
and we also use the differential geometry notation 
\begin{align}
\label{form26}
X(\bs{z})=X^{(\bs{z})}(\bs{z})\pa_{\bs{z}}=&X^{(z)}(\bs{z})\pa_{z}+ X^{(\bar{z})}(\bs{z})\pa_{\bar{z}}
\\\notag=&\sum_{a\in\mb{N}_*}\big(X^{(z_a)}(\bs{z})\pa_{z_a}+ X^{(\bar{z}_a)}(\bs{z})\pa_{\bar{z}_a}\big).
\end{align}

\indent In a brief statement, we use the convenient notation $\zeta=(\zeta_{j})_{j=(\delta,a)\in\mb{U}_{3}\times\mb{N}_*}$ with $\mb{U}_{3}=\{\pm1,0\}$, where 
\begin{equation}
\label{form3-23-notation}
\zeta_{j}=\left\{\begin{aligned}&z^2_{a},\qquad\; \qquad\text{when}\;\delta=1,
\\&\bar{z}^2_{a}, \qquad\;\qquad\text{when}\;\delta=-1,
\\&I_a:=|z_a|^2, \quad\text{when}\;\delta=0.
\end{aligned}\right.
\end{equation}
Let $\bar{j}=(-\delta,a)$, then $\bar{\zeta}_{j}=\zeta_{\bar{j}}$. Set $|j|=a$ and define 
\begin{equation}
\label{formzetanorm}
\|\zeta\|_s:=\sum_{j\in\mb{U}_{3}\times\mb{N}_*}|j|^{2s}|\zeta_{j}|=3\|z\|_s^2.
\end{equation}

For $\bs{j}=(j_{1},\cdots,j_{l})\in(\mb{U}_{3}\times\mb{N}_*)^{l}$, denote the monomial $\zeta_{\bs{j}}=\zeta_{j_{1}}\cdots \zeta_{j_{l}}$.
Denote by $\ms{M}_{2l+1}$ the set of homogeneous polynomial vector fields in \eqref{form26} of order $2l+1$ with each component 
\begin{equation}
\label{form25}
X^{(z_a)}(\bs{z})=z_a\sum_{\bs{j}\in(\mb{U}_{3}\times\mb{N}_*)^{l}}\tilde{X}^{(z_a, z_a)}_{\bs{j}}\zeta_{\bs{j}}+\bar{z}_{a}\sum_{\bs{j}\in(\mb{U}_{3}\times\mb{N}_*)^{l}}\tilde{X}^{(z_a, \bar{z}_a)}_{\bs{j}}\zeta_{\bs{j}},
\end{equation}
\begin{equation}
\label{form25-5}
X^{(\bar{z}_a)}(\bs{z})=\bar{z}_a\sum_{\bs{j}\in(\mb{U}_{3}\times\mb{N}_*)^{l}}\tilde{X}^{(\bar{z}_a, \bar{z}_a)}_{\bs{j}}\zeta_{\bar{\bs{j}}}+z_{a}\sum_{\bs{j}\in(\mb{U}_{3}\times\mb{N}_*)^{l}}\tilde{X}^{(\bar{z}_a, z_a)}_{\bs{j}}\zeta_{\bar{\bs{j}}},
\end{equation}
where the coefficients $\tilde{X}^{(z_a, z_a)}_{\bs{j}},\tilde{X}^{(z_a, \bar{z}_a)}_{\bs{j}}, \tilde{X}^{(\bar{z}_a, \bar{z}_a)}_{\bs{j}},\tilde{X}^{(\bar{z}_a, z_a)}_{\bs{j}}\in\mb{C}$ satisfy
\begin{equation}
\tilde{X}^{(z_a, z_a)}_{\bs{j}}=\overline{\tilde{X}^{(\bar{z}_a, \bar{z}_a)}_{\bs{j}}},\quad \tilde{X}^{(z_a, \bar{z}_a)}_{\bs{j}}=\overline{\tilde{X}^{(\bar{z}_a, z_a)}_{\bs{j}}},
\end{equation}
\begin{equation}
\label{defnorm}
\|X\|_{\ell^{\infty}}:=\sup_{\substack{a\in\mb{N}_*\\\bs{j}\in(\mb{U}_{3}\times\mb{N}_*)^{l}}}\Big\{\big|\tilde{X}^{(z_a, z_a)}_{\bs{j}}\big|,\big|\tilde{X}^{(z_a, \bar{z}_a)}_{\bs{j}}\big|\Big\}<+\infty,
\end{equation}
and for any permutation $\sigma$, one has
\begin{equation*}
\tilde{X}^{(z_a, z_a)}_{\sigma(\bs{j})}=\tilde{X}^{(z_a, z_a)}_{\bs{j}},
\quad\tilde{X}^{(z_a, \bar{z}_a)}_{\sigma(\bs{j})}=\tilde{X}^{(z_a, \bar{z}_a)}_{\bs{j}}.
\end{equation*} 
Especially, if the coefficients in \eqref{form25} and \eqref{form25-5} satisfy the reversible condition
\begin{equation}
\label{form26-12-1}
\tilde{X}^{(z_a, z_a)}_{\bs{j}}=-\tilde{X}^{(\bar{z}_a, \bar{z}_a)}_{\bs{j}},\quad \tilde{X}^{(z_a, \bar{z}_a)}_{\bs{j}}=-\tilde{X}^{(\bar{z}_a, z_a)}_{\bs{j}},
\end{equation}
then we say that the above polynomial vector field $X(\bs{z})$ belongs to the set $\ms{M}^{{\rm rev}}_{2l+1}$.
%
%
If the coefficients \eqref{form25} and \eqref{form25-5} satisfy the anti-reversible condition
\begin{equation}
\label{form27-12-1}
\tilde{X}^{(z_a, z_a)}_{\bs{j}}=\tilde{X}^{(\bar{z}_a, \bar{z}_a)}_{\bs{j}},\quad \tilde{X}^{(z_a, \bar{z}_a)}_{\bs{j}}=\tilde{X}^{(\bar{z}_a, z_a)}_{\bs{j}},
\end{equation}
then we say that the above polynomial vector field $X(\bs{z})$ belongs to the set $\ms{M}^{{\rm an-rev}}_{2l+1}$.

Remark that for any $X(\bs{z})\in\ms{M}_{2l+1}$, one has 
$\overline{X^{(z_a)}(\bs{z})}=X^{(\bar{z}_a)}(\bs{z})$, and the $\bar{z}_a$-component $X^{(\bar{z}_a)}(\bs{z})$ is uniquely determined by  the $z_a$-component $X^{(z_a)}(\bs{z})$. Especially, the coefficients of $X(\bs{z})\in \ms{M}^{{\rm rev}}_{2l+1}$ are purely imaginary, and the coefficients of $X(\bs{z})\in \ms{M}^{{\rm an-rev}}_{2l+1}$ are real.
Then these vector fields are estimated in the following lemma.

\begin{lemma}
\label{le21-1}
For $s\geq0$ and $X(\bs{z})\in\ms{M}_{2l+1}$, we have
\begin{equation}
\label{form27}
\|X(\bs{z})\|_{s}<3^{l+1}\|X\|_{\ell^{\infty}}\|z\|_s\|z\|_{0}^{2l}.
\end{equation}
\end{lemma}
\begin{proof}
In view of \eqref{form26}, one has
\begin{equation}
\label{form27-1}
\|X(\bs{z})\|_{s}=\sqrt{\|X^{(z)}(\bs{z})\|^2_{s}+\|X^{(\bar{z})}(\bs{z})\|^2_{s}}=\sqrt{2}\|X^{(z)}(\bs{z})\|_{s}.
\end{equation}
In view of \eqref{form25} and the definition of norm \eqref{defnorm}, one has
\begin{align}
\label{form27-2}
&\|X^{(z)}(\bs{z})\|_{s}=\Big(\sum_{a\in \mb{N}_*}a^{2s}\big|X^{(z_a)}\big|^2\Big)^{\frac{1}{2}}
\\\notag\leq&\|X\|_{\ell^{\infty}}\Big(\sum_{a\in \mb{N}_*}a^{2s}\Big||\bar{z}_a|\sum_{\bs{j}\in(\mb{U}_{3}\times\mb{N}_*)^{l}}|\zeta_{\bs{j}}|+|z_{a}|\sum_{\bs{j}\in(\mb{U}_{3}\times\mb{N}_*)^{l}}|\zeta_{\bs{j}}|\Big|^2\Big)^{\frac{1}{2}}
\\\notag\leq&\|X\|_{\ell^{\infty}}\Big(\sum_{a\in \mb{N}_*}a^{2s}\big(|\bar{z}_a|+|z_{a}|\big)^2\Big)^{\frac{1}{2}}\sum_{\bs{j}\in(\mb{U}_{3}\times\mb{N}_*)^{l}}|\zeta_{\bs{j}}|
\\\notag=&2\|X\|_{\ell^{\infty}}\|z\|_s\|\zeta\|_{0}^{l}.
\end{align}
By  \eqref{formzetanorm}, \eqref{form27-1} and \eqref{form27-2}, one has
\begin{align*}
\|X(\bs{z})\|_{s}\leq&2\sqrt{2}\|X\|_{\ell^{\infty}}\|z\|_s(3\|z\|_0^2)^{l}
\\<&3^{l+1}\|X\|_{\ell^{\infty}}\|z\|_s\|z\|_{0}^{2l},
\end{align*}
where the last inequality follows from the fact $2\sqrt{2}\times3^{l}<3^{l+1}$.
\end{proof}

Next, we define resonant polynomial vector fields. For $l\in\mb{N}_*$ and $b\in\mb{Z}$, denote
\begin{align*}
\mc{R}^b_{l}:=&\{\bs{j}=(\delta_k,a_k)_{k=1}^l\in(\mb{U}_{3}\times\mb{N}_*)^{l}\mid \Delta_{\bs{j}}:=\sum_{k=1}^l\delta_ka_{k}=b\}.
\end{align*}
Write $\mc{R}^{b}:=\bigcup_{l\geq1}\mc{R}_{l}^{b}$. Similarly with $\ms{M}_{2l+1}$, 
denote by $\ms{R}_{2l+1}$ the set of resonant polynomial vector fields of order $2l+1$ with $\mc{R}^0_{l}$ or $\mc{R}^a_{l}$ instead of $(\mb{U}_{3}\times\mb{N}_*)^{l}$ in \eqref{form25} and \eqref{form25-5}.
Concretely, a vector field $K(\bs{z})\in\ms{R}_{2l+1}$ is of the form
\begin{equation}
\label{form31} 
K(\bs{z})=\sum_{a\in\mb{N}_*}\big(K^{(z_a)}(\bs{z})\pa_{z_a}+ K^{(\bar{z}_a)}(\bs{z})\pa_{\bar{z}_a}\big)
\end{equation}
with each component
\begin{align}
\label{form32}
&K^{(z_a)}(\bs{z})=z_a\sum_{\bs{j}\in\mc{R}^0_{l}}\tilde{K}^{(z_a, z_a)}_{\bs{j}}\zeta_{\bs{j}}+\bar{z}_{a}\sum_{\bs{j}\in\mc{R}^a_{l}}\tilde{K}^{(z_a, \bar{z}_a)}_{\bs{j}}\zeta_{\bs{j}},
\\\label{form32-4-30}&K^{(\bar{z}_a)}(\bs{z})=\bar{z}_a\sum_{\bs{j}\in\mc{R}^0_{l}}\tilde{K}^{(\bar{z}_a, \bar{z}_a)}_{\bs{j}}\zeta_{\bar{\bs{j}}}+z_{a}\sum_{\bs{j}\in\mc{R}^a_{l}}\tilde{K}^{(\bar{z}_a, z_a)}_{\bs{j}}\zeta_{\bar{\bs{j}}}.
\end{align}
Especially, if the coefficients satisfy the reversible condition \eqref{form26-12-1} or anti-reversible condition \eqref{form27-12-1}, then we say $K\in\ms{R}^{{\rm rev}}_{2l+1}$ or $K\in\ms{R}^{{\rm an-rev}}_{2l+1}$, respectively.

Especially, denote the integrable index sets
\begin{equation*}
\mc{I}_{l}^{0}=\big\{\bs{j}=(\delta_k,a_k)_{k=1}^l\in\mc{R}_{l}^{0}\mid \exists\;\text{permutation}\;\sigma,\;\text{s.t.}\;\forall k,\; \delta_{k}=-\delta_{\sigma(k)},\;a_{k}=a_{\sigma(k)}\big\},
\end{equation*}
and for $a\in\mb{N}_*$,
$$\mc{I}_{l}^{a}:=\{\bs{j}\in\mc{R}^a_{l}\mid \big(\bs{j},(-1,a)\big)\in\mc{I}_{l+1}^{0}\}.$$
Write $\mc{I}^{0}:=\bigcup_{l\geq1}\mc{I}_{l}^{0}$ and $\mc{I}^{a}:=\bigcup_{l\geq1}\mc{I}_{l}^{a}$. 
Similarly with $\ms{R}_{2l+1}$, 
denote by $\ms{I}_{2l+1}$ the set of  integrable polynomial vector fields of order $2l+1$ with $\mc{I}^0_{l}$ and $\mc{I}^a_{l}$ instead of $\mc{R}^0_{l}$ and $\mc{R}^a_{l}$ in \eqref{form32}  and \eqref{form32-4-30}.
Remark that an  integrable polynomial vector field $Z(\bs{z})\in\ms{I}^{{\rm rev}}_{2l+1}$ could be rewritten as
\begin{equation}
\label{form32'}
Z(\bs{z})=\sum_{a\in\mb{N}_*}\sum_{\bs{d}\in\mb{N}_*^{l}}\tilde{Z}^{(z_a,z_a)}_{\bs{d}}I_{\bs{d}}(z_a\pa_{z_a}-\bar{z}_a\pa_{\bar{z}_a}).
\end{equation}
%
%

Finally, for two vector fields $X(\bs{z})$ and $Y(\bs{z})$, introduce their commutator 
$$[X,Y](\bs{z}):=DX(\bs{z})[Y(\bs{z})]-DY(\bs{z})[X(\bs{z})],$$ 
where the $z_a$-component is
\begin{align}
\label{form26'}
[X,Y]^{(z_a)}=&\frac{\pa X^{(z_a)}}{\pa \bs{z}}Y^{(\bs{z})}-\frac{\pa Y^{(z_a)}}{\pa \bs{z}}X^{(\bs{z})}
\\\notag=&\sum_{b\in\mb{N}_*}\Big(\frac{\pa X^{(z_a)}}{\pa z_b}Y^{(z_b)}+\frac{\pa X^{(z_a)}}{\pa \bar{z}_b}Y^{(\bar{z}_b)}-\frac{\pa Y^{(z_a)}}{\pa z_b}X^{(z_b)}-\frac{\pa Y^{(z_a)}}{\pa \bar{z}_b}X^{(\bar{z}_b)}\Big).
\end{align}
%
Given a vector field $X$, its transformed vector field under the time 1 flow generated by $\chi$ is
\begin{equation}
\label{form26-9-26}
e^{ad_{\chi}}X:=\sum_{k=0}^{+\infty}\frac{1}{k!}ad_{\chi}^kX
\end{equation}
where $ad^0_{\chi}X=X$ and  $ad_{\chi}^kX:=[ad_{\chi}^{k-1}X, \chi]$. 
Similarly with Lemma 4.12 in \cite{BG06}, we have the following result.

\begin{lemma}
\label{le22-12}
For $X\in \ms{M}^{{\rm rev}}_{2l_1+1}$ and $\chi\in \ms{M}^{{\rm an-rev}}_{2l_2+1}$, one has $[X,\chi]\in\ms{M}^{{\rm rev}}_{2(l_1+l_2)+1}$ with
\begin{equation}
\label{form211-9-25}
\|[X,\chi]\|_{\ell^{\infty}}<6(l_1+l_2+1)\|X\|_{\ell^{\infty}}\|\chi\|_{\ell^{\infty}}.
\end{equation}
\end{lemma}
\begin{proof}
In view of the definition of the commutator, one has the component 
\begin{align}
\label{formp1}
[X,\chi]^{(z_a)}(\bs{z})=&\sum_{\bs{j}'\in(\mb{U}_{3}\times\mb{N}_*)^{l_1}}\tilde{X}^{(z_a, z_a)}_{\bs{j}'}\frac{\pa(z_a\zeta_{\bs{j}'})}{\pa\bs{z}}\chi^{(\bs{z})}+\sum_{\bs{j}'\in(\mb{U}_{3}\times\mb{N}_*)^{l_1}}\tilde{X}^{(z_a, \bar{z}_a)}_{\bs{j}'}\frac{\pa(\bar{z}_{a}\zeta_{\bs{j}'})}{\pa\bs{z}}\chi^{(\bs{z})}
\\\notag&-\sum_{\bs{j}''\in(\mb{U}_{3}\times\mb{N}_*)^{l_2}}\tilde{\chi}^{(z_a, z_a)}_{\bs{j}''}\frac{(\pa z_a\zeta_{\bs{j}''})}{\pa\bs{z}}X^{(\bs{z})}+\sum_{\bs{j}''\in(\mb{U}_{3}\times\mb{N}_*)^{l_2}}\tilde{\chi}^{(z_a, \bar{z}_a)}_{\bs{j}''}\frac{\pa(\bar{z}_{a}\zeta_{\bs{j}''})}{\pa\bs{z}}X^{(\bs{z})}
\\\notag:=&z_a\sum_{\bs{j}\in(\mb{U}_{3}\times\mb{N}_*)^{l_1+l_2}}\tilde{Y}^{(z_a, z_a)}_{\bs{j}}\zeta_{\bs{j}}+\bar{z}_{a}\sum_{\bs{j}\in(\mb{U}_{3}\times\mb{N}_*)^{l_1+l_2}}\tilde{Y}^{(z_a, \bar{z}_a)}_{\bs{j}}\zeta_{\bs{j}},
\end{align}
\begin{align}
\label{formp2}
[X,\chi]^{(\bar{z}_a)}(\bs{z})=&\sum_{\bs{j}'\in(\mb{U}_{3}\times\mb{N}_*)^{l_1}}\tilde{X}^{(\bar{z}_a, \bar{z}_a)}_{\bs{j}'}\frac{\pa(\bar{z}_a\zeta_{\bar{\bs{j}}'})}{\pa\bs{z}}\chi^{(\bs{z})}+\sum_{\bs{j}'\in(\mb{U}_{3}\times\mb{N}_*)^{l_1}}\tilde{X}^{(\bar{z}_a, z_a)}_{\bs{j}'}\frac{\pa(z_{a}\zeta_{\bar{\bs{j}}'})}{\pa\bs{z}}\chi^{(\bs{z})}
\\\notag&-\sum_{\bs{j}''\in(\mb{U}_{3}\times\mb{N}_*)^{l_2}}\tilde{\chi}^{(\bar{z}_a, \bar{z}_a)}_{\bs{j}''}\frac{(\pa \bar{z}_a\zeta_{\bar{\bs{j}}''})}{\pa\bs{z}}X^{(\bs{z})}+\sum_{\bs{j}''\in(\mb{U}_{3}\times\mb{N}_*)^{l_2}}\tilde{\chi}^{(\bar{z}_a, z_a)}_{\bs{j}''}\frac{\pa(z_{a}\zeta_{\bar{\bs{j}}''})}{\pa\bs{z}}X^{(\bs{z})}
\\\notag:=&\bar{z}_a\sum_{\bs{j}\in(\mb{U}_{3}\times\mb{N}_*)^{l_1+l_2}}\tilde{Y}^{(\bar{z}_a, \bar{z}_a)}_{\bs{j}}\zeta_{\bar{\bs{j}}}+z_{a}\sum_{\bs{j}\in(\mb{U}_{3}\times\mb{N}_*)^{l_1+l_2}}\tilde{Y}^{(\bar{z}_a, z_a)}_{\bs{j}}\zeta_{\bar{\bs{j}}}.
\end{align}
Without loss of generality, let us consider this term $z_a\tilde{Y}^{(z_a, z_a)}_{\bs{j}}\zeta_{\bs{j}}$ in \eqref{formp1}, where $\bs{j}$ is generated by $\bs{j}'\in(\mb{U}_{3}\times\mb{N}_*)^{l_1}$ and $\bs{j}''\in(\mb{U}_{3}\times\mb{N}_*)^{l_2}$.

{\bf{Case 1}}: this term is generated by taking the partial derivative with respect to $\bs{z}_a$. There are at most four combinations to obtain $z_a$, i.e., 
$$\frac{\pa (z_a\zeta_{\bs{j}'})}{\pa z_a}z_a\zeta_{\bs{j}''}\quad\text{or}\quad\frac{\pa (\bar{z}_a\zeta_{\bs{j}'})}{\pa \bar{z}_a}z_a\zeta_{\bs{j}''}\quad\text{or}\quad\frac{\pa (z_a\zeta_{\bs{j}''})}{\pa z_a}z_a\zeta_{\bs{j}'}\quad\text{or}\quad\frac{\pa(\bar{z}_a\zeta_{\bs{j}''})}{\pa\bar{z}_a}z_a\zeta_{\bs{j}'}.$$
In this case, one has $\bs{j}=(\bs{j}',\bs{j}'')$.  

{\bf{Case 2}}: this term is generated by taking the partial derivative with respect to $\zeta$. 
%
Remark that for any $j\in\bs{j}$, there are at most three combinations to obtain the term $\zeta_{j}$. Concretely, $\zeta_{j}=I_{b}$ is generated by
$$\frac{\pa z_b^2}{\pa z_b}\bar{z}_b\quad\text{or}\quad \frac{\pa\bar{z}_b^2}{\pa \bar{z}_b}z_b\quad\text{or}\quad \frac{\pa I_b}{\pa\bs{z}_b}\bs{z}_b;$$
$\zeta_{j}=z^2_{b}$ is generated by
$$\frac{\pa z_b^2}{\pa z_b}z_b\quad\text{or}\quad \frac{\pa I_b}{\pa\bar{z}_b}z_b;$$ 
and $\zeta_{j}=\bar{z}^2_{b}$ is generated by
$$\frac{\pa \bar{z}_b^2}{\pa\bar{z}_b}\bar{z}_b\quad\text{or}\quad \bar{z}^2_{b}=\frac{\pa I_b}{\pa z_b}\bar{z}_b.$$
In this case, there are at most one different index in $\bs{j}$ with $(\bs{j}',\bs{j}'')$, and $z_a$ is well determined.

To sum up, the coefficient $\tilde{Y}^{(z_a, z_a)}_{\bs{j}}$ consists of at most $4+2\times3(l_1+l_2)$ terms of the form $\tilde{X}^{(\bs{z}_b, \bs{z}_b)}_{\bs{j'}}\tilde{\chi}^{(\bs{z}_c,\bs{z}_c)}_{\bs{j}''}$. Thus, one has 
\begin{align*}
\|[X,\chi]\|_{\ell^{\infty}}\leq&\big(4+2\times3(l_1+l_2)\big)\|X\|_{\ell^{\infty}}\|\chi\|_{\ell^{\infty}}<6(l_1+l_2+1)\|X\|_{\ell^{\infty}}\|\chi\|_{\ell^{\infty}}.
\end{align*}
Moreover, by $X\in \ms{M}^{{\rm rev}}_{2l_1+1}$ and $\chi\in \ms{M}^{{\rm an-rev}}_{2l_2+1}$,
one has 
$$\tilde{X}^{(\bs{z}_b, \bs{z}_b)}_{\bs{j'}}\tilde{\chi}^{(\bs{z}_c,\bs{z}_c)}_{\bs{j}''}=-\tilde{X}^{(\bar{\bs{z}}_b,\bar{\bs{z}}_b)}_{\bs{j'}}\tilde{\chi}^{(\bar{\bs{z}}_c,\bar{\bs{z}}_c)}_{\bs{j}''},$$
which implies $[X,\chi]\in\ms{M}^{{\rm rev}}_{2(l_1+l_2)+1}$. 
\end{proof}

\subsection{Resonant normal form theorem}
\label{sec22}
Recall the transformed system \eqref{form218-12-1}. Fix an integer $r\geq2$, and expand $f(y):=\big(1+2\varphi(y)\big)^{-\frac{3}{2}}$ up to the order $r-1$, namely
$$f(y)=f(0)+f'(0)y+\frac{f''(0)}{2!}y^2+\cdots\frac{f^{(r-1)}(0)}{(r-1)!}y^{r-1}+O(y^{r}).$$
By direct calculation, we have $f(0)=1$ and $f'(0)=-3$.
%
Then the system \eqref{form218-12-1} becomes
\begin{align}
  \label{form28}
\pa_{\tau}\bs{z}=Z_1(\bs{z})+\sum_{l=1}^{r}P_{2l+1}(\bs{z})+R_{\geq2r+3}(\bs{z}),
\end{align}
where 
\begin{equation}
\label{form213-9-24}
Z_1(\bs{z})=-{\rm i}\sum_{a\in\mb{N}_*}a(z_a\pa_{z_a}-\bar{z}_a\pa_{\bar{z}_a}),
\end{equation}
\begin{equation}
\label{form210-12-20}
P_{2l+1}(\bs{z})=-\frac{{\rm i}f^{(l-1)}(0)}{4^l(l-1)!}\left(\sum_{b\in\mb{N}_*}\frac{z_b^2+\bar{z}_b^2+2I_b}{b}\right)^{l-1}\left(\sum_{b\in\mb{N}_*}(z_b^2-\bar{z}_b^2)\right)\sum_{a\in\mb{N}_*}(\bar{z}_{a}\pa_{z_a}+z_a\pa_{\bar{z}_a}),
\end{equation}
and $R_{\geq2r+3}(\bs{z})$ is the remainder term. 
%
Especially, we have
\begin{equation}
\label{form29}
P_3(\bs{z})=-\frac{{\rm i}}{4}\left(\sum_{b\in\mb{N}_*}(z_b^2-\bar{z}_b^2)\right)\sum_{a\in\mb{N}_*}(\bar{z}_{a}\pa_{z_a}+z_a\pa_{\bar{z}_a}),
\end{equation}
\begin{equation}
\label{form210}
P_5(\bs{z})=\frac{3{\rm i}}{16}\left(\sum_{b\in\mb{N}_*}\frac{z_b^2+\bar{z}_b^2+2I_b}{b}\right)\left(\sum_{b\in\mb{N}_*}(z_b^2-\bar{z}_b^2)\right)\sum_{a\in\mb{N}_*}(\bar{z}_{a}\pa_{z_a}+z_a\pa_{\bar{z}_a}).
\end{equation}
Moreover, we have $P_{2l+1}(\bs{z})\in\ms{M}^{{\rm rev}}_{2l+1}$, and there exists a  positive constant $C_0$ such that
\begin{equation}
\label{form211}
\|P_{2l+1}\|_{\ell^{\infty}}\leq C_0^{l},
\end{equation} 
\begin{equation}
\label{form212}
\|R_{\geq2r+3}(\bs{z})\|_s\leq C_0^r\|z\|_s^{2r+3}.
\end{equation} 

For convenience, we use the vector field 
\begin{equation}
\label{form218}
X:=Z_1+\sum_{l=1}^{r}P_{2l+1}+R_{\geq2r+3}
\end{equation}
to respresent the system \eqref{form28}.
In this subsection, we will eliminate the non-resonant parts of $P_{2l+1}$ to prove the following resonant normal form theorem.

\begin{theorem}
\label{th31}
Fix an integer $r\geq2$. Denote 
$$C_1=\max\{C_0, 20\}, \quad \varepsilon_{0}=(6rC_1^{r+1})^{-\frac{1}{2}}.$$ 
For any $s\geq0$ and $0<\varepsilon\leq\varepsilon_{0}$, there exists a nearly identity coordinate transformation $\phi^{(1)}:B_{s}(3\varepsilon)\to B_{s}(4\varepsilon)$ such that the vector field \eqref{form218} is transformed into 
\begin{equation}
\label{form33}
X'= Z_1+Z_3+\sum_{l=2}^{r}K_{2l+1}+R'_{\geq2r+3},
\end{equation}
where
\begin{enumerate}[(i)]
 \item  the transformation is bound and satisfies the estimate
    \begin{equation}
    \label{form36}
    \|\bs{z}-\phi^{(1)}(\bs{z})\|_{s}<9C_1\|z\|_{s}^{3},
    \end{equation}
    and the same estimate is fulfilled by the inverse transformation;
\item $Z_3$ is an integrable  polynomial vector field with the form
    \begin{equation}
    \label{form34}
    Z_{3}(\bs{z})=-\frac{{\rm i}}{4}\sum_{a\in\mb{N}_*}I_a(z_a\pa_{z_a}-\bar{z}_a\pa_{\bar{z}_a});
    \end{equation}
\item the vector field $K_{2l+1}\in\ms{R}^{{\rm rev}}_{2l+1}$ satisfies the coefficient estimate
     \begin{equation}
    \label{form37}
    \|K_{2l+1}\|_{\ell^{\infty}}\leq C_1^{l^2};
     \end{equation}
 \item the remainder term $R'_{\geq2r+3}$ satisfies the estimate
    \begin{equation}
    \label{form38}
    \|R'_{\geq2r+3}(\bs{z})\|_{s}<(3C_1)^{r(r+1)}\|z\|_s^{2r+3}.
     \end{equation}
\end{enumerate}
 Moreover, the $z_a$-component of $K_5(\bs{z})$ is 
\begin{equation}
\label{form35}
 K^{(z_a)}_{5}(\bs{z})=Z^{(z_a)}_5(\bs{z})+{\rm i}\frac{3a}{32}\bar{z}_a\Big(\sum_{b_1+b_2=a}\frac{z_{b_1}^2z_{b_2}^2}{b_1b_2}+\sum_{b-c=a}\frac{2z_{b}^2\bar{z}_{c}^2}{bc}\Big)
\end{equation}
with its integrable part
\begin{equation}
\label{form35'} 
Z^{(z_a)}_5(\bs{z})={\rm i}z_a\left(\frac{27}{64a}I^2_a+\frac{1}{8}I_a\sum_{d\neq a}\Big(\frac{3}{d}+\frac{d}{d^2-a^2}\Big)I_d-\frac{a}{16}\sum_{d\neq a}\frac{1}{d^2-a^2}I^2_d\right).
 \end{equation}
\end{theorem}

\begin{remark}
\label{re21-12}
The above theorem for $r=2$ is contained in \cite{BH21}. Actually, they obtained the second step resonant normal form of Kirchhoff equation on $d$-dimensional torus.
\end{remark}

%
We eliminate the non-resonant parts of $P_{2l+1}$ in \eqref{form210-12-20} by using the linear vector field $Z_1$ in \eqref{form213-9-24}.
So before proving \Cref{th31}, we solve the homological equation with respect to $Z_1$, which will be used to construct the transformation $\phi^{(1)}$.
\begin{lemma}
\label{le31}
For $l\in\mb{N}_*$ and $P\in\ms{M}^{{\rm rev}}_{2l+1}$, there exist homogenous polynomial vector fields $\chi\in\ms{M}^{{\rm an-rev}}_{2l+1}$ and $K\in\ms{R}^{{\rm rev}}_{2l+1}$ 
such that
\begin{equation}
 \label{form39}
[Z_1,\chi]+P=K.
\end{equation}
Moreover, $\chi$ and $K$ fulfill the coefficient estimates
\begin{equation}
 \label{form310}
\|\chi\|_{\ell^{\infty}}\leq\frac{1}{2}\|P\|_{\ell^{\infty}}\quad\text{and}\quad\|K\|_{\ell^{\infty}}\leq\|P\|_{\ell^{\infty}}.
\end{equation}
\end{lemma}
\begin{proof}
In view of the definition of $\ms{M}_{2l+1}$, we write the components of $P$ in the form:
\begin{align*}
&P^{(z_a)}(\bs{z})=z_a\sum_{\bs{j}\in(\mb{U}_{3}\times\mb{N}_*)^{l}}\tilde{P}^{(z_a, z_a)}_{\bs{j}}\zeta_{\bs{j}}+\bar{z}_{a}\sum_{\bs{j}\in(\mb{U}_{3}\times\mb{N}_*)^{l}}\tilde{P}^{(z_a, \bar{z}_a)}_{\bs{j}}\zeta_{\bs{j}},
\\&P^{(\bar{z}_a)}(\bs{z})=\bar{z}_a\sum_{\bs{j}\in(\mb{U}_{3}\times\mb{N}_*)^{l}}\tilde{P}^{(\bar{z}_a, \bar{z}_a)}_{\bs{j}}\zeta_{\bar{\bs{j}}}+z_{a}\sum_{\bs{j}\in(\mb{U}_{3}\times\mb{N}_*)^{l}}\tilde{P}^{(\bar{z}_a, z_a)}_{\bs{j}}\zeta_{\bar{\bs{j}}}.
\end{align*}
Let $\chi(\bs{z})=\sum_{a\in\mb{N}_*}\big(\chi^{(z_a)}(\bs{z})\pa_{z_a}+ \chi^{(\bar{z}_a)}(\bs{z})\pa_{\bar{z}_a}\big)$ with
\begin{align*}
&\chi^{(z_a)}(\bs{z})=z_a\sum_{\bs{j}\in(\mb{U}_{3}\times\mb{N}_*)^{l}}\tilde{\chi}^{(z_a, z_a)}_{\bs{j}}\zeta_{\bs{j}}+\bar{z}_{a}\sum_{\bs{j}\in(\mb{U}_{3}\times\mb{N}_*)^{l}}\tilde{\chi}^{(z_a, \bar{z}_a)}_{\bs{j}}\zeta_{\bs{j}},
\\&\chi^{(\bar{z}_a)}(\bs{z})=\bar{z}_a\sum_{\bs{j}\in(\mb{U}_{3}\times\mb{N}_*)^{l}}\tilde{\chi}^{(\bar{z}_a, \bar{z}_a)}_{\bs{j}}\zeta_{\bar{\bs{j}}}+z_{a}\sum_{\bs{j}\in(\mb{U}_{3}\times\mb{N}_*)^{l}}\tilde{\chi}^{(\bar{z}_a, z_a)}_{\bs{j}}\zeta_{\bar{\bs{j}}},
\end{align*}
where
\begin{align}
\label{formchi1}
\tilde{\chi}^{(z_a, z_a)}_{\bs{j}}=&
\begin{cases}
\frac{{\rm i}\tilde{P}^{(z_a, z_a)}_{\bs{j}}}{2\Delta_{\bs{j}}},&\text{if}\;\bs{j}\notin\mc{R}^0_l, \\
0,&\text{if}\;\bs{j}\in\mc{R}^0_{l},
\end{cases}
\end{align}
\begin{align}
\label{formchi2}
\tilde{\chi}^{(z_a, \bar{z}_a)}_{\bs{j}}=&
\begin{cases}
\frac{  {\rm i}\tilde{P}^{(z_a, \bar{z}_a)}_{\bs{j}}}{2(\Delta_{\bs{j}}-a)},&\text{if}\;\bs{j}\notin\mc{R}^a_l, \\
0,&\text{if}\;\bs{j}\in\mc{R}^a_{l}.
\end{cases}
\end{align}
\begin{align}
\label{formchi3}
\tilde{\chi}^{(\bar{z}_a, \bar{z}_a)}_{\bs{j}}=&
\begin{cases}
-\frac{{\rm i}\tilde{P}^{(\bar{z}_a, \bar{z}_a)}_{\bs{j}}}{2\Delta_{\bs{j}}},&\text{if}\;\bs{j}\notin\mc{R}^0_l, \\
0,&\text{if}\;\bs{j}\in\mc{R}^0_{l},
\end{cases}
\end{align}
and 
\begin{align}
\label{formchi4}
\tilde{\chi}^{(\bar{z}_a, z_a)}_{\bs{j}}=&
\begin{cases}
-\frac{ {\rm i}\tilde{P}^{(\bar{z}_a, z_a)}_{\bs{j}}}{2(\Delta_{\bs{j}}-a)},&\text{if}\;\bs{j}\notin\mc{R}^a_l, \\
0,&\text{if}\;\bs{j}\in\mc{R}^a_{l}.
\end{cases}
\end{align}
Then the homological equation \eqref{form39} holds with 
\begin{align}
\label{form241-12-20}
&K^{(z_a)}(\bs{z})=z_a\sum_{\bs{j}\in\mc{R}^0_{l}}\tilde{P}^{(z_a, z_a)}_{\bs{j}}\zeta_{\bs{j}}+\bar{z}_{a}\sum_{\bs{j}\in\mc{R}^a_{l}}\tilde{P}^{(z_a, \bar{z}_a)}_{\bs{j}}\zeta_{\bs{j}},
\\\label{form242-12-20}
&K^{(\bar{z}_a)}(\bs{z})=\bar{z}_a\sum_{\bs{j}\in\mc{R}^0_{l}}\tilde{P}^{(\bar{z}_a, \bar{z}_a)}_{\bs{j}}\zeta_{\bar{\bs{j}}}+z_{a}\sum_{\bs{j}\in\mc{R}^a_{l}}\tilde{P}^{(\bar{z}_a, z_a)}_{\bs{j}}\zeta_{\bar{\bs{j}}}.
\end{align}
Notice that $P\in\ms{M}_{2l+1}^{\rm rev}$. Hence, by \eqref{formchi1}--\eqref{form242-12-20}, we have $\chi\in\ms{M}^{\rm an-rev}_{2l+1}$ and $K\in\ms{R}^{\rm rev}_{2l+1}$ with the estimates \eqref{form310}.
\end{proof}

In order to state the iterative lemma, let $\varepsilon_{i}=\varepsilon(4-\frac{i}{r})$ for any positive integer $i\leq r$. We proceed by induction from $i-1$ to $i$. Initially, $P^{(0)}_{2l+1}=P_{2l+1}$ and $R^{(0)}_{\geq2r+3}=R_{\geq2r+3}$ in \eqref{form218}.
\begin{lemma}[Iterative Lemma]
\label{le32}
Assume that the vector field
\begin{equation}
\label{form242-vec}
X^{(i-1)}=Z_1+\sum_{l=1}^{i-1}K_{2l+1}+\sum_{l=i}^{r}P^{(i-1)}_{2l+1}+R^{(i-1)}_{\geq2r+3}
\end{equation}
is defined in $B_{s}(\varepsilon_{i-1})$, where the polynomial vector fields $K_{2l+1}\in\ms{R}^{{\rm rev}}_{2l+1}$, $P^{(i-1)}_{2l+1}\in\ms{M}^{{\rm rev}}_{2l+1}$ satisfy the coefficient estimates
  \begin{equation}
  \label{form313}
  \|K_{2l+1}\|_{\ell^{\infty}}\leq C_1^{l^2},\qquad\|P^{(i-1)}_{2l+1}\|_{\ell^{\infty}}\leq C_1^{il},
  \end{equation}
and the remainder term $R^{(i-1)}_{\geq2r+3}$ satisfies the estimate
  \begin{equation}
  \label{form314}
  \|R^{(i-1)}_{\geq2r+3}(\bs{z})\|_{s}<(3C_1)^{ir}\|z\|_s^{2r+3}.
  \end{equation}
Then there exists a homogeneous  polynomial vector field $\chi_{2i+1}\in\ms{M}^{{\rm an-rev}}_{2i+1}$ with the coefficient estimate
\begin{equation}
\label{form316}
\|\chi_{2i+1}\|_{\ell^{\infty}}\leq\frac{1}{2}C_1^{i^2},
\end{equation}
such that the transformed field under the time 1 flow generated by $\chi_{2i+1}$ is
\begin{equation}
\label{form311}
X^{(i)}:=e^{ad_{\chi_{2i+1}}}X^{(i-1)}=Z_1+\sum_{l=1}^{i}K_{2l+1}+\sum_{l=i+1}^{r}P^{(i)}_{2l+1}+R^{(i)}_{\geq2r+3},
\end{equation}
where the time 1 flow $\Phi^1_{\chi_{2i+1}}: B_{s}(\varepsilon_{i})\to  B_{s}(\varepsilon_{i-1})$ is bound with the estimate
\begin{equation}
\label{form312}
    \|\bs{z}-\Phi^1_{\chi_{2i+1}}(\bs{z})\|_{s}<\frac{3^{i+1}C_1^{i^2}}{2}\|z\|_{s}^{2i+1},
    \end{equation}
and $K_{2l+1}\in\ms{R}^{{\rm rev}}_{2l+1}$, $P^{(i)}_{2l+1}\in\ms{M}^{{\rm rev}}_{2l+1}$, $R^{(i)}_{\geq2r+3}$ satisfy the same estimates \eqref{form313}, \eqref{form314} with $i$ in place of $i-1$.
%
\end{lemma}
\begin{proof}
%
By \Cref{le31}, there exist vector fields $\chi_{2i+1}\in\ms{M}^{{\rm an-rev}}_{2i+1}$ and $K_{2i+1}\in\ms{R}^{{\rm rev}}_{2i+1}$ 
such that
\begin{equation}
 \label{form318}
[Z_1,\chi_{2i+1}]+P^{(i-1)}_{2i+1}=K_{2i+1}
\end{equation}
with the coefficient estimates
\begin{align}
\label{form319}
&\|\chi_{2i+1}\|_{\ell^{\infty}}\leq\frac{1}{2}\|P^{(i-1)}_{2i+1}\|_{\ell^{\infty}}\leq \frac{1}{2}C_1^{i^2},
\\\label{form319'}&\|K_{2i+1}\|_{\ell^{\infty}}\leq\|P^{(i-1)}_{2i+1}\|_{\ell^{\infty}}\leq C_1^{i^2}.
\end{align}
%
By \Cref{le21-1} and the estimate \eqref{form319}, one has
\begin{align}
\label{form320}
\sup_{\bs{z}\in B_{s}(\varepsilon_{i-1})}\|\chi_{2i+1}(\bs{z})\|_s<&3^{i+1}\|\chi_{2i+1}\|_{\ell^{\infty}}\|z\|_s^{2i+1}
\\\notag\leq&\frac{3^{i+1}C_1^{i^2}}{2}\|z\|_s^{2i+1}.
\end{align}
For $\bs{z}\in B_{s}(\varepsilon_{i-1})$, one has 
$\|z\|_s\leq\frac{\varepsilon_{i-1}}{\sqrt{2}}\leq2\sqrt{2}\varepsilon$. 
By the fact $\varepsilon\leq(6rC_1^{r+1})^{-\frac{1}{2}}<(72\sqrt{2}rC_1^{r})^{-\frac{1}{2}}$, one has
\begin{align}
\label{form320-12-1}
\sup_{\bs{z}\in B_{s}(\varepsilon_{i-1})}\|\chi_{2i+1}(\bs{z})\|_s<&\frac{3^{i+1}C_1^{i^2}}{2}(2\sqrt{2}\varepsilon)^{2i+1}
\\\notag\leq&\frac{\varepsilon}{r}=\varepsilon_{i-1}-\varepsilon_i.
\end{align}
Hence, the time 1 flow $\Phi^{1}_{\chi_{2i+1}}$  generated by $\chi_{2i+1}$ is well defined in $B_{s}(\varepsilon_{i})$ and satisfies the estimate
%
\begin{equation}
\label{form322-9-26}
\sup_{\bs{z}\in B_{s}(\varepsilon_{i})}\|\bs{z}-\Phi^1_{\chi_{2i+1}}(\bs{z})\|_{s}\leq\sup_{\bs{z}\in B_{s}(\varepsilon_{i-1})}\|\chi_{2i+1}(\bs{z})\|_s<\frac{3^{i+1}C_1^{i^2}}{2}\|z\|_s^{2i+1},
\end{equation}
which is the estimate \eqref{form312}.
By the formula \eqref{form26-9-26} and \eqref{form318}, one has
\begin{align}
e^{ad_{\chi_{2i+1}}}X^{(i-1)}=&Z_1+\sum_{l=1}^{i-1}K_{2l+1}+[Z_1,\chi_{2i+1}]+\sum_{l=i}^{r}P^{(i-1)}_{2l+1}+\sum_{k=2}^{+\infty}\frac{1}{k!}ad_{\chi_{2i+1}}^kZ_1
\\\notag&+\sum_{k=1}^{+\infty}\frac{1}{k!}ad_{\chi_{2i+1}}^k\Big(\sum_{l=1}^{i-1}K_{2l+1}+\sum_{l=i}^{r}P^{(i-1)}_{2l+1}\Big)+e^{ad_{\chi_{2i+1}}}R^{(i-1)}_{\geq2r+3}
\\\notag:=&Z_1+\sum_{l=1}^{i}K_{2l+1}+\sum_{l=i+1}^{r}P^{(i)}_{2l+1}+R^{(i)}_{\geq2r+3},
\end{align}
where for $l=i+1,\cdots, r$, one has
\begin{align}
\label{form324-9-26}
P^{(i)}_{2l+1}=&P^{(i-1)}_{2l+1}+\sum_{ik=l,k\geq2}\frac{1}{k!}ad_{\chi_{2i+1}}^{k}Z_1
\\\notag&+\sum_{\substack{m+ik=l\\k\geq1, 1\leq m\leq i-1}}\frac{1}{k!}ad_{\chi_{2i+1}}^kK_{2m+1}+\sum_{\substack{m+ik=l\\k\geq1, i\leq m\leq l}}\frac{1}{k!}ad_{\chi_{2i+1}}^kP^{(i-1)}_{2m+1},
\end{align}
%
and
\begin{align}
\label{form325-9-26}
R^{(i)}_{\geq2r+3}=&\sum_{ik\geq r+1}\frac{1}{k!}ad_{\chi_{2i+1}}^{k}Z_1+\sum_{\substack{m+ik\geq r+1\\1\leq m\leq i-1}}\frac{1}{k!}ad_{\chi_{2i+1}}^kK_{2m+1}
\\\notag&+\sum_{\substack{m+ik\geq r+1\\i\leq m\leq l}}\frac{1}{k!}ad_{\chi_{2i+1}}^kP^{(i-1)}_{2m+1}+e^{ad_{\chi_{2i+1}}}R^{(i-1)}_{\geq2r+3}.
\end{align}
Remark that in \eqref{form324-9-26}, the notation $\sum_{ik=l, k\geq2}$ represents the sum with respect to integer $k$ satisfying the conditions $ik=l$ and $k\geq2$, which is empty for $\frac{l}{i}\notin\mb{N}$; the notation $\sum_{\substack{m+ik=l\\k\geq1, 1\leq m\leq i-1}}$ represents the sum with respect to  integers $k,m$ satisfying the conditions $m+ik=l$, $k\geq1$ and $1\leq m\leq i-1$.

Next, we will estimate the coefficients of $P^{(i)}_{2l+1}$.
%
%
By \Cref{le22-12}, the coefficient estimates \eqref{form313}, \eqref{form319} and the fact $m+ik=l$, we have 
\begin{align}
\label{form326-9-16}
d_1(k,m):=&\Big\|\frac{1}{k!}ad_{\chi_{2i+1}}^kP^{(i-1)}_{2m+1}\Big\|_{\ell^{\infty}}
\\\notag\leq&\frac{1}{k!}\|P^{(i-1)}_{2m+1}\|_{\ell^{\infty}}\big(6\|\chi_{2i+1}\|_{\ell^{\infty}}\big)^k\prod_{n=1}^k(m+1+in)
\\\notag\leq&\frac{3^k}{k!}C_1^{im+i^2k}\prod_{n=1}^k(m+1+in)
\\\notag=&(3i)^k\binom{\frac{l+1}{i}}{k}C_1^{il},
\end{align}
where in the last equality we introduce the generalized combination number 
$$\binom{\alpha}{k}:=\frac{\alpha(\alpha-1)\cdots(\alpha-k+1)}{k!}.$$
%
Similarly with \eqref{form326-9-16}, one has
\begin{align}
\label{form330-9-26}
d_2(k,m):=&\Big\|\frac{1}{k!}ad_{\chi_{2i+1}}^kK_{2m+1}\Big\|_{\ell^{\infty}}
\\\notag\leq&\frac{1}{k!}\|K_{2m+1}\|_{\ell^{\infty}}\big(6\|\chi_{2i+1}\|_{\ell^{\infty}}\big)^k\prod_{n=1}^k(m+1+in)
\\\notag\leq&\frac{3^k}{k!}C_1^{m^2+i^2k}\prod_{n=1}^k(m+1+in)
\\\notag<&(3i)^k\binom{\frac{l+1}{i}}{k}C_1^{il},
\end{align}
where the last inequality follows from $m^2+i^2k<im+i^2k=il$.
By the homological equation \eqref{form318}, the coefficient estimates \eqref{form313}, \eqref{form319} and the fact $ik=l$, we have
\begin{align}
\label{form332-9-26}
d_3(k):=&\Big\|\frac{1}{k!}ad_{\chi_{2i+1}}^{k}Z_  1\Big\|_{\ell^{\infty}}
\\\notag\leq&\frac{1}{k!}\|P^{(i-1)}_{2i+1}\|_{\ell^{\infty}}\big(6\|\chi_{2i+1}\|_{\ell^{\infty}}\big)^{k-1}\prod_{n=1}^{k-1}(i+1+in)
\\\notag\leq&\frac{3^k}{k!}C_1^{i^2+i^2(k-1)}\prod_{n=1}^{k-1}(i+1+in)
\\\notag\leq&\frac{6^{l}}{2}C_1^{il},
\end{align}
where the last inequality follows from $i^2+i^2(k-1)=il$ and
\begin{align*}
\frac{3^k}{k!}\prod_{n=1}^{k-1}(i+1+in)\leq&\frac{3^k}{k!}\prod_{n=1}^{k-1}\big((i+1)(n+1)\big)
\\=&3^k(i+1)^{k-1}
\\\leq&3^{ik}\times2^{i(k-1)}
\\\leq&\frac{6^{l}}{2}.
\end{align*}
%
In view of \eqref{form324-9-26}, by the coefficient estimates  \eqref{form313} and \eqref{form326-9-16}--\eqref{form332-9-26}, one has
\begin{align}
\label{form260-12-1}
\|P^{(i)}_{2l+1}\|_{\ell^{\infty}}\leq&\|P^{(i-1)}_{2l+1}\|_{\ell^{\infty}}+\sum_{ik=l,k\geq2}d_3(k)+\sum_{\substack{m+ik=l\\k\geq1, 1\leq m\leq i-1}}d_2(k,m)+\sum_{\substack{m+ik=l\\k\geq1, i\leq m\leq l}}d_1(k,m)
\\\notag\leq&C_1^{il}+\frac{6^{l}}{2}C_1^{il}+\sum_{\substack{m+ik=l\\k\geq1, 1\leq m\leq l}}(3i)^k\binom{\frac{l+1}{i}}{k}C_1^{il}.
\end{align}
Notice that by binomial theorem, one has
\begin{align}
\label{form261-12-1}
\sum_{\substack{m+ik=l\\k\geq1, 1\leq m\leq l}}(3i)^k\binom{\frac{l+1}{i}}{k}
\leq&\sum_{k=1}^{[\frac{l-1}{i}]}(3i)^{k}\binom{\big[\frac{l+1}{i}\big]+1}{k}
\\\notag<&(1+3i)^{[\frac{l+1}{i}]+1}
\\\notag\leq&4^{i(\frac{l+1}{i}+1)}
\\\notag\leq&16^{l},
\end{align}
where the last inequality follows from $i+1\leq l$. 
Thus by \eqref{form260-12-1} and \eqref{form261-12-1}, we have
\begin{align}
\label{form262-12-1}
\|P^{(i)}_{2l+1}\|_{\ell^{\infty}}< (1+\frac{6^{l}}{2}+16^{l})C_1^{il}\leq C_1^{(i+1)l},
\end{align}
where the last inequality follows from $1+\frac{6^{l}}{2}+16^{l}\leq C_1^l$.

Finally, we estimate the remainder term $R^{(i)}_{\geq2r+3}$ in \eqref{form325-9-26}. Divide it into two parts $R=R_1+R_2$,
where 
\begin{align}
\label{form333-9-28}
R_1=&e^{ad_{\chi_{2i+1}}}R^{(i-1)}_{\geq2r+3}=(D\Phi^1_{\chi_{2i+1}})^{-1}R^{(i-1)}_{\geq2r+3}\circ\Phi^1_{\chi_{2i+1}},
\\\label{form334-9-28}
R_2=&\sum_{ik\geq r+1}\frac{1}{k}ad_{\chi_{2i+1}}^{k}Z_1+\sum_{\substack{m+ik\geq r+1\\1\leq m\leq i-1}}\frac{1}{k!}ad_{\chi_{2i+1}}^kK_{2m+1}+\sum_{\substack{m+ik\geq r+1\\i\leq m\leq l}}\frac{1}{k!}ad_{\chi_{2i+1}}^kP^{(i-1)}_{2m+1}
\\\notag:=&\sum_{l\geq r+1}Q_{2l+1}
\end{align}
with
\begin{align*}
Q_{2l+1}=\sum_{ik=l,k\geq2}ad_{\chi_{2i+1}}^{k}Z_1
+\sum_{\substack{m+ik=l\\k\geq1, 1\leq m\leq i-1}}\frac{1}{k!}ad_{\chi_{2i+1}}^kK_{2m+1}+\sum_{\substack{m+ik=l\\k\geq1, i\leq m\leq l}}\frac{1}{k!}ad_{\chi_{2i+1}}^kP^{(i-1)}_{2m+1}.
\end{align*}
Thus
\begin{equation}
\label{form326}
\|R^{(i)}_{\geq2r+3}\|_s\leq\|R_1\|_s+\|R_2\|_s.
\end{equation}
In view of  \eqref{form333-9-28}, by \eqref{form314} and \eqref{form312}, one has 
\begin{align}
\label{form327}
\|R_1\|_s\leq&\|(D\Phi^1_{\chi_{2i+1}}(\bs{z}))^{-1}\|_{\ms{P}_s\to\ms{P}_s}\|R^{(i-1)}_{\geq2r+3}(\Phi^1_{\chi_{2i+1}}(\bs{z}))\|_{s}
\\\notag<&2(3C_1)^{ir}(2\|z\|_{s})^{2r+3}
\\\notag\leq&\frac{1}{3}(3C_1)^{(i+1)r}\|z\|_{s}^{2r+3},
\end{align}
where the last inequality follows from $2^{2r+4}\leq \frac{1}{3}(3C_1)^r$.
The same as the coefficient estimate of $P^{(i)}_{2l+1}$ in \eqref{form262-12-1}, one has
\begin{equation}
\label{form328}
\|Q_{2l+1}\|_{\ell^{\infty}}\leq C_1^{il}.
\end{equation}
By \Cref{le21-1}, \eqref{form328} and the fact $\|z\|_s\leq(6rC_1^{r+1})^{-\frac{1}{2}}<(6C_1^r)^{-\frac{1}{2}}$, one has
\begin{align}
\label{form329}
\|R_2\|_s\leq\sum_{l\geq r+1}\|Q_{2l+1}\|_s<&\sum_{l\geq r+1} 3^{l+1}\|Q_{2l+1}\|_{\ell^{\infty}}\|z\|_s\|z\|_{0}^{2l}
\\\notag\leq&\sum_{l\geq r+1}3^{l+1}C_1^{il}\|z\|_{s}^{2l+1}
\\\notag\leq&2\times3^{r+2}C_1^{i(r+1)}\|z\|_{s}^{2r+3}
\\\notag\leq&\frac{2}{3}(3C_1)^{(i+1)r}\|z\|_{s}^{2r+3},
\end{align}
where the last inequality follows form $3^{2-ir}C_1^{i-r}\leq\frac{1}{3}$ for $1\leq i\leq r$ and $r\geq2$. 
To sum up, by \eqref{form326}, \eqref{form327} and \eqref{form329}, the remainder term satisfies the estimate 
$$\|R^{(i)}_{\geq2r+3}\|_s<(3C_1)^{(i+1)r}\|z\|_{s}^{2r+3}.$$
\end{proof}

\begin{proof}[Proof of $\Cref{th31}$]
Initially, by the estimates \eqref{form211} and \eqref{form212}, the assumptions \eqref{form313} and \eqref{form314} in \Cref{le32} are satisfied.
After $r$ iterations, one has
$$X^{(r)}=Z_1+\sum_{l=1}^{r}K_{2l+1}+R^{(r)}_{\geq2r+3}.$$
Let $\phi^{(1)}=\Phi^1_{\chi_{3}}\circ\cdots\circ\Phi^1_{\chi_{2r+1}}$,  and then we have the vector field \eqref{form33} with  $Z_3:=K_3$ and $R'_{\geq2r+3}:=R^{(r)}_{\geq2r+3}$. 
Moreover, the estimates \eqref{form37} and \eqref{form38} follow from \eqref{form313} and \eqref{form314} with $r$ in place of $i-1$, respectively.
By \eqref{form312}, the transformation $\phi^{(1)}:B_{s}(3\varepsilon)\to B_{s}(4\varepsilon)$ satisfies the estimate
\begin{align*}
\sup_{\bs{z}\in B_{s}(3\varepsilon)}\|\bs{z}-\phi^{(1)}(\bs{z})\|_{s}
&\leq\sum_{i=1}^{r}\sup_{\bs{z}\in B_{s}(\varepsilon_{i})}\|\bs{z}-\Phi^1_{\chi_{2i+1}}(\bs{z})\|_{s}
\\&<\sum_{i=1}^{r}\frac{3^{i+1}C_1^{i^2}}{2}\sup_{\bs{z}\in B_{s}(\varepsilon_{i})}\|z\|_{s}^{2i+1}
\\&\leq 9C_1\|z\|_s^3,
\end{align*}
where the last inequality follows from the fact $\varepsilon\leq(6rC_1^{r+1})^{-\frac{1}{2}}<(6C_1^{r+1})^{-\frac{1}{2}}$.
Notice that $K_3$ is the resonant part of $P_3$ in \eqref{form29}, and thus 
$$K_3(\bs{z})=-\frac{{\rm i}}{4}\sum_{a\in\mb{N}_*}(z^2_a\bar{z}_a\pa
_{z_a}-\bar{z}^2_az_a\pa_{\bar{z}_a})=-\frac{{\rm i}}{4}\sum_{a\in\mb{N}_*}I_a(z_a\pa_{z_a}-\bar{z}_a\pa_{\bar{z}_a}),$$
which is actually an integrable vector field.

Now, the only remaining task is to compute $K_5$. According to the proof of \Cref{le32}, $K_5$ is the resonant part of $P^{(1)}_5$.
By \eqref{form324-9-26} and the homological equation $[Z_1,\chi_{3}]+P_{3}=K_{3}$, one has 
\begin{equation}
\label{form330}
P^{(1)}_5=\frac{1}{2}ad_{\chi_3}^2Z_1+ad_{\chi_3}P_3+P_5=\frac{1}{2}[K_3+P_3,\chi_3]+P_5.
\end{equation}
%
%
In view of \eqref{form29},  by the proof of \Cref{le31}, one has 
\begin{equation}
\label{form331}
\chi_3(\bs{z})=\frac{1}{8}\sum_{a\in\mb{N}_*}\left(\Big(\sum_{b\neq a}\frac{z^2_b}{b-a}+\sum_{b\in\mb{N}_*}\frac{\bar{z}^2_b}{b+a}\Big)\bar{z}_a\pa_{z_a}+\Big(\sum_{b\neq a}\frac{\bar{z}^2_b}{b-a}+\sum_{b\in\mb{N}_*}\frac{z^2_b}{b+a}\Big)z_a\pa_{\bar{z}_a}\right).
\end{equation}
By a direct calculation, the $z_a$-component of $[P_3,\chi_3](\bs{z})$ is
\begin{align}
\label{form332}
[P_3,\chi_3]^{(z_a)}(\bs{z})=&-\frac{{\rm i}z_{a}}{32}\Big(\sum_{b\in\mb{N}_*}(z_b^2-\bar{z}_b^2)\Big)\Big(\sum_{b\neq a}\frac{\bar{z}^2_b-z_b^2}{b-a}+\sum_{b\in\mb{N}_*}\frac{z^2_b-\bar{z}_b^2}{b+a}\Big)
\\\notag&-\frac{{\rm i}\bar{z}_a}{16}\sum_{b\in\mb{N}_*}I_b\Big(\sum_{c\neq b}\frac{z^2_c}{c-b}+\sum_{c\in\mb{N}_*}\frac{\bar{z}^2_c}{c+b}-\sum_{c\neq b}\frac{\bar{z}^2_c}{c-b}-\sum_{c\in\mb{N}_*}\frac{z^2_c}{c+b}\Big)
\\\notag&+\frac{{\rm i}\bar{z}_a}{16}\Big(\sum_{b\in\mb{N}_*}(z_b^2-\bar{z}_b^2)\Big)\Big(\sum_{c\neq a}\frac{I_c}{c-a}+\sum_{c\in\mb{N}_*}\frac{I_c}{c+a}\Big),
\end{align}
where the resonant part of the first line is
\begin{equation}
\label{form273-12-27}
-\frac{{\rm i}z_{a}}{16}\Big(\sum_{b\neq a}\frac{I^2_b}{b-a}-\sum_{b\in\mb{N}_*}\frac{I_b^2}{b+a}\Big)=\frac{{\rm i}z_{a}I_a^2}{32a}+\frac{{\rm i}az_{a}}{8}\sum_{b\neq a}\frac{I^2_b}{a^2-b^2};
\end{equation}
the resonant part of the second line  is
\begin{equation}
\label{form274-12-27}
-\frac{{\rm i}z_aI_a}{16}\Big(\sum_{b\neq a}\frac{I_b}{a-b}-\sum_{b\in\mb{N}_*}\frac{I_b}{a+b}\Big)=\frac{{\rm i}z_aI^2_a}{32a}-\frac{{\rm i}z_aI_a}{8}\sum_{b\neq a}\frac{bI_b}{a^2-b^2};
\end{equation}
and the resonant part of the third line is
\begin{equation}
\label{form275-12-27}
\frac{{\rm i}z_aI_a}{16}\Big(\sum_{c\neq a}\frac{I_c}{c-a}+\sum_{c\in\mb{N}_*}\frac{I_c}{c+a}\Big)=\frac{{\rm i}z_aI^2_a}{32a}-\frac{{\rm i}z_aI_a}{8}\sum_{b\neq a}\frac{bI_b}{a^2-b^2}.
\end{equation}
Taking sum of \eqref{form273-12-27}--\eqref{form275-12-27} and multiplying $\frac{1}{2}$, we give the resonant part of $\frac{1}{2}[P_3,\chi_3]^{(z_a)}$: 
\begin{equation}
\label{form333}
\frac{3{\rm i}z_aI^2_a}{64a}-\frac{{\rm i}z_aI_a}{8}\sum_{b\neq a}\frac{bI_b}{a^2-b^2}+\frac{{\rm i}az_a}{16}\sum_{b\neq a}\frac{I_b^2}{a^2-b^2}.
\end{equation}
In view of \eqref{form210}, the resonant part of $P_5^{(z_a)}(\bs{z})$ is equal to:
\begin{align}
\label{form334}
&\frac{3{\rm i}\bar{z}_a}{16}\sum_{b_1+b_2=a}\frac{z^2_{b_1}z^2_{b_2}}{b_1}+\frac{3{\rm i}\bar{z}_a}{16}\sum_{b-c=a}(\frac{1}{c}-\frac{1}{b})z^2_{b}\bar{z}^2_{c}+\frac{3{\rm i}z_aI_a}{16}\sum_{b\in\mb{N}_*}\frac{2I_b}{b}
\\\notag=&\frac{3a{\rm i}\bar{z}_a}{32}\sum_{b_1+b_2=a}\frac{z^2_{b_1}z^2_{b_2}}{b_1b_2}+\frac{3a{\rm i}\bar{z}_a}{16}\sum_{b-c=a}\frac{z^2_{b}\bar{z}^2_{c}}{bc}+\frac{3{\rm i}z_aI_a}{8}\sum_{b\in\mb{N}_*}\frac{I_b}{b}.
\end{align}
Notice that all terms in $[K_3,\chi_3]$ are non-resonant. %
Hence, taking sum of  \eqref{form333} and \eqref{form334},  we get $K^{(z_a)}_5$ in \eqref{form35}.
\end{proof}


\section{Small divisors and rational vector fields}
\label{sec3}
In this section, we firstly introduce truncated resonant vector fields and two suitable small divisor conditions. Taking these conditions into account, we define rational vector fields. Then we estimate the commutator and solve the homological equation for rational vector fields.

\subsection{Truncation and small divisors}
\label{sec31}
For $\bs{j}=(j_{k})_{k=1}^{l}\in(\mb{U}_{3}\times\mb{N}_*)^{l}$, let $j_{1}^{*}\geq\cdots\geq j_{l}^{*}$ denote the decreasing rearrangement of  $\{|j_{1}|,\cdots,|j_{l}|\}$, and denote $\mu_{\min}(\bs{j}):=j_l^*$. 
For $l\in\mb{N}_*$, $b\in\mb{Z}$ and $N\geq1$, denote the truncated index sets
\begin{align*}
\mc{R}^{b}_{l,N}:=&\{\bs{j}\in\mc{R}^b_{l}\mid j_{1}^{*}\leq N\}.
\end{align*}
Then we divide the resonant polynomial vector field $K(\bs{z})\in\ms{R}_{2l+1}$ into two parts: 
\begin{equation}
\label{form41} 
K(\bs{z})=K^{\leq N}(\bs{z})+K^{>N}(\bs{z}),
\end{equation}
where 
the $z_a$-component of $K^{\leq N}(\bs{z})\in\ms{R}_{2l+1}$ is the truncated resonant polynomial
$$K^{(z_a)}_{\leq N}(\bs{z})=
\left\{\begin{aligned}
z_a\sum_{\bs{j}\in\mc{R}^0_{l,N}}\tilde{K}^{(z_a,z_a)}_{\bs{j}}\zeta_{\bs{j}}+\bar{z}_a\sum_{\bs{j}\in\mc{R}^a_{l,N}}\tilde{K}^{(z_a,\bar{z}_a)}_{\bs{j}}\zeta_{\bs{j}},&\quad  \text{for}\;a\leq N,
\\z_a\sum_{\bs{j}\in\mc{R}^0_{l,N}}\tilde{K}^{(z_a,z_a)}_{\bs{j}}\zeta_{\bs{j}},\qquad\qquad&\quad  \text{for}\; a>N,
\end{aligned}\right.$$
and the $\bar{z}_a$-component of $K^{\leq N}(\bs{z})$ is the conjugate of $K^{(z_a)}_{\leq N}(\bs{z})$.

Now we estimate the remainder term $K^{>N}(\bs{z})$ in the following lemma.
\begin{lemma}
\label{le30}
For any $s\geq0$, the vector field $K^{>N}(\bs{z})$ in \eqref{form41} satisfies 
\begin{equation}
\label{form42}
\|K^{>N}(\bs{z})\|_s<\frac{3^{l+1}l^{2s}}{N^{2s}}\|K\|_{\ell^{\infty}}\|z\|_s^3\|z\|_{0}^{2l-2}.
\end{equation}
\end{lemma}
\begin{proof}
According to the definition, the $z_a$-component of $K^{>N}(\bs{z})$ is 
\begin{equation}
\label{form33-12-29}
K^{(z_a)}_{>N}(\bs{z})=
\left\{\begin{aligned}
z_a\sum_{\substack{\bs{j}\in\mc{R}_{l}^0\\j_{1}^{*}>N}}\tilde{K}^{(z_a,z_a)}_{\bs{j}}\zeta_{\bs{j}}+\bar{z}_a\sum_{\substack{\bs{j}\in\mc{R}_{l}^a\\j_{1}^{*}>N}}\tilde{K}^{(z_a,\bar{z}_a)}_{\bs{j}}\zeta_{\bs{j}},\;  &\text{for}\;a\leq N,
\\z_a\sum_{\substack{\bs{j}\in\mc{R}_{l}^0\\j_{1}^{*}>N}}\tilde{K}^{(z_a,z_a)}_{\bs{j}}\zeta_{\bs{j}}+\bar{z}_a\sum_{\bs{j}\in\mc{R}_{l}^a}\tilde{K}^{(z_a,\bar{z}_a)}_{\bs{j}}\zeta_{\bs{j}}, \quad  &\text{for}\; a>N.
\end{aligned}\right.
\end{equation}
Notice that for $\bs{j}=(\delta_k,a_k)_{k=1}^l\in\mc{R}^a_{l}$ with $a>N$, one has $\sum_{k=1}^{l}\delta_ka_k=a>N$ and thus $j_{1}^{*}>\frac{N}{l}$. 
%
In view of \eqref{form33-12-29}, for any $\bs{j}\in\mc{R}^0_{l}$ or $\mc{R}^a_{l}$ in  $K^{(z_a)}_{>N}(\bs{z})$, we have $j_{1}^{*}>\frac{N}{l}.$
Similarly with the proof of \Cref{le21-1}, one has
\begin{align*}
\|K^{>N}(\bs{z})\|_{s}=&\sqrt{2}\Big(\sum_{a\in \mb{N}_*}a^{2s}\big|K_{>N}^{(z_a)}\big|^2\Big)^{\frac{1}{2}}
\\\leq&\sqrt{2}\|K\|_{\ell^{\infty}}\Big(\sum_{a\in \mb{N}_*}a^{2s}\big(|\bar{z}_a|+|z_{a}|\big)^2\Big)^{\frac{1}{2}}\sum_{\substack{\bs{j}\in(\mb{U}_{3}\times\mb{N}_*)^{l}\\j_{1}^{*}>\frac{N}{l}}}|\zeta_{\bs{j}}|
\\<&2\sqrt{2}\frac{3^ll^{2s}}{N^{2s}}\|K\|_{\ell^{\infty}}\|z\|_s^3\|z\|_{0}^{2l-2},
\end{align*}
which implies \eqref{form42}.
\end{proof}

Remark that by taking $N$ big enough depending on $\varepsilon$,   the remainder term $K^{>N}(\bs{z})$ is sufficiently small as to not  influence the long time stability result. Thus, we only need eliminate the non-integrable parts of $K^{\leq N}$. 

Next, for the transformed vector field \eqref{form33} in \Cref{th31}, we introduce modified frequencies and small divisor conditions.
In view of \eqref{form34}, one has the cubic truncated integrable polynomial vector field 
\begin{equation}
\label{form33-1-20}
Z_3^{\leq N}(\bs{z})=-{\rm i}\sum_{a\in\mb{N}_*}\omega_a^{(2)}(z_a\pa_{z_a}-\bar{z}_a\pa_{\bar{z}_a})
\end{equation} 
with
\begin{equation}
\label{form43}
\omega^{(2)}_a:=\left\{\begin{aligned}\frac{1}{4}I_a,\;\quad&\text{for}\; a\leq N, 
\\0,\qquad&\text{for}\; a>N;
\end{aligned}\right.
\end{equation}
and in view of \eqref{form35'}, one has the quintic truncated integrable polynomial vector field 
\begin{align}
\label{form10-2}
Z_5^{\leq N}=-{\rm i}\sum_{a\in\mb{N}_*}\omega^{(4)}_a(z_a\pa_{z_a}-\bar{z}_a\pa_{\bar{z}_a})
\end{align}
with
\begin{equation}
\label{form44}
\omega^{(4)}_a:=\left\{\begin{aligned}-\frac{27}{64a}I^2_a-\frac{1}{8}I_a\sum_{a\neq d\leq N}\Big(\frac{3}{d}+\frac{d}{d^2-a^2}\Big)I_d+\frac{a}{16}\sum_{a\neq d\leq N}\frac{1}{d^2-a^2}I^2_d,
\quad&\text{for}\; a\leq N, 
\\\frac{a}{16}\sum_{a\neq d\leq N}\frac{1}{d^2-a^2}I^2_d,\qquad\qquad\qquad\qquad&\text{for}\; a>N.\end{aligned}\right.
\end{equation}
For convenience, let $\omega^{(2)}_0=\omega^{(4)}_0=0$. For $\bs{j}=(\delta_k,a_k)_{k=1}^l\in\mc{R}^b_l$ with $\Delta_{\bs{j}}=b\geq0$, define
\begin{equation}
\label{formomega2}
\Omega_{\bs{j}}^{(2)}(I):=2\Big(\sum_{k=1}^l\delta_k\omega^{(2)}_{a_k}-\omega^{(2)}_{b}\Big),
\end{equation}
\begin{equation}
\label{formomega4}
\Omega_{\bs{j}}^{(4)}(I):=2\Big(\sum_{k=1}^l\delta_k(\omega^{(2)}_{a_k}+\omega^{(4)}_{a_k})-(\omega^{(2)}_{b}+\omega^{(4)}_{b})\Big).
\end{equation}

For any $\bs{j}=(j_{k})_{k=1}^{l}\in\mc{R}^0_l$,  
%
denote $Irr({\bs{j}})$ as the irreducible part of $\bs{j}$, namely the subsequence of maximal length $(j'_1,\cdots,j'_{l'})$ containing no action in the sense that $j'_i\neq\bar{j'_k}$ for all $i,k=1,\cdots,l'$.
Similarly, for any $\bs{j}\in\mc{R}^a_l$, denote the  irreducible part $Irr({\bs{j}})$ with the above condition and additionally  deleting one index $j'_m=(1,a)$ if it exists. 
%
%
For example, consider the monomial $\zeta_{\bs{j}}=z^2_{b_1}z^2_{b_2}z^2_{b_3}z^2_{b_4}\bar{z}^2_{c_1}\bar{z}^2_{c_2}I_d$ with its index
$$\bs{j}:=\big((1,b_1),(1, b_2),(1,b_3),(1, b_4),(-1,c_1),(-1,c_2),(0,d)\big)$$ 
satisfying $b_1=c_1$ and $b_2,b_3,b_4\neq c_2$. 
%
%
If $\Delta_{\bs{j}}:=b_1+b_2+b_3+b_4-c_1-c_2\neq b_2,b_3,b_4$, then 
$$Irr(\bs{j})=\big((1, b_2),(1, b_3),(1, b_4),(-1,c_2)\big)\quad\text{and}\quad\zeta_{Irr(\bs{j})}=z^2_{b_2}z^2_{b_3}z^2_{b_4}\bar{z}^2_{c_2};$$
and if $\Delta_{\bs{j}}=b_2$, then 
$$Irr(\bs{j})=\big((1, b_3),(1,b_4),(-1,c_2)\big)\quad\text{and}\quad\zeta_{Irr(\bs{j})}=z^2_{b_3}z^2_{b_4}\bar{z}^2_{c_2}.$$

Remark that 
$\Omega_{\bs{j}}^{(2)}(I)=\Omega_{Irr(\bs{j})}^{(2)}(I)$, $\Omega_{\bs{j}}^{(4)}(I)=\Omega_{Irr(\bs{j})}^{(4)}(I)$. In particular, for $\bs{j}\in\mc{I}^0$ or $\mc{I}^a$, one has $Irr(\bs{j})=\emptyset$ and $\Omega_{\bs{j}}^{(2)}(I)=\Omega_{\bs{j}}^{(4)}(I)=0$. 
Let
$$Irr(\mc{R})=\{Irr(\bs{j})\mid\bs{j}\in\mc{R}^b\;\text{with}\; b\geq0\}.$$
Denote $\#\bs{j}$ as the length of $\bs{j}$. Remark that for any $\bs{j}\in Irr(\mc{R})$, one has $\#\bs{j}\geq2$. 
Especially, if $\bs{j}\in Irr(\mc{R})\cap\mc{R}^{0}$, then $\#\bs{j}\geq3$.
Moreover, let
$$Irr(\mc{R}_{\leq N})=\{\bs{j}\in Irr(\mc{R}) \mid j_{1}^{*}\leq N\;\text{and}\;\Delta_{\bs{j}}\leq N\}.$$
Then for $\bs{j}\in Irr(\mc{R}_{\leq N})$,  $\Omega_{\bs{j}}^{(2)}(I)$ and $\Omega_{\bs{j}}^{(4)}(I)$ only depend on the action $\{I_d\}_{d\leq N}$.

%

For $r\geq2$, $N\geq1$ and $\gamma>0$, we say that $z\in \ell_{s}^2$ belongs to the open set $\mc{U}_{\gamma}^{N}$, if for any $\bs{j}\in Irr(\mc{R}_{\leq N})$ with $\#\bs{j}=l\leq r$, one has
\begin{align}
\label{form49}
&|\Omega_{\bs{j}}^{(2)}(I)|>\gamma\|z\|_{s}^{2}N^{-4l-2}\kappa_{\bs{j}}^{-2s},
\\\label{form410}
&|\Omega_{\bs{j}}^{(4)}(I)|>\gamma\|z\|_{s}^{2}N^{-4l-2}\max\{\kappa_{\bs{j}}^{-2s},\gamma\|z\|_{s}^{2}\},
\end{align}
where 
$$\kappa_{\bs{j}}=\left\{\begin{aligned}
\;\mu_{\min}(\bs{j}),\qquad&\;\text{if}\quad \Delta_{\bs{j}}=0, 
\\\min\{\mu_{\min}(\bs{j}),\Delta_{\bs{j}}\},\;&\;\text{if}\quad \Delta_{\bs{j}}>0.
\end{aligned}\right.$$
%

The following lemma shows that the set $\mc{U}_{\gamma}^{N}$ is stable with respect to the action $I$.
\begin{lemma}
\label{le41}
Fix $s\geq0$ and $r\geq2$. For $N\geq1$ and $0<\gamma<1$, let $z\in\mc{U}_{\gamma}^{N}$. If $z'\in\ell^2_{s}$ satisfies the conditions
\begin{equation}
\label{form411}
\|z'\|_{s}\leq4\|z\|_{s}\quad\text{and}\quad
\sup_{a\leq N}|a|^{2s}|I'_{a}-I_{a}|\leq\frac{\gamma^{2}\|z\|_{s}^{2}}{288(r+1)N^{4r+3}},
\end{equation}
then $z'\in\mc{U}_{\gamma/2}^{N}$.
\end{lemma}
\begin{proof}
In view of the definition \eqref{formomega2}, for any $\bs{j}=(\delta_k,a_k)_{k=1}^l\in Irr(\mc{R}_{\leq N})$ with $l\leq r$, if $\Delta_{\bs{j}}=0$, then
\begin{align}
\label{form412}
\big|\Omega^{(2)}_{\bs{j}}(I')-\Omega^{(2)}_{\bs{j}}(I)\big|
&\leq\frac{1}{2}\sum_{k=1}^l|I'_{a_{k}}-I_{a_{k}}|
\\\notag&\leq \frac{l}{2}\mu_{\min}^{-2s}(\bs{j})\sup_{a\leq N}a^{2s}|I'_{a}-I_{a}|;
\end{align}
and if $b:=\Delta_{\bs{j}}>0$, then
\begin{align}
\label{form412-12}
\big|\Omega^{(2)}_{(\bs{j})}(I')-\Omega^{(2)}_{(\bs{j})}(I)\big|
\leq&\frac{1}{2}\Big(\sum_{k=1}^l|I'_{a_{k}}-I_{a_{k}}|+|I'_{b}-I_{b}|\Big)
\\\notag\leq&\frac{l+1}{2}\big(\min\{\mu_{\min}(\bs{j}),b\}\big)^{-2s}\sup_{a\leq N}a^{2s}|I'_{a}-I_{a}|.
\end{align}
%
Then for any  $\bs{j}\in Irr(\mc{R}_{\leq N})$, by \eqref{form412}, \eqref{form412-12} and 
$$\sup_{a\leq N}a^{2s}|I'_{a}-I_{a}|\leq\frac{\gamma^{2}\|z\|_{s}^{2}}{288(r+1)N^{4r+3}}\leq\frac{\gamma\|z\|_{s}^{2}}{2(l+1)N^{4l+2}},$$
we have
\begin{equation}
\label{form317-12-29}
\big|\Omega^{(2)}_{\bs{j}}(I')-\Omega^{(2)}_{\bs{j}}(I)\big|\leq\frac{\gamma}{4}\|z\|_{s}^{2}N^{-4l-2}\kappa^{-2s}_{\bs{j}},
\end{equation}
and thus
\begin{align}
\label{32-form316}
|\Omega_{\bs{j}}^{(2)}(I')|\geq&|\Omega_{\bs{j}}^{(2)}(I)|-\big|\Omega^{(2)}_{\bs{j}}(I')-\Omega^{(2)}_{\bs{j}}(I)\big|
\\\notag>&\frac{\gamma}{2}\|z\|_{s}^{2}N^{-4l-2}\kappa^{-2s}_{\bs{j}}.
\end{align}

For any $\bs{j}\in Irr(\mc{R}_{\leq N})$, write $\Omega^{(4,4)}_{\bs{j}}:=\Omega^{(4)}_{\bs{j}}-\Omega^{(2)}_{\bs{j}}$, and then $\Omega_{\bs{j}}^{(4,4)}(I)$ only depends on $\{\omega^{(4)}_a\}_{a\leq N}$, which is a quadratic homogenous polynomial with respect to $\{I_a\}_{a\leq N}$.
By a direct calculation and the condition $\|z'\|_{s}\leq4\|z\|_{s}$, we have
\begin{small}
\begin{align*}
&|\omega^{(4)}_a(I)-\omega^{(4)}_a(I')|
\\\leq&\frac{27}{64}|I^2_a-I'^2_a|+\frac{1}{2}\Big(I_a\sum_{a\neq d\leq N}|I_d-I'_d|+|I_a-I'_a|\sum_{a\neq d\leq N}I'_d\Big)+\frac{1}{16}\sum_{a\neq d\leq N}|I^2_d-I'^2_d|
\\\leq&\Big(\frac{27(\|z\|_{s}^{2}+\|z'\|_{s}^{2})}{64}+\frac{N\|z\|_{s}^{2}+\|z'\|_{s}^{2}}{2}+\frac{\|z\|_{s}^{2}+\|z'\|_{s}^{2}}{16}\Big)\sup_{a\leq N}|I'_{a}-I_{a}|
\\\leq&\Big(\frac{1039}{64}+\frac{N}{2}\Big)\|z\|_{s}^{2}\sup_{a\leq N}|I'_{a}-I_{a}|
\\<&18N\|z\|_{s}^{2}\sup_{a\leq N}|I'_{a}-I_{a}|.
\end{align*}
\end{small}%
Then in view of \eqref{formomega2} and \eqref{formomega4}, for any $\bs{j}\in Irr(\mc{R}_{\leq N})$, one has
\begin{align}
\label{form413}
&\big|\Omega^{(4,4)}_{\bs{j}}(I')-\Omega^{(4,4)}_{\bs{j}}(I)\big|
\\\notag\leq&2\Big(\sum_{k=1}^{l}|\omega^{(4)}_{a_{k}}(I)-\omega^{(4)}_{a_{k}}(I')|+|\omega^{(4)}_{b}(I)-\omega^{(4)}_{b}(I')|\Big)
\\\notag\leq&36(l+1)N\|z\|_{s}^{2}\sup_{a\leq N}|I'_{a}-I_{a}|
\\\notag\leq&\frac{\gamma^2}{8}\|z\|_{s}^{4}N^{-4l-2},
\end{align}
where $b=\Delta_{\bs{j}}\geq0$ and the last inequality follows from 
$$\sup_{a\leq N}|I'_{a}-I_{a}|\leq\frac{\gamma^{2}\|z\|_{s}^{2}}{288(r+1)N^{4r+3}}\leq\frac{\gamma^{2}\|z\|_{s}^{2}}{288(l+1)N^{4l+3}}.$$
By \eqref{form317-12-29} and  \eqref{form413} , one has
\begin{align*}
&\big|\Omega_{\bs{j}}^{(4)}(I')-\Omega_{\bs{j}}^{(4)}(I)\big|
\\\leq&2\max\Big\{\big|\Omega_{\bs{j}}^{(2)}(I')-\Omega_{\bs{j}}^{(2)}(I)\big|,\big|\Omega_{\bs{j}}^{(4,4)}(I')-\Omega_{\bs{j}}^{(4,4)}(I)\big|\Big\}
\\\leq&\frac{\gamma}{2}\|z\|_{s}^{2}N^{-4l-2}\max\{\kappa^{-2s}_{\bs{j}},\frac{\gamma}{2}\|z\|_{s}^{2}\},
\end{align*}
and then
\begin{align}
\label{32-form319}
|\Omega_{\bs{j}}^{(4)}(I')|
\geq&|\Omega_{\bs{j}}^{(4)}(I)|-\big|\Omega_{\bs{j}}^{(4)}(I')-\Omega_{\bs{j}}^{(4)}(I)\big|
\\\notag>&\frac{\gamma}{2}\|z\|_{s}^{2}N^{-4l-2}\max\{\kappa^{-2s}_{\bs{j}},\frac{\gamma}{2}\|z\|_{s}^{2}\}.
\end{align}
To sum up, we conclude that $z'\in\mc{U}_{\gamma/2}^{N}$ by \eqref{32-form316} and \eqref{32-form319}.
\end{proof}

\subsection{Rational vector fields}
\label{sec32}

We want to find a transformation to eliminate the non-integrable part of $K^{\leq N}_{5}$, which is the truncation of $K_5$ in \Cref{th31}. This transformation arises from a vector field $\chi'_{3}$ associated with the small divisor $\Omega^{(2)}$ in \eqref{formomega2}. 
Then we need eliminate the non-integrable part of higher order truncated resonant vector fields $K^{\leq N}_{2l+1}$ by using $Z_3^{\leq N}$ or $Z_3^{\leq N}+Z^{\leq N}_5$, where the small divisors $\Omega^{(2)}$ and $\Omega^{(4)}$ in \eqref{formomega2} and \eqref{formomega4} are encountered. To better characterize the class of vector fields, we introduce rational vector fields.

In the following, we use the notation $a\lesssim b$ to mean that $a\leq Cb$ with a positive constant $C$. 
%
For any $b\in\mb{N}$, denote 
$\mc{R}^b_{\leq N}=\bigcup_{\alpha\in\mb{N}_*}\mc{R}^b_{\alpha,N}$.
Firstly, we define a family of index sets to control the rational fractions.
\begin{definition}
\label{def41}
For any $l\in\mb{N}_*$ and $N\geq1$, denote by $\mc{H}_{l,N}$ the family of all the elements $\Gamma_{l}=\bigcup_{a\in\mb{N}_*}\Gamma^{a}_{l}$ and $\Gamma'_{l}=\bigcup_{a\leq N}\Gamma'^{a}_{l}$ with
\begin{eqnarray}
\label{form414}
\Gamma_{l}^a & \subseteq & \{a\}\times\mc{R}^0_{\leq N}\times\bigcup_{\beta\geq0}\big(Irr(\mc{R}_{\leq N})\big)^{\beta}\times\bigcup_{\beta'\geq0}\big(Irr(\mc{R}_{\leq N})\big)^{\beta'}\times\mb{N}, 
\\\label{form415}
\Gamma_{l}'^a & \subseteq & \{a\}\times\mc{R}^a_{\leq N}\times\bigcup_{\beta\geq0}\big(Irr(\mc{R}_{\leq N})\big)^{\beta}\times\bigcup_{\beta'\geq0}\big(Irr(\mc{R}_{\leq N})\big)^{\beta'}\times\mb{N}, 
\end{eqnarray}
where for any $(a,\bs{j},\bs{h},\bs{k},n)\in\Gamma_{l}$ or $\Gamma'_{l}$, 
\begin{enumerate}
\item\label{def41-1} one has 
\begin{itemize}
\item[(1-1)] the order $l$, i.e., $\#\bs{j}-\#\bs{h}-2\#\bs{k}=l$,
\item[(1-2)] $\bs{j}\in\mc{R}^0_{\leq N}$ or $\mc{R}^a_{\leq N}$ with $\#\bs{j}\geq2$,
\item[(1-3)] $\bs{h}=(\bs{h}_{m})_{m=1}^{\#\bs{h}}$ with $\bs{h}_{m}\in Irr(\mc{R}_{\leq N})$ and $2\leq\#\bs{h}_{m}\leq\#\bs{j}$,
\item[(1-4)] $\bs{k}=(\bs{k}_{m})_{m=1}^{\#\bs{k}}$ with $\bs{k}_{m}\in Irr(\mc{R}_{\leq N})$ and $2\leq\#\bs{k}_{m}\leq\#\bs{j}$,
\item[(1-5)] $n\leq\#\bs{h}$;
\end{itemize}
\item\label{def41-2} if $a$ and $\bs{j}$ are fixed, then there is a finite number of indices $(a,\bs{j},\bs{h},\bs{k},n)$ and this number is controlled by $\alpha:=\#\bs{j}$, i.e.,
\begin{align*}
\max_{a,\; \bs{j}\in\mc{R}_{\alpha,N}^{0}\cup\mc{R}_{\alpha,N}^{a}}\#(\Gamma_{l}\cup\Gamma'_{l};\ a, \bs{j})\leq(2\alpha)^{2\alpha-3}
\end{align*}
with the notation $\#(\Gamma_{l}\cup\Gamma'_{l};\ a, \bs{j}):=\#\big\{(\tilde{a},\tilde{\bs{j}},\bs{h},\bs{k},n)\in\Gamma_{l}\cup\Gamma'_{l} \;\mid\; \tilde{a}=a,  \tilde{\bs{j}}=\bs{j}\big\}$;
\item\label{def41-3} one has the control condition
\begin{equation}
\label{form416}
\prod_{m=1}^{\#\bs{h}}\kappa_{\bs{h}_{m}}\leq\prod_{m=1}^{\#\bs{j}}j_{m}^*.
\end{equation}
\end{enumerate}
\end{definition}

Then we define the family of resonant rational vector fields.

\begin{definition}
\label{def42}
Being given $\Gamma_{l}, \Gamma'_{l}\in\mc{H}_{l,N}$, define the rational vector field of order $2l+1$:
\begin{equation}
\label{form418}
Q_{\Gamma_{l},\Gamma'_{l}}(\bs{z})=Q_{\Gamma_{l},\Gamma'_{l}}^{(z)}(\bs{z})\pa_{z}+Q_{\Gamma_{l},\Gamma'_{l}}^{(\bar{z})}(\bs{z})\pa_{\bar{z}}
\end{equation}
with the $z$-component 
\begin{align}
\label{form419}
Q_{\Gamma_{l},\Gamma'_l}^{(z)}(\bs{z})\pa_{z}=&Q_{\Gamma_{l}}^{(z)}(\bs{z})\pa_{z}+Q_{\Gamma'_{l}}^{(z)}(\bs{z})\pa_{z}
\\\notag=&\sum_{J:=(a,\bs{j},\bs{h},\bs{k},n)\in\Gamma_{l}}z_aQ_{J}^{(z_a)}(\bs{z})\pa_{z_a}+\sum_{J:=(a,\bs{j},\bs{h},\bs{k},n)\in\Gamma'_{l}}\bar{z}_aQ_{J}^{(z_a)}(\bs{z})\pa_{z_a},
\end{align}
\begin{equation}
\label{form419-10-3-2}
Q_{J}^{(z_a)}(\bs{z})=\tilde{Q}^{(z_a)}_{J}\frac{\zeta_{\bs{j}}}{\prod\limits_{m=1}^{n}\Omega^{(2)}_{\bs{h}_{m}}\prod\limits_{m=n+1}^{\#\bs{h}}\Omega^{(4)}_{\bs{h}_{m}}\prod\limits_{m=1}^{\#\bs{k}}\Omega^{(4)}_{\bs{k}_{m}}},
\end{equation}
and the $\bar{z}$-component
\begin{align}
\label{form420}
Q_{\Gamma_{l},\Gamma'_l}^{(\bar{z})}(\bs{z})\pa_{\bar{z}}=&Q_{\Gamma_{l}}^{(\bar{z})}(\bs{z})\pa_{\bar{z}}+Q_{\Gamma'_{l}}^{(\bar{z})}(\bs{z})\pa_{\bar{z}}
\\\notag=&\sum_{J:=(a,\bs{j},\bs{h},\bs{k},n)\in\Gamma_{l}}\bar{z}_aQ_{J}^{(\bar{z}_a)}(\bs{z})\pa_{\bar{z}_a}
+\sum_{J:=(a,\bs{j},\bs{h},\bs{k},n)\in\Gamma'_{l}}z_aQ_{J}^{(\bar{z}_a)}(\bs{z})\pa_{\bar{z}_a},
\end{align}
\begin{equation}
\label{form420-10-3-2}
Q_{J}^{(\bar{z}_a)}(\bs{z})=\tilde{Q}^{(\bar{z}_a)}_{J}\frac{\zeta_{\bar{\bs{j}}}}{\prod\limits_{m=1}^{n}\Omega^{(2)}_{\bs{h}_{m}}\prod\limits_{m=n+1}^{\#\bs{h}}\Omega^{(4)}_{\bs{h}_{m}}\prod\limits_{m=1}^{\#\bs{k}}\Omega^{(4)}_{\bs{k}_{m}}},
\end{equation}
where
the coefficients $\tilde{Q}^{(z_a)}_{J},\tilde{Q}^{(\bar{z}_a)}_{J}\in\mb{C}$ satisfy $\tilde{Q}^{(z_a)}_{J}=\overline{\tilde{Q}^{(\bar{z}_a)}_{J}}$,
\begin{equation}
\label{form421}
\|Q_{\Gamma_{l},\Gamma'_{l}}\|_{\ell^{\infty}}:=\sup_{J\in\Gamma_{l}\bigcup\Gamma'_{l}}|\tilde{Q}^{(z_a)}_{J}|<+\infty,
\end{equation}
and for any permutation $\sigma_1, \sigma_2, \sigma_3$,
\begin{align*}
\tilde{Q}^{(z_a)}_J=\tilde{Q}^{(z_a)}_{(a,\bs{j},\bs{h},\bs{k},n)}=\tilde{Q}^{(z_a)}_{(a,\sigma_{1}(\bs{j}),\sigma_{2}(\bs{h}),\sigma_{3}(\bs{k}),n)},
\\\tilde{Q}^{(\bar{z}_a)}_J=\tilde{Q}^{(\bar{z}_a)}_{(a,\bs{j},\bs{h},\bs{k},n)}=\tilde{Q}^{(\bar{z}_a)}_{(a,\sigma_{1}(\bs{j}),\sigma_{2}(\bs{h}),\sigma_{3}(\bs{k}),n)}.
\end{align*} 
Denote by $\ms{H}_{l,N}$ the set of the above rational vector fields.
Especially, for any $J\in\Gamma_{l}\bigcup\Gamma'_{l}$, if the coefficients satisfy the reversible condition
\begin{equation}
\label{formrev-1}
\tilde{Q}^{(z_a)}_{J}=-\tilde{Q}^{(\bar{z}_a)}_J,
\end{equation}
then we say that the rational vector field $Q_{\Gamma_{l},\Gamma'_{l}}$ belongs to the set $\ms{H}_{l, N}^{{\rm rev}}$;
and if the coefficients satisfy the anti-reversible condition
\begin{equation}
\label{formanrev-1}
\tilde{Q}^{(z_a)}_{J}=\tilde{Q}^{(\bar{z}_a)}_J,
\end{equation}
then we say that the rational vector field $Q_{\Gamma_{l},\Gamma'_{l}}$ belongs to the set $\ms{H}_{l, N}^{{\rm an-rev}}$.
%
\end{definition}

The following three remarks help us better understand the above two definitions.
\begin{remark}
\label{re41}
There are some properties of rational vector fields.
\begin{enumerate}[(1)]
\item\label{re41-1} The condition $\bs{h}_{m},\bs{k}_{m}\in Irr(\mc{R}_{\leq N})$ ensures that the denominators in \Cref{def42} are not equal to zero and only depend on the action $\{I_d\}_{d\leq N}$.
\item \label{re41-0} For any $Q_{\Gamma_{l},\Gamma'_{l}}\in\ms{H}_{l,N}$, one has 
$$\overline{Q^{(z)}_{\Gamma_{l}}(\bs{z})}=Q^{(\bar{z})}_{\Gamma_{l}}(\bs{z}),\quad \overline{Q^{(z)}_{\Gamma'_{l}}(\bs{z})}=Q^{(\bar{z})}_{\Gamma'_{l}}(\bs{z}),$$ 
and thus the $\bar{z}$-component $Q^{(\bar{z})}_{\Gamma_{l},\Gamma'_{l}}\pa_{\bar{z}}$ is uniquely determined by the $z$-component $Q^{(z)}_{\Gamma_{l},\Gamma'_{l}}\pa_{z}$. 
Especially, the coefficients of $Q_{\Gamma_{l},\Gamma'_{l}}\in \ms{H}^{{\rm rev}}_{l,N}$ are purely imaginary, and the coefficients of $Q_{\Gamma_{l},\Gamma'_{l}}(\bs{z})\in \ms{H}^{{\rm an-rev}}_{l,N}$ are real.
\item \label{re41-2} 
 The $z$-component \eqref{form419} is rewritten as
\begin{equation}
\label{formre12-31}
Q_{\Gamma_{l},\Gamma'_l}^{(z)}(\bs{z})\pa_{z}=\sum_{a\in\mb{N}_*}z_aQ_{\Gamma^a_{l}}^{(z_a)}(\bs{z})\pa_{z_a}+\sum_{a\leq N}\bar{z}_aQ_{\Gamma'^a_{l}}^{(z_a)}(\bs{z})\pa_{z_a}
\end{equation}
with 
\begin{align}
\label{form423}
Q^{(z_a)}_{\Gamma_{l}^a}(\bs{z})=\sum_{(\alpha,\beta,\beta')\in\mc{F}_{l}}\sum_{(a,\bs{j},\bs{h},\bs{k},n)\in\Gamma_{l}^{(a,\alpha,\beta,\beta')}}\tilde{Q}^{(z_a)}_{J}\frac{\zeta_{\bs{j}}}{\prod\limits_{m=1}^{n}\Omega^{(2)}_{\bs{h}_{m}}\prod\limits_{m=n+1}^{\beta}\Omega^{(4)}_{\bs{h}_{m}}\prod\limits_{m=1}^{\beta'}\Omega^{(4)}_{\bs{k}_{m}}},
\\\label{form424}
Q^{(z_a)}_{\Gamma'^a_{l}}(\bs{z})=\sum_{(\alpha,\beta,\beta')\in\mc{F}_{l}}\sum_{(a,\bs{j},\bs{h},\bs{k},n)\in\Gamma_{l}'^{(a,\alpha,\beta,\beta')}}\tilde{Q}^{(z_a)}_{J}\frac{\zeta_{\bs{j}}}{\prod\limits_{m=1}^{n}\Omega^{(2)}_{\bs{h}_{m}}\prod\limits_{m=n+1}^{\beta}\Omega^{(4)}_{\bs{h}_{m}}\prod\limits_{m=1}^{\beta'}\Omega^{(4)}_{\bs{k}_{m}}},
\end{align}
where  $\mc{F}_{l}\subseteq\mb{N}^{3}$ is a set with its elements satisfying $\alpha-\beta-2\beta'=l$, and
\begin{align*}
\Gamma^{(a,\alpha, \beta, \beta')}_{l}=&\Gamma_{l}^a\bigcap\Big(\{a\}\times\mc{R}^0_{\alpha,N}\times\big(Irr(\mc{R}_{\leq N})\big)^{\beta}\times\big(Irr(\mc{R}_{\leq N})\big)^{\beta'}\times\mb{N}\Big), 
\\\Gamma'^{(a,\alpha, \beta, \beta')}_{l}=&\Gamma_{l}'^a\bigcap\Big(\{a\}\times\mc{R}_{\alpha,N}^a\times\big(Irr(\mc{R}_{\leq N})\big)^{\beta}\times\big(Irr(\mc{R}_{\leq N})\big)^{\beta'}\times\mb{N}\Big).
\end{align*} 
In particular, by assuming that 
$$\mc{J}_{\Gamma_l,\Gamma'_l}:=\sup_{(a,\bs{j},\bs{h},\bs{k},n)\in\Gamma_{l}\bigcup\Gamma'_{l}}\{\#\bs{j}\}<+\infty,$$
it follows that $\mc{F}_{l}\subseteq\mb{N}^{3}$ is a finite set with $\#\mc{F}_{l}\leq\mc{J}_{\Gamma_l,\Gamma'_l}^3$. 
It is similar for the $\bar{z}$-component \eqref{form420}.
%
\end{enumerate}
\end{remark}

\begin{remark}
\label{re41-6} 
There are several notes for small divisors in \Cref{def42}.
\begin{enumerate}[(1)]
\item\label{re41-3} The denominators in the solutions of homological equations are constituted by the denominators from the preceding terms, along with newly introduced small divisors. These small divisors are determined by the numerators in the previous terms. 
\item \label{re41-7} For two rational vector fields, the denominators of their commutator are composed of the denominators from each of the vector fields, supplemented by additional small divisors. These additional small divisors are also derived from the denominators of the vector fields. 
\end{enumerate}
In summary, during the iterative process, the small divisors in the denominators of rational vector field \eqref{form418} arise from solving homological equations and commutator operations. Thus it is natural to propose the conditions (1-3),  (1-4) and \eqref{def41-2} in \Cref{def41}, where the first two conditions imply that the lengths of all indices $\bs{h}_{m},\bs{k}_{m}$ in the denominator are controlled by the length of $\bs{j}$ in the numerator.
\end{remark}

\begin{remark}
There are some notes for the control condition \eqref{form416} in \Cref{def41}.
\begin{enumerate}[(1)]
\item\label{re41-4} When we use the small divisor conditions \eqref{form49} and \eqref{form410} to estimate $\Omega_{\bs{h}_m}^{(2)}$, $\Omega_{\bs{h}_m}^{(4)}$ and $\Omega_{\bs{k}_m}^{(4)}$ in \eqref{form419-10-3-2} and \eqref{form420-10-3-2}, this control condition  is adequate to bound the rational vector fields in the iterative process. 
\item\label{re41-5} 
This control condition ensures that the solution of the homological equation is also well defined by \Cref{def42}.
\item\label{re41-5-420} This control condition is well kept in the commutator of rational vector fields.
\end{enumerate}
\end{remark}

Then we estimate resonant rational vector fields.
\begin{lemma}
\label{le42}
Fix $r\geq2$, $s\geq0$, $N\geq1$ and $\gamma\in(0,1)$. 
Being given $\Gamma_{l}, \Gamma'_{l}\in\mc{H}_{l,N}$ satisfying $\mc{J}_{\Gamma_l,\Gamma'_l}<+\infty$,
consider the rational vector field $Q_{\Gamma_{l},\Gamma'_{l}}\in\ms{H}_{l,N}$.
%
 For $z\in\mc{U}_{\gamma}^{N}$, one has
\begin{equation}
\label{form425}
\|Q_{\Gamma_l,\Gamma'_l}(\bs{z})\|_s\leq 6\sqrt{2}\mc{J}_{\Gamma_l,\Gamma'_l}^3\max_{(\alpha,\beta,\beta')\in\mc{F}_{l}}\frac{(12\alpha^2)^{\alpha-1}N^{(4\alpha+2)(\beta+\beta')}}{\gamma^{\beta+2\beta'}}\|Q_{\Gamma_{l},\Gamma'_{l}}\|_{\ell^{\infty}}\|z\|_{s}^{2l+1}.
\end{equation}
%
\end{lemma}
\begin{proof}
In view of \eqref{re41-2} in \Cref{re41},  rewrite the $z$-component \eqref{formre12-31} in the form
$$Q_{\Gamma_{l},\Gamma'_l}^{(z)}(\bs{z})\pa_z=\sum_{a\in\mb{N}_*}\sum_{(\alpha,\beta,\beta')\in\mc{F}_{l}}\big(Q_1^{(a,\alpha,\beta,\beta')}(\bs{z})+Q_2^{(a,\alpha,\beta,\beta')}(\bs{z})\big)\pa_{z_a}$$ 
with
\begin{align}
\label{form339-1-2-1}
&Q_1^{(a,\alpha,\beta,\beta')}(\bs{z})=z_a\sum_{J:=(a,\bs{j},\bs{h},\bs{k},n)\in\Gamma_{l}^{(a,\alpha,\beta,\beta')}}\tilde{Q}^{(z_a)}_{J}\zeta_{\bs{j}}f_{J}(I),
\\&Q_2^{(a,\alpha,\beta,\beta')}(\bs{z})=\left\{\begin{aligned}
\bar{z}_a\sum_{J:=(a,\bs{j},\bs{h},\bs{k},n)\in\Gamma_{l}'^{(a,\alpha,\beta,\beta')}}\tilde{Q}^{(z_a)}_{J}\zeta_{\bs{j}}f_{J}(I), &\quad\text{for}\;a\leq N,
\\0,\qquad\qquad\qquad&\quad\text{for}\;a>N,\end{aligned}\right.
\end{align}
where
\begin{equation}
\label{form426}
f_{J}(I)=\frac{1}{\prod\limits_{m=1}^{n}\Omega^{(2)}_{\bs{h}_{m}}(I)\prod\limits_{m=n+1}^{\beta}\Omega^{(4)}_{\bs{h}_{m}}(I)\prod\limits_{m=1}^{\beta'}\Omega^{(4)}_{\bs{k}_{m}}(I)}.
\end{equation}
According to \eqref{re41-0} in \Cref{re41} and $\#\mc{F}_{l}\leq\mc{J}_{\Gamma_l,\Gamma'_l}^3$, one has
\begin{align}
\label{form427}
&\|Q_{\Gamma_{l},\Gamma'_l}(\bs{z})\|_{s}=\sqrt{2}\|Q_{\Gamma_{l},\Gamma'_l}^{(z)}(\bs{z})\|_s
\\\notag\leq&\sqrt{2}\mc{J}_{\Gamma_l,\Gamma'_l}^3\max_{(\alpha,\beta,\beta')\in\mc{F}_{l}}\Big(\big(\sum_{a\in\mb{N}_*}a^{2s}\big|Q^{(a,\alpha,\beta,\beta')}_1(\bs{z})\big|^2_{s}\big)^{\frac{1}{2}}+\big(\sum_{a\in\mb{N}_*}a^{2s}\big|Q^{(a,\alpha,\beta,\beta')}_2(\bs{z})\big|^2_{s}\big)^{\frac{1}{2}}\Big).
\end{align}

In view of \eqref{form339-1-2-1}, one has
\begin{align}
\label{form428}
&\Big(\sum_{a\in\mb{N}_*}a^{2s}\big|Q^{(a,\alpha,\beta,\beta')}_1(\bs{z})\big|^2_{s}\Big)^{\frac{1}{2}}
\\\notag=&\Big(\sum_{a\in\mb{N}_*}a^{2s}\big|z_a\sum_{J\in\Gamma_{l}^{(a,\alpha,\beta,\beta')}}\tilde{Q}^{(z_a)}_{J}\zeta_{\bs{j}}f_{J}(I)\big|^2\Big)^{\frac{1}{2}}
\\\notag\leq&\|z\|_{s}\|Q_{\Gamma_{l},\Gamma'_{l}}\|_{\ell^{\infty}}\sup_{a\in\mb{N}_*}\sum_{J\in\Gamma_{l}^{(a,\alpha,\beta,\beta')}}|\zeta_{\bs{j}}||f_{J}(I)|.
\end{align}
By the small divisor conditions \eqref{form49} and \eqref{form410}, one has
\begin{equation}
\label{form430}
|\Omega^{(2)}_{\bs{h}_{m}}|,|\Omega^{(4)}_{\bs{h}_{m}}|>\gamma\|z\|_{s}^{2}N^{-4\alpha-2}\kappa_{\bs{h}_{m}}^{-2s}
\quad\text{and}\quad
|\Omega^{(4)}_{\bs{k}_{m}}|>\gamma^2\|z\|_{s}^{4}N^{-4\alpha-2}.
\end{equation}
Hence, by \eqref{form426}, \eqref{form430} and the control condition \eqref{form416}, one has
\begin{align}
\label{form431}
|f_{J}(I)|&\leq\Big(\prod_{m=1}^{\beta}\kappa^{2s}_{\bs{h}_{m}}\Big)\Big(\frac{N^{4\alpha+2}}{\gamma\|z\|_{s}^{2}}\big)^{\beta}\Big(\frac{N^{4\alpha+2}}{\gamma^{2}\|z\|_{s}^{4}}\Big)^{\beta'}
\\\notag&\leq\big(\prod_{m=1}^{\alpha}j_m^*\big)^{2s}\frac{N^{(4\alpha+2)(\beta+\beta')}}{\gamma^{\beta+2\beta'}\|z\|_{s}^{2\beta+4\beta'}}
\end{align}
with $\#\bs{j}=\alpha$.
By \eqref{form431}, the condition \eqref{def41-2} of \eqref{form414} and \eqref{formzetanorm}, one has
\begin{align}
\label{form432}
&\sum_{J\in\Gamma_{l}^{(a,\alpha,\beta,\beta')}}|\zeta_{\bs{j}}||f_{J}(I)|
\\\notag\leq&\frac{N^{(4\alpha+2)(\beta+\beta')}}{\gamma^{\beta+2\beta'}\|z\|_{s}^{2\beta+4\beta'}}\sum_{J\in\Gamma_{l}^{(a,\alpha,\beta,\beta')}}\big(\prod_{m=1}^{\alpha}j_m^*\big)^{2s}|\zeta_{\bs{j}}|
\\\notag\leq&\frac{N^{(4\alpha+2)(\beta+\beta')}}{\gamma^{\beta+2\beta'}\|z\|_{s}^{2\beta+4\beta'}}(2\alpha)^{2\alpha-2}\sum_{\bs{j}\in(\mb{U}_3\times\mb{N}_*)^{\alpha}}\big(\prod_{m=1}^{\alpha}j_m^*\big)^{2s}|\zeta_{\bs{j}}|
\\\notag\leq&(2\alpha)^{2\alpha-2}\frac{N^{(4\alpha+2)(\beta+\beta')}}{\gamma^{\beta+2\beta'}\|z\|_{s}^{2\beta+4\beta'}}\big(3\|z\|_s^2\big)^{\alpha}
\\\notag=&3(12\alpha^2)^{\alpha-1}\frac{N^{(4\alpha+2)(\beta+\beta')}}{\gamma^{\beta+2\beta'}}\|z\|_{s}^{2l},
\end{align}
where the last inequality follows from the fact $\alpha-\beta-2\beta'=l$.
By \eqref{form428} and \eqref{form432}, one has 
\begin{align}
\label{form433}
\big(\sum_{a\in\mb{N}_*}a^{2s}\big|Q^{(a,\alpha,\beta,\beta')}_1(\bs{z})\big|^2_{s}\big)^{\frac{1}{2}}
\leq3(12\alpha^2)^{\alpha-1}\|Q_{\Gamma_{l},\Gamma'_{l}}\|_{\ell^{\infty}}\frac{N^{(4\alpha+2)(\beta+\beta')}}{\gamma^{\beta+2\beta'}}\|z\|_{s}^{2l+1}.
\end{align}
Similarly, one has
\begin{align}
\label{form434}
\big(\sum_{a\in\mb{N}_*}a^{2s}\big|Q^{(a,\alpha,\beta,\beta')}_2(\bs{z})\big|^2_{s}\big)^{\frac{1}{2}}
\leq3(12\alpha^2)^{\alpha-1}\|Q_{\Gamma_{l},\Gamma'_{l}}\|_{\ell^{\infty}}\frac{N^{(4\alpha+2)(\beta+\beta')}}{\gamma^{\beta+2\beta'}}\|z\|_{s}^{2l+1}.
\end{align}
To sum up, the estimate \eqref{form425} follows from \eqref{form427}, \eqref{form433} and \eqref{form434}.
\end{proof}

\subsection{Commutator of rational vector fields}
\label{sec33}
The following lemma shows that the structure of reversible rational vector field is well kept in the commutator of a reversible rational vector field and an anti-reversible rational vector field.

\begin{lemma}
\label{le43}
For $Q_{\Gamma_{l_1},\Gamma'_{l_1}}\in\ms{H}^{{\rm rev}}_{l_{1},N}$ and $Q_{\Gamma_{l_2},\Gamma'_{l_2}}\in\ms{H}^{{\rm an-rev}}_{l_{2},N}$, there exists $Q_{\Gamma_{l},\Gamma'_{l}}\in\ms{H}^{{\rm rev}}_{l,N}$ with $l:=l_{1}+l_{2}$ and $\mc{J}_{\Gamma_{l},\Gamma'_{l}}\leq \mc{J}_{\Gamma_{l_1},\Gamma'_{l_1}}+\mc{J}_{\Gamma_{l_2},\Gamma'_{l_2}}+2$, such that
\begin{equation}
\label{form435}
[Q_{\Gamma_{l_1},\Gamma'_{l_1}},Q_{\Gamma_{l_2},\Gamma_{l_2}}]=Q_{\Gamma_{l},\Gamma'_{l}}.
\end{equation}
Moreover, by assuming that $\mc{J}_{\Gamma_{l_1},\Gamma'_{l_1}}\lesssim l_1,\mc{J}_{\Gamma_{l_2},\Gamma'_{l_2}}\lesssim l_2$,  there exists a positive constant $C_2$ such that
\begin{equation}
\label{form436}
\|Q_{\Gamma_{l},\Gamma'_{l}}\|_{\ell^{\infty}}\leq   C_2l^4\|Q_{\Gamma_{l_1},\Gamma'_{l_1}}\|_{\ell^{\infty}}\|Q_{\Gamma_{l_2},\Gamma'_{l_2}}\|_{\ell^{\infty}}.
\end{equation}
\end{lemma}
\begin{proof}
Similarly with the commutator of polynomial vector fields in \Cref{le22-12}, the commutator of a reversible vector field and an anti-reversible vector field is a reversible vector field. In the following,  we show that the structure of rational fraction is well kept in commutator of rational vector fields.

In view of \Cref{def42}, rewrite \eqref{form418} in the form
$$Q_{\Gamma_{l},\Gamma'_{l}}(\bs{z})=\sum_{a\in\mb{N}_*}\Big(Q_{\Gamma_{l}^a,\Gamma'^a_{l}}^{(z_a)}(\bs{z})\pa_{z_a}+\overline{Q_{\Gamma_{l}^a,\Gamma'^a_{l}}^{(z_a)}(\bs{z})}\pa_{\bar{z}_a}\Big)$$
where for $a\leq N$, the $z_a$-component
\begin{align*}
Q_{\Gamma^a_{l},\Gamma'^a_{l}}^{(z_a)}(\bs{z})=z_a\sum_{J:=(a,\bs{j},\bs{h},\bs{k},n)\in\Gamma^a_{l}}\tilde{Q}^{(z_a)}_{J}f_J(I)\zeta_{\bs{j}}+\bar{z}_a\sum_{J:=(a,\bs{j},\bs{h},\bs{k},n)\in\Gamma'^a_{l}}\tilde{Q}^{(z_a)}_{J}f_J(I)\zeta_{\bs{j}}
\end{align*}
and for $a>N$, the $z_a$-component
\begin{align*}
Q_{\Gamma^a_{l},\Gamma'^a_{l}}^{(z_a)}(\bs{z})=z_a\sum_{J:=(a,\bs{j},\bs{h},\bs{k},n)\in\Gamma^a_{l}}\tilde{Q}^{(z_a)}_{J}f_J(I)\zeta_{\bs{j}},
\end{align*}
with 
$$f_J=\frac{1}{\prod\limits_{m=1}^{n}\Omega^{(2)}_{\bs{h}_{m}}}\frac{1}{\prod\limits_{m=n+1}^{\#\bs{h}}\Omega^{(4)}_{\bs{h}_{m}}}\frac{1}{\prod\limits_{m=1}^{\#\bs{k}}\Omega^{(4)}_{\bs{k}_{m}}}:=f^{(2)}_Jf^{(4,2)}_Jf^{(4,4)}_J.$$
We discuss the $z_a$-component $[Q_{\Gamma_{l_1},\Gamma'_{l_1}}, Q_{\Gamma_{l_2},\Gamma'_{l_2}}]^{(z_a)}$ in four different cases.

1) Consider the first case that denominators do not appear in the commutator. Then $Q^{(z_a)}_{\Gamma_{l}^{(1)},\Gamma_{l}'^{(1)}}$ is of the form
\begin{align}
\label{formq1}
&\sum_{J'\in\Gamma^a_{l_1}}
\tilde{Q}^{(z_a)}_{J'}f_{J'}(I)\frac{\pa(z_a\zeta_{\bs{j}'})}{\pa\bs{z}}Q_{\Gamma_{l_2},\Gamma'_{l_2}}^{(\bs{z})}+\sum_{J'\in\Gamma'^a_{l_1}}
\tilde{Q}^{(z_a)}_{J'}f_{J'}(I)\frac{\pa(\bar{z}_a\zeta_{\bs{j}'})}{\pa\bs{z}}Q_{\Gamma_{l_2},\Gamma'_{l_2}}^{(\bs{z})}
\\\notag-&\sum_{J''\in\Gamma^a_{l_2}}\tilde{Q}^{(z_a)}_{J''}f_{J''}(I)\frac{\pa(z_a\zeta_{\bs{j}''})}{\pa\bs{z}}Q_{\Gamma_{l_1},\Gamma'_{l_1}}^{(\bs{z})}-\sum_{J''\in\Gamma'^a_{l_2}}
\tilde{Q}^{(z_a)}_{J''}f_{J''}(I)\frac{\pa(\bar{z}_a\zeta_{\bs{j}''})}{\pa\bs{z}}Q_{\Gamma_{l_1},\Gamma'_{l_1}}^{(\bs{z})}.
\end{align}
For any $J\in\Gamma_{l}^{(1)}\bigcup\Gamma_{l}'^{(1)}$, there exist $J'\in\Gamma_{l_1}\bigcup\Gamma_{l_1}'$ and $J''\in\Gamma_{l_2}\bigcup\Gamma_{l_2}'$ such that 
\begin{equation}
\label{form7-10-1}
\#\bs{j}=\#\bs{j}'+\#\bs{j}'',\quad \bs{h}=(\bs{h}',\bs{h}''),\quad \bs{k}=(\bs{k}',\bs{k}''),\quad n=n_1+n_2.
\end{equation}
Hence,
\begin{align}
\label{form7-12-2}
\#\bs{j}-\#\bs{h}-2\#\bs{k}&=\#\bs{j}'+\#\bs{j}''-(\#\bs{h}'+\#\bs{h}'')-2(\#\bs{k}'+\#\bs{k}'')
\\\notag&=l_{1}+l_{2}=l.
\end{align}

Similarly with the proof of \Cref{le22-12}, there are at most one different index in $\bs{j}$ with $(\bs{j}',\bs{j}'')$, and the difference is in the set $\mb{U}_3$. 
Then in view of \eqref{form7-10-1},

$\star$ if $\bs{j}$  and $\#\bs{j}'=\alpha_1$ are fixed, then $\#\bs{j}''=\alpha-\alpha_1$, and there are at most $3\alpha$ different $(\bs{j}',\bs{j}'')$; 

$\star$ fixing $a$, if $a'=a$, then there are at most $\alpha_1+1$ different $a''$, and if $a''=a$, then there are at most $\alpha-\alpha_1+1$ different $a'$, i.e., there are at most $\alpha+2$ different $(a',a'')$;

$\star$ if $J'$, $J''$ are fixed, then $\bs{h},\bs{k},n$ are determined. Hence, one has
\begin{align}
\label{form7-12-1}
&\max_{a,\; \bs{j}\in\mc{R}_{\alpha,N}^{0}\cup\mc{R}_{\alpha,N}^{a}}\#(\Gamma_{l}\cup\Gamma'_{l};\ a, \bs{j})
\\\notag\leq&\sum_{\alpha_1=2}^{\alpha-2}3\alpha(\alpha+2)\max_{a',\; \bs{j}'\in\mc{R}_{\alpha_1,N}^{0}\cup\mc{R}_{\alpha_1,N}^{a'}}\#(\Gamma_{l_1}\cup\Gamma'_{l_1};\ a', \bs{j}')
\\\notag&\quad\qquad\qquad\max_{a'',\; \bs{j}''\in\mc{R}_{\alpha-\alpha_1,N}^{0}\cup\mc{R}_{\alpha-\alpha_1,N}^{a''}}\#(\Gamma_{l_2}\cup\Gamma'_{l_2};\ a'', \bs{j}'')
\\\notag\leq&3\alpha(\alpha+2)\sum_{\alpha_1=2}^{\alpha-2}(2\alpha_1)^{2\alpha_1-3}(2\alpha-2\alpha_1)^{2\alpha-2\alpha_1-3}
\\\notag<& \frac{3}{8}(2\alpha)^{2\alpha-3},
\end{align}
where the last inequality follows from 
$$(\alpha+2)(\alpha-3)<\alpha^2\quad\text{and}\quad (2\alpha_1)^{2\alpha_1-3}(2\alpha-2\alpha_1)^{2\alpha-2\alpha_1-3}<(2\alpha)^{2\alpha-6}.$$

%
%
In view of \eqref{form7-10-1}, one has
\begin{align*}
\prod_{m=1}^{\#\bs{h}}\kappa_{\bs{h}_{m}}&=\Big(\prod_{m=1}^{\#\bs{h}'}\kappa_{\bs{h}'_{m}}\Big)\Big(\prod_{m=1}^{\#\bs{h}''}\kappa_{\bs{h}''_{m}}\Big)
\\&\leq\prod_{m=1}^{\#\bs{j}'}(j'_{m})^*\prod_{m=1}^{\#\bs{j}''}(j''_{m})^*
\\&=\prod_{m=1}^{\#\bs{j}}j^*_{m},
\end{align*}
which is \eqref{form416} in \Cref{def41}. Hence, one has $\Gamma_{l}^{(1)},\Gamma_{l}'^{(1)}\in\mc{H}_{l,N}$ with $l=l_1+l_2$.

Rewrite \eqref{formq1} in the form \eqref{form418}--\eqref{form420-10-3-2}. In view of \eqref{form7-10-1}, for any given $J\in\Gamma^{(1)}_l\bigcup\Gamma'^{(1)}_l$ , if $n_1,\#\bs{h}',\#\bs{k}'$ are fixed, then $n_2, \bs{h}',\bs{h}'',\bs{k}',\bs{k}'',\#\bs{j}',\#\bs{j}''$ are determined, and similarly with the proof of \Cref{le22-12},  there are at most $3\#\bs{j}+1$ different $((a',\bs{j}'), (a'',\bs{j}''))$. By $n_1,\#\bs{h}',\#\bs{k}',\#\bs{j}\leq \mc{J}_{\Gamma_{l},\Gamma'_{l}}$, one has 
\begin{align}
\label{form7-14-1}
\|Q_{\Gamma_{l}^{(1)},\Gamma_{l}'^{(1)}}\|_{\ell^{\infty}}\lesssim \mc{J}_{\Gamma_{l},\Gamma'_{l}}^4\|Q_{\Gamma_{l_1},\Gamma_{l_1}'}\|_{\ell^{\infty}}\|Q_{\Gamma_{l_2},\Gamma_{l_2}'}\|_{\ell^{\infty}},
\end{align}
and thus $Q_{\Gamma_{l}^{(1)},\Gamma_{l}'^{(1)}}\in\ms{H}_{l,N}^{{\rm rev}}$.
%

2) Consider the second case that $\Omega_{\bs{h}}^{(2)}$ appears in the commutator. Without loss of generality, let $n_{1},n_2\geq1$ and then $Q^{(z_a)}_{\Gamma_{l}^{(2)},\Gamma_{l}'^{(2)}}$ is of the form
\begin{align}
\label{formq2}
&z_a\sum_{J'\in\Gamma^a_{l_1}}
\tilde{Q}^{(z_a)}_{J'}f^{(4,2)}_{J'}f^{(4,4)}_{J'}\zeta_{\bs{j}'}\frac{\pa f^{(2)}_{J'}}{\pa\bs{z}}Q_{\Gamma_{l_2},\Gamma'_{l_2}}^{(\bs{z})}
+\bar{z}_a\sum_{J'\in\Gamma'^a_{l_1}}
\tilde{Q}^{(z_a)}_{J'}f^{(4,2)}_{J'}f^{(4,4)}_{J'}\zeta_{\bs{j}'}\frac{\pa f^{(2)}_{J'}}{\pa\bs{z}}Q_{\Gamma_{l_2},\Gamma'_{l_2}}^{(\bs{z})}
\\\notag-&z_a\sum_{J''\in\Gamma^a_{l_2}}\tilde{Q}^{(z_a)}_{J''}f^{(4,2)}_{J''}f^{(4,4)}_{J''}\zeta_{\bs{j}''}\frac{\pa f^{(2)}_{J''}}{\pa\bs{z}}Q_{\Gamma_{l_1},\Gamma'_{l_1}}^{(\bs{z})}
-\bar{z}_a\sum_{J''\in\Gamma'^a_{l_2}}
\tilde{Q}^{(z_a)}_{J''}f^{(4,2)}_{J''}f^{(4,4)}_{J''}\zeta_{\bs{j}''}\frac{\pa f^{(2)}_{J''}}{\pa\bs{z}}Q_{\Gamma_{l_1},\Gamma'_{l_1}}^{(\bs{z})}.
\end{align}
Then for any $J\in\Gamma_{l}^{(2)}\bigcup\Gamma_{l}'^{(2)}$, there exist $J'\in\Gamma_{l_1}\bigcup\Gamma_{l_1}'$ and $J''\in\Gamma_{l_2}\bigcup\Gamma_{l_2}'$ such that  
\begin{equation}
\label{form7-10-2}
\begin{array}{l}
\displaystyle \#\bs{j}=\#\bs{j}'+\#\bs{j}''+1,
\\\displaystyle \bs{h}=(\bs{h}_i,\bs{h}',\bs{h}'') \quad \text{with}\quad\bs{h}_i\in\{\bs{h}'_1,\cdots,\bs{h}'_{n_1},\bs{h}''_1,\cdots,\bs{h}''_{n_2}\},
\\\displaystyle \bs{k}=(\bs{k}',\bs{k}''),
\\\displaystyle n=n_1+n_2+1. 
\end{array}
\end{equation}
Hence,
\begin{align*}
\#\bs{j}-\#\bs{h}-2\#\bs{k}&=\#\bs{j}'+\#\bs{j}''+1-(\#\bs{h}'+\#\bs{h}''+1)-2(\#\bs{k}'+\#\bs{k}'')
\\&=l_{1}+l_{2}=l.
\end{align*}

In view of \eqref{form7-10-2}, write $\bs{j}=(j,\bs{j}',\bs{j}'')$ with $j\in\mb{U}_3\times\mb{N}_*$.

$\star$ If $\bs{j}$ and $\#\bs{j}'=\alpha_1$ are fixed, then $\bs{j}',\bs{j}''$ are determined and $\#\bs{j}''=\alpha-\alpha_1-1$; 

$\star$ fixing $a$, one has $(a',a'')=(a,|j|)$ or $(a',a'')=(|j|,a)$; 

$\star$ if $J'$, $J''$ are fixed, then $\bs{k},n$ are determined and there are at
most $n_1+n_2$ different $\bs{h}$. Hence, by the fact $n_1+n_2\leq\#\bs{h}'+\#\bs{h}''<\alpha$, one has
\begin{align}
\label{form7-12-2'}
&\max_{a,\; \bs{j}\in\mc{R}_{\alpha,N}^{0}\cup\mc{R}_{\alpha,N}^{a}}\#(\Gamma_{l}^{(2)}\cup\Gamma'^{(2)}_{l};\ a, \bs{j})
\\\notag<&\sum_{\alpha_1=2}^{\alpha-2}2\alpha\max_{a',\; \bs{j}'\in\mc{R}_{\alpha_1,N}^{0}\cup\mc{R}_{\alpha_1,N}^{a'}}\#(\Gamma_{l_1}\cup\Gamma'_{l_1};\ a', \bs{j}')
\\\notag&\;\quad\max_{a'',\; \bs{j}''\in\mc{R}_{\alpha-\alpha_1-1,N}^{0}\cup\mc{R}_{\alpha-\alpha_1-1,N}^{a''}}\#(\Gamma_{l_2}\cup\Gamma'_{l_2};\ a'', \bs{j}'')
\\\notag\leq&2\alpha\sum_{\alpha_1=2}^{\alpha-2}(2\alpha_1)^{2\alpha_1-3}(2\alpha-2\alpha_1-2)^{2(\alpha-\alpha_1-1)-3}
\\\notag<& \frac{1}{2}(2\alpha)^{2\alpha-6},
\end{align}
where the last inequality follows from 
$(2\alpha_1)^{2\alpha_1-3}(2\alpha-2\alpha_1-2)^{2\alpha-2\alpha_1-5}<(2\alpha)^{2\alpha-8}$.

Now, we will prove the control condition \eqref{form416}. In view of \eqref{formq2}, if $\Delta_{\bs{h}_i}=0$, then  $|j|\geq\mu_{\min}(\bs{h}_i)=\kappa_{\bs{h}_i}$, and if $\Delta_{\bs{h}_i}^{(0)}>0$, then  $|j|\geq\min\{\mu_{\min}(\bs{h}_i),\Delta_{\bs{h}_i}\}=\kappa_{\bs{h}_i}$. Thus $\kappa_{\bs{h}_i}\leq j$, and then
\begin{align*}
\prod_{m=1}^{\#\bs{h}}\kappa_{\bs{h}_{m}}&=\kappa_{\bs{h}_{i}}\Big(\prod_{m=1}^{\#\bs{h}'}\kappa_{\bs{h}'_{m}}\Big)\Big(\prod_{m=1}^{\#\bs{h}''}\kappa_{\bs{h}''_{m}}\Big)
\\&\leq |j|\prod_{m=1}^{\#\bs{j}'}(j'_{m})^*\prod_{m=1}^{\#\bs{j}''}(j''_{m})^*
\\&=\prod_{m=1}^{\#\bs{j}}j_{m}^*.
\end{align*}
Hence, one has $\Gamma_{l}^{(2)},\Gamma_{l}'^{(2)}\in\mc{H}_{l,N}$ with $l=l_1+l_2$. 

Rewrite \eqref{formq2} in the form \eqref{form418}--\eqref{form420-10-3-2}. 
In view of \eqref{form7-10-2}, for any given $J\in\Gamma_{l}^{(2)}\bigcup\Gamma_l'^{(2)}$, if $\#\bs{j}',\#\bs{k}',n_1$ are fixed, then $J',J''$ are determined. 
Notice that the coefficient $\tilde{Q}^{(z_a)}_{J}$ depends on the derivate of $\Omega_{\bs{h}_i}^{(2)}$, which is less than $\|\Omega_{\bs{h}_i}^{(2)}\|_{\ell^{\infty}}$ up to a constant multiply. In view of the definitions \eqref{form43} and \eqref{formomega2}, the coefficients of $\Omega^{(2)}$ satisfy $\|\Omega^{(2)}_{\bs{h}_i}\|_{\ell^{\infty}}\leq\frac{1}{2}$. By $\#\bs{j}',\#\bs{k}',n_1\leq \mc{J}_{\Gamma_{l},\Gamma'_{l}}$, one has 
\begin{align}
\label{form7-14-2}
\|Q_{\Gamma_{l}^{(1)},\Gamma_{l}'^{(1)}}\|_{\ell^{\infty}}\lesssim \mc{J}_{\Gamma_{l},\Gamma'_{l}}^3\|Q_{\Gamma_{l_1},\Gamma_{l_1}'}\|_{\ell^{\infty}}\|Q_{\Gamma_{l_2},\Gamma_{l_2}'}\|_{\ell^{\infty}},
\end{align}
and then $Q_{\Gamma_{l}^{(2)},\Gamma_{l}'^{(2)}}\in\ms{H}_{l, N}^{{\rm rev}}$.

3) Consider the third case that $\Omega^{(4)}_{\bs{h}}$ appears in the commutator. Then $Q^{(z_a)}_{\Gamma_{l}^{(3)},\Gamma_{l}'^{(3)}}$ is of the form
\begin{align}
\label{formq3}
&z_a\sum_{J'\in\Gamma^a_{l_1}}
\tilde{Q}^{(z_a)}_{J'}f^{(2)}_{J'}f^{(4,4)}_{J'}\zeta_{\bs{j}'}\frac{\pa f^{(4,2)}_{J'}}{\pa\bs{z}}Q_{\Gamma_{l_2},\Gamma'_{l_2}}^{(\bs{z})}
+\bar{z}_a\sum_{J'\in\Gamma'^a_{l_1}}
\tilde{Q}^{(z_a)}_{J'}f^{(2)}_{J'}f^{(4,4)}_{J'}\zeta_{\bs{j}'}\frac{\pa f^{(4,2)}_{J'}}{\pa\bs{z}}Q_{\Gamma_{l_2},\Gamma'_{l_2}}^{(\bs{z})}
\\\notag-&z_a\sum_{J''\in\Gamma^a_{l_2}}\tilde{Q}^{(z_a)}_{J''}f^{(2)}_{J''}f^{(4,4)}_{J''}\zeta_{\bs{j}''}\frac{\pa f^{(4,2)}_{J''}}{\pa\bs{z}}Q_{\Gamma_{l_1},\Gamma'_{l_1}}^{(\bs{z})}
-\bar{z}_a\sum_{J''\in\Gamma'^a_{l_2}}
\tilde{Q}^{(z_a)}_{J''}f^{(2)}_{J''}f^{(4,4)}_{J''}\zeta_{\bs{j}''}\frac{\pa f^{(4,2)}_{J''}}{\pa\bs{z}}Q_{\Gamma_{l_1},\Gamma'_{l_1}}^{(\bs{z})}.
\end{align}
We divide $Q^{(z_a)}_{\Gamma_{l}^{(3)},\Gamma_{l}'^{(3)}}$ into the following two parts:

$\bullet$ One part $Q^{(z_a)}_{\Gamma_{l}^{(3,1)},\Gamma_{l}'^{(3,1)}}$, where $\Gamma_{l}^{(3,1)}, \Gamma_{l}'^{(3,1)}$ and the coefficient $\tilde{Q}^{(z_a)}_{J}$ depend on the derivate of the linear part of $\Omega_{\bs{h}}^{(4)}$. 
Then for any $J\in\Gamma_{l}^{(3,1)}\bigcup\Gamma_{l}'^{(3,1)}$, there exist $J'\in\Gamma_{l_1}\bigcup\Gamma_{l_1}'$ and $J''\in\Gamma_{l_2}\bigcup\Gamma_{l_2}'$ such that 
\begin{equation}
\label{form7-10-3}
\begin{array}{l}
\displaystyle \#\bs{j}=\#\bs{j}'+\#\bs{j}''+1,
\\\displaystyle \bs{h}=(\bs{h}_i,\bs{h}',\bs{h}'') \quad \text{with}\;\bs{h}_i\in\{\bs{h}'_{n_1+1},\cdots,\bs{h}'_{\#\bs{h}'},\bs{h}''_{n_2+1},\cdots,\bs{h}''_{\#\bs{h}''}\},
\\\displaystyle \bs{k}=(\bs{k}',\bs{k}''),
 \\\displaystyle n=n_1+n_2. 
\end{array}
\end{equation}
It is similar with the case 2).

$\bullet$ The other part $Q^{(z_a)}_{\Gamma_{l}^{(3,2)},\Gamma_{l}'^{(3,2)}}$, where $\Gamma_{l}^{(3,2)}, \Gamma_{l}'^{(3,2)}$ and the coefficient $\tilde{Q}^{(z_a)}_{J}$ depend on the derivate of the quadratic part of $\Omega_{\bs{h}}^{(4)}$. Then for any $J\in\Gamma_{l}^{(3,2)}\bigcup\Gamma_{l}'^{(3,2)}$, there exist $J'\in\Gamma_{l_1}\bigcup\Gamma_{l_1}'$ and $J''\in\Gamma_{l_2}\bigcup\Gamma_{l_2}'$ such that  
\begin{equation}
\label{form7-10-4}
\begin{array}{l}
\displaystyle \#\bs{j}=\#\bs{j}'+\#\bs{j}''+2,
\\\displaystyle\bs{h}=(\bs{h}',\bs{h}''),
\\\displaystyle \bs{k}=(\bs{h}_i,\bs{k}',\bs{k}'') \quad \text{with}\;\bs{h}_i\in\{\bs{h}'_{n_1+1},\cdots,\bs{h}'_{\#\bs{h}'},\bs{h}''_{n_2+1},\cdots,\bs{h}''_{\#\bs{h}''}\},
\\\displaystyle n=n_1+n_2.
\end{array}
\end{equation}
%
Hence,
\begin{align*}
\#\bs{j}-\#\bs{h}-2\#\bs{k}=&\#\bs{j}'+\#\bs{j}''+2-(\#\bs{h}'+\#\bs{h}'')-2(\#\bs{k}'+\#\bs{k}''+1)
\\=&l_{1}+l_{2}=l.
\end{align*}

In view of \eqref{form7-10-4}, write $\bs{j}=(j_1, j_2,\bs{j}',\bs{j}'')$ with $j_1, j_2\in\mb{U}_3\times\mb{N}_*$.

$\star$ if $\bs{j}$ and $\#\bs{j}'=\alpha_1$ are fixed, then  $\bs{j}',\bs{j}''$ are determined and $\#\bs{j}''=\alpha-\alpha_1-2$;

$\star$ fixing $a$, if $a'=a$, then $a''=|j_1|$ or $|j_2|$, and if $a''=a$, then $a'=|j_1|$ or $|j_2|$; 

$\star$ if $J'$, $J''$ are fixed, then $\bs{h},n$ are determined, and there are at
most $(\#\bs{h}'+\#\bs{h}'')$ different $\bs{k}$. Hence, by the fact $\#\bs{h}'+\#\bs{h}''<\alpha$, one has
\begin{align}
\label{form7-12-3}
&\max_{a,\; \bs{j}\in\mc{R}_{\alpha,N}^{0}\cup\mc{R}_{\alpha,N}^{a}}\#(\Gamma_{l}^{(3,2)}\cup\Gamma'^{(3,2)}_{l};\ a, \bs{j})
\\\notag<&\sum_{\alpha_1=2}^{\alpha-2}4\alpha\max_{a',\;\bs{j}'\in\mc{R}_{\alpha_1,N}^{0}\cup\mc{R}_{\alpha_1,N}^{a'}}\#(\Gamma_{l_1}\cup\Gamma'_{l_1};\ a', \bs{j}')
\\\notag&\;\quad\max_{a'',\; \bs{j}''\in\mc{R}_{\alpha-\alpha_1-2,N}^{0}\cup\mc{R}_{\alpha-\alpha_1-2,N}^{a''}}\#(\Gamma_{l_2}\cup\Gamma'_{l_2};\ a'', \bs{j}'')
\\\notag\leq&4\alpha\sum_{\alpha_1=2}^{\alpha-2}(2\alpha_1)^{2\alpha_1-3}(2\alpha-2\alpha_1-4)^{2(\alpha-\alpha_1-2)-3}
\\\notag<&(2\alpha)^{2\alpha-8},
\end{align}
where the last inequality follows from $(2\alpha_1)^{2\alpha_1-3}(2\alpha-2\alpha_1-4)^{2\alpha-2\alpha_1-7}<(2\alpha)^{2\alpha-10}$.

Notice that
\begin{align*}
\prod_{m=1}^{\#\bs{h}}\kappa_{\bs{h}_{m}}&=\Big(\prod_{m=1}^{\#\bs{h}'}\kappa_{\bs{h}'_{m}}\Big)\Big(\prod_{m=1}^{\#\bs{h}''}\kappa_{\bs{h}''_{m}}\Big)
\\&\leq\prod_{m=1}^{\#\bs{j}'}(j'_{m})^*\prod_{m=1}^{\#\bs{j}''}(j''_{m})^*
\\&\leq\prod_{m=1}^{\#\bs{j}}j_{m}^*,
\end{align*}
which is the control condition \eqref{form416}. Hence, one has $\Gamma_{l}^{(3,2)},\Gamma_{l}'^{(3,2)}\in\mc{H}_{l,N}$ with $l=l_1+l_2$.

Rewrite \eqref{formq3} in the form  \eqref{form418}--\eqref{form420-10-3-2}. In view of \eqref{form7-10-2}, for any given $J\in\Gamma_{l}^{(3,2)}\bigcup\Gamma_l'^{(3,2)}$, if $\#\bs{j}',\#\bs{h}',n_1$ are fixed, then $J',J''$ are determined. 
Notice that the coefficient $\tilde{Q}^{(z_a)}_{J}$ depends on the derivate of $\Omega_{\bs{h}_i}^{(4)}-\Omega_{\bs{h}_i}^{(2)}$, , which is less than $\|\Omega_{\bs{h}_i}^{(4)}-\Omega_{\bs{h}_i}^{(2)}\|_{\ell^{\infty}}$ up to a constant multiply. 
In view of the definitions \eqref{form43} and \eqref{form44}--\eqref{formomega4}, the coefficients of $\Omega_{\bs{h}_i}^{(4)}-\Omega_{\bs{h}_i}^{(2)}$ satisfy $\|\Omega_{\bs{h}_i}^{(4)}-\Omega_{\bs{h}_i}^{(2)}\|_{\ell^{\infty}}\leq1$. Similarly with \eqref{form7-14-2}, by $\#\bs{j}',\#\bs{h}',n_1\leq\mc{J}_{\Gamma_{l},\Gamma'_{l}}$, one has
\begin{align}
\label{form7-14-3}
\|Q_{\Gamma_{l}^{(1)},\Gamma_{l}'^{(1)}}\|_{\ell^{\infty}}\lesssim \mc{J}_{\Gamma_{l},\Gamma'_{l}}^3\|Q_{\Gamma_{l_1},\Gamma_{l_1}'}\|_{\ell^{\infty}}\|Q_{\Gamma_{l_2},\Gamma_{l_2}'}\|_{\ell^{\infty}},
\end{align}
and thus $Q_{\Gamma_{l}^{(3,2)},\Gamma_{l}'^{(3,2)}}\in\ms{H}_{l,N}^{{\rm rev}}$.

\indent 4) Consider the fourth case that $\Omega^{(4)}_{\bs{k}}$ appears in  the commutator. Then $Q^{(z_a)}_{\Gamma_{l}^{(4)},\Gamma_{l}'^{(4)}}$ is of the form
\begin{align}
\label{formq4}
&z_a\sum_{J'\in\Gamma^a_{l_1}}
\tilde{Q}^{(z_a)}_{J'}f^{(2)}_{J'}f^{(4,2)}_{J'}\zeta_{\bs{j}'}\frac{\pa f^{(4,4)}_{J'}}{\pa\bs{z}}Q_{\Gamma_{l_2},\Gamma'_{l_2}}^{(\bs{z})}
+\bar{z}_a\sum_{J'\in\Gamma'^a_{l_1}}
\tilde{Q}^{(z_a)}_{J'}f^{(2)}_{J'}f^{(4,2)}_{J'}\zeta_{\bs{j}'}\frac{\pa f^{(4,4)}_{J'}}{\pa\bs{z}}Q_{\Gamma_{l_2},\Gamma'_{l_2}}^{(\bs{z})}
\\\notag-&z_a\sum_{J''\in\Gamma^a_{l_2}}\tilde{Q}^{(z_a)}_{J''}f^{(2)}_{J''}f^{(4,2)}_{J''}\zeta_{\bs{j}''}\frac{\pa f^{(4,4)}_{J''}}{\pa\bs{z}}Q_{\Gamma_{l_1},\Gamma'_{l_1}}^{(\bs{z})}
-\bar{z}_a\sum_{J''\in\Gamma'^a_{l_2}}
\tilde{Q}^{(z_a)}_{J''}f^{(2)}_{J''}f^{(4,2)}_{J''}\zeta_{\bs{j}''}\frac{\pa f^{(4,4)}_{J''}}{\pa\bs{z}}Q_{\Gamma_{l_1},\Gamma'_{l_1}}^{(\bs{z})}.
\end{align}
We divide $Q^{(z_a)}_{\Gamma_{l}^{(4)},\Gamma_{l}'^{(4)}}$ into the following two parts:
\\ $\bullet$ One part $Q^{(z_a)}_{\Gamma_{l}^{(4,1)},\Gamma_{l}'^{(4,1)}}$, where $\Gamma_{l}^{(4,1)}, \Gamma_{l}'^{(4,1)}$ and the coefficient $\tilde{Q}^{(z_a)}_{J}$ depend on the derivate of the linear part of $\Omega_{\bs{k}}^{(4)}$. 
Then for any $J\in\Gamma_{l}^{(4,1)}\bigcup\Gamma_{l}'^{(4,1)}$, there exist $J'\in\Gamma_{l_1}\bigcup\Gamma_{l_1}'$ and $J''\in\Gamma_{l_2}\bigcup\Gamma_{l_2}'$ such that 
\begin{equation}
\label{form7-10-5}
\begin{array}{l}
\displaystyle \#\bs{j}=\#\bs{j}'+\#\bs{j}''+1,
\\\displaystyle \bs{h}=(\bs{k}_i,\bs{h}',\bs{h}'') \quad \text{with}\;\bs{k}_i\in\{\bs{k}'_{1},\cdots,\bs{k}'_{\#\bs{k}'},\bs{k}''_{1},\cdots,\bs{k}''_{\#\bs{k}''}\},
\\\displaystyle \bs{k}=(\bs{k}',\bs{k}'')
 \\\displaystyle n=n_1+n_2. 
\end{array}
\end{equation}
\\$\bullet$ The other part $Q^{(z_a)}_{\Gamma_{l}^{(4,2)},\Gamma_{l}'^{(4,2)}}$, where $\Gamma_{l}^{(3,2)}, \Gamma_{l}'^{(3,2)}$ and the coefficient $\tilde{Q}^{(z_a)}_{J}$ depend on the derivate of the quadratic part of $\Omega_{\bs{k}}^{(4)}$. Then for any $J\in\Gamma_{l}^{(4,2)}\bigcup\Gamma_{l}'^{(4,2)}$, there exist $J'\in\Gamma_{l_1}\bigcup\Gamma_{l_1}'$ and $J''\in\Gamma_{l_2}\bigcup\Gamma_{l_2}'$ such that  
\begin{equation}
\label{form7-10-6}
\begin{array}{l}
\displaystyle \#\bs{j}=\#\bs{j}'+\#\bs{j}''+2,
\\\displaystyle\bs{h}=(\bs{h}',\bs{h}''),
\\\displaystyle \bs{k}=(\bs{k}_i,\bs{k}',\bs{k}'') \quad \text{with}\;\bs{k}_i\in\{\bs{k}'_{1},\cdots,\bs{k}'_{\#\bs{k}'},\bs{k}''_{1},\cdots,\bs{k}''_{\#\bs{k}''}\},
\\\displaystyle n=n_1+n_2.
\end{array}
\end{equation}
The proof is parallel to the case 3) with $\bs{k}_i$ in place of $\bs{h}_i$.
Then the details are omitted.
\\\indent To sum up, let $\Gamma_{l}=\Gamma_{l}^{(1)}\cup\Gamma_{l}^{(2)}\cup\Gamma_{l}^{(3)}\cup\Gamma_{l}^{(4)}$ and $\Gamma'_{l}=\Gamma_{l}'^{(1)}\cup\Gamma_{l}'^{(2)}\cup\Gamma_{l}'^{(3)}\cup\Gamma_{l}'^{(4)}$, then the control condition \eqref{form416}  hold for any $J\in\Gamma_{l}\bigcup\Gamma'_l$. 
By \eqref{form7-12-1}, \eqref{form7-12-2'} and \eqref{form7-12-3}, we conclude that 
\begin{align*}
\max_{a,\; \bs{j}\in\mc{R}_{\alpha,N}^{0}\cup\mc{R}_{\alpha,N}^{a}}\#(\Gamma_{l}\cup\Gamma'_{l};\ a, \bs{j})
<& \frac{3}{8}(2\alpha)^{2\alpha-3}+ \frac{3}{2}(2\alpha)^{2\alpha-6}+2(2\alpha)^{2\alpha-8}
\\<&(2\alpha)^{2\alpha-3}.
\end{align*}
%
Hence, $\Gamma_{l},\Gamma'_l\in\mc{H}_{l,N}$ with $l=l_1+l_2$.
Moreover, the coefficient estimate \eqref{form436} follows from \eqref{form7-14-1}, \eqref{form7-14-2} and \eqref{form7-14-3}, and thus $Q_{\Gamma_{l},\Gamma_{l}'}\in\ms{H}_{l,N}^{{\rm rev}}$.
\end{proof}
According to the above proof, we get more relationships between indexes of the family $\{\mc{H}_{l,N}\}_{l\in\mb{N}_*}$ in \Cref{le43}.
\begin{remark}
\label{re42}
For any $(a,\bs{j},\bs{h},\bs{k},n)\in\Gamma_{l}\bigcup\Gamma'_l$, one has
\begin{equation}
\label{form425-10-1}
\#\bs{h}\leq\max_{J'\in\Gamma_{l_1}\bigcup\Gamma'_{l_1}}\{\#\bs{h}'\}+\max_{J''\in\Gamma_{l_2}\bigcup\Gamma'_{l_2}}\{\#\bs{h}''\}+1,
\end{equation}
\begin{equation}
\label{form425'-10-1}
\#\bs{k}\leq\max_{J'\in\Gamma_{l_1}\bigcup\Gamma'_{l_1}}\{\#\bs{k}'\}+\max_{J''\in\Gamma_{l_2}\bigcup\Gamma'_{l_2}}\{\#\bs{k}''\}+1,
\end{equation}
\begin{equation}
\label{form426-10-1}
\#\bs{h}+\#\bs{k}\leq\max_{J'\in\Gamma_{l_1}\bigcup\Gamma'_{l_1}}\{\#\bs{h}'+\#\bs{k}'\}+\max_{J''\in\Gamma_{l_2}\bigcup\Gamma'_{l_2}}\{\#\bs{h}''+\#\bs{k}''\}+1,
\end{equation}
with $J':=(a',\bs{j}',\bs{h}',\bs{k}',n_1)\in\Gamma_{l_1}\bigcup\Gamma'_{l_1}$ and $J'':=(a'',\bs{j}'',\bs{h}'',\bs{k}'',n_2)\in\Gamma_{l_2}\bigcup\Gamma'_{l_2}$.
\end{remark}

\section{Rational normal form and homological equations}
\label{sec4}
At first, we introduce integrable rational vector fields.
\begin{definition}
\label{defintrat} 
Being given $\Gamma_{l}, \Gamma'_{l}\in\mc{H}_{l,N}$, 
define the integrable rational vector field of order $2l+1$:
\begin{equation}
\label{formintrat}
Z_{\Gamma_l,\Gamma'_l}(\bs{z})=Z_{\Gamma_l}(\bs{z})+Z_{\Gamma'_l}(\bs{z})\in\ms{H}_{l, N}^{{\rm rev}}
\end{equation}
with
\begin{equation}
Z_{\Gamma'_l}(\bs{z})=\sum_{\substack{J:=(a,\bs{j},\bs{h},\bs{k},n)\in\Gamma_{l}\\\bs{j}\in\mc{I}^{0}}}\tilde{Q}^{(z_a)}_{J}f_{(\bs{h},\bs{k},n)}(I)\big(z_a\zeta_{\bs{j}}\pa_{z_a}-\bar{z}_a\zeta_{\bar{\bs{j}}}\pa_{\bar{z}_a}\big),
\end{equation}
\begin{equation}
Z_{\Gamma'_l}(\bs{z})=\sum_{\substack{J:=(a,\bs{j},\bs{h},\bs{k},n)\in\Gamma'_{l}\\\bs{j}\in\mc{I}^{a}}}\tilde{Q}^{(z_a)}_{J}f_{(\bs{h},\bs{k},n)}(I)\big(\bar{z}_a\zeta_{\bs{j}}\pa_{z_a}-z_a\zeta_{\bar{\bs{j}}}\pa_{\bar{z}_a}\big),
\end{equation}
where ${\rm i}\tilde{Q}^{(z_a)}_{J}\in\mb{R}$ and 
\begin{equation}
\label{form43-1-17}
f_{(\bs{h},\bs{k},n)}(I)=\frac{1}{\prod\limits_{m=1}^{n}\Omega^{(2)}_{\bs{h}_{m}}}\frac{1}{\prod\limits_{m=n+1}^{\#\bs{h}}\Omega^{(4)}_{\bs{h}_{m}}}\frac{1}{\prod\limits_{m=1}^{\#\bs{k}}\Omega^{(4)}_{\bs{k}_{m}}}.
\end{equation}
\end{definition}

More generally, it is crucial to define the rational normal form.
\begin{definition}
\label{def43} 
Being given $\Gamma_{l}\in\mc{H}_{l, N}$, define the rational normal form of order $2l+1$:
\begin{equation}
\label{formrnf}
\tilde{Z}_{\Gamma_{l}}(\bs{z})=\sum_{J:=(a,\bs{j},\bs{h},\bs{k},n)\in\Gamma_{l}}\tilde{Q}_{J}^{(z_a)}f_{(\bs{h},\bs{k},n)}(I)\big(z_a\zeta_{\bs{j}}\pa_{z_a}-\bar{z}_a\zeta_{\bar{\bs{j}}}\pa_{\bar{z}_a})\in\ms{H}_{l, N}^{{\rm rev}},
\end{equation}
where ${\rm i}\tilde{Q}^{(z_a)}_{J}\in\mb{R}$ satisfies
$$\tilde{Q}^{(z_a)}_{(a,\bs{j},\bs{h},\bs{k},n)}=\tilde{Q}^{(z_a)}_{(a,\bar{\bs{j}},\bs{h},\bs{k},n)}.$$
%
\end{definition}

\begin{remark} 
\label{renf}
There are two notes for rational normal form.
\begin{enumerate}[(1)]
\item The rational norma form \eqref{formrnf} is rewritten as 
\begin{equation}
\label{nf}
\tilde{Z}_{\Gamma_{l}}(\bs{z})=\sum_{J:=(a,\bs{j},\bs{h},\bs{k},n)\in\Gamma_{l}}\frac{\tilde{Q}^{(z_a)}_{J}}{2}f_{(\bs{h},\bs{k},n)}(I)(\zeta_{\bs{j}}+\zeta_{\bar{\bs{j}}})\big(z_a\pa_{z_a}-\bar{z}_a\pa_{\bar{z}_a}\big).
\end{equation}

\item\label{renf-2} 
The integrable rational vector field $Z_{\Gamma_l,\Gamma'_l}(\bs{z})$ in \eqref{formintrat} is a rational normal form. Concretely, for $\bs{j}\in\mc{I}^0$ in $J\in\Gamma_{l}$, there exists a permutation $\sigma$ such that $\sigma(\bs{j})=\bar{\bs{j}}$, and thus 
$$\tilde{Q}^{(z_a)}_{(a,\bs{j},\bs{h},\bs{k},n)}=\tilde{Q}^{(z_a)}_{(a,\sigma(\bs{j}),\bs{h},\bs{k},n)}=\tilde{Q}^{(z_a)}_{(a,\bar{\bs{j}},\bs{h},\bs{k},n)};$$
%
for $\bs{j}\in\mc{I}^a$ in $J\in\Gamma'_{l}$, one has 
$\bar{z}_a\zeta_{\bs{j}}=z_a\zeta_{\tilde{\bs{j}}}I_a$ for some $\tilde{\bs{j}}\in\mc{I}^0$, and thus there exists $\tilde{\Gamma}_l\in\mc{H}_{l,N}$ such that $Z_{\Gamma'_l}$ is rewritten as an integrable rational vector field $Z_{\tilde{\Gamma}_l}$, which implies it is also a rational normal form.
\end{enumerate}
\end{remark}

\begin{remark}
\label{re41-12}
The above rational normal form $\tilde{Z}_{\Gamma_{l}}(\bs{z})$ has no effect on the action $I_a$ for any $a\in\mb{N}_*$, i.e., 
\begin{equation}
DI_a[\tilde{Z}_{\Gamma_{l}}]=\bar{z}_a\tilde{Z}^{(z_a)}_{\Gamma_{l}}+z_a\tilde{Z}^{(\bar{z}_a)}_{\Gamma_{l}}=0.
\end{equation}
Concretely, for the system $\pa_t\bs{z}=Z_{\Gamma_{l}}(\bs{z})$, one has
\begin{align*}
\pa_t I_a=&\bar{z}_a\pa_tz_a+z_a\pa_t\bar{z}_a
\\=&\sum_{J:=(a,\bs{j},\bs{h},\bs{k},n)\in\Gamma_{l}^a}\frac{\tilde{Q}^{(z_a)}_{J}}{2}f_{(\bs{h},\bs{k},n)}(I)(\zeta_{\bs{j}}+\zeta_{\bar{\bs{j}}})(\bar{z}_az_a-z_a\bar{z}_a)
\\=&0.
\end{align*} 
%
As a corollary, one has $\pa_t(\|z\|^2_s)=0$.
\end{remark}

The following lemma solve the homological equation associated with the integrable polynomial vector field
\begin{equation}
\label{formz3z5}
Z^{\leq N}_3(\bs{z})+Z_5^{\leq N}(\bs{z})=-{\rm i}\sum_{a\in\mb{N}_*}(\omega_a^{(2)}+\omega_a^{(4)})(z_a\pa_{z_a}-\bar{z}_a\pa_{\bar{z}_a}),
\end{equation} 
which will be used to construct the nearly identity coordinate transformation to eliminate the non-normal form terms in $K_{2l+1}^{\leq N}$.

\begin{lemma}
\label{le44}
For $l\geq3$ and $Q_{\Gamma_l,\Gamma'_l}\in\ms{H}^{{\rm rev}}_{l,N}$, there exist a rational vector field $\chi_{\Gamma_{l-2},\Gamma'_{l-2}}\in\ms{H}^{{\rm an-rev}}_{l-2,N}$, an integrable rational vector field $Z_{\Gamma_l,\Gamma'_l}\in\ms{H}_{l,N}^{{\rm rev}}$,  and two  rational normal forms $\tilde{Z}_{\tilde{\Gamma}_{l}}\in\ms{H}_{l,N}^{{\rm rev}}$,  $\tilde{Z}_{\tilde{\Gamma}_{l-1}}\in\ms{H}_{l-1,N}^{{\rm rev}}$, such that
\begin{equation}
\label{form450}
[Z_3^{\leq N}+Z_5^{\leq N},\chi_{\Gamma_{l-2},\Gamma'_{l-2}}]+Q_{\Gamma_l,\Gamma'_l}=Z_{\Gamma_l,\Gamma'_l}+\tilde{Z}_{\tilde{\Gamma}_l}+\tilde{Z}_{\tilde{\Gamma}_{l-1}}.
\end{equation}
Moreover, one has the coefficient estimates
\begin{equation}
 \label{form454}
\|\chi_{\Gamma_{l-2},\Gamma'_{l-2}}\|_{\ell^{\infty}}\leq\|Q_{\Gamma_l,\Gamma'_l}\|_{\ell^{\infty}},
\end{equation}
\begin{equation}
\label{formcoe1}
\|Z_{\Gamma_{l},\Gamma'_l}\|_{\ell^{\infty}}\leq\|Q_{\Gamma_l,\Gamma'_l}\|_{\ell^{\infty}},
\end{equation}
\begin{equation}
\label{form456-1-19}
\|\tilde{Z}_{\tilde{\Gamma}_{l}}\|_{\ell^{\infty}}\leq2\|Q_{\Gamma_l,\Gamma'_l}\|_{\ell^{\infty}}, 
\end{equation}
\begin{equation}
\label{formcoe2}
\|\tilde{Z}_{\tilde{\Gamma}_{l-1}}\|_{\ell^{\infty}}\leq\frac{1}{2}\|Q_{\Gamma_l,\Gamma'_l}\|_{\ell^{\infty}}.
\end{equation}
\end{lemma}
\begin{proof}
For convenience, we write $\Gamma'^a_l=\emptyset$ for $a>N$.
Then in view of \Cref{def42}, rewrite $Q_{\Gamma_l,\Gamma'_l}(\bs{z})$ in the form 
\begin{equation}
\label{form46-1-18}
Q_{\Gamma_{l},\Gamma'_{l}}(\bs{z})=\sum_{a\in\mb{N}_*}\big(Q_{\Gamma_{l},\Gamma'_{l}}^{(z_a)}(\bs{z})\pa_{z_a}+\overline{Q_{\Gamma_{l},\Gamma'_{l}}^{(z_a)}(\bs{z})}\pa_{\bar{z}_a}\big)
\end{equation}
with the $z_a$-component 
$Q_{\Gamma_{l},\Gamma'_{l}}^{(z_a)}(\bs{z})=z_aQ_{\Gamma^a_{l}}^{(z_a)}(\bs{z})+\bar{z}_aQ_{\Gamma'^a_{l}}^{(z_a)}(\bs{z})$,
where
\begin{align}
\label{form47-1-18}
Q_{\Gamma^a_{l}}^{(z_a)}(\bs{z})=&\sum_{J:=(a,\bs{j},\bs{h},\bs{k},n)\in\Gamma^a_{l}}\tilde{Q}^{(z_a)}_{J}f_{(\bs{h},\bs{k},n)}(I)\zeta_{\bs{j}},
\\\label{form48-1-18}
Q_{\Gamma'^a_{l}}^{(z_a)}(\bs{z})=&\sum_{J:=(a,\bs{j},\bs{h},\bs{k},n)\in\Gamma'^a_{l}}\tilde{Q}^{(z_a)}_{J}f_{(\bs{h},\bs{k},n)}(I)\zeta_{\bs{j}}
\end{align}
with ${\rm i}\tilde{Q}^{(z_a)}_{J}\in\mb{R}$ and $f_{(\bs{h},\bs{k},n)}(I)$ defined in \eqref{form43-1-17}.
Then we are going to solve the homological equation \eqref{form450}.

Let $\chi(\bs{z})=\sum_{a\in\mb{N}_*}\chi^{(z_a)}(\bs{z})\pa_{z_a}+\overline{\chi^{(z_a)}(\bs{z})}\pa_{\bar{z}_a}$ with the $z_a$-component 
\begin{align}
 \label{form455}
\chi^{(z_a)}(\bs{z})=&z_a\sum_{\substack{J:=(a,\bs{j},\bs{h},\bs{k},n)\in\Gamma^a_{l}\\\bs{j}\notin\mc{I}^0}}\frac{{\rm i}\tilde{Q}^{(z_a)}_{J}f_{(\bs{h},\bs{k},n)}(I)\zeta_{\bs{j}}}{\Omega^{(4)}_{\bs{j}}}
+\bar{z}_a\sum_{\substack{J:=(a,\bs{j},\bs{h},\bs{k},n)\in\Gamma'^a_{l}\\\bs{j}\notin\mc{I}^a}}\frac{{\rm i}\tilde{Q}^{(z_a)}_{J}f_{(\bs{h},\bs{k},n)}(I)\zeta_{\bs{j}}}{\Omega^{(4)}_{\bs{j}}}.
\end{align}
%
%
Then there exist $\Gamma_{l-2}, \Gamma'_{l-2}\in\mc{H}_{l-2,N}$ such that $\chi=\chi_{\Gamma_{l-2},\Gamma'_{l-2}}\in\ms{H}_{l-2,N}^{{\rm an-rev}}$ satisfying the estimate \eqref{form454}. Precisely, for any $J'\in\Gamma_{l-2}\cup\Gamma'_{l-2}$, there exists $(a,\bs{j},\bs{h},\bs{k},n)\in\Gamma_{l}\cup\Gamma'_l$ such that 
\begin{equation}
\label{form457}
J':=(a,\bs{j}',\bs{h}',\bs{k}',n')=\big(a,\bs{j},\bs{h}, \big(\bs{k},Irr(\bs{j})\big), n\big),
\end{equation}
%
and thus
\begin{equation*}
\prod_{m=1}^{\#\bs{h}'}\kappa_{\bs{h}'_{m}}=\prod_{m=1}^{\#\bs{h}}\kappa_{\bs{h}_{m}}\leq \prod_{m=1}^{\#\bs{j}}j^*_{m}=\prod_{m=1}^{\#\bs{j}'}(j'_{m})^*,
\end{equation*}
which is the control condition \eqref{form416}.

By a straightforward calculation, one has the equation
\begin{equation}
\label{form451}
[Z_3^{\leq N}+Z_{5}^{\leq N},\chi]^{(z_a)}+Q^{(z_a)}_{\Gamma_l,\Gamma'_l}=Z_{\Gamma_l,\Gamma'_l}^{(z_a)}+R_1^{(z_a)}+R_2^{(z_a)},
\end{equation}
where $Z_{\Gamma_l,\Gamma'_l}^{(z_a)}$ is the integrable part of $Q^{(z_a)}_{\Gamma_l,\Gamma'_l}$, i.e.,
\begin{align}
\label{formz}
Z_{\Gamma_l,\Gamma'_l}^{(z_a)}=z_a\sum_{\substack{J:=(a,\bs{j},\bs{h},\bs{k},n)\in\Gamma^a_{l}\\\bs{j}\in\mc{I}^{0}}}\tilde{Q}^{(z_a)}_{J}\zeta_{\bs{j}}f_{(\bs{h},\bs{k},n)}(I)+\bar{z}_a\sum_{\substack{J:=(a,\bs{j},\bs{h},\bs{k},n)\in\Gamma'^a_{l}\\\bs{j}\in\mc{I}^{a}}}\tilde{Q}^{(z_a)}_{J}\zeta_{\bs{j}}f_{(\bs{h},\bs{k},n)}(I),
\end{align} 
and $R_1^{(z_a)}$, $R_2^{(z_a)}$ are two new remainder terms, i.e.,
\begin{equation}
\label{formres}
R_1^{(z_a)}=-{\rm i}z_a\Big(\frac{\pa \omega_a^{(2)}}{\pa z_a}\chi^{(z_a)}+\frac{\pa \omega_a^{(2)}}{\pa \bar{z}_a}\chi^{(\bar{z}_a)}\Big),
\end{equation}
\begin{equation}
\label{formres2}
R_2^{(z_a)}=-{\rm i}z_a\Big(\frac{\pa \omega_a^{(4)}}{\pa z_a}\chi^{(z_a)}+\frac{\pa \omega_a^{(4)}}{\pa \bar{z}_a}\chi^{(\bar{z}_a)}\Big).
\end{equation}
%
%
Let $Z_{\Gamma_l,\Gamma'_l}(\bs{z})=\sum_{a\in\mb{N}_*}\big(Z_{\Gamma_l,\Gamma'_l}^{(z_a)}(\bs{z})\pa_{z_a}+\overline{Z^{(z_a)}_{\Gamma_l,\Gamma'_l}(\bs{z})}\pa_{\bar{z}_a}\big)$, and thus $Z_{\Gamma_l,\Gamma'_l}$ is an integrable rational vector field of order $2l+1$ satisfying the estimate \eqref{formcoe1}.

%

For any $a\in\mb{N}_*$, one has 
\begin{align}
\label{form413-1-20}
DI_a[\chi]:=&\Big(\frac{\pa I_a}{\pa z_a}\chi^{(z_a)}+\frac{\pa I_a}{\pa \bar{z}_a}\chi^{(\bar{z}_a)}\Big)
\\\notag=&\sum_{\substack{J:=(a,\bs{j},\bs{h},\bs{k},n)\in\Gamma^a_{l}\\\bs{j}\notin\mc{I}^0}}\frac{{\rm i}\tilde{Q}^{(z_a)}_{J}f_{(\bs{h},\bs{k},n)}(I)I_a(\zeta_{\bs{j}}+\zeta_{\bar{\bs{j}}})}{\Omega^{(4)}_{\bs{j}}}
\\\notag&+\sum_{\substack{J:=(a,\bs{j},\bs{h},\bs{k},n)\in\Gamma'^a_{l}\\\bs{j}\notin\mc{I}^a}}\frac{{\rm i}\tilde{Q}^{(z_a)}_{J}f_{(\bs{h},\bs{k},n)}(I)(\bar{z}_a^2\zeta_{\bs{j}}+z_a^2\zeta_{\bar{\bs{j}}})}{\Omega^{(4)}_{\bs{j}}}.
\end{align}
In view of \eqref{form43} and \eqref{formres}, for $a\leq N$, one has
\begin{align}
\label{formmod1}
R_1^{(z_a)}(\bs{z})=-{\rm i}\frac{z_a}{4}DI_a[\chi],
\end{align}
and for $a>N$, one has $R_1^{(z_a)}=0$.
%
Notice that for any $J\in\Gamma'^a_{l}$, one has $(\bs{j},(-1,a))\in\mc{R}^0$. 
Then there exists $\tilde{\Gamma}_{l-1}\in\mc{H}_{l-1,N}$ such that 
\begin{equation}
\label{form415-1-18-23}
R_1^{(z_a)}(\bs{z})=z_a\sum_{J':=(a,\bs{j}',\bs{h}',\bs{k}',n')\in\tilde{\Gamma}^{a}_{l-1}}\frac{\tilde{Q}^{(z_a)}_{J'}}{2}f_{(\bs{h}',\bs{k}',n')}(I)(\zeta_{\bs{j}'}+\zeta_{\bar{\bs{j}}'}),
\end{equation}
where for any $J'\in\tilde{\Gamma}_{l-1}$, there exists $J\in\Gamma_{l}$ such that 
 \begin{equation}
\label{form422-1-18-2}
(a,\bs{j}',\bs{h}',\bs{k}',n')=\big(a,\big((0,a),\bs{j}\big),\bs{h}, \big(\bs{k},Irr(\bs{j})\big), n\big),
\end{equation}
 or there exists $J\in\Gamma'_l$ such that
\begin{equation}
\label{form423-1-18-2}
(a,\bs{j}',\bs{h}',\bs{k}',n')=\big(a,\big((-1,a),\bs{j}\big),\bs{h}, \big(\bs{k},Irr(\bs{j})\big), n\big).
\end{equation}
Thus $R_1(\bs{z}):=\tilde{Z}_{\tilde{\Gamma}_{l-1}}(\bs{z})\in\ms{H}^{\rm rev}_{l-1,N}$ is a rational normal form of order $2l-1$ satisfying the estimate \eqref{formcoe2}.
In view of  \eqref{form44},  \eqref{formres2} and \eqref{form413-1-20}, for $a\leq N$, one has
\begin{align}
\label{formmod2}
R_2^{(z_a)}(\bs{z})=&{\rm i}z_a\Big(\frac{27}{32a}I_a+\frac{1}{8}\sum_{a\neq d\leq N}\Big(\frac{3}{d}+\frac{d}{d^2-a^2}\Big)I_d\Big)DI_a[\chi]
\\\notag&{\rm i}\frac{z_a}{8}\sum_{a\neq d\leq N}\Big(\Big(\frac{3}{d}+\frac{d}{d^2-a^2}\Big)I_a-\frac{a}{d^2-a^2}I_d\Big)DI_d[\chi],
\end{align}
and for $a>N$, one has
\begin{equation}
\label{formresN}
R_2^{(z_a)}(\bs{z})=-{\rm i}\frac{z_a}{8}\sum_{a\neq d\leq N}\frac{a}{d^2-a^2}I_dDI_d[\chi].
\end{equation}
%
Similarly with \eqref{formres}, there exists $\tilde{\Gamma}_{l}\in\mc{H}_{l,N}$ such that 
\begin{equation}
\label{form427-1-18-23}
R_2^{(z_a)}(\bs{z})=z_a\sum_{J':=(a,\bs{j}',\bs{h}',\bs{k}',n')\in\tilde{\Gamma}^{a}_{l}}\frac{\tilde{Q}^{(z_a)}_{J'}}{2}f_{(\bs{h}',\bs{k}',n')}(I)(\zeta_{\bs{j}'}+\zeta_{\bar{\bs{j}}'}),
\end{equation}
 where for any $J'\in\tilde{\Gamma}_{l}$, there exists $J\in\Gamma_{l}$ such that 
 \begin{equation}
\label{form428-1-18-2}
(a,\bs{j}',\bs{h}',\bs{k}',n')=\big(a,\big(j_1,j_2,\bs{j}\big),\bs{h}, \big(\bs{k},Irr(\bs{j})\big), n\big)
\end{equation}
 with $j_1=(0,a)$ or $(0,d)$ and $j_2=(0,a)$ or $(0,d)$;
 or there exists $J\in\Gamma'_{l}$ such that
\begin{equation}
\label{form429-1-18-2}
(a,\bs{j}',\bs{h}',\bs{k}',n')=\big(a,\big(j_3,j_4,\bs{j}\big),\bs{h}, \big(\bs{k},Irr(\bs{j})\big), n\big)
\end{equation}
with $j_3=(0,a)$ or $(0,d)$ and $j_4=(-1,a)$ or $(-1,d)$.
Thus $R_2(\bs{z}):=\tilde{Z}_{\tilde{\Gamma}_{l}}(\bs{z})\in\ms{H}^{\rm rev}_{l,N}$ is a rational normal form of order $2l+1$ satisfying the estimate \eqref{form456-1-19}.
\end{proof}
%

According to the proof, we have more relationships between indexes of the family $\{\mc{H}_{l,N}\}_{l\geq N}$ in \Cref{le44}.
\begin{remark}
\label{re43}
 Consider these elements $\Gamma_{l},\Gamma'_{l}, \Gamma_{l-2},\Gamma'_{l-2}, \tilde{\Gamma}_{l}, \tilde{\Gamma}_{l-1}$ in \Cref{le44}. 

$\bullet$ For any $(a',\bs{j}',\bs{h}',\bs{k}',n')\in\Gamma_{l-2}\bigcup\Gamma'_{l-2}$, there exists $(a,\bs{j},\bs{h},\bs{k},n)\in\Gamma_{l}\bigcup\Gamma'_{l}$ such that 
\begin{align}
\label{form491}
\bs{h}'=\bs{h}, \qquad\#\bs{k}'=\#\bs{k}+1, \quad n'=n.
\end{align}

$\bullet$ For any $(a',\bs{j}',\bs{h}',\bs{k}',n')\in\tilde{\Gamma}_{l}$, there exists $(a,\bs{j},\bs{h},\bs{k},n)\in\Gamma_{l}\bigcup\Gamma'_{l}$ such that 
\begin{equation}
\label{form489}
\bs{h}'=\bs{h}, \qquad\#\bs{k}'=\#\bs{k}+1, \quad n'=n.
\end{equation}

$\bullet$ For any $(a',\bs{j}',\bs{h}',\bs{k}',n')\in\tilde{\Gamma}_{l-1}$, there exists $(a,\bs{j},\bs{h},\bs{k},n)\in\Gamma_{l}\bigcup\Gamma'_{l}$ such that 
\begin{equation}
\label{form490}
\bs{h}'=\bs{h}, \qquad\#\bs{k}'=\#\bs{k}+1, \quad n'=n.
\end{equation}
\end{remark}

For $l=2$,  we solve the homological equation associated solely with $Z_3^{\leq N}$ in the following lemma.
\begin{lemma}
\label{le51}
Consider the vector field $K(\bs{z})\in\ms{R}^{{\rm rev}}_{5}$ with the $z_a$-component
\begin{equation}
\label{form59}
K^{(z_a)}(\bs{z})=\left\{
\begin{aligned}{\rm i}\frac{3a}{32}\bar{z}_a\Big(\sum_{\substack{b_1+b_2=a\\b_1,b_2\leq N}}\frac{z_{b_1}^2z_{b_2}^2}{b_1b_2}+\sum_{\substack{b-c=a\\b,c\leq N}}\frac{2z_{b}^2\bar{z}_{c}^2}{bc}\Big),\quad&\text{for}\;\;a\leq N,
\\0,\qquad\qquad\qquad\qquad&\text{for}\;\;a>N.
\end{aligned}\right.
\end{equation}
Then there exist $\Gamma_1,\Gamma'_1\in\mc{H}_{1,N}$ such that the rational vector field $\chi_{\Gamma_1,\Gamma'_1}\in\ms{H}^{{\rm an-rev}}_{1,N}$ is the solution of the homological equation
\begin{equation}
\label{form510}
[Z_3^{\leq N},\chi_{\Gamma_1,\Gamma'_{1}}]+K=0,
\end{equation}
where for any $(a,\bs{j},\bs{h},\bs{k},n)\in\Gamma_{1}\bigcup\Gamma'_{1}$, one has
\begin{equation}
\label{form511}
a\leq N,\quad n=\#\bs{h}=1 \;\text{or}\; 2, \quad \bs{k}=\emptyset, 
\end{equation}
and the control condition 
\begin{equation}
\label{form512}
\prod_{m=1}^{\#\bs{h}}\kappa_{\bs{h}_{m}}\leq\prod_{m=2}^{\#\bs{j}}\bs{j}_m^*.
\end{equation}
Moreover, one has the coefficient estimate
\begin{equation}
 \label{form513}
\|\chi_{\Gamma_l,\Gamma'_{1}}\|_{\ell^{\infty}}\leq\|K\|_{\ell^{\infty}}.
\end{equation}
\end{lemma}
\begin{proof}
We will construct the solution $\chi_{\Gamma_l,\Gamma'_1}$ in two steps.

{\bf{Step 1}}:
Let $S(\bs{z})=\sum_{a\leq N}S^{(z_a)}(\bs{z})\pa_{z_a}+\overline{S^{(z_a)}(\bs{z})}\pa_{\bar{z}_a}$ with the $z_a$-component 
\begin{equation}
 \label{form514}
S^{(z_a)}(\bs{z})=-\frac{3a}{32}\bar{z}_a\Big(\sum_{\substack{b_1+b_2=a\\b_1,b_2\leq N}}\frac{z_{b_1}^2z_{b_2}^2}{b_1b_2\Omega^{(2)}_{((1,b_1),(1,b_2))}}+\sum_{\substack{b-c=a\\b,c\leq N}}\frac{2z_{b}^2\bar{z}_{c}^2}{bc\Omega^{(2)}_{((1,b)(-1,c))}}\Big).
\end{equation}
Then there exists $\Gamma'_1\in\mc{H}_{1,N}$ such that $S=Q_{\Gamma'_1}\in\ms{H}^{{\rm an-rev}}_{1,N}$, where for $(a,\bs{j},\bs{h},\bs{k},n)\in\Gamma'_{1}$, one has 
\begin{equation}
\label{form515}
\begin{aligned}
\bs{j}\in\mc{R}^a_{2,N}, &\quad \bs{h}=(\bs{h}_1)=(\bs{j}),\quad \bs{k}=\emptyset, \quad n=\#\bs{h}=1,
\\&\kappa_{\bs{h}_1}=\min\{\mu_{\min}(\bs{j}),\Delta_{\bs{j}}\}\leq j_2^*.
\end{aligned}
\end{equation}
In view of  \eqref{form33-1-20}, one has 
\begin{equation}
\label{form516}
 [Z_3^{\leq N},S]^{(z_a)}+K^{(z_a)}=-{\rm i}\frac{z_a}{4}DI_a[S].
\end{equation}

Remark that $DI_a[S]=\bar{z}_aS^{(z_a)}+z_a\overline{S^{(z_a)}}$ is  the sum of two conjugate rational polynomials.
To eliminate the remainder term $-{\rm i}\frac{z_a}{4}DI_a[S]$ in \eqref{form516}, in the next step, we will construct a rational vector field $M(\bs{z})=\sum_{a\in\mb{N}_*}(M^{(z_a)}(\bs{z})\pa_{z_a}+\overline{M^{(z_a)}(\bs{z})}\pa_{\bar{z}_a})\in\ms{H}_{1,N}^{{\rm an-rev}}$ 
with its component of the form
$$M^{(z_a)}(\bs{z})=z_a(M^a(\bs{z})-\overline{M^a(\bs{z})}).$$
%
Fortunately, for rational vector fields of this type, one has
\begin{align}
\label{formmod3}
DI_a[M]=0.
\end{align}

{\bf{Step 2}}: 
Let $M(\bs{z})=\sum_{a\leq N}(M^a(\bs{z})-\overline{M^a(\bs{z})})(z_a\pa_{z_a}-\bar{z}_a\pa_{\bar{z}_a})$ 
with 
\begin{align}
 \label{form517}
M^a(\bs{z})=&-\frac{3a}{128}\Big(\sum_{\substack{b_1+b_2=a\\b_1,b_2\leq N}}\frac{\bar{z}^2_az_{b_1}^2z_{b_2}^2}{b_1b_2\Omega^{(2)}_{((1,b_1),(1,b_2))}\Omega^{(2)}_{((-1,a),(1,b_1),(1,b_2))}}
\\\notag&\qquad\qquad+\sum_{\substack{b-c=a\\b,c\leq N}}\frac{2\bar{z}^2_az_{b}^2\bar{z}_{c}^2}{bc\Omega^{(2)}_{((1,b),(-1,c))}\Omega^{(2)}_{((-1,a),(1,b),(-1,c))}}\Big).
\end{align} 
Then there exists $\Gamma_1\in\mc{H}_{1,N}$ such that $M=Q_{\Gamma_1}\in\ms{H}^{{\rm an-rev}}_{1,N}$ with the $z_a$-component 
\begin{align}
 \label{form518}
M^{(z_a)}=&z_a(M^a(\bs{z})-\overline{M^a(\bs{z})})
\\\notag=&-\frac{3a}{128}\Big(\sum_{\substack{b_1+b_2=a\\b_1,b_2\leq N}}\frac{z_a(\bar{z}^2_az_{b_1}^2z_{b_2}^2-z_a^2\bar{z}_{b_1}^2\bar{z}_{b_2}^2)}{b_1b_2\Omega^{(2)}_{((1,b_1),(1,b_2))}\Omega^{(2)}_{((-1,a),(1,b_1),(1,b_2))}}
\\\notag&\qquad\qquad+\sum_{\substack{b-c=a\\b,c\leq N}}\frac{2z_a(\bar{z}_a^2z_{b}^2\bar{z}_{c}^2-z^2_a\bar{z}_{b}^2z_{c}^2)}{bc\Omega^{(2)}_{((1,b),(-1,c))}\Omega^{(2)}_{((-1,a),(1,b),(-1,c))}}\Big),
\end{align}
 where for $(a,\bs{j},\bs{h},\bs{k},n)\in\Gamma_{1}$, one has 
 \begin{equation}
\label{form519}
\bs{j}\in\mc{R}^0_{3,N}, \quad n=\#\bs{h}=2, \quad \bs{k}=\emptyset\quad \text{and}\quad \kappa_{\bs{h}_1}=\kappa_{\bs{h}_2}=\mu_{\min}(\bs{j}).
\end{equation}
Thus 
\begin{equation}
\label{formrelation}
\kappa_{\bs{h}_1}\kappa_{\bs{h}_2}=\mu^2_{\min}(\bs{j})\leq j_2^*j_3^*.
\end{equation}
In view of  \eqref{form514}, \eqref{formmod3} and \eqref{form518}, one has 
\begin{equation}
\label{form520}
 [Z_3^{\leq N},M]^{(z_a)}-{\rm i}\frac{z_a}{4}\Big(\bar{z}_aS^{(z_a)}+z_a\overline{S^{(z_a)}}\Big)=0.
\end{equation}

To sum up, by the equations \eqref{form516} and \eqref{form520}, one has the homological equation
\begin{equation}
\label{form510-1-20-1}
[Z_3^{\leq N},S+M]^{(z_a)}+K^{(z_a)}=0.
\end{equation}
Let $\chi_{\Gamma_l,\Gamma'_{1}}(\bs{z})=S(\bs{z})+M(\bs{z})$. Then according to \Cref{def41} and \Cref{def42},  $\chi_{\Gamma_l,\Gamma'_{1}}(\bs{z})\in\ms{H}_{1,N}^{{\rm an-rev}}$ satisfies the homological equation \eqref{form510}. The relationship \eqref{form511} and the control condition \eqref{form512} follow from \eqref{form515}, \eqref{form519} and \eqref{formrelation}. Moreover, \eqref{form514} and \eqref{form518} implies the coefficient estimate \eqref{form513}. 
\end{proof}

However, if we use \Cref{le44} for $l=3$ to eliminate septic terms, then according to the homological equation \eqref{form450}, there exists a remaining rational normal form $\tilde{Z}_{\tilde{\Gamma}_2}$, which is of the same order as $Z_{5}^{\leq N}$. 
Thus we will not employ $Z_{3}^{\leq N}+Z_{5}^{\leq N}$ to deal with septic terms. Instead, we will use the following lemma, which involves the homological equation associated solely with $Z_3^{\leq N}$.

\begin{lemma}
\label{le52}
Consider $Q_{\Gamma_l,\Gamma'_l}\in\ms{H}^{{\rm rev}}_{l,N}$ with $l\geq3$, where for any $(a,\bs{j},\bs{h},\bs{k},n)\in\Gamma_l\bigcup\Gamma'_{l}$, one has
\begin{equation}
\label{form522}
n=\#\bs{h}, \quad \bs{k}=\emptyset, 
\end{equation}
and the control condition 
\begin{equation}
\label{form523}
\prod_{m=1}^{\#\bs{h}}\kappa_{\bs{h}_{m}}\leq\prod_{m=2}^{\#\bs{j}}j_m^*.
\end{equation}
Then there exist a rational vector field $\chi_{\Gamma_{1-1},\Gamma'_{1-1}}\in\ms{H}^{{\rm an-rev}}_{l-1,N}$,  an integrable rational vector field $Z_{\Gamma_l,\Gamma'_l}\in\ms{H}_{l,N}^{{\rm rev}}$  and a rational normal form $\tilde{Z}_{\tilde{\Gamma}_l}\in\ms{H}^{{\rm rev}}_{l,N}$  such that
\begin{equation}
\label{form524}
[Z_3^{\leq N},\chi_{\Gamma_{l-1},\Gamma'_{l-1}}]+Q_{\Gamma_{l},\Gamma'_{l}}=Z_{\Gamma_{l},\Gamma'_{l}}+\tilde{Z}_{\tilde{\Gamma}_l},
\end{equation}
where for any $(a,\bs{j}',\bs{h}',\bs{k}',n')\in\Gamma_{l-1}\bigcup\Gamma'_{1-1}$, there exists $(a,\bs{j},\bs{h},\bs{k},n)\in\Gamma_l\bigcup\Gamma'_{l}$ such that 
\begin{equation}
\label{form525}
n'=\#\bs{h}'=\#\bs{h}+1\quad \text{and}\quad \bs{k}'=\bs{k}=\emptyset;
\end{equation}
for any $(a,\bs{j}',\bs{h}',\bs{k}',n')\in\tilde{\Gamma}_{l}$, there exists $(a,\bs{j},\bs{h},\bs{k},n)\in\Gamma_l\bigcup\Gamma'_{l}$ such that 
\begin{equation}
\label{form526}
n'=\#\bs{h}'=\#\bs{h}+1 \quad \text{and}\quad \bs{k}'=\bs{k}=\emptyset.
\end{equation}
Moreover, we have the coefficient estimates
\begin{equation}
 \label{form528}
\|\chi_{\Gamma_{l-1},\Gamma'_{l-1}}\|_{\ell^{\infty}}\leq\|Q_{\Gamma_l,\Gamma'_l}\|_{\ell^{\infty}},
\end{equation}
\begin{equation}
\label{formz1coe}
\|Z_{\Gamma_{l},\Gamma'_l}\|_{\ell^{\infty}}\leq\|Q_{\Gamma_l,\Gamma'_l}\|_{\ell^{\infty}},
\end{equation}
\begin{equation}
\label{formz2coe}
\|\tilde{Z}_{\tilde{\Gamma}_{l}}\|_{\ell^{\infty}}\leq\frac{1}{2}\|Q_{\Gamma_l,\Gamma'_l}\|_{\ell^{\infty}}.
\end{equation}
\end{lemma}
\begin{proof}
In view of \Cref{def42}, write 
$$Q_{\Gamma_l,\Gamma'_l}=\sum_{a\in\mb{N}_*}\big(Q_{\Gamma_{l},\Gamma'_{l}}^{(z_a)}(\bs{z})\pa_{z_a}+\overline{Q_{\Gamma_{l},\Gamma'_{l}}^{(z_a)}(\bs{z})}\pa_{\bar{z}_a}\big)\in\ms{H}^{{\rm rev}}_{l,N}$$
 with the $z_a$-component
\begin{equation}
\label{form530}
Q_{\Gamma_{l},\Gamma'_{l}}^{(z_a)}=z_a\sum_{J:=(a,\bs{j},\bs{h},\emptyset,\#\bs{h})\in\Gamma^a_{l}}\frac{\tilde{Q}^{(z_a)}_{J}\zeta_{\bs{j}}}{\prod\limits_{m=1}^{\#\bs{h}}\Omega^{(2)}_{\bs{h}_{m}}}+\bar{z}_a\sum_{J:=(a,\bs{j},\bs{h},\emptyset,\#\bs{h})\in\Gamma'^a_{l}}\frac{\tilde{Q}^{(z_a)}_{J}\zeta_{\bs{j}}}{\prod\limits_{m=1}^{\#\bs{h}}\Omega^{(2)}_{\bs{h}_{m}}},
\end{equation}
where $\Gamma'^a_l=\emptyset$ for $a>N$. 

Let $\chi(\bs{z})=\sum_{a\in\mb{N}_*}\chi^{(z_a)}(\bs{z})\pa_{z_a}+\overline{\chi^{(z_a)}(\bs{z})}\pa_{\bar{z}_a}$ with the $z_a$-component 
\begin{align}
 \label{form532}
\chi^{(z_a)}(\bs{z})=z_a\sum_{\substack{J:=(a,\bs{j},\bs{h},\emptyset,\#\bs{h})\in\Gamma^a_{l}\\\bs{j}\notin\mc{I}^0}}\frac{{\rm i}\tilde{Q}^{(z_a)}_{J}\zeta_{\bs{j}}}{\Omega^{(2)}_{\bs{j}}\prod\limits_{m=1}^{\#\bs{h}}\Omega^{(2)}_{\bs{h}_{m}}}
+\bar{z}_a\sum_{\substack{J:=(a,\bs{j},\bs{h},\emptyset,\#\bs{h})\in\Gamma'^a_{l}\\\bs{j}\notin\mc{I}^a}}\frac{{\rm i}\tilde{Q}^{(z_a)}_{J}\zeta_{\bs{j}}}{\Omega^{(2)}_{\bs{j}}\prod\limits_{m=1}^{\#\bs{h}}\Omega^{(2)}_{\bs{h}_{m}}}.
\end{align}
Then there exist $\Gamma_{l-1},\Gamma'_{l-1}\in\mc{H}_{l-1,N}$ such that $\chi=\chi_{\Gamma_{l-1},\Gamma'_{l-1}}\in\ms{H}^{{\rm an-rev}}_{l-1,N}$ satisfying the estimate \eqref{form528},
where for any $J'\in\Gamma_{l-1}\bigcup\Gamma'_{l-1}$, there exists $(a,\bs{j},\bs{h},\bs{k},n)\in\Gamma_{l}\bigcup\Gamma'_{l}$ such that 
\begin{equation}
\label{form533}
J':=(a,\bs{j}',\bs{h}',\bs{k}',n')=\big(a,\bs{j},\big(\bs{h},Irr(\bs{j})\big),\bs{k},n+1\big).
\end{equation}
Thus the relationship \eqref{form525} holds.
In view of \eqref{form533}, by the fact $\kappa_{Irr(\bs{j})}\leq j_1^*$ and  the control condition \eqref{form523}, one has
$$\prod_{m=1}^{\#\bs{h}'}\kappa_{\bs{h}'_{m}}=\kappa_{Irr(\bs{j})}\prod_{m=1}^{\#\bs{h}}\kappa_{\bs{h}_{m}}\leq j_1^*\prod_{m=2}^{\#\bs{j}}j_m^*=\prod_{m=1}^{\#\bs{j}'}(j'_m)^*.$$
By \eqref{form34}, \eqref{form530} and \eqref{form532}, one has 
\begin{equation}
\label{form534}
 [Z_3^{\leq N},\chi]^{(z_a)}+Q_{\Gamma_{l},\Gamma'_{l}}^{(z_a)}=Z_{\Gamma_l,\Gamma'_l}^{(z_a)}+R^{(z_a)},
\end{equation}
where $Z_{\Gamma_l,\Gamma'_l}^{(z_a)}$ is the integrable part of $Q^{(z_a)}_{\Gamma_l,\Gamma'_l}$, i.e.,
\begin{equation}
\label{form531}
Z_{\Gamma_l,\Gamma'_l}^{(z_a)}(\bs{z})=z_a\sum_{\substack{J:=(a,\bs{j},\bs{h},\emptyset,\#\bs{h})\in\Gamma^a_{l}\\\bs{j}\in\mc{I}^0}}\frac{\tilde{Q}^{(z_a)}_{J}\zeta_{\bs{j}}}{\prod\limits_{m=1}^{\#\bs{h}}\Omega^{(2)}_{\bs{h}_{m}}}
+\bar{z}_a\sum_{\substack{J:=(a,\bs{j},\bs{h},\emptyset,\#\bs{h})\in\Gamma'^a_{l}\\\bs{j}\in\mc{I}^a}}\frac{\tilde{Q}^{(z_a)}_{J}\zeta_{\bs{j}}}{\prod\limits_{m=1}^{\#\bs{h}}\Omega^{(2)}_{\bs{h}_{m}}},
\end{equation}
and $R^{(z_a)}$ is a new remainder term, i.e.,
\begin{equation}
 \label{form535}
R^{(z_a)}=-{\rm i}\frac{z_a}{4}DI_a[\chi].
\end{equation}

In view of \eqref{form531}, let $Z_{\Gamma_l,\Gamma'_l}(\bs{z})=\sum_{a\in\mb{N}_*}\big(Z_{\Gamma_l,\Gamma'_l}^{(z_a)}(\bs{z})\pa_{z_a}+\overline{Z^{(z_a)}_{\Gamma_l,\Gamma'_l}(\bs{z})}\pa_{\bar{z}_a}\big)$, and thus $Z_{\Gamma_l,\Gamma'_l}$ is an integrable rational vector field of order $2l+1$ satisfying the estimate 
\eqref{formz1coe}.
%
%
In view of \eqref{form535}, by a straightforward calculation, one has 
\begin{align}
\label{form461-1-19}
R^{(z_a)}=&\frac{z_a}{4}\Big(\sum_{\substack{J:=(a,\bs{j},\bs{h},\emptyset,\#\bs{h})\in\Gamma^a_{l}\\\bs{j}\notin\mc{I}^0}}\frac{\tilde{Q}^{(z_a)}_{J}(\zeta_{\bs{j}}+\zeta_{\bar{\bs{j}}})I_a}{\Omega^{(2)}_{\bs{j}}\prod\limits_{m=1}^{\#\bs{h}}\Omega^{(2)}_{\bs{h}_{m}}}
+\sum_{\substack{J:=(a,\bs{j},\bs{h},\emptyset,\#\bs{h})\in\Gamma'^a_{l}\\\bs{j}\notin\mc{I}^a}}\frac{\tilde{Q}^{(z_a)}_{J}(\bar{z}_a^2\zeta_{\bs{j}}+z_a^2\zeta_{\bar{\bs{j}}})}{\Omega^{(2)}_{\bs{j}}\prod\limits_{m=1}^{\#\bs{h}}\Omega^{(2)}_{\bs{h}_{m}}}\Big).
\end{align}
Then there exists $\tilde{\Gamma}_{l}\in\mc{H}_{l,N}$ such that
\begin{equation}
\label{form462-1-19}
R^{(z_a)}(\bs{z})=z_a\sum_{J':=(a,\bs{j}',\bs{h}',\bs{k}',n')\in\tilde{\Gamma}^a_l}\frac{\tilde{Q}^{(z_a)}_{J'}}{2}f_{(\bs{h}',\bs{k}',n')}(I)(\zeta_{\bs{j}'}+\zeta_{\bar{\bs{j}}'}),
\end{equation}
 where for any $J'\in\tilde{\Gamma}_{l}$, there exists $J\in\Gamma_{l}$ such that  
 \begin{equation}
 \label{form463-1-19}
(a,\bs{j}',\bs{h}',\bs{k}',n')=\big(a,\big(\bs{j},(0,a)\big),\big(\bs{h},Irr(\bs{j})\big),\bs{k},n+1\big);
\end{equation}
or there exits $J\in\Gamma'_{l}$ such that  
 \begin{equation}
  \label{form464-1-19}
(a,\bs{j}',\bs{h}',\bs{k}',n')=\big(a,\big(\bs{j},(-1,a)\big),\big(\bs{h},Irr(\bs{j})\big),\bs{k},n+1\big).
\end{equation}
%
Thus $R_2(\bs{z}):=Z_{\tilde{\Gamma}_{l}}(\bs{z})\in\ms{H}^{{\rm rev}}_{l,N}$ is a rational normal form of order $2l+1$ with the relationship \eqref{form526} and the estimate \eqref{formz2coe}.
\end{proof}

If we rely solely on the integrable vector field $Z_3^{\leq N}$ to eliminate $K_{2l+1}$ in \Cref{th31}, then the control condition \eqref{form416} in \Cref{def41} could not be  preserved  after several iterations.  To address this issue, we will employ $Z^{\leq N}_3$ specifically to deal with $K_5$ (via \Cref{le51}) and $K_7$ (via \Cref{le52}), while using  $Z^{\leq N}_3+Z_5^{\leq N}$ to deal with the higher order terms $K_{2l+1}$ for $l\geq4$ (via \Cref{le44}). This strategy  will be further elaborated in the next section. 

\section{Rational normal form theorems}
\label{sec5}
In this section, we are going to normalize the vector field \eqref{form33} in \Cref{th31}. 
Rewrite the vector field \eqref{form33} as
$$X'= Z_1+Z_3^{\leq N}+\sum_{l=2}^{r}K^{\leq N}_{2l+1}+Z_3^{>N}+\sum_{l=2}^{r}K^{>N}_{2l+1}+R'_{\geq2r+3},$$
where the $z_a$-component of truncated resonant vector field $K_{2l+1}^{\leq N}$ is of the form
  \begin{align}
    \label{form52-3-17}
    K^{(z_a)}_{2l+1, \leq N}(\bs{z})=
\left\{\begin{aligned}
z_a\sum_{\bs{j}\in\mc{R}^0_{l,N}}\tilde{K}^{(z_a,z_a)}_{\bs{j}}\zeta_{\bs{j}}+\bar{z}_a\sum_{\bs{j}\in\mc{R}^a_{l,N}}\tilde{K}^{(z_a,\bar{z}_a)}_{\bs{j}}\zeta_{\bs{j}},\; &\text{for}\;a\leq N,
\\z_a\sum_{\bs{j}\in\mc{R}^0_{l,N}}\tilde{K}^{(z_a,z_a)}_{\bs{j}}\zeta_{\bs{j}},\qquad\qquad\quad  &\text{for}\; a>N,
\end{aligned}\right.
    \end{align}
the $z_a$-component of remainder term $K^{>N}_{2l+1}$ is of the form
    \begin{align}
    \label{form52}
    K^{(z_a)}_{2l+1, >N}(\bs{z})=
\left\{\begin{aligned}
z_a\sum_{\bs{j}\in\mc{R}^0_{l}\backslash\mc{R}^0_{l,N}}\tilde{K}^{(z_a,z_a)}_{\bs{j}}\zeta_{\bs{j}}+\bar{z}_a\sum_{\bs{j}\in\mc{R}^a_{l}\backslash\mc{R}^a_{l,N}}\tilde{K}^{(z_a,\bar{z}_a)}_{\bs{j}}\zeta_{\bs{j}},\; &\text{for}\;a\leq N,
\\z_a\sum_{\bs{j}\in\mc{R}^0_{l}\backslash\mc{R}^0_{l,N}}\tilde{K}^{(z_a,z_a)}_{\bs{j}}\zeta_{\bs{j}}+\bar{z}_a\sum_{\bs{j}\in\mc{R}^a_{l}}\tilde{K}^{(z_a,\bar{z}_a)}_{\bs{j}}\zeta_{\bs{j}},\quad  &\text{for}\; a>N.
\end{aligned}\right.
    \end{align}
%
Concretely, we will eliminate the non-normal form parts of the truncated resonant vector fields $K_{2l+1}^{\leq N}$ with $l=2,\cdots,r$ by the integrable polynomial vector fields $Z^{\leq N}_3$ and $Z^{\leq N}_3+Z_5^{\leq N}$.
\subsection{Elimination of quintic and septic terms}
\label{sec51}
The following theorem implies that after two rational normal form steps,  the non-integrable parts of $K_5^{\leq N}$ and   the non-normal form parts of $K_7^{\leq N}$ are eliminated.
\begin{theorem}
\label{th51}
Fix an integer $r\geq4$.
For any $s\geq0$ and $0<\varepsilon\ll 1$, $N\geq r$, $\gamma\in(0,1)$ satisfying $\varepsilon\leq\gamma^{2}(2C_3N)^{-29r-7}$ with some constant $C_3:=\max\{C_1^2, C_2^2\}$, there exists a nearly identity coordinate transformation 
$$\phi^{(2)}:\Big(\mc{U}^{N}_{\frac{\gamma}{2^{r-2}}}\times\mc{U}^{N}_{\frac{\gamma}{2^{r-2}}}\Big)\bigcap B_{s}(2\varepsilon)\to\Big(\mc{U}^{N}_{\frac{\gamma}{2^{r}}}\times\mc{U}^{N}_{\frac{\gamma}{2^{r}}}\Big)\bigcap B_{s}(3\varepsilon)$$ 
such that the vector field \eqref{form33} is transformed into 
\begin{align}
\label{form51}
X''= &Z_1+Z_3^{\leq N}+Z^{\leq N}_5+Z_{\Gamma_3,\Gamma'_3}+\tilde{Z}_{\tilde{\Gamma}_3}+\sum_{l=4}^rQ_{\Gamma_{l},\Gamma'_l}+R''_{\geq2r+3}
\\\notag&+(D\phi^{(2)})^{-1}\big(Z_3^{>N}+\sum_{l=2}^{r}K^{>N}_{2l+1}+R'_{\geq2r+3}\big)\circ\phi^{(2)},
\end{align}
where
\begin{enumerate}[(i)]
\item  the transformation $\phi^{(2)}$ satisfies the estimate
    \begin{equation}
    \label{form53}
    \|\bs{z}-\phi^{(2)}(\bs{z})\|_{s}\leq\frac{3^{11}\times4^{r}N^{28}}{\gamma^{2}}\|z\|_{s}^{3},
    \end{equation}
    and the same estimate is fulfilled by the inverse transformation;
\item $Z_{\Gamma_{3},\Gamma'_3}\in\ms{H}_{3,N}^{{\rm rev}}$ is an integrable rational vector field satisfying the coefficient estimate 
\begin{equation}
\label{formth51-1-25-1}
\|Z_{\Gamma_3,\Gamma'_3}\|_{\ell^{\infty}}\leq C_3^5,
\end{equation}
 and for any $(a,\bs{j},\bs{h},\bs{k},n)\in\tilde{\Gamma}_{3}$, one has
 \begin{equation}
 \label{form55-1-23}
  n=\#\bs{h}\leq3,\; \bs{k}=\emptyset;
  \end{equation}
\item $\tilde{Z}_{\tilde{\Gamma}_{3}}\in\ms{H}_{3,N}^{{\rm rev}}$ is a rational normal form defined by \Cref{def43} satisfying the coefficient estimate 
\begin{equation}
\label{formth51-1-25-2}
\|\tilde{Z}_{\tilde{\Gamma}_{3}}\|_{\ell^{\infty}}\leq C_3^5,
\end{equation}
 and for any $(a,\bs{j},\bs{h},\bs{k},n)\in\tilde{\Gamma}_{3}$, one has
 \begin{equation}
 \label{form55}
  n=\#\bs{h}\leq4,\; \bs{k}=\emptyset;
  \end{equation}
\item $Q_{\Gamma_{l},\Gamma'_l}\in\ms{H}_{l,N}^{{\rm rev}}$ is defined by \Cref{def42} satisfying the coefficient estimate 
\begin{equation}
\label{form56}
\|Q_{\Gamma_{l},\Gamma'_l}\|_{\ell^{\infty}}\leq  C_3^{l^2},
\end{equation}
and for any $(a,\bs{j},\bs{h},\bs{k},n)\in\Gamma_{l}\bigcup\Gamma'_{l}$, one has
 \begin{equation}
 \label{form57}
  n=\#\bs{h}\leq3l-5,\; \bs{k}=\emptyset;
  \end{equation}
\item the remainder term $R''_{\geq2r+3}$ satisfies the estimate
  \begin{equation}
  \label{form58}
  \|R''_{\geq2r+3}(\bs{z})\|_{s}\leq\frac{(2C_3N)^{48r^{2}}}{\gamma^{3r-2}}\|z\|_{s}^{2r+3}.
  \end{equation}
\end{enumerate}
\end{theorem}

\begin{proof}
We are going to apply \Cref{le51} and \Cref{le52} to perform two steps of rational normal form process.

{\bf Step 1: Elimination of the quintic term.}
In view of \eqref{form35}, by \Cref{le51}, there exists a rational vector field $\chi_{\Gamma_1, \Gamma'_1}\in\ms{H}^{{\rm an-rev}}_{1,N}$ such that
\begin{equation}
 \label{form539}
[Z_3^{\leq N},\chi_{\Gamma_1,\Gamma'_1}]+K_{5}^{\leq N}=Z_{5}^{\leq N}
\end{equation}
with the coefficient estimate
\begin{align}
\label{form540}
\|\chi_{\Gamma_1,\Gamma'_1}\|_{\ell^{\infty}}\leq\|K_5-Z_5\|_{\ell^{\infty}}\leq \frac{3}{32},
\end{align}
where for any $(a,\bs{j},\bs{h},\bs{k},n)\in\Gamma_1\bigcup\Gamma'_{1}$, one has
\begin{equation}
\label{form541}
n=\#\bs{h}=1 \;\text{or}\; 2, \quad \bs{k}=\emptyset, 
\end{equation}
and the control condition 
\begin{equation}
\label{form542}
\prod_{m=1}^{\#\bs{h}}\kappa_{\bs{h}_{m}}\leq\prod_{m=2}^{\#\bs{j}}j^*_m.
\end{equation}

In view of \eqref{re41-2} in \Cref{re41} and \eqref{form541}, there exists a finite set $\mc{F}_1\subset\mb{N}^3$ such that 
$$\Gamma_1=\bigcup_{(\alpha,\beta,\beta')\in\mc{F}_1}\Gamma^{(\alpha,\beta,\beta')}_1\quad\text{and}\quad \Gamma'_1=\bigcup_{(\alpha,\beta,\beta')\in\mc{F}_1}\Gamma'^{(\alpha,\beta,\beta')}_1$$ 
with $\beta\leq 2$, $\beta'=0$ and $\alpha=1+\beta+2\beta'\leq3$. 
Hence, by \Cref{le42}  and  \eqref{form540}, for any $z\in\mc{U}^{N}_{\frac{\gamma}{2^{r-1}}}$, one has
\begin{align}
\label{form543}
\|\chi_{\Gamma_1,\Gamma'_1}(\bs{z})\|_s\leq&6\sqrt{2}\times3^3\max_{(\alpha,\beta,\beta')\in\mc{F}_1}\frac{(12\alpha^2)^{\alpha-1}N^{(4\alpha+2)(\beta+\beta')}}{(\gamma/2^{r-1})^{\beta+2\beta'}}\|\chi_{\Gamma_1,\Gamma'_1}\|_{\ell^{\infty}}\|z\|_s^{3}
\\\notag\leq&3^4\times2\sqrt{2}\frac{(12\times3^2)^2N^{(4\times3+2)\times2}}{(\gamma/2^{r-1})^2}\frac{3}{32}\|z\|_s^{3}
\\\notag=&3^{11}\sqrt{2}\frac{4^{r-1}N^{28}}{\gamma^2}\|z\|_s^{3}.
\end{align}
Denote by the flow $\Phi_{\chi_{\Gamma_1,\Gamma'_1}}^{t}$ generated by $\chi_{\Gamma_1,\Gamma'_1}$. Then for any $|t|\leq 1$, as long as $\|\Phi_{\chi_{\Gamma_1,\Gamma'_1}}^{t}(\bs{z})\|_{s}\leq\frac{3}{2}\|\bs{z}\|_{s}$, one has
\begin{equation}
\label{form544}
\|\bs{z}-\Phi_{\chi_{\Gamma_1,\Gamma'_1}}^{t}(\bs{z})\|_{s}\leq\|\chi_{\Gamma_1,\Gamma'_1}(\bs{z})\|_{s}\leq3^{11}\sqrt{2}\frac{4^{r-1}N^{28}}{\gamma^{2}}\|z\|_{s}^{3}.
\end{equation}
Now, we need to make sure the condition \eqref{form411} in \Cref{le41}.
Notice that
\begin{align}
\label{form545}
\sup_{a\in\mb{N}_*}|a|^{2s}\big||\Phi_{\chi_{\Gamma_1,\Gamma'_1}}^{t}(z)_{a}|^{2}-I_{a}\big|
&\leq\|\Phi_{\chi_{\Gamma_1,\Gamma'_1}}^{t}(z)-z\|_{s}(\|\Phi_{\chi_{\Gamma_1,\Gamma'_1}}^{t}(z)\|_{s}+\|z\|_{s})
\\\notag&\leq\frac{5}{2}\|\chi_{\Gamma_1,\Gamma'_1}(z)\|_{s}\|z\|_{s}.
\end{align}
For $\bs{z}\in B_{s}(\frac{5}{2}\varepsilon)$, one has $\|z\|_s\leq\frac{5\varepsilon}{2\sqrt{2}}$. 
By \eqref{form543}--\eqref{form545} and the fact 
\begin{equation}
\label{forme11-10-1}
\varepsilon\leq\frac{\gamma^{2}}{(2C_3N)^{29r+7}}<\Big(\frac{(\gamma/2^{r-1})^{4}\sqrt{2}}{160\times3^{13}(r+1)N^{4r+31}}\Big)^{\frac{1}{2}},
\end{equation}
the transformation
\begin{equation}
\label{form516-1-23}
\Phi_{\chi_{\Gamma_1,\Gamma'_1}}^{t}:\Big(\mc{U}^{N}_{\frac{\gamma}{2^{r-1}}}\times\mc{U}^{N}_{\frac{\gamma}{2^{r-1}}}\Big)\bigcap B_{s}(\frac{5}{2}\varepsilon)\to\Big(\mc{U}^{N}_{\frac{\gamma}{2^{r}}}\times\mc{U}^{N}_{\frac{\gamma}{2^{r}}}\Big)\bigcap B_{s}(3\varepsilon)
\end{equation}
 is well defined for $|t|\leq 1$ satisfying the estimate \eqref{form544}.    

Similarly with the proof \Cref{le43} and \Cref{re42}, consider $\Gamma_{l_1}, \Gamma'_{l_1}\in\mc{H}_{l_1,N}$ and $ \Gamma_{l_2},\Gamma'_{l_2}\in\mc{H}_{l_2,N}$ with their elements satisfying the relationships
\begin{equation}
\label{form550}
n=\#\bs{h} \quad\text{and}\quad \quad \bs{k}=\emptyset, 
\end{equation}
for any $(a,\bs{j},\bs{h},\bs{k},n)\in\Gamma_{l_1}\bigcup\Gamma'_{l_1}$ or $\Gamma_{l_2}\bigcup\Gamma'_{l_2}$. 
Then for $Q_{\Gamma_{l_1},\Gamma'_{l_1}}\in\ms{H}^{{\rm rev}}_{l_1,N}$ and $Q_{\Gamma_{l_2},\Gamma'_{l_2}}\in\ms{H}^{{\rm an-rev}}_{l_2,N}$, there exists $Q_{\Gamma_{l_1+l_2},\Gamma'_{l_1+l_2}}\in\ms{H}^{{\rm rev}}_{l_1+l_2,N}$ such that 
\begin{equation}
\label{form551}
[Q_{\Gamma_{l_1},\Gamma_{l_1}}, Q_{\Gamma_{l_2},\Gamma'_{l_2}}]=Q_{\Gamma_{l_1+l_2},\Gamma'_{l_1+l_2}},
\end{equation}
where for any $(a',\bs{j}',\bs{h}',\bs{k}',n')\in\Gamma_{l_1+l_2}\bigcup\Gamma'_{l_1+l_2}$, we still have the relationships \eqref{form550}, i.e.,
\begin{align}
\label{form552}
n'=\#\bs{h}'\leq&\max_{(a,\bs{j},\bs{h},\bs{k},n)\in\Gamma_{l_1}\bigcup\Gamma'_{l_1}}\{\#\bs{h}\}+\max_{(a,\bs{j},\bs{h},\bs{k},n)\in\Gamma_{l_2}\bigcup\Gamma'_{l_2}}\{\#\bs{h}\}+1,\quad \bs{k}'=\emptyset.
\end{align}
Moreover, the coefficient estimate \eqref{form436} still holds.
%

By the fact $[Z_1, \chi_{\Gamma_1,\Gamma'_1}]=0$, the homological equation \eqref{form539} and \eqref{form550}--\eqref{form552},  the vector field \eqref{form33} is transformed into
\begin{align}
\label{form547}
e^{ad_{\chi_{\Gamma_1,\Gamma'_1}}}X'
=&Z_1+Z_3^{\leq N}+Z_5^{\leq N}+\sum_{l=3}^{r}K_{\Gamma_{l},\Gamma'_{l}}+R^{(1)}_{\geq 2r+3}
\\\notag&+(D\Phi^1_{\chi_{\Gamma_1,\Gamma'_1}})^{-1}\big(Z_3^{>N}+\sum_{l=2}^{r}K^{>N}_{2l+1}+R'_{\geq2r+3}\big)\circ\Phi^1_{\chi_{\Gamma_1,\Gamma'_1}},
\end{align}
where for $l=3,\dots,r$, one has $\Gamma_l,\Gamma'_l\in\mc{H}_{l,N}$, 
\begin{equation}
\label{form548}
K_{\Gamma_{l},\Gamma'_{l}}=\frac{1}{(l-1)!}ad_{\chi_{\Gamma_1,\Gamma'_1}}^{l-1}Z_3^{\leq N}+\sum_{\substack{m+k=l\\k\geq0, 2\leq m\leq r}}\frac{1}{k!}ad_{\chi_{\Gamma_1,\Gamma'_1}}^kK_{2m+1}^{\leq N},
\end{equation}
and 
\begin{align}
\label{form549}
R^{(1)}_{\geq2r+3}(\bs{z})=&\int_0^1\frac{(1-t)^{r-1}}{(r-1)!}(D\Phi^t_{\chi_{\Gamma_1,\Gamma'_1}})^{-1}ad_{\chi_{\Gamma_1,\Gamma'_1}}^rZ_3^{\leq N}\circ\Phi^t_{\chi_{\Gamma_1,\Gamma'_1}}dt
\\\notag&+\sum_{m=2}^{r}\int_0^1\frac{(1-t)^{r-l}}{(r-m)!}(D\Phi^t_{\chi_{\Gamma_1,\Gamma'_1}})^{-1}ad_{\chi_{\Gamma_1,\Gamma'_1}}^{r-m+1}K_{2m+1}^{\leq N}\circ\Phi^t_{\chi_{\Gamma_1,\Gamma'_1}}dt.
\end{align}
Moreover, for any $(a,\bs{j},\bs{h},\bs{k},n)\in\Gamma_l\bigcup\Gamma'_l$, one has
\begin{equation}
\label{form553}
n=\#\bs{h}\leq 3(l-2),\quad \bs{k}=\emptyset, 
\end{equation}
and thus
\begin{equation}
\label{form553-1-23}
\#\bs{j}=l+\#\bs{h}+2\#\bs{k}\leq 4l-6<4l.
\end{equation}
In view of \eqref{form539} and \eqref{form548}, by the coefficient estimates \eqref{form37},  \eqref{form436}, \eqref{form540} and \eqref{form553-1-23}, one has
\begin{align}
\label{form554}
\|K_{\Gamma_{l},\Gamma'_{l}}\|_{\ell^{\infty}}\leq&\|K_{2l+1}^{\leq N}\|_{\ell^{\infty}}+\frac{l^{4}\cdots4^{4}\times3^{4}}{(l-1)!}\big(C_2\|\chi_{\Gamma_1,\Gamma'_1}\|_{\ell^{\infty}}\big)^{l-2}\|Z_{5}^{\leq N}-K_{5}^{\leq N}\|_{\ell^{\infty}}
\\\notag&+\sum_{m=2}^{l-1}\frac{l^{4}\cdots(m+2)^{4}(m+1)^{4}}{(l-m)!}\big(C_2\|\chi_{\Gamma_1,\Gamma'_1}\|_{\ell^{\infty}}\big)^{l-m}\|K_{2m+1}^{\leq N}\|_{\ell^{\infty}}
\\\notag\leq&C_1^{l^2}+l^{4(l-2)}\big(\frac{3}{32}C_2\big)^{l-2}C_{1}^{4}+\sum_{m=2}^{l-1}l^{4(l-m)}\big(\frac{3}{32}C_2\big)^{l-m}C_{1}^{m^2}
\\\notag\leq&C^{l^2}+C^{l(l-2)}\frac{C^{2l}}{l}+\sum_{m=2}^{l-1}C^{l(l-m)}\frac{C^{lm}}{l}
\\\notag<&2C^{l^2}
\end{align}
with taking $C=\max\{C_2, C_1\}$, where the penultimate inequality follows from 
$$\frac{3}{32}C_2l^4\leq C^l\quad \text{and}\quad C_1^{m^2}\leq \frac{C^{lm}}{l} \;\text{for}\; m=2,\cdots, l-1.$$

{\bf Step 2: Elimination of the septic term.}
By \eqref{form539}, \eqref{form548} and  \eqref{form553} for $l=3$, one has
\begin{align}
\label{form555}
K_{\Gamma_{3},\Gamma'_{3}}=&\frac{1}{2}[[Z_3^{\leq N},\chi_{\Gamma_1,\Gamma'_1}],\chi_{\Gamma_1,\Gamma'_1}]+[K_5^{\leq N},\chi_{\Gamma_1,\Gamma'_1}]+K_7^{\leq N}
\\\notag=&\frac{1}{2}[Z_{5}^{\leq N}+K_5^{\leq N},\chi_{\Gamma_1,\Gamma'_1}]+K_7^{\leq N},
\end{align}
where for any $(a,\bs{j},\bs{h},\bs{k},n)\in\Gamma_3\bigcup\Gamma'_3$, one has
\begin{equation}
\label{form526-1-23}
n=\#\bs{h}\leq 3,\quad \bs{k}=\emptyset, \quad\#\bs{j}=3+\#\bs{h}+2\#\bs{k}\leq 6.
\end{equation}
Notice that by \eqref{form542}, the control condition \eqref{form523} is satisfied by $K_{\Gamma_{3},\Gamma'_{3}}$.
Then by \Cref{le52}, there exists a rational vector field $\chi_{\Gamma_2,\Gamma'_2}\in\ms{H}^{{\rm an-rev}}_{2,N}$, an integrable rational vector field $Z_{\Gamma_3,\Gamma'_3}\in\ms{H}_{3,N}^{{\rm rev}}$  and a rational normal form $\tilde{Z}_{\tilde{\Gamma}_3}\in\ms{H}^{{\rm rev}}_{3,N}$  such that
\begin{equation}
\label{form556}
[Z_3^{\leq N},\chi_{\Gamma_{2},\Gamma'_{2}}]+K_{\Gamma_{3},\Gamma'_{3}}=Z_{\Gamma_3,\Gamma'_3}+\tilde{Z}_{\tilde{\Gamma}_3},
\end{equation}
where for any $(a,\bs{j},\bs{h},\bs{k},n)\in\Gamma_{2}\bigcup\Gamma'_{2}$, one has
\begin{equation}
\label{form557}
n=\#\bs{h}\leq3+1=4\quad \text{and}\quad \bs{k}=\emptyset;
\end{equation}
for any $(a,\bs{j},\bs{h},\bs{k},n)\in\tilde{\Gamma}_{3}$, one has 
\begin{equation}
\label{form558}
n=\#\bs{h}\leq3+1=4 \quad \text{and}\quad \bs{k}=\emptyset.
\end{equation}
Moreover, by \eqref{form528}--\eqref{formz2coe} and \eqref{form554}, one has the coefficient estimates
\begin{align}
 \label{form559}
\|\chi_{\Gamma_{2},\Gamma'_{2}}\|_{\ell^{\infty}}\leq\|K_{\Gamma_3,\Gamma'_3}\|_{\ell^{\infty}}\leq2C^{9},
\\\label{form559-1-23-1}
\|Z_{\Gamma_{3},\Gamma'_{3}}\|_{\ell^{\infty}}\leq\|K_{\Gamma_3,\Gamma'_3}\|_{\ell^{\infty}}\leq2C^{9},
\\\label{form559-1-23-2}
\|\tilde{Z}_{\tilde{\Gamma}_{3}}\|_{\ell^{\infty}}\leq\frac{1}{2}\|K_{\Gamma_3,\Gamma'_3}\|_{\ell^{\infty}}\leq C^{9}.
\end{align}
By the fact $2C^9<C_3^5$, the estimates \eqref{form559-1-23-1} and \eqref{form559-1-23-2} imply \eqref{formth51-1-25-1} and \eqref{formth51-1-25-2}, respectively.

In view of \eqref{re41-2} in \Cref{re41} and \eqref{form557}, 
 there exists a finite set $\mc{F}_2\subset\mb{N}^3$ such that
$$\Gamma_2=\bigcup_{(\alpha,\beta,\beta')\in\mc{F}_2}\Gamma^{(\alpha,\beta,\beta')}_2\quad\text{and}\quad \Gamma'_2=\bigcup_{(\alpha,\beta,\beta')\in\mc{F}_2}\Gamma'^{(\alpha,\beta,\beta')}_2$$
with $\beta\leq 4$, $\beta'=0$ and $\alpha=2+\beta+2\beta'\leq6$. 
Hence, by \Cref{le42}  and  \eqref{form559}, for any $\bs{z}\in\mc{U}^{N}_{\frac{\gamma}{2^{r-2}}}$, one has
\begin{align}
\label{form561}
\|\chi_{\Gamma_{2},\Gamma'_2}(\bs{z})\|_s\leq&6\sqrt{2}\times6^3\max_{(\alpha,\beta,\beta')\in\mc{F}_2}\frac{(12\alpha^2)^{\alpha-1}N^{(4\alpha+2)(\beta+\beta')}}{(\gamma/2^{r-2})^{\beta+2\beta'}}\|\chi_{\Gamma_2,\Gamma'_2}\|_{\ell^{\infty}}\|z\|_s^{5}
\\\notag\leq&6^4\sqrt{2}\frac{(12\times6^2)^5N^{(4\times6+2)\times4}}{(\gamma/2^{r-2})^4}2C^{9}\|z\|_s^{5}
\\\notag=&3^{19}\sqrt{2}C^9\frac{2^{4r+17}N^{104}}{\gamma^4}\|z\|_s^{5}.
\end{align}
Denote by the flow $\Phi_{\chi_{\Gamma_{2},\Gamma'_2}}^{t}$ generated by $\chi_{\Gamma_{2},\Gamma'_2}$. Then for any $|t|\leq 1$, as long as $\|\Phi_{\chi_{\Gamma_{2},\Gamma'_2}}^{t}(z)\|_{s}\leq\frac{3}{2}\|z\|_{s}$, one has
\begin{equation}
\label{form562}
\|\Phi_{\chi_{\Gamma_{2},\Gamma'_2}}^{t}(\bs{z})-\bs{z}\|_{s}\leq\|\chi_{\Gamma_{2},\Gamma'_2}(\bs{z})\|_{s}\leq3^{19}\sqrt{2}C^9\frac{2^{4r+17}N^{104}}{\gamma^4}\|z\|_s^{5}.
\end{equation}
Now, we make sure the condition \eqref{form411} in \Cref{le41}.
Notice that
\begin{align}
\label{form563}
\sup_{a\in\mb{N}_*}|a|^{2s}\big||\Phi_{\chi_{\Gamma_{2},\Gamma'_2}}^{t}(z)_{a}|^{2}-I_{a}\big|
&\leq\|\Phi_{\chi_{\Gamma_{2},\Gamma'_2}}^{t}(z)-z\|_{s}(\|\Phi_{\chi_{\Gamma_{2},\Gamma'_2}}^{t}(z)\|_{s}+\|z\|_{s})
\\\notag&\leq\frac{5}{2}\|\chi_{\Gamma_{2},\Gamma'_2}(z)\|_{s}\|z\|_{s}.
\end{align}
For $\bs{z}\in B_{s}(2\varepsilon)$, one has $\|z\|_s\leq\sqrt{2}\varepsilon$. 
By \eqref{form561}--\eqref{form563} and the fact
\begin{equation}
\label{form564}
\varepsilon\leq\frac{\gamma^{2}}{(2C_3N)^{29r+7}}<\Big(\frac{(\gamma/2^{r-2})^{6}\sqrt{2}}{5\times3^{21}\times2^{32}(r+1)C^9N^{4r+107}}\Big)^{\frac{1}{4}},
\end{equation}
the transformation
\begin{equation}
\label{form539-1-23}
\Phi_{\chi_{\Gamma_{2},\Gamma'_2}}^{t}:\Big(\mc{U}^{N}_{\frac{\gamma}{2^{r-2}}}\times\mc{U}^{N}_{\frac{\gamma}{2^{r-2}}}\Big)\bigcap B_{s}(2\varepsilon)\to\Big(\mc{U}^{N}_{\frac{\gamma}{2^{r-1}}}\times\mc{U}^{N}_{\frac{\gamma}{2^{r-1}}}\Big)\bigcap B_{s}(\frac{5}{2}\varepsilon)
\end{equation}
 is well defined for  $|t|\leq 1$ satisfying the estimate \eqref{form562}.

By \eqref{form516-1-23} and \eqref{form539-1-23}, the transformation $\phi^{(2)}:=\Phi_{\chi_{\Gamma_1, \Gamma'_1}}^1\circ\Phi_{\chi_{\Gamma_2, \Gamma'_2}}^1$ is well defined in $\Big(\mc{U}^{N}_{\frac{\gamma}{2^{r-2}}}\times\mc{U}^{N}_{\frac{\gamma}{2^{r-2}}}\Big)\bigcap B_{s}(2\varepsilon)$. 
By \eqref{form544}, \eqref{form562} and the fact $\varepsilon\leq\frac{\gamma^{2}}{(C_2N)^{29r+7}}$, one has
\begin{align*}
 \|\bs{z}-\phi^{(2)}(\bs{z})\|_{s}\leq&\|\bs{z}-\Phi_{\chi_{\Gamma_{2},\Gamma'_2}}^{1}(\bs{z})\|_{s}+ \|\Phi_{\chi_{\Gamma_{2},\Gamma'_2}}^{1}(\bs{z})-\Phi_{\chi_{\Gamma_1,\Gamma'_1}}^{1}(\Phi_{\chi_{\Gamma_{2},\Gamma'_2}}^{1}(\bs{z}))\|_{s}
 \\\leq&3^{19}\sqrt{2}C^9\frac{2^{4r+17}N^{104}}{\gamma^4}\|z\|_s^{5}+3^{11}\sqrt{2}\frac{4^{r-1}N^{28}}{\gamma^{2}}\|\Phi_{\chi_{\Gamma_{2},\Gamma'_2}}^{1}(z)\|_{s}^{3}
 \\\leq&\frac{3^{11}\times4^{r}N^{28}}{\gamma^{2}}\|z\|_{s}^{3},
\end{align*}
which is the estimate \eqref{form53}. 

By the fact $[Z_1, \chi_{\Gamma_{2},\Gamma'_2}]=0$, the homological equation \eqref{form556} and \Cref{le43}, the vector field \eqref{form547} is transformed into  
\begin{align}
\label{form567}
X'':=&e^{ad_{\chi_{\Gamma_2,\Gamma'_2}}}e^{ad_{\chi_{\Gamma_1,\Gamma'_1}}}X'
\\\notag=&Z_1+Z_3^{\leq N}+Z^{\leq N}_5+Z_{\Gamma_3,\Gamma'_3}+\tilde{Z}_{\tilde{\Gamma}_3}+\sum_{l=4}^rQ_{\Gamma_{l},\Gamma'_l}+R''_{\geq2r+3}
\\\notag&+(D\Phi^1_{\chi_{\Gamma_2,\Gamma'_2}})^{-1}(D\Phi^1_{\chi_{\Gamma_1,\Gamma'_1}})^{-1}\big(Z_3^{>N}+\sum_{l=2}^{r}K^{>N}_{2l+1}+R'_{\geq2r+3}\big)\circ\Phi^1_{\chi_{\Gamma_1,\Gamma'_1}}\circ\Phi^1_{\chi_{\Gamma_2, \Gamma'_2}},
\end{align}
where for $l=4,\dots,r$, one has $\Gamma_l,\Gamma'_l\in\mc{H}_{l,N}$, 
\begin{align}
\label{form568}
Q_{\Gamma_{l},\Gamma'_{l}}=&\sum_{\substack{m+2k=l\\m=1,2}}\frac{1}{k!}ad_{\chi_{\Gamma_2,\Gamma'_2}}^{k}Z_{2m+1}^{\leq N}
+\sum_{\substack{m+2k=l\\k\geq0, 3\leq m\leq r}}\frac{1}{k!}ad_{\chi_{\Gamma_2,\Gamma'_2}}^kK_{\Gamma_m,\Gamma'_m},
\end{align}
\begin{equation}
\label{form542-1-23}
R''_{\geq2r+3}=R^{(2)}_{\geq2r+3}+(D\Phi^1_{\chi_{\Gamma_2, \Gamma'_2}})^{-1}R^{(1)}_{\geq 2r+3}\circ\Phi^1_{\chi_{\Gamma_2, \Gamma'_2}}
\end{equation}
with
\begin{align}
\label{form569}
R^{(2)}_{\geq2r+3}=&\sum_{m=1}^{2}\int_0^1\frac{(1-t)^{[\frac{l-m}{2}]}}{[\frac{l-m}{2}]!}(D\Phi^t_{\chi_{\Gamma_2, \Gamma'_2}})^{-1}ad_{\chi_{\Gamma_2,\Gamma'_2}}^{[\frac{r-m}{2}]+1}Z_{2m+1}^{\leq N}\circ\Phi^t_{\chi_{\Gamma_2,\Gamma'_2}}dt
\\\notag&+\sum_{m=3}^{r}\int_0^1\frac{(1-t)^{[\frac{r-m}{2}]}}{[\frac{r-m}{2}]!}(D\Phi^t_{\chi_{\Gamma_2, \Gamma'_2}})^{-1}ad_{\chi_{\Gamma_2,\Gamma'_2}}^{[\frac{r-m}{2}]+1}K_{\Gamma_{m},\Gamma'_{m}}\circ\Phi^t_{\chi_{\Gamma_2, \Gamma'_2}}dt.
\end{align}
Moreover, by \eqref{form550}--\eqref{form552}, \eqref{form553} and \eqref{form556}--\eqref{form558}, for any $(a,\bs{j},\bs{h},\bs{k},n)\in\Gamma_l\bigcup\Gamma'_l$, one has
\begin{equation}
\label{form570}
n=\#\bs{h}\leq 5(\frac{l-1}{2}-1)+4\leq3l-5,\quad \bs{k}=\emptyset,
\end{equation}
and thus 
\begin{equation}
\label{form5701-1-23}
\#\bs{j}=l+\#\bs{h}+2\#\bs{k}\leq 4l-5<4l.
\end{equation}
Next, we will prove the coefficient estimate \eqref{form56}. 
In view of \eqref{form556} and \eqref{form568}, by \Cref{le43}, \eqref{form554}, \eqref{form559}--\eqref{form559-1-23-2} and \eqref{form5701-1-23}, one has
\begin{align}
\label{form570'}
\|Q_{\Gamma_{l},\Gamma'_{l}}\|_{\ell^{\infty}}\leq&\|K_{\Gamma_{l},\Gamma'_{l}}\|_{\ell^{\infty}}+\frac{l^{4}\cdots6^{4}\times4^{4}}{(\frac{l-2}{2})!}\big(C_2\|\chi_{\Gamma_2, \Gamma'_2}\|_{\ell^{\infty}}\big)^{\frac{l-2}{2}}\|Z_{5}^{\leq N}\|_{\ell^{\infty}}
\\\notag&+\frac{l^{4}\cdots7^{4}\times5^{4}}{(\frac{l-1}{2})!}\big(C_2\|\chi_{\Gamma_2, \Gamma'_2}\|_{\ell^{\infty}}\big)^{\frac{l-3}{2}}\|Z_{\Gamma_3,\Gamma'_3}+\tilde{Z}_{\tilde{\Gamma}_3}-K_{\Gamma_{3},\Gamma'_{3}}\|_{\ell^{\infty}}
\\\notag&+\sum_{m=3}^{l-1}\frac{l^{4}\cdots(m+4)^{4}(m+2)^{4}}{(\frac{l-m}{2})!}\big(C_2\|\chi_{\Gamma_2, \Gamma'_2}\|_{\ell^{\infty}}\big)^{\frac{l-m}{2}}\|K_{\Gamma_m,\Gamma'_m}\|_{\ell^{\infty}}
\\\notag\leq&2C^{l^2}+l^{2(l-3)}(2C_2C^{9})^{\frac{l-3}{2}}(3C^{9})+\sum_{m=2}^{l-1}l^{2(l-m)}(2C_2C^{9})^{\frac{l-m}{2}}(2C^{m^2})
\\\notag\leq&\frac{1}{3}C_3^{l^2}+C_3^{l(l-3)}\frac{C_3^{3l}}{3}+\sum_{m=2}^{l-1}C_3^{l(l-m)}\frac{C_{3}^{lm}}{3l}
\\\notag\leq&C_3^{l^2},
\end{align}
where the penultimate inequality follows from 
$$(2C_2C^9)^{\frac{1}{2}}l^2\leq C_3^l,\quad 3C^9\leq\frac{C_3^{3l}}{3}\quad\text{and}\quad 2C^{m^2}\leq \frac{C_3^{lm}}{3l} \;\text{for}\; m=2,\cdots, l-1.$$

{\bf Step 3:  Estimation the remainder term $R''_{\geq2r+3}$ in \eqref{form542-1-23}.}
Consider the remainder term of the form
\begin{equation}
\label{form571}
\tilde{R}:=\int_{0}^{1}g(t)(D\Phi^t_{\chi})^{-1}Q_{\Gamma_{l},\Gamma'_{l}}\circ\Phi_{\chi}^{t}dt
\end{equation}
with $\|g(t)\|_{L^{\infty}}\leq1$. Then one has
\begin{align}
\label{form572}
\|\tilde{R}(\bs{z})\|_{s}\leq&\|g(t)\|_{L^{\infty}}\|(D\Phi_{\chi}^{t})^{-1}Q_{\Gamma_{l},\Gamma'_{l}}\circ\Phi_{\chi}^{t}(\bs{z})\|_{s}
\\\notag\leq&\|(D\Phi^t_{\chi})^{-1}\|_{\ms{P}_{s}\mapsto\ms{P}_{s}}\|Q_{\Gamma_{l},\Gamma'_{l}}(\Phi_{\chi}^{t}(\bs{z}))\|_{s}.
\end{align}

Now, we estimate $(D\Phi^1_{\chi_{\Gamma_2, \Gamma'_2}})^{-1}R^{(1)}_{\geq 2r+3}\circ\Phi^1_{\chi_{\Gamma_2, \Gamma'_2}}$. Similarly with the proof of \eqref{form553}--\eqref{form554}, there exist $\Gamma_{r+1},\Gamma'_{r+1}\in\mc{H}_{r+1,N}$ 
such that
 $$\frac{1}{(r-1)!}ad_{\chi_{\Gamma_1,\Gamma'_1}}^{r}Z_3^{\leq N}+\sum_{m=2}^{r}\frac{1}{(r-m)!}ad_{\chi_{\Gamma'_1}}^{r-m+1}K_{2m+1}^{\leq N}=Q_{\Gamma_{r+1},\Gamma'_{r+1}}$$
with the coefficient estimate
\begin{equation}
\label{form573}
\|Q_{\Gamma_{r+1},\Gamma'_{r+1}}\|_{\ell^{\infty}}\leq 2C^{(r+1)^2},
\end{equation}
where for any $(a,\bs{j},\bs{h},\bs{k},n)\in\Gamma_{r+1}\bigcup\Gamma'_{r+1}$, one has
\begin{equation}
\label{form574}
n=\#\bs{h}\leq 3(r-1),\quad \bs{k}=\emptyset, \end{equation}
and thus
\begin{equation}
\label{form574-1-23}
\#\bs{j}=r+1+\#\bs{h}+2\#\bs{k}\leq 2(2r-1).
\end{equation}
%
%
%
In view of \eqref{re41-2} in \Cref{re41}, by \Cref{le42}, the estimates \eqref{form573}--\eqref{form574-1-23} and the fact $\|\Phi_{\chi_{\Gamma_1,\Gamma'_1}}^{t}(\bs{z}))\|_{s}<2\|z\|_{s}$, one has
\begin{align} 
\label{form575}
&\|Q_{\Gamma_{r+1},\Gamma'_{r+1}}(\Phi_{\chi_{\Gamma_1,\Gamma'_1}}^{t}(\bs{z}))\|_{s}
\\\notag\leq&6\sqrt{2}(4r-2)^3\frac{(12(4r-2)^2)^{4r-3}N^{6(8r-3)(r-1)}}{(\gamma/2^{r-1})^{3(r-1)}}\|Q_{\Gamma_{r+1},\Gamma'_{r+1}}\|_{\ell^{\infty}}\|\Phi_{\chi_{\Gamma_1,\Gamma'_1}}^{t}(z)\|_{s}^{2r+3}
\\\notag\leq&6\sqrt{2}\frac{8^{(r-1)^2}12^{4r-3}(4r-2)^{8r-3}N^{6(8r-3)(r-1)}}{\gamma^{3(r-1)}}2C^{(r+1)^2}(2\|z\|_{s})^{2r+3}
\\\notag=&\sqrt{2}\frac{2^{3r^2+20r-4}3^{4r-2}C^{(r+1)^2}(r-\frac{1}{2})^{8r-3}N^{6(8r-3)(r-1)}}{\gamma^{3(r-1)}}\|z\|_{s}^{2r+3}
\\\notag\leq&\frac{(2CN)^{48r(r-1)}}{\gamma^{3(r-1)}}\|z\|_{s}^{2r+3},
\end{align}
where the last inequality follows from $C^{(r+1)^2}< C^{48r(r-1)}$ and
$$2^{3r^2+20r-4}3^{4r-2}\sqrt{2}(r-\frac{1}{2})^{8r-3}<2^{3r^2+20r-4+(8r-4)+1+(8r-3)r}<2^{48r(r-1)}.$$
Hence, in view of \eqref{form549}, by \eqref{form572},  \eqref{form575} and the fact $\|(D\Phi^t_{\chi_{\Gamma_1,\Gamma'_1}})^{-1}\|_{\ms{P}_{s}\mapsto\ms{P}_{s}}<2$, one has
\begin{align}
\label{form576}
\|R^{(1)}_{\geq2r+3}(\bs{z})\|_{s}\leq&\|(D\Phi^t_{\chi_{\Gamma_1,\Gamma'_1}})^{-1}\|_{\ms{P}_{s}\mapsto\ms{P}_{s}}\|Q_{\Gamma_{r+1},\Gamma'_{r+1}}(\Phi_{\chi_{\Gamma_1,\Gamma'_1}}^{t}(\bs{z}))\|_{s}
\\\notag\leq&\frac{2(2CN)^{48r(r-1)}}{\gamma^{3(r-1)}}\|z\|_{s}^{2r+3}.
\end{align}
By \eqref{form576}, the facts $\|\Phi_{\chi_{\Gamma_2, \Gamma'_2}}^{1}(\bs{z})\|_{s}<2\|z\|_{s}$ and  $\|(D\Phi^t_{\chi_{\Gamma_2,\Gamma'_2}})^{-1}\|_{\ms{P}_{s}\mapsto\ms{P}_{s}}<2$, one has
\begin{align}
\label{form577}
&\|(D\Phi^1_{\chi_{\Gamma_2, \Gamma'_2}})^{-1}R^{(1)}_{\geq 2r+3}\circ\Phi^1_{\chi_{\Gamma_2, \Gamma'_2}}(\bs{z})\|_{s}
\\\notag\leq&\|(D\Phi^t_{\chi_{\Gamma_2,\Gamma'_2}})^{-1}\|_{\ms{P}_{s}\mapsto\ms{P}_{s}}\|R^{(1)}_{\geq2r+3}(\Phi^1_{\chi_{\Gamma_2, \Gamma'_2}}(\bs{z}))\|_{s}
\\\notag\leq&\frac{4(2CN)^{48r(r-1)}}{\gamma^{3(r-1)}}\|\Phi^1_{\chi_{\Gamma_2, \Gamma'_2}}(\bs{z})\|_{s}^{2r+3}
\\\notag\leq&\frac{2^{2r+5}(2CN)^{48r(r-1)}}{\gamma^{3(r-1)}}\|z\|_{s}^{2r+3}.
\end{align}

Next, we will estimate $R^{(2)}_{\geq2r+3}$ in \eqref{form569}.
Similarly with the proof of \eqref{form570} and \eqref{form570'}, there exist $\Gamma_{r+1},\Gamma'_{r+1}\in\mc{H}_{r+1,N}$ and $\Gamma_{r+2},\Gamma'_{r+2}\in\mc{H}_{r+2,N}$ such that
\begin{align}
&\sum_{m=1}^{2}\frac{1}{[\frac{l-m}{2}]!}(D\Phi^t_{\chi_{\Gamma_2, \Gamma'_2}})^{-1}ad_{\chi_{\Gamma_2,\Gamma'_2}}^{[\frac{r-m}{2}]+1}Z_{2m+1}^{\leq N}+\sum_{m=3}^{r}\frac{1}{[\frac{r-l}{2}]!}(d\Phi^t_{\chi_{\Gamma_2, \Gamma'_2}})^{-1}ad_{\chi_{\Gamma_2,\Gamma'_2}}^{[\frac{r-m}{2}]+1}K_{\Gamma_{m},\Gamma'_{m}}
\\\notag=&
Q_{\Gamma_{r+1},\Gamma'_{r+1}}+Q_{\Gamma_{r+2},\Gamma'_{r+2}},
\end{align}
with the coefficient estimates
\begin{equation}
\label{form579}
\|Q_{\Gamma_{r+1},\Gamma'_{r+1}}\|_{\ell^{\infty}}\leq C_3^{(r+1)^2}\quad\text{and}\quad\|Q_{\Gamma_{r+2},\Gamma'_{r+2}}\|_{\ell^{\infty}}\leq C_3^{(r+2)^2},
\end{equation}
where for any $(a,\bs{j},\bs{h},\bs{k},n)\in\Gamma_{r+1}\bigcup\Gamma'_{r+1}$, 
\begin{equation}
\label{form580'}
n=\#\bs{h}\leq 3r-2,\quad \bs{k}=\emptyset, \quad\#\bs{j}=r+1+\#\bs{h}+2\#\bs{k}\leq 4r-1;
\end{equation}
and for any $(a,\bs{j},\bs{h},\bs{k},n)\in\Gamma_{r+2}\bigcup\Gamma'_{r+2}$, one has
\begin{equation}
\label{form580}
n=\#\bs{h}\leq 3r+1,\quad \bs{k}=\emptyset, \quad\#\bs{j}=r+2+\#\bs{h}+2\#\bs{k}\leq 4r+3.
\end{equation}
In view of \eqref{form539-1-23}, by \eqref{re41-2} in \Cref{re41}, \Cref{le42}, the estimates \eqref{form579}--\eqref{form580} and the fact $\|\Phi_{\chi_{\Gamma_2, \Gamma'_2}}^{t}(\bs{z}))\|_{s}<2\|z\|_{s}$ for $|t|\leq1$, one has
\begin{align} 
\label{form581}
&\|Q_{\Gamma_{r+1},\Gamma'_{r+1}}(\Phi_{\chi_{\Gamma_2, \Gamma'_2}}^{t}(\bs{z}))\|_{s}
\\\notag\leq&6\sqrt{2}(4r-1)^3\frac{(12(4r-1)^2)^{4r-1}N^{2(8r-1)(3r-2)}}{(\gamma/2^{r-2})^{3r-2}}\|Q_{\Gamma_{r+1},\Gamma'_{r+1}}\|_{\ell^{\infty}}\|\Phi_{\chi_{\Gamma_2, \Gamma'_2}}^{t}(z)\|_{s}^{2r+3}
\\\notag\leq&\sqrt{2}\frac{2^{(r-2)(3r-2)+24r+1}3^{4r}r^{8r+1}N^{2(8r-1)(3r-2)}}{\gamma^{3r-2}}C_3^{(r+1)^2}(2\|z\|_{s})^{2r+3}
\\\notag\leq&\frac{(2C_3N)^{48r^2-38r+4}}{8\gamma^{3r-2}}\|z\|_{s}^{2r+3},
\end{align}
where the last inequality follows form $C_3^{(r+1)^2}\leq C_3^{48r^2-38r+4}$ and 
$$2^{(r-2)(3r-2)+24r+1}3^{4r}\sqrt{2}r^{8r+1}<2^{(r-2)(3r-2)+24r+1+8r+1+r(8r+1)}<2^{48r^2-38r+4}.$$
Similarly, one has
\begin{align} 
\label{form582}
&\|Q_{\Gamma_{r+2},\Gamma'_{r+2}}(\Phi_{\chi_{\Gamma_2, \Gamma'_2}}^{t}(\bs{z}))\|_{s}
\\\notag\leq&6\sqrt{2}(4r+3)^3\frac{(12(4r+3)^2)^{4r+3}N^{2(8r+7)(3r+1)}}{(\gamma/2^{r-2})^{3r+1}}\|Q_{\Gamma_{r+2},\Gamma'_{r+2}}\|_{\ell^{\infty}}\|\Phi_{\chi_{\Gamma_2, \Gamma'_2}}^{t}(z)\|_{s}^{2r+5}
\\\notag\leq&\sqrt{2}\frac{2^{(r-2)(3r+1)+24(r+1)+1}3^{4(r+1)}(r+1)^{8r+9}N^{2(8r+7)(3r+1)}}{\gamma^{3r+1}}C_3^{(r+2)^2}(2\|z\|_{s})^{2r+5}
\\\notag\leq&\frac{(2C_3N)^{48r^2+58r+14}}{8\gamma^{3r+1}}\|z\|_{s}^{2r+5},
\end{align}
where the last inequality follows form $C_3^{(r+2)^2}\leq C_3^{48r^2+58r+14}$ and 
\begin{align*}
2^{(r-2)(3r+1)+24(r+1)+1}3^{4(r+1)}\sqrt{2}(r+1)^{8r+9}<&2^{(r-2)(3r+1)+24(r+1)+1+8(r+1)+1+r(8r+9)}
\\<&2^{48r^2+58r+14}.
\end{align*}
Hence, by \eqref{form572}, \eqref{form581}, \eqref{form582} and the fact 
$\varepsilon\leq\gamma^{2}(2C_3N)^{-29r-7}$, one has
\begin{align}
\label{form583}
\|R^{(2)}_{\geq2r+3}(\bs{z})\|_{s}\leq&2\big(\|Q_{\Gamma_{r+1},\Gamma'_{r+1}}(\Phi_{\chi_{\Gamma_2, \Gamma'_2}}^{t}(\bs{z}))\|_{s}+\|Q_{\Gamma_{r+2},\Gamma'_{r+2}}(\Phi_{\chi_{\Gamma_2, \Gamma'_2}}^{t}(\bs{z}))\|_{s}\big)
\\\notag\leq&\frac{(2C_3N)^{48r^2-38r+4}}{4\gamma^{3r-2}}\|z\|_{s}^{2r+3}+\frac{(2C_3N)^{48r^2+58r+14}}{4\gamma^{3r+1}}\|z\|_{s}^{2r+5}
\\\notag\leq&\frac{(2C_3N)^{48r^2}}{2\gamma^{3r-2}}\|z\|_{s}^{2r+3}.
\end{align}
In view of \eqref{form542-1-23}, by \eqref{form577} and \eqref{form583}, one has
\begin{align*}
\|R''_{\geq2r+3}(\bs{z})\|_{s}\leq&\|(D\Phi^1_{\chi_{\Gamma_2, \Gamma'_2}})^{-1}R^{(1)}_{\geq 2r+3}\circ\Phi^1_{\chi_{\Gamma_2, \Gamma'_2}}(\bs{z})\|_{s}+\|R^{(2)}_{\geq2r+3}(\bs{z})\|_{s}
\\\leq&\frac{2^{2r+5}(2CN)^{48r(r-1)}}{\gamma^{3(r-1)}}\|z\|_{s}^{2r+3}+\frac{(2C_3N)^{48r^2}}{2\gamma^{3r-2}}\|z\|_{s}^{2r+3}
\\\leq&\frac{(2C_3N)^{48r^2}}{\gamma^{3r-2}}\|z\|_{s}^{2r+3},
\end{align*}
where the last inequality follows from the fact $C_3=C^2>8C$.
\end{proof}

\subsection{Elimination of higher order terms}
\label{sec52}
This subsection aims at eliminating the non-normal form parts of the rational 
vector field $Q_{\Gamma_{l},\Gamma_{l}'}$ in \Cref{th51} by using the integrable vector field $Z_{3}^{\leq N}+Z_{5}^{\leq N}$.

\begin{theorem}
\label{th52}
Fix an integer $r\geq4$. For any $s\geq0$ and $0<\varepsilon\ll 1$, $N\geq r$, $\gamma\in(0,1)$ satisfying $\varepsilon\leq \gamma^{\frac{13}{2}}(C_3N)^{-1128(r-1)}$,
there exists a nearly identity coordinate transformation $$\phi^{(3)}:\Big(\mc{U}^{N}_{\frac{\gamma}{2}}\times\mc{U}^{N}_{\frac{\gamma}{2}}\Big)\bigcap B_{s}(\frac{3}{2}\varepsilon)\to\Big(\mc{U}^{N}_{\frac{\gamma}{2^{r-2}}}\times\mc{U}^{N}_{\frac{\gamma}{2^{r-2}}}\Big)\bigcap B_{s}(2\varepsilon)$$ such that the vector field \eqref{form51} is transformed into 
\begin{align}
\label{form584}
X'''= &Z_1+Z_3^{\leq N}+Z_{5}^{\leq N}+\sum_{l=3}^{r}Z_{\Gamma_{l},\Gamma'_{l}}+\sum_{l=3}^{r}\tilde{Z}_{\tilde{\Gamma}_{l}}+\sum_{l=3}^{r-1}\tilde{Z}_{\tilde{\tilde{\Gamma}}_{l}}+R'''_{\geq2r+3}
\\\notag&+(D\phi^{(3)})^{-1}(D\phi^{(2)})^{-1}\big(Z_3^{>N}+\sum_{l=2}^{r}K^{>N}_{2l+1}+R'_{\geq2r+3}\big)\circ\phi^{(2)}\circ\phi^{(3)},
\end{align}
 where 
\begin{enumerate}[(i)]
\item  the transformation $\phi^{(3)}$ satisfies the estimate
    \begin{equation}
    \label{form585}
    \|\bs{z}-\phi^{(3)}(\bs{z})\|_{s}\leq\frac{2^{9(r-2)}(C_3N)^{2256}}{\gamma^{9}}\|z\|_s^{5},
    \end{equation}
and the same estimate is fulfilled by the inverse transformation;
  \item the integrable rational vector field $Z_{\Gamma_{l}, \Gamma'_l}\in\ms{H}_{l,N}^{{\rm rev}}$ is defined by \Cref{defintrat} satisfying the coefficient estimate
\begin{equation}
\label{form566-1-26}
 \|Z_{\Gamma_{l},\Gamma'_l}\|_{\ell^{\infty}}\leq(C_3r)^{4(l-1)(l-2)};
 \end{equation}
\item the rational normal form terms $\tilde{Z}_{\tilde{\Gamma}_{l}},\tilde{Z}_{\tilde{\tilde{\Gamma}}_{l}}\in\ms{H}_{l,N}^{{\rm rev}}$ are defined by \Cref{def43} satisfying the coefficient estimates
\begin{align}
\label{form586}
\|Z_{\tilde{\Gamma}_{l}}\|_{\ell^{\infty}}\leq(C_3r)^{4(l-1)(l-2)}\quad\text{and}\quad
\|\tilde{Z}_{\tilde{\tilde{\Gamma}}_{l}}\|_{\ell^{\infty}}\leq(C_3r)^{4(l-1)l};
\end{align}
\item the remainder term $R'''_{\geq2r+3}$ satisfies the estimate
  \begin{equation}
  \label{form587}
  \|R'''_{\geq2r+3}(\bs{z})\|_{s}\leq\frac{4^{(r-3)(r+3)}(C_3N)^{564(r-1)^2}}{\gamma^{13r-17}}\|z\|_{s}^{2r+3}.
  \end{equation}
\end{enumerate}
\end{theorem}
We prove this theorem via an iterative process comprising $(r-3)$ steps. In order to state the iterative lemma, we need some notations. For $i=4,\cdots, r$, denote
\begin{equation}
\label{form588}
\varepsilon_i=\varepsilon\big(2-\frac{i-3}{2(r-3)}\big).
\end{equation}
We proceed by induction from $i-1$ to $i$. 
Initially, write
\begin{equation}
\label{form589}
X^{(3)}=Z_1+Z_3^{\leq N}+Z^{\leq N}_5+Z_{\Gamma_3,\Gamma'_3}+\tilde{Z}_{\tilde{\Gamma}_3}+\sum_{l=4}^rQ_{\Gamma_{l},\Gamma'_l}+R''_{\geq2r+3},
\end{equation}
with $Q_{\Gamma_{l},\Gamma'_l}$ and $R''_{\geq2r+3}$ in \Cref{th51}.
\begin{lemma}
\label{le53}
Assume that the vector field
\begin{align}
\label{form590-1-24}
X^{(i-1)}=&Z_1+Z_3^{\leq N}+Z_{5}^{\leq N}+\sum_{l=3}^{i-1}Z_{\Gamma_{l}^{(l-1)},\Gamma'^{(l-1)}_{l}}+\sum_{l=3}^{i-1}\tilde{Z}_{\tilde{\Gamma}_{l}}+\sum_{l=3}^{i-2}\tilde{Z}_{\tilde{\tilde{\Gamma}}_{l}}
\\\notag&+\sum_{l=i}^{r}Q_{\Gamma^{(i-1)}_{l},\Gamma'^{(i-1)}_{l}}+R^{(i-1)}_{\geq2r+3}
\end{align}
is defined in $\big(\mc{U}^{N}_{\frac{\gamma}{2^{r+2-i}}}\times\mc{U}^{N}_{\frac{\gamma}{2^{r+2-i}}}\big)\bigcap B_{s}(\varepsilon_{i-1})$, 
where
\begin{enumerate}[(i)]
 \item the integrable rational vector field $Z_{\Gamma_{l}^{(l-1)},\Gamma'^{(l-1)}_l}\in\ms{H}_{l,N}^{{\rm rev}}$ is defined by \Cref{defintrat} satisfying the coefficient estimate
\begin{equation}
\label{form593-1}
 \|Z_{\Gamma_{l}^{(l-1)},\Gamma'^{(l-1)}_l}\|_{\ell^{\infty}}\leq C_3(C_3r)^{4(l-2)(l-1)-4}.
 \end{equation}
Moreover, for any $(a,\bs{j},\bs{h},\bs{k},n)\in\Gamma_{l}^{(l-1)}\bigcup\Gamma'^{(l-1)}_l$, when  $l=3$, one has $\#\bs{h}\leq3$, $\#\bs{k}=0$, and when $l\geq4$, one has
  \begin{equation}
  \label{form594-1}
  \#\bs{k}\leq3(l-4), \quad\#\bs{h}+\#\bs{k}\leq10l-33;
  \end{equation}
 \item the rational normal form $\tilde{Z}_{\tilde{\Gamma}_{l}}\in\ms{H}_{l,N}^{{\rm rev}}$ is defined by \Cref{def43} satisfying the coefficient estimate
\begin{equation}
\label{form593-1-24}
 \|\tilde{Z}_{\tilde{\Gamma}_{l}}\|_{\ell^{\infty}}\leq2C_3(C_3r)^{4(l-2)(l-1)-4}.
 \end{equation}
Moreover, for any $(a,\bs{j},\bs{h},\bs{k},n)\in\tilde{\Gamma}_{l}$, when  $l=3$, one has $\#\bs{h}\leq4$, $\#\bs{k}=0$, and when  $l\geq4$, one has
  \begin{equation}
  \label{form594-1-24}
 \#\bs{k}\leq3(l-4)+1, \quad \#\bs{h}+\#\bs{k}\leq10l-32;
  \end{equation}
 \item the rational normal form $\tilde{Z}_{\tilde{\tilde{\Gamma}}_{l}}\in\ms{H}_{l,N}^{{\rm rev}}$ is defined by \Cref{def43} satisfying the coefficient estimate
\begin{equation}
\label{form593}
 \|\tilde{Z}_{\tilde{\tilde{\Gamma}}_{l}}\|_{\ell^{\infty}}\leq\frac{C_3}{2}(C_3r)^{4(l-1)l-4}.
 \end{equation}
Moreover, for any $(a,\bs{j},\bs{h},\bs{k},n)\in\tilde{\tilde{\Gamma}}_{l}$, one has
  \begin{equation}
  \label{form594}
\#\bs{k}\leq3(l-3)+1,\quad \#\bs{h}\leq\#\bs{h}+\#\bs{k}\leq10l-22;
  \end{equation}
\item the resonant rational vector field $Q_{\Gamma^{(i-1)}_{l},\Gamma'^{(i-1)}_{l}}\in\ms{H}_{l,N}^{{\rm rev}}$ is defined by \Cref{def42} satisfying the coefficient estimate
 \begin{equation}
 \label{form595}
 \|Q_{\Gamma^{(i-1)}_{l},\Gamma'^{(i-1)}_{l}}\|_{\ell^{\infty}}\leq C_3(C_3r)^{2(l-2)(l+i-2)-4}.
 \end{equation}
Moreover, for any $(a,\bs{j},\bs{h},\bs{k},n)\in\Gamma^{(i-1)}_{l}\bigcup\Gamma'^{(i-1)}_{l}$, when $i=4$, one has $\#\bs{h}\leq3l-5$, $\bs{k}=\emptyset$, and when $i\geq5$, rewriting uniquely as $l=(i-3)n_1+n_2$ with $n_1\in\mb{N}_*$, $n_2\in\{2,1,0,-1,\cdots,-(i-6)\}$, one has
  \begin{equation}
  \label{form596}
 \#\bs{k}\leq3l-4(n_1+1),\quad\#\bs{h}+\#\bs{k}\leq10l-11(n_1+1);
  \end{equation} 
%
 \item the remainder term $R^{(i-1)}_{\geq2r+3}$ satisfies the estimate
   \begin{equation}
   \label{form597}
   \|R^{(i-1)}_{\geq2r+3}(\bs{z})\|_{s}\leq\frac{4^{(i-4)(r+3)}(C_3N)^{564(r-1)^2}}{\gamma^{13r-17}}\|z\|_{s}^{2r+3}.
    \end{equation}
\end{enumerate}
Then there exists a rational vector field $\chi_{\Gamma_{i-2},\Gamma'_{i-2}}\in\ms{H}_{i-2,N}^{{\rm an-rev}}$
such that  the transformed field under the time 1 flow generated by $\chi_{\Gamma_{i-2},\Gamma'_{i-2}}$ is
\begin{align}
\label{form590}
X^{(i)}:=&e^{ad_{\chi_{2i+1}}}X^{(i-1)}
\\\notag=&Z_1+Z_3^{\leq N}+Z_{5}^{\leq N}+\sum_{l=3}^{i}Z_{\Gamma_{l}^{(l-1)},\Gamma'^{(l-1)}_{l}}+\sum_{l=3}^{i}\tilde{Z}_{\tilde{\Gamma}_{l}}+\sum_{l=3}^{i-1}\tilde{Z}_{\tilde{\tilde{\Gamma}}_{l}}
\\\notag&+\sum_{l=i+1}^{r}Q_{\Gamma^{(i)}_{l},\Gamma'^{(i)}_{l}}+R^{(i)}_{\geq2r+3},
\end{align}
where the transformation 
$$\Phi_{\chi_{\Gamma_{i-2},\Gamma'_{i-2}}}^{1}:\big(\mc{U}^{N}_{\frac{\gamma}{2^{r+1-i}}}\times\mc{U}^{N}_{\frac{\gamma}{2^{r+1-i}}}\big)\bigcap B_{s}(\varepsilon_i)\to\big(\mc{U}^{N}_{\frac{\gamma}{2^{r+2-i}}}\times\mc{U}^{N}_{\frac{\gamma}{2^{r+2-i}}}\big)\bigcap B_{s}(\varepsilon_{i-1})$$
 satisfies the estimate
    \begin{equation}
    \label{form591}
    \|\bs{z}-\Phi^{1}_{\chi_{\Gamma_{i-2},\Gamma'_{i-2}}}(\bs{z})\|_{s}\leq\frac{2^{(13i-43)(r+1-i)}(C_3N)^{564(i-2)^2}}{\gamma^{13i-43}}\|z\|_s^{2i-3},
    \end{equation}
and the same estimate is fulfilled by the inverse transformation.
Moreover, the integrable rational vector field $Z_{\Gamma_{l}^{(l-1)},\Gamma'^{(l-1)}_l}\in\ms{H}_{l,N}^{{\rm rev}}$, the rational normal form terms  $\tilde{Z}_{\tilde{\Gamma}_{l}}, \tilde{Z}_{\tilde{\tilde{\Gamma}}_{l}}\in\ms{H}_{l,N}^{{\rm rev}}$, the resonant rational vector field $Q_{\Gamma^{(i)}_{l},\Gamma'^{(i)}_{l}}\in\ms{H}_{l,N}^{{\rm rev}}$ and the remainder term $R^{(i)}_{\geq2r+3}$ satisfy the same estimates \eqref{form593-1}--\eqref{form597} with $i$ in place of $i-1$.
\end{lemma}

\begin{proof}
We are going to eliminate the non-normal form parts of the resonant  rational vector field  $Q_{\Gamma_{i}^{(i-1)},\Gamma_{i}'^{(i-1)}}$. 
By \Cref{le44}, there exist a rational vector field $\chi_{\Gamma_{i-2},\Gamma'_{i-2}}\in\ms{H}^{{\rm an-rev}}_{i-2,N}$, an integrable rational vector field $Z_{\Gamma_i^{(i-1)},\Gamma_{i}'^{(i-1)}}\in\ms{H}^{{\rm rev}}_{i,N}$ and two rational normal form terms $\tilde{Z}_{\tilde{\Gamma}_{i}}\in\ms{H}_{i,N}^{{\rm rev}}$,  $\tilde{Z}_{\tilde{\tilde{\Gamma}}_{i-1}}\in\ms{H}_{i-1,N}^{{\rm rev}}$ such that
\begin{equation}
\label{form598}
[Z_3^{\leq N}+Z_5^{\leq N},\chi_{\Gamma_{i-2},\Gamma'_{i-2}}]+Q_{\Gamma^{(i-1)}_i,\Gamma'^{(i-1)}_i}=Z_{\Gamma_i^{(i-1)},\Gamma_{i}'^{(i-1)}}+\tilde{Z}_{\tilde{\Gamma}_i}+\tilde{Z}_{\tilde{\tilde{\Gamma}}_{i-1}},
\end{equation}
with the coefficient estimates
\begin{equation}
\label{form5101}
\|\chi_{\Gamma_{i-2},\Gamma'_{i-2}}\|_{\ell^{\infty}}\leq\|Q_{\Gamma^{(i-1)}_i,\Gamma'^{(i-1)}_i}\|_{\ell^{\infty}}\leq C_3(C_3r)^{4(i-2)(i-1)-4},
\end{equation}
\begin{equation}
 \label{form599}
\|Z_{\Gamma^{(i-1)}_{i},\Gamma'^{(i-1)}_i}\|_{\ell^{\infty}}\leq\|Q_{\Gamma^{(i-1)}_i,\Gamma'^{(i-1)}_i}\|_{\ell^{\infty}}\leq C_3(C_3r)^{4(i-2)(i-1)-4},
\end{equation}
\begin{equation}
 \label{form5100}
 \|\tilde{Z}_{\tilde{\Gamma}_{i}}\|_{\ell^{\infty}}\leq2\|Q_{\Gamma^{(i-1)}_i,\Gamma'^{(i-1)}_i}\|_{\ell^{\infty}}\leq2C_3(C_3r)^{4(i-2)(i-1)-4},
 \end{equation}
\begin{equation}
 \label{form5100-1-24}
 \|\tilde{Z}_{\tilde{\tilde{\Gamma}}_{i-1}}\|_{\ell^{\infty}}\leq\frac{1}{2}\|Q_{\Gamma^{(i-1)}_i,\Gamma'^{(i-1)}_i}\|_{\ell^{\infty}}\leq\frac{C_3}{2}(C_3r)^{4(i-2)(i-1)-4}.
 \end{equation}
Notice that for any $(a,\bs{j},\bs{h},\bs{k},n)\in\Gamma^{(i-1)}_{i}\bigcup\Gamma'^{(i-1)}_{i}$, when $i=4$, one has
\begin{equation}
\label{form5103-3-7-1}
 \#\bs{k}=0, \quad\#\bs{h}+\#\bs{k}\leq3i-5=7;
  \end{equation}
when $i\geq5$, one has $i=(i-3)n_1+n_2$ with $n_1=2$, $n_2=-(i-6)$, and thus
\begin{equation}
\label{form5103}
\#\bs{k}\leq3i-12,\quad \#\bs{h}+\#\bs{k}\leq10i-33.
 \end{equation}
Thus by \Cref{re43},  for any $(a,\bs{j},\bs{h},\bs{k},n)\in\Gamma_{i-2}\bigcup\Gamma'_{i-2}$ or $\tilde{\Gamma}_{i}$ or $\tilde{\tilde{\Gamma}}_{i-1}$ with $i\geq4$, one has
\begin{equation}
\label{form5104}
\#\bs{k}\leq3i-11\quad\text{and} \quad\#\bs{h}+\#\bs{k}\leq10i-32.
\end{equation}

{\bf{Step 1: The transformation generated by $\chi_{\Gamma_{i-2},\Gamma'_{i-2}}$. }}
In view of \eqref{re41-2} in \Cref{re41} and \eqref{form5104},  there exists a finite set $\mc{F}_{i-2}\subset\mb{N}^3$ such that
$$\Gamma_{i-2}=\bigcup_{(\alpha,\beta,\beta')\in\mc{F}_{i-2}}\Gamma^{(\alpha,\beta,\beta')}_{i-2}\quad\text{and}\quad\Gamma'_{i-2}=\bigcup_{(\alpha,\beta,\beta')\in\mc{F}_{i-2}}\Gamma'^{(\alpha,\beta,\beta')}_{i-2}$$
with 
\begin{equation}
\label{form3-7-592}
\begin{aligned}
&\beta'=\#\bs{k}\leq3i-11,
\\&\beta+\beta'=\#\bs{h}+\#\bs{k}\leq10i-32,
\\&\alpha=i-2+\beta+2\beta'\leq14i-45.
\end{aligned}
\end{equation}
Hence, by \Cref{le42}, \eqref{form3-7-592}, the inequalities $\sqrt{12}(14i-45)<50i$ and  the coefficient estimate \eqref{form5101},  for any $\bs{z}\in\mc{U}^{N}_{\frac{\gamma}{2^{r+1-i}}}$, one has
\begin{align}
\label{form5105}
&\|\chi_{\Gamma_{i-2},\Gamma'_{i-2}}(\bs{z})\|_s
\\\notag<&6\sqrt{2}\times(14i-45)^3\max_{(\alpha,\beta,\beta')\in\mc{F}_{i-2}}\frac{(\sqrt{12}\alpha)^{2(\alpha-1)}N^{(4\alpha+2)(\beta+\beta')}}{(\gamma/2^{r+1-i})^{\beta+2\beta'}}\|\chi_{\Gamma_{i-2},\Gamma'_{i-2}}\|_{\ell^{\infty}}\|z\|_s^{2i-3}
\\\notag\leq&6\sqrt{2}\times(14i)^3\frac{(50i)^{4(7i-23)}N^{2(28i-89)(10i-32)}}{(\gamma/2^{r+1-i})^{13i-43}}C_3(C_3r)^{4(i-2)(i-1)-4}\|z\|_s^{2i-3}
\\\notag\leq&\frac{2^{(13i-43)(r+1-i)}(C_3N)^{564(i-2)^2}}{\gamma^{13i-43}}\|z\|_s^{2i-3},
\end{align}
where the last inequality follows form the inequalities
\begin{align*}
N^{2(28i-89)(10i-32)}r^{4(i-2)(i-1)-4}\leq&N^{2(28i-89)(10i-32)+4(i-2)(i-1)-4}
\\\leq&N^{564(i-2)^2}
\end{align*}
and 
\begin{align*}
6\sqrt{2}\times(14i)^3(50i)^{4(7i-23)}C_3C_3^{4(i-2)(i-1)-4}\leq&(50i)^{28i-89}C_3^{4(i-2)(i-1)-3}
\\\leq&C_3^{(i-2)(28i-89)+4(i-2)(i-1)-3}
\\<&C_3^{564(i-2)^2}.
\end{align*}
Denote by the flow $\Phi_{\chi_{\Gamma_{i-2},\Gamma'_{i-2}}}^{t}$ generated by $\chi_{\Gamma_{i-2},\Gamma'_{i-2}}$. Then for any $|t|\leq 1$, as long as $\|\Phi_{\chi_{\Gamma_{i-2},\Gamma'_{i-2}}}^{t}(z)\|_{s}\leq\frac{3}{2}\|z\|_{s}$, one has
\begin{align}
\label{form5106}
\|\Phi_{\chi_{\Gamma_{i-2},\Gamma'_{i-2}}}^{t}(\bs{z})-\bs{z}\|_{s}\leq&\|\chi_{\Gamma_{i-2},\Gamma'_{i-2}}(\bs{z})\|_{s}
\\\notag\leq&\frac{2^{(13i-43)(r+1-i)}(C_3N)^{564(i-2)^2}}{\gamma^{13i-43}}\|z\|_s^{2i-3}.
\end{align}
Now, we make sure the condition \eqref{form411} in \Cref{le41}.
Notice that
\begin{align}
\label{form5107}
&\sup_{a\in\mb{N}_*}|a|^{2s}\big||\Phi_{\chi_{\Gamma_{i-2},\Gamma'_{i-2}}}^{t}(z)_{a}|^{2}-I_{a}\big|
\\\notag&\leq\|\Phi_{\chi_{\Gamma_{i-2},\Gamma'_{i-2}}}^{t}(z)-z\|_{s}(\|\Phi_{\chi_{\Gamma_{i-2},\Gamma'_{i-2}}}^{t}(z)\|_{s}+\|z\|_{s})
\\\notag&\leq\frac{5}{2}\|\chi_{\Gamma_{i-2},\Gamma'_{i-2}}(z)\|_{s}\|z\|_{s}.
\end{align}
For $\bs{z}\in B_{s}(\varepsilon_i)$, one has $\|z\|_s\leq\frac{\varepsilon_i}{\sqrt{2}}<\sqrt{2}\varepsilon$. 
By \eqref{form5105} and \eqref{form5107} and a bootstrap argument, for any
$$\varepsilon\leq\frac{\gamma^{\frac{13}{2}}}{(C_3N)^{1128(r-1)}}<\frac{1}{\sqrt{2}}\Big(\frac{(\gamma/2^{r-l+1})^{13i-41}}{144\times5(r+1)N^{4r+3}(C_3N)^{564(i-2)^2}}\Big)^{\frac{1}{2(i-2)}},$$
the transformation
$$\Phi_{\chi_{\Gamma_{i-2},\Gamma'_{i-2}}}^{t}:\big(\mc{U}^{N}_{\frac{\gamma}{2^{r+1-i}}}\times\mc{U}^{N}_{\frac{\gamma}{2^{r+1-i}}}\big)\bigcap B_{s}(\varepsilon_i)\to \big(\mc{U}^{N}_{\frac{\gamma}{2^{r+2-i}}}\times\mc{U}^{N}_{\frac{\gamma}{2^{r+2-i}}}\big)\bigcap B_{s}(\varepsilon_{i-1})$$
 is well defined for any $|t|\leq 1$ satisfying the estimate \eqref{form5106}.

By the fact $[Z_1, \chi_{\Gamma_{i-2},\Gamma'_{i-2}}]=0$, the homological equation \eqref{form598} and \Cref{le43}, the vector field \eqref{form590-1-24} is transformed into the vector field \eqref{form590}, where for $l=i+1,\cdots, r$, one has
\begin{align}
\label{form5113}
Q_{\Gamma^{(i)}_{l},\Gamma'^{(i)}_{l}}=&\sum_{i+k(i-2)=l}\frac{d_1(k)}{(k+1)!}+\sum_{i-1+k(i-2)=l}\frac{d_2(k)}{(k+1)!}+\sum_{\substack{m+k(i-2)=l\\k\geq1,3\leq m\leq i-1}}\frac{d_3(k,m)}{k!}
\\\notag&+\sum_{\substack{m+k(i-2)=l\\k\geq1,3\leq m\leq i-2}}\frac{d_4(k,m)}{k!}+\sum_{\substack{m+k(i-2)=l\\k\geq0, i\leq m\leq r}}\frac{d_5(k,m)}{k!}
\end{align}
with $\Gamma^{(i)}_{l},\Gamma'^{(i)}_{l}\in\mc{H}_{l,N}$ and 
\begin{align}
\label{form5113-1-25-1}
d_1(k)=&ad_{\chi_{\Gamma_{i-2},\Gamma'_{i-2}}}^{k}\big(Z_{\Gamma^{(i-1)}_i,\Gamma'^{(i-1)}_i}+\tilde{Z}_{\tilde{\Gamma}_i}-Q_{\Gamma^{(i-1)}_i,\Gamma'^{(i-1)}_i}\big)
\\\label{form5113-1-25-2}
d_2(k)=&ad_{\chi_{\Gamma_{i-2},\Gamma'_{i-2}}}^{k}\tilde{Z}_{\tilde{\tilde{\Gamma}}_{i-1}}
\\\label{form5113-1-25-3}
d_3(k,m)=&ad_{\chi_{\Gamma_{i-2},\Gamma'_{i-2}}}^{k}\big(Z_{\Gamma^{(m-1)}_m,\Gamma'^{(m-1)}_m}+\tilde{Z}_{\tilde{\Gamma}_m}\big)
\\\label{form5113-1-25-4}
d_4(k,m)=&ad_{\chi_{\Gamma_{i-2},\Gamma'_{i-2}}}^{k}\tilde{Z}_{\tilde{\tilde{\Gamma}}_m}
\\\label{form5113-1-25-5}
d_5(k,m)=&ad_{\chi_{\Gamma_{i-2},\Gamma'_{i-2}}}^kQ_{\Gamma^{(i-1)}_{m},\Gamma'^{(i-1)}_{m}};
\end{align}
and the remainder term 
\begin{equation}
\label{form5114}
R^{(i)}_{\geq2r+3}=R_1+R_2
\end{equation}
with
\begin{align}
\label{form5123}
R_1&=\int_0^1\frac{(1-t)^{[\frac{r-2}{i-2}]}}{[\frac{r-2}{i-2}]!}(D\Phi^t_{\chi_{\Gamma_{i-2},\Gamma'_{i-2}}})^{-1}d_1\Big([\frac{r-i}{i-2}]+1\Big)\circ\Phi^t_{\chi_{\Gamma_{i-2},\Gamma'_{i-2}}}dt
\\\notag&+\int_0^1\frac{(1-t)^{[\frac{r-1}{i-2}]}}{[\frac{r-1}{i-2}]!}(D\Phi^t_{\chi_{\Gamma_{i-2},\Gamma'_{i-2}}})^{-1}
d_2\Big([\frac{r-i+1}{i-2}]+1\Big)\circ\Phi^t_{\chi_{\Gamma_{i-2},\Gamma'_{i-2}}}dt
\\\notag&+\sum_{m=3}^{i-1}\int_0^1\frac{(1-t)^{[\frac{r-i+1}{i-2}]}}{[\frac{r-m}{i-2}]!}(D\Phi^t_{\chi_{\Gamma_{i-2},\Gamma'_{i-2}}})^{-1}d_3\Big([\frac{r-m}{i-2}]+1,m\Big)\circ\Phi^t_{\chi_{\Gamma_{i-2},\Gamma'_{i-2}}}dt
\\\notag&+\sum_{m=3}^{i-2}\int_0^1\frac{(1-t)^{[\frac{r-m}{i-2}]}}{[\frac{r-m}{i-2}]!}(D\Phi^t_{\chi_{\Gamma_{i-2},\Gamma'_{i-2}}})^{-1}d_4\Big([\frac{r-m}{i-2}]+1,m\Big)\circ\Phi^t_{\chi_{\Gamma_{i-2},\Gamma'_{i-2}}}dt
\\\notag&+\sum_{m=i}^{r}\int_0^1\frac{(1-t)^{[\frac{r-m}{i-2}]}}{[\frac{r-m}{i-2}]!}(D\Phi^t_{\chi_{\Gamma_{i-2},\Gamma'_{i-2}}})^{-1}d_5\Big([\frac{r-m}{i-2}]+1,m\Big)\circ\Phi^t_{\chi_{\Gamma_{i-2},\Gamma'_{i-2}}}dt,
\end{align}
\begin{equation}
\label{form5123-1-26}
R_2=(D\Phi^1_{\chi_{\Gamma_{i-2},\Gamma'_{i-2}}})^{-1}R^{(i-1)}_{\geq 2r+3}\circ\Phi^1_{\chi_{\Gamma_{i-2},\Gamma'_{i-2}}}.
\end{equation}

{\bf{Step 2: Estimation the number of small divisors in the new rational vector field  $Q_{\Gamma^{(i)}_{l},\Gamma'^{(i)}_{l}}$.}}
%
We will show that for any $(a,\bs{j},\bs{h},\bs{k},n)\in\Gamma^{(i)}_{l}\bigcup\Gamma'^{(i)}_{l}$ with $l\geq i+1$, rewriting as $l=(i-2)n_1+n_2$ with $n_1\in\mb{N}_*$ and $n_2\in\{2,1,0,-1,\cdots,-(i-5)\}$, one has
 $$\#\bs{k}\leq3l-4(n_1+1),\quad\#\bs{h}+\#\bs{k}\leq10l-11(n_1+1),$$
which is the relationship \eqref{form596} for the  resonant rational vector field  $Q_{\Gamma^{(i)}_{l},\Gamma'^{(i)}_{l}}$.
In view of \eqref{form5113}, we divide it into the following three parts:

\noindent $\bullet$ the part of $\Gamma^{(i)}_{l}\bigcup\Gamma'^{(i)}_{l}$ determined by the rational normal form termxs $\{\tilde{Z}_{\tilde{\tilde{\Gamma}}_{m}}\}_{m=3}^{i-1}$ and the rational vector field $\chi_{\Gamma_{i-2},\Gamma'_{i-2}}$.

Notice that $l=k(i-2)+m$ with $k=n_1-1$ and $m=i-2+n_2$. Then by \Cref{re42} and the relationships \eqref{form594}, \eqref{form5104}, one has 
\begin{align}
\label{form5117}
\#\bs{k}\leq&3(m-3)+1+(3i-11+1)k
\\\notag=&3l-4(n_1+1),
\end{align}
\begin{align}
\label{form5118}
\#\bs{h}+\#\bs{k}\leq&10m-22+(10i-32+1)k
\\\notag=&10l-11(n_1+1).
\end{align}
$\bullet$ the part of $\Gamma^{(i)}_{l}\bigcup\Gamma'^{(i)}_{l}$ determined by the rational normal form terms $\{Z_{\Gamma^{(m-1)}_m,\Gamma'^{(m-1)}_m}+\tilde{Z}_{\tilde{\Gamma}_{m}}\}_{m=3}^{i}$ and the rational vector field $\chi_{\Gamma_{i-2},\Gamma'_{i-2}}$.

Notice that  $l=k(i-2)+m$ with $k=n_1-1$ and $m=i-2+n_2$.  Then by \Cref{re42} and the relationships \eqref{form594-1}, \eqref{form594-1-24}, \eqref{form5104}, when $m=3$, one has
\begin{align}
\label{form-3-18-1}
\#\bs{k}\leq&0+(3i-11+1)k
\\\notag=&3l-4n_1-5
\\\notag<&3l-4(n_1+1),
\end{align}
\begin{align}
\label{form-3-18-2}
\#\bs{h}+\#\bs{k}\leq&4+(10i-32+1)k
\\\notag=&10l-11n_1-15
\\\notag<&10l-11(n_1+1);
\end{align}
when $m\neq3$, one has
\begin{align}
\label{form-3-7-1}
\#\bs{k}\leq&3(m-4)+1+(3i-11+1)k
\\\notag=&3l-4n_1-7
\\\notag<&3l-4(n_1+1),
\end{align}
\begin{align}
\label{form-3-7-2}
\#\bs{h}+\#\bs{k}\leq&10m-32+(10i-32+1)k
\\\notag=&10l-11n_1-21
\\\notag<&10l-11(n_1+1).
\end{align}
$\bullet$ the part of $\Gamma^{(i)}_{l}\bigcup\Gamma'^{(i)}_{l}$ determined by the rational vector fields $\{Q_{\Gamma^{(i-1)}_{m},\Gamma'^{(i-1)}_{m}}\}_{m=i}^{l}$ and the rational vector field $\chi_{\Gamma_{i-2},\Gamma'_{i-2}}$.

For $i=4$, by \Cref{re42}, the relationship \eqref{form5104}, the fact $2k=l-m\leq l-4$ and the inequality $l>2n_1$, one has
\begin{align}
\label{form5117-3-18}
\#\bs{k}\leq0+2k\leq l-4<3l-4(n_1+1),
\end{align}
\begin{align}
\label{form5118-3-18}
\#\bs{h}+\#\bs{k}\leq&3m-5+9k\leq\frac{9}{2}l-11
<10l-11(n_1+1).
\end{align}

For $i\geq5$, one has $l=k(i-2)+m$, where $m=(i-3)n_1'+n_2'$ with $n'_1\in\mb{N}_*$ and $n_2'\in\{2,1,0,-1,\cdots,-(i-6)\}$. 
Hence, one has
$k(i-2)+(i-3)n_1'+n_2=(i-2)n_1+n_2$, i.e., 
\begin{equation}
\label{form3-19-1}
k+n_1'-n_1=\frac{n_2-n_2'+n_1'}{i-2}.
\end{equation}
By the fact $n_2\geq-(i-5)$, $n_2\leq2$ and $n_1'\geq1$, one has 
\begin{equation}
\label{form3-19-2}
n_2-n_2'+n_1'\geq 4-i>2-i.
\end{equation}
By \eqref{form3-19-1} and \eqref{form3-19-2}, we have $k+n_1'-n_1>-1$, which implies $k+n_1'-n_1\geq 0$. Thus by \Cref{re42} and the relationships \eqref{form596},  \eqref{form5104}, one has
%
\begin{align}
\label{form5117-3-7}
\#\bs{k}\leq&3m-4(n_1'+1)+(3i-11+1)k
\\\notag=&3l-4(k+n_1'+1)
\\\notag\leq&3l-4(n_1+1),
\end{align}
\begin{align}
\label{form5118-3-7}
\#\bs{h}+\#\bs{k}\leq&10m-11(n'_1+1)+(10i-32+1)k
\\\notag=&10l-11(k+n'_1+1)
\\\notag\leq&10l-11(n_1+1).
\end{align}
%

{\bf{Step 3: Estimation the coefficient of rational vector field $Q_{\Gamma^{(i)}_{l},\Gamma'^{(i)}_{l}}$ in \eqref{form5113}.}} 
Notice that for any $(a,\bs{j},\bs{h},\bs{k},n)\in\Gamma^{(i-1)}_{m}\bigcup\Gamma'^{(i-1)}_{m}\bigcup\tilde{\Gamma}_m\bigcup\tilde{\tilde{\Gamma}}_m\bigcup\Gamma^{(i)}_{m}\bigcup\Gamma'^{(i)}_{m}$,  one has 
\begin{align*}
\#\bs{h}+\#\bs{k}\leq10m-22, \quad \#\bs{k}\leq3m-8,
\end{align*}
and thus
\begin{align}
\label{form5103-1-25}
\#\bs{j}=m+(\#\bs{h}+\#\bs{k})+\#\bs{k}\leq14m-30
<14(m-2).
\end{align}
Then in view of \eqref{form5113-1-25-3} and \eqref{form5103-1-25}, by \Cref{le43} and the coefficient estimates \eqref{form593-1}, \eqref{form593-1-24}, \eqref{form5101},
one has
\begin{align}
\label{form5104-1-25}
\big|d_3(k,m)\big|
\leq&\big(C_2l^4\|\chi_{\Gamma_{i-2},\Gamma'_{i-2}}\|_{\ell^{\infty}}\big)^k\big(\|Z_{\Gamma^{(m-1)}_{m},\Gamma'^{(m-1)}_m}+\tilde{Z}_{\tilde{\Gamma}_{m}}\|_{\ell^{\infty}}\big)
\\\notag\leq&3C_3\Big(C_2l^4C_3(C_3r)^{4(i-2)(i-1)-4}\Big)^k(C_3r)^{4(m-2)(m-1)-4}
\\\notag=&3C_3\Big(\frac{C_2l^4}{C_3^3r^4}\Big)^k(C_3r)^{4(i-2)(i-1)k+4(m-2)(m-1)-4}.
\end{align}
By $(i-2)k+m=l$, $m\leq i$ and $l\geq i+1$, one has
\begin{align}
\label{form3-19-3}
4(i-2)(i-1)k+4(m-2)(m-1)\leq&4(l-m)(i-1)+4(m-2)(i-1)
\\\notag=&4(l-2)(i-1)
\\\notag<&2(l-2)(l+i-1).
\end{align}
For any $k\geq1$, one has
\begin{equation}
\label{form3-19-c}
3\Big(\frac{C_2l^4}{C_3^3r^4}\Big)^k<\frac{1}{7},
\end{equation}
and thus, by \eqref{form5104-1-25}--\eqref{form3-19-c}, one has
\begin{align}
\label{form5104-1-25-d3}
\big|d_3(k,m)\big|
\leq&\frac{C_3}{7}(C_3r)^{2(l-2)(l+i-1)-4}.
\end{align}

Similarly, in view of \eqref{form5113-1-25-1}, by the coefficient estimates \eqref{form5101}--\eqref{form5100} and the fact $(i-2)k+i=l$, one has
\begin{align}
\label{form5104-1-25-d1}
\big|d_1(k)\big|\leq&\big(C_2l^4\|\chi_{\Gamma_{i-2},\Gamma'_{i-2}}\|_{\ell^{\infty}}\big)^k\big(\|Z_{\Gamma^{(i-1)}_{i},\Gamma'^{(i-1)}_i}+\tilde{Z}_{\tilde{\Gamma}_{i}}-Q_{\Gamma^{(i-1)}_i,\Gamma'^{(i-1)}_i}\|_{\ell^{\infty}}\big)
\\\notag\leq&3C_3\Big(C_2l^4C_3(C_3r)^{4(i-2)(i-1)-4}\Big)^k(C_3r)^{4(i-2)(i-1)-4}
\\\notag=&3C_3\Big(\frac{C_2l^4}{C_3^3r^4}\Big)^k(C_3r)^{4(l-2)(i-1)-4}
\\\notag\leq&\frac{C_3}{7}(C_3r)^{2(l-2)(l+i-1)-4},
\end{align}
where the last inequality follows from the inequalities \eqref{form3-19-c} and $i-1\leq l-2$.

In view of \eqref{form5113-1-25-4} and \eqref{form5103-1-25}, by \Cref{le43} and the coefficient estimates \eqref{form593} and \eqref{form5101}, one has
\begin{align}
\label{form5104-1-25-2}
\big|d_4(k,m)\big|
\leq&\big(C_2l^4\|\chi_{\Gamma_{i-2},\Gamma'_{i-2}}\|_{\ell^{\infty}}\big)^k \|\tilde{Z}_{\tilde{\tilde{\Gamma}}_{m}}\|_{\ell^{\infty}}
\\\notag\leq&\frac{C_3}{2}\Big(C_2l^4C_3(C_3r)^{4(i-2)(i-1)-4}\Big)^k(C_3r)^{4(m-1)m-4}
\\\notag=&\frac{C_3}{2}\Big(\frac{C_2l^4}{C_3^3r^4}\Big)^k(C_3r)^{4(i-2)(i-1)k+4(m-1)m-4}.
\end{align}
By $(i-2)k+m=l$, $m\leq i-1$ and $l\geq i+1$, one has
\begin{align}
\label{form3-19-4}
4(i-2)(i-1)k+4(m-1)m\leq&4(i-1)(l-m)+4(m-1)(i-1)
\\\notag=&4(i-1)(l-1)
\\\notag\leq&2(i-1)(l-1)+2(l-2)i
\\\notag\leq&2(l-2)(l+i-1).
\end{align}
Thus, by \eqref{form5104-1-25-2}, \eqref{form3-19-4} and the inequality \eqref{form3-19-c}, one has
\begin{align}
\label{form5104-1-25-d4}
\big|d_4(k,m)\big|
\leq\frac{C_3}{7}(C_3r)^{2(l-2)(l+i-1)-4}.
\end{align}
In view of \eqref{form5113-1-25-2},  by the estimate \eqref{form5104-1-25-d4}, one has
\begin{equation}
\label{form5104-1-25-d2}
\big|d_2(k)\big|=\big|d_4(k,i-1)\big|\leq\frac{C_3}{7}(C_3r)^{2(l-2)(l+i-1)-4}.
\end{equation}

In view of \eqref{form5113-1-25-5}, by \Cref{le43} and the coefficient estimates \eqref{form595} and \eqref{form5101}, one has
\begin{align}
\label{form5104-1-25-3}
\big|d_5(k,m)\big|
\leq&\big(C_2l^4\|\chi_{\Gamma_{i-2},\Gamma'_{i-2}}\|_{\ell^{\infty}}\big)^k \|Q_{\Gamma^{(i-1)}_{m},\Gamma'^{(i-1)}_{m}}\|_{\ell^{\infty}}
\\\notag\leq&\Big(C_2l^4C_3(C_3r)^{4(i-2)(i-1)-4}\Big)^kC_3(C_3r)^{2(m-2)(m+i-2)-4}
\\\notag=&\Big(\frac{C_2l^4}{C_3^3r^4}\Big)^kC_3(C_3r)^{4(i-2)(i-1)k+2(m-2)(m+i-2)-4}.
\end{align}
By $(i-2)k+m=l$ and $m\geq i$, one has
\begin{align}
\label{form3-19-5}
4(i-2)(i-1)k+2(m-2)(m+i-2)=&4(i-2)(l-2)+2(m-2)(m-i)
\\\notag\leq&4(i-1)(l-2)+2(l-2)(l-i)
\\\notag<&2(l-2)(l+i-1).
\end{align}
%
Then by \eqref{form5104-1-25-3}, \eqref{form3-19-5} and  the inequality \eqref{form3-19-c}, one has
\begin{align}
\label{form5104-1-25-d5}
\big|d_5(k,m)\big|\leq\frac{C_3}{7}(C_3r)^{2(l-2)(l+i-1)-4}.
\end{align}

Hence, in view of \eqref{form5113}, by the estimates \eqref{form5104-1-25-d3}, \eqref{form5104-1-25-d1}, \eqref{form5104-1-25-d4}, \eqref{form5104-1-25-d2} and \eqref{form5104-1-25-d5}, one has
\begin{align*}
\|Q_{\Gamma^{(i)}_{l},\Gamma'^{(i)}_{l}}\|_{\ell^{\infty}}
\leq&\sum_{i+k(i-2)=l}\frac{|d_1(k)|}{(k+1)!}+\sum_{i-1+k(i-2)=l}\frac{|d_2(k)|}{(k+1)!}+\sum_{\substack{m+k(i-2)=l\\k\geq1,3\leq m\leq i-1}}\frac{|d_3(k,m)|}{k!}
\\&+\sum_{\substack{m+k(i-2)=l\\k\geq1,3\leq m\leq i-2}}\frac{|d_4(k,m)|}{k!}+\sum_{\substack{m+k(i-2)=l\\k\geq0, i\leq m\leq r}}\frac{|d_5(k,m)|}{k!}
\\\notag\leq&\frac{C_3}{7}(C_3r)^{2(l-2)(l+i-1)-4}\Big(\sum_{i+k(i-2)=l}1+\sum_{i-1+k(i-2)=l}1+\sum_{\substack{m+k(i-2)=l\\k\geq1,3\leq m\leq i-1}}1
\\\qquad&+\sum_{\substack{m+k(i-2)=l\\k\geq1,3\leq m\leq i-2}}1+\sum_{\substack{m+k(i-2)=l\\k\geq0, i\leq m\leq r}}\frac{1}{k!}\Big).
\end{align*}
Denote by $M$ the part in parentheses of the above inequality. Then
\begin{align*}
M\leq1+1+1+1+\sum_{k\geq0}\frac{1}{k!}
<7,
\end{align*}
and thus
\begin{equation}
\label{form5122}
\|Q_{\Gamma^{(i)}_{l},\Gamma'^{(i)}_{l}}\|_{\ell^{\infty}}
\leq C_3(C_3r)^{2(l-2)(l+i-1)-4}.
\end{equation}

{\bf{Step 4: Estimation the new remainder term $R^{(i)}_{\geq2r+3}$ in \eqref{form5114}.}} 
%
Firstly, we estimate the reminder term $R_1$ in \eqref{form5123}. Similarly with $Q_{\Gamma^{(i)}_{l},\Gamma'^{(i)}_{l}}$ in \eqref{form5113} satisfying the relationship \eqref{form596} and the estimate \eqref{form5122}, 
there exist $\{\Gamma_{m},\Gamma'_{m}\}_{m=r+1}^{r+i-2}$ with $\Gamma_{m},\Gamma'_{m}\in\mc{H}_{m,N}$ such that
\begin{align}
\label{form5124}
\sum_{m=r+1}^{r+i-2}Q_{\Gamma_{m},\Gamma'_{m}}=&d_1\Big([\frac{r-i}{i-2}]+1\Big)+d_2\Big([\frac{r-i+1}{i-2}]+1\Big)+\sum_{m=3}^{i-1}d_3\Big([\frac{r-m}{i-2}]+1,m\Big)
\\\notag&+\sum_{m=3}^{i-2}d_4\Big([\frac{r-m}{i-2}]+1,m\Big)+\sum_{m=i}^{r}d_5\Big([\frac{r-m}{i-2}]+1,m\Big)
\end{align}
with the coefficient estimate
\begin{equation}
\label{form5125}
\|Q_{\Gamma_{m},\Gamma'_{m}}\|_{\ell^{\infty}}\leq C_3(C_3r)^{2(m-2)(m+i-1)-4},
\end{equation}
%
where for any $(\bs{j},\bs{h},\bs{k},n)\in\Gamma_{m}\bigcup\Gamma'_{m}$, one has 
\begin{equation}
\label{form5126}
\begin{aligned}
&\#\bs{k}\leq3m-8,
\\&\#\bs{h}+\#\bs{k}\leq10m-22, 
\\&\#\bs{j}=\#\bs{h}+2\#\bs{k}+m\leq14m-30.
\end{aligned}
\end{equation}
In view of \eqref{re41-2} in \Cref{re41} and \eqref{form5126}, there exists a finite set $\mc{F}_m\subset\mb{N}^3$ such that 
$$\Gamma_m=\bigcup_{(\alpha,\beta,\beta')\in\mc{F}_m}\Gamma^{(\alpha,\beta,\beta')}_m\quad\text{and}\quad \Gamma'_m=\bigcup_{(\alpha,\beta,\beta')\in\mc{F}_m}\Gamma'^{(\alpha,\beta,\beta')}_m$$ 
with 
$$\beta'<3m-8,\quad \beta+\beta'<10m-22, \quad\alpha<14m-30.$$ 
Thus by \Cref{le42}, the fact $\sqrt{12}\alpha<50m$, \eqref{form5125} and $\|\Phi_{\chi_{\Gamma_{i-2},\Gamma'_{i-2}}}^{t}(\bs{z}))\|_{s}\leq2\|z\|_{s}$, one has
\begin{align*} 
&\|Q_{\Gamma_{m},\Gamma'_{m}}(\Phi_{\chi_{\Gamma_{i-2},\Gamma'_{i-2}}}^{t}(\bs{z}))\|_{s}
\\\notag\leq&6\sqrt{2}\times(14m)^3\max_{(\alpha,\beta,\beta')\in\mc{F}_{m}}\frac{(\sqrt{12}\alpha)^{2(\alpha-1)}N^{(4\alpha+2)(\beta+\beta')}}{(\gamma/2^{r+1-i})^{\beta+2\beta'}}\|Q_{\Gamma_{m},\Gamma'_{m}}\|_{\ell^{\infty}}\|\Phi_{\chi_{\Gamma_{i-2},\Gamma'_{i-2}}}^{t}(z)\|_{s}^{2m+1}
\\\notag\leq&6\sqrt{2}\times(14m)^3\frac{(50m)^{2(14m-31)}N^{2(28m-59)(10m-22)}}{(\gamma/2^{r+1-i})^{13m-30}}C_3(C_3r)^{2(m-2)(m+i-1)-4}\big(2\|z\|_{s}\big)^{2m+1}.
\end{align*}
Remark that for $4\leq i\leq r\leq m-1$, one has 
\begin{align*}
N^{2(28m-59)(10m-22)}r^{2(m-2)(m+i-1)-4}\leq&N^{2(28m-59)(10m-22)+4(m-2)(m-1)-4}
\\\leq&N^{564(m-2)^2};
\end{align*}
%
and by the fact $C_3\geq400$, one has
\begin{align*}
&6\sqrt{2}\times(14m)^3(50m)^{2(14m-31)}2^{(r+1-i)(13m-30)}C_3C_3^{2(m-2)(m+i-1)-4}2^{2m+1}
\\\leq&(50m)^{28m-59}2^{(m-4)(13m-30)+2m+1}C_3^{4(m-2)(m-1)-3}
\\\leq&C_3^{(m-2)(28m-59)+(m-4)(13m-30)+2m+1+4(m-2)(m-1)-3}
\\\leq&C_3^{564(m-2)^2}.
\end{align*}
Hence, one has
\begin{align}
\label{form5127}
\|Q_{\Gamma_{m},\Gamma'_{m}}(\Phi_{\chi_{\Gamma_{i-2},\Gamma'_{i-2}}}^{t}(\bs{z}))\|_{s}\leq\frac{(C_3N)^{564(m-2)^2}}{\gamma^{13m-30}}\|z\|_{s}^{2m+1}.
\end{align}
In view of \eqref{form5123} and \eqref{form5124}, by \eqref{form572}, \eqref{form5127} and $\|(D\Phi_{\chi_{\Gamma_{i-2},\Gamma'_{i-2}}}^{t})^{-1}\|_{\ms{P}_{s}\mapsto\ms{P}_{s}}\leq2$, for any $\varepsilon\leq \gamma^{\frac{13}{2}}(C_3N)^{-1128(r-1)}$, one has
\begin{align} 
\label{form5128}
\|R_1(\bs{z})\|_s\leq&\|(D\Phi_{\chi_{\Gamma_{i-2},\Gamma'_{i-2}}}^{t})^{-1}\|_{\ms{P}_{s}\mapsto\ms{P}_{s}}\sum_{m=r+1}^{r+i-2}\|Q_{\Gamma_{m},\Gamma'_{m}}(\Phi_{\chi_{\Gamma_{i-2},\Gamma'_{i-2}}}^{t}(\bs{z}))\|_{s}
\\\notag\leq&2\sum_{m=r+1}^{r+i-2}\frac{(C_3N)^{564(m-2)^2}}{\gamma^{13m-30}}\|z\|_{s}^{2m+1}
\\\notag\leq&\frac{4(C_3N)^{564(r-1)^2}}{\gamma^{13r-17}}\|z\|_{s}^{2r+3}.
\end{align}

In view of \eqref{form5123-1-26}, by \eqref{form597} and the fact $\|\Phi_{\chi_{\Gamma_{i-2},\Gamma'_{i-2}}}^{1}(\bs{z}))\|_{s}\leq2\|z\|_{s}$, one has
\begin{align}
\label{form5129}
\|R_2(\bs{z})\|_{s}\leq&\|(D\Phi_{\chi_{\Gamma_{i-2},\Gamma'_{i-2}}}^{t})^{-1}\|_{\ms{P}_{s}\mapsto\ms{P}_{s}}\|R^{(i-1)}_{\geq2r+3}(\Phi^1_{\chi_{\Gamma_{i-2},\Gamma'_{i-2}}}(\bs{z}))\|_{s}
\\\notag\leq&2\frac{4^{(i-4)(r+3)}(C_3N)^{564(r-1)^2}}{\gamma^{13r-17}}\big(2\|z\|_{s}\big)^{2r+3}
\\\notag=&\frac{4^{(i-4)(r+3)+r+2}(C_3N)^{564(r-1)^2}}{\gamma^{13r-17}}\|z\|_{s}^{2r+3}.
\end{align}
In view of \eqref{form5114}, by \eqref{form5128} and \eqref{form5129}, one has
\begin{align*}
\|R^{(i)}_{\geq2r+3}(\bs{z})\|_{s}\leq&\|R_1(\bs{z})\|_{s}+\|R_2(\bs{z})\|_{s}
\\\leq&\big(4+4^{(i-4)(r+3)+r+2}\big)\frac{(C_3N)^{564(r-1)^2}}{\gamma^{13r-17}}\|z\|_{s}^{2r+3}
\\\leq&\frac{4^{(i-3)(r+3)}(C_3N)^{564(r-1)^2}}{\gamma^{13r-17}}\|z\|_{s}^{2r+3}.
\end{align*}
This completes the proof of \Cref{le53}.
\end{proof}

\begin{proof}[Proof of Theorem {\sl\ref{th52}}]
Initially, in view of \eqref{form589}, by  \eqref{formth51-1-25-1}--\eqref{form58} in \Cref{th51}, the assumptions \eqref{form593-1}-- \eqref{form597} in \Cref{le53} are satisfied.
After $r$ iterations, one has
$$X^{(r)}=Z_1+Z_3^{\leq N}+Z_{5}^{\leq N}+\sum_{l=3}^{r}Z_{\Gamma_{l}^{(l-1)},\Gamma'^{(l-1)}_{l}}+\sum_{l=3}^{r}\tilde{Z}_{\tilde{\Gamma}_{l}}+\sum_{l=3}^{r-1}\tilde{Z}_{\tilde{\tilde{\Gamma}}_{l}}+R^{(r)}_{\geq2r+3}.$$
Let $\phi^{(3)}:=\Phi_{\chi_{\Gamma_{2},\Gamma'_{2}}}^{1}\circ\cdots\circ\Phi_{\chi_{\Gamma_{r-2},\Gamma'_{r-2}}}^{1}$, and then we have the transformed vector field  \eqref{form584} with  $Z_{\Gamma_{l},\Gamma'_l}:=Z_{\Gamma_{l}^{(l-1)},\Gamma'^{(l-1)}_{l}}$ and $R'''_{\geq2r+2}:=R^{(r)}_{\geq2r+3}$.
Moreover, the estimates \eqref{form566-1-26} and \eqref{form586} follow from \eqref{form593-1}, \eqref{form593-1-24} and \eqref{form593}; and the estimate \eqref{form587} follows from \eqref{form597} with $r$ in place of $i-1$.
By the estimate \eqref{form591}, for any $\varepsilon\leq \gamma^{\frac{13}{2}}(C_3N)^{-1128(r-1)}$, the transformation $$\phi^{(3)}:\Big(\mc{U}^{N}_{\frac{\gamma}{2}}\times\mc{U}^{N}_{\frac{\gamma}{2}}\Big)\bigcap B_{s}(\frac{3}{2}\varepsilon)\to\Big(\mc{U}^{N}_{\frac{\gamma}{2^{r-2}}}\times\mc{U}^{N}_{\frac{\gamma}{2^{r-2}}}\Big)\bigcap B_{s}(2\varepsilon)$$ 
satisfies the estimate
\begin{align}
\sup_{\bs{z}\in B_{s}(\frac{3}{2}\varepsilon)}\|\bs{z}-\phi^{(3)}(\bs{z})\|_{s}\leq&\sum_{i=4}^{r}\sup_{\bs{z}\in B_{s}(\varepsilon_i)}\|\bs{z}-\Phi^{1}_{\chi_{\Gamma_{i-2},\Gamma'_{i-2}}}(\bs{z})\|_{s}
\\\notag\leq&\sum_{i=4}^{r}\frac{2^{(13i-43)(r+1-i)}(C_3N)^{564(i-2)^2}}{\gamma^{13i-43}}\|z\|_s^{2i-3}
\\\notag\leq& \frac{2^{9(r-3)+1}(C_3N)^{2256}}{\gamma^{9}}\|z\|_s^{5},
\end{align}
which implies the estimate \eqref{form585}.
\end{proof}

\section{Proof of \Cref{th11}}
\label{sec6}
In this section, firstly, we estimate the measure. Then combining with the previous  normal form theorems, we prove the long time stability in \Cref{th11}.

\subsection{Measure estimate}
\label{sec61}
Introduce the Gaussian measure
\begin{equation}
\label{form61}
d\mu=\frac{e^{-\sum_{m\in\mb{N}_{*}}m^{2s}|z_m|^{2}}dzd\bar{z}}{\int_{\sum_{m\in\mb{N}_{*}}m^{2s-2}|z_m|^{2}\leq\frac{1}{2}}e^{-\sum_{m\in\mb{N}_{*}}m^{2s}|z_m|^{2}}dzd\bar{z}}.
\end{equation}
Let $z_{m}=I_{m}^{\frac{1}{2}}e^{{\rm i}\theta_{m}}$ with $I_{m}>0$, $\theta_{m}\in[0,2\pi)$, and thus the measure \eqref{form61} is equivalent to the measure
\begin{equation}
\label{form62}
d\mu=\frac{e^{-\sum_{m\in\mb{N}_{*}}m^{2s}I_{m}}dId\theta}{\int_{\sum_{m\in\mb{N}_{*}}m^{2s-2}I_{m}\leq\frac{1}{2}}e^{-\sum_{m\in\mb{N}_{*}}|m|^{2s}I_{m}}dId\theta}.
\end{equation}
Then we have the following measure estimate.
\begin{lemma}
\label{le61}
Fix $r\geq2$, $s\geq1$, $\gamma>0$ and $N\geq2(r+1)$. If
\begin{equation}
\label{form63}
\varepsilon^{2}\leq\frac{2\gamma}{(r+1)N^{4r+2}},
\end{equation}
then there exists a positive constant $\lambda$ depending on $r,s$ such that
\begin{equation}
\label{form64}
\mu(\varepsilon z\in\mc{U}_{\gamma}^{N})\geq1-\lambda\gamma.
\end{equation}
\end{lemma}
In order to prove the above measure estimate, we need the next two lemmas.
\begin{lemma}
\label{le62}
Fix an integer $r\geq2$. For $s\geq1$, $N>1$ and $\gamma>0$,  there exists a positive constant $\lambda_{1}$ such that
\begin{equation}
\label{form65}
\mu\big(|\Omega_{\bs{j}}^{(2)}(\varepsilon^{2}I)|>\gamma'\;\text{for any}\;\bs{j}\in Irr(\mc{R}_{\leq N})\;\text{with}\;\#\bs{j}\leq r\big)\geq1-\lambda_{1}\gamma,
\end{equation}
where $\gamma'=\frac{1}{2}\gamma\varepsilon^{2}N^{-4l-2}\kappa_{\bs{j}}^{-2s}$ with $l:=\#\bs{j}$.
\end{lemma}

\begin{proof}
Consider the complementary set of \eqref{form65}:
$$\Theta^{(1)}=\{\exists \;\bs{j}\in Irr(\mc{R}_{\leq N})\;\text{with}\;\#\bs{j}\leq r\;\text{such that}\;|\Omega_{\bs{j}}^{(2)}(\varepsilon^{2}I)|\leq\gamma'\}.$$
For any $\bs{j}\in Irr(\mc{R}_{\leq N})$, one has
\begin{align}
\label{form67}
&\mu\big(|\Omega_{\bs{j}}^{(2)}(\varepsilon^{2}I)|\leq\gamma'\big)
\\\notag=&\lim_{M\to\infty}\frac{\int_{\sum_{m\leq M}m^{2s-2}I_{m}\leq\frac{1}{2},|\Omega_{\bs{j}}^{(2)}(\varepsilon^{2}I)|<\gamma'}e^{-\sum_{m\leq M}m^{2s}I_{m}}\prod_{m\leq M}dI_{m}}{\int_{\sum_{m\leq M}m^{2s-2}I_{m}\leq\frac{1}{2}} e^{-\sum_{m\leq M}m^{2s}I_{m}}\prod_{m\leq M}dI_{m}}
\\\notag=&\lim_{M\to\infty}\frac{\int_{\sum_{m\leq M}m^{-2}x_m\leq\frac{1}{2},|\Omega_{\bs{j}}^{(2)}(\varepsilon^{2}I)|<\gamma'}e^{-\sum_{m\leq M}x_{m}}\prod_{m\leq M}dx_{m}}{\int_{\sum_{m\leq M}m^{-2}x_{m}\leq\frac{1}{2}} e^{-\sum_{m\leq M}x_{m}}\prod_{m\leq M}dx_{m}}
\end{align}
with $x_{m}=m^{2s}I_{m}$.  Let $A_1=\frac{1}{2}\big(\sum_{m\in\mb{N}_*}m^{-\frac{3}{2}}\big)^{-1}$ and then
\begin{equation}
\label{form68}
\{\sum_{m\leq M}m^{-2}x_{m}\leq\frac{1}{2}\}\supseteq\{\forall m\leq M, m^{-\frac{1}{2}}x_m\leq A_1\}.
\end{equation}
Hence, the denominator of \eqref{form67} has the lower bound
\begin{equation}
\label{form69}
\prod_{m\leq M}\int_{0}^{A_1\sqrt{m}}e^{-x_m}dx_{m}.
\end{equation}
For $\bs{j}=(\delta_k,a_k)_{k=1}^l$, let $\kappa_{\bs{j}}=\tilde{a}$, and then by $|\Omega_{\bs{j}}^{(2)}(\varepsilon^{2}I)|\leq\gamma'=\frac{1}{2}\gamma\varepsilon^{2}N^{-4l-2}\kappa_{\bs{j}}^{-2s}$, one has
\begin{equation}
\label{form610}
A_2-\gamma N^{-4l-2}\leq x_{\tilde{a}}\leq A_2+\gamma N^{-4l-2}.
\end{equation}
with $A_2=\tilde{a}^{2s}\big|\sum_{a_k\neq \tilde{a}}\delta_kI_{a_k}-I_{\Delta_{\bs{j}}}\big|$ or $\tilde{a}^{2s}\big|\sum_{k=1}^{l}\delta_kI_{a_k}\big|$.
By \eqref{form67}, \eqref{form69} and \eqref{form610}, one has
\begin{align*}
\mu(|\Omega_{\bs{j}}^{(2)}(\varepsilon^{2}I)|\leq\gamma')
&\leq\lim_{M\to\infty}\frac{\prod_{m\neq\tilde{a},m\leq M}(\int_{0}^{+\infty} e^{-x_{m}}dx_{m})\int_{A_2-\gamma N^{-4l-2}}^{A_2+\gamma N^{-4l-2}}dx_{\tilde{a}}}{\prod_{m\leq M}\int_{0}^{A_1\sqrt{m}}e^{-x_m}dx_{m}}
\\&\leq\lim_{M\to\infty}\frac{2\gamma N^{-4l-2}}{\prod_{m\leq M}(1-e^{-A_1\sqrt{m}})}
\\&=\frac{2\gamma}{N^{4l+2}\prod_{m\in\mb{N}_*}(1-e^{-A_1\sqrt{m}})}.
\end{align*}
Hence,  there exists a positive constant $\lambda_1$ such that
\begin{align*}
\mu(\Theta^{(1)})\leq&\sum_{l=2}^{r}\sum_{\substack{\bs{j}\in Irr(\mc{R}_{\leq N})\\\#\bs{j}=l}}\frac{2\gamma}{N^{4l+2}\prod_{m\in\mb{N}_*}(1-e^{-A_1\sqrt{m}})}
\\\leq&\frac{2\gamma}{\prod_{m\in\mb{N}_*}(1-e^{-A_1\sqrt{m}})}\sum_{l=2}^{r}N^{-3l-2}
\\\leq&\lambda_1\gamma,
\end{align*}
which implies the estimate \eqref{form65}. 
\end{proof}

For convenience, we consider $\tilde{\Omega}_{\bs{j}}^{(4)}(I):=\Omega_{\bs{j}}^{(2)}(I)+\tilde{\Omega}_{\bs{j}}^{(4,4)}(I)$ 
with
\begin{align}
\label{form611}
\tilde{\Omega}_{\bs{j}}^{(4,4)}=2\Big(\sum_{k=1}^l\delta_k\tilde{\omega}^{(4)}_{a_k}-\tilde{\omega}^{(4)}_{\Delta_{\bs{j}}}\Big),
\end{align}
where $\tilde{\omega}^{(4)}_0=0$ and for any $a\in\mb{N}_*$,
\begin{align}
\label{form612}
\tilde{\omega}^{(4)}_a=\frac{a}{16}\sum_{\substack{d\leq N\\d\neq a_1,\cdots,a_l,\Delta_{\bs{j}}}}\frac{I^2_d}{d^2-a^2}.
\end{align}

\begin{lemma}
\label{le63}
Fix an integer $r\geq2$. For $s\geq1$, $N\geq2(r+1)$ and $\gamma>0$, there exists a positive constant $\lambda_{2}$ depending on $r,s$ such that
\begin{align}
\label{form613}
\mu\big(|\tilde{\Omega}_{\bs{j}}^{(4)}(\varepsilon^{2}I)|>\gamma''\;\text{for any}\;\bs{j}\in Irr(\mc{R}_{\leq N})\;\text{with}\;\#\bs{j}\leq r\big)\geq1-\lambda_{2}\gamma,
\end{align}
where $\gamma''=\gamma\varepsilon^{2}N^{-4l-2}\max\{\kappa_{\bs{j}}^{-2s},\gamma\varepsilon^{2}\}$ with $l:=\#\bs{j}$.
\end{lemma}

\begin{proof}
Consider the complementary set of \eqref{form613}:
$$\Theta^{(2)}=\{\exists \;\bs{j}\in Irr(\mc{R}_{\leq N})\;\text{with}\;\#\bs{j}\leq r\;\text{such that}\;|\tilde{\Omega}_{\bs{j}}^{(4)}(\varepsilon^{2}I)|\leq\gamma''\}.$$
Let $x_{a}=a^{2s}I_{a}$, and then 
\begin{align}
\label{form615}
&\mu\big(|\tilde{\Omega}_{\bs{j}}^{(4)}(\varepsilon^{2}I)|\leq\gamma''\big)
\\\notag=&\lim_{M\to\infty}\frac{\int_{\sum_{a\leq M}m^{2s-2}I_{m}<\frac{1}{2},|\tilde{\Omega}_{\bs{j}}^{(4)}(\varepsilon^{2}I)|\leq\gamma''}e^{-\sum_{m\leq M}m^{2s}I_{m}}\prod_{m\leq M}dI_{m}}{\int_{\sum_{m\leq M}m^{2s-2}I_{m}<\frac{1}{2}} e^{-\sum_{m\leq M}m^{2s}I_{m}}\prod_{m\leq M}dI_{m}}
\\\notag=&\lim_{M\to\infty}\frac{\int_{\sum_{m\leq M}m^{-2}x_{m}\leq\frac{1}{2},|\tilde{\Omega}_{\bs{j}}^{(4)}(\varepsilon^{2}I)|<\gamma''}e^{-\sum_{m\leq M}x_{m}}\prod_{m\leq M}dx_{m}}{\int_{\sum_{m\leq M}m^{-2}x_{m}\leq\frac{1}{2}} e^{-\sum_{m\leq M}x_{m}}\prod_{m\leq M}dx_{m}}.
\end{align}
Now, we are going to estimate \eqref{form613} in two different ways.
\\\indent On one hand, in view of \eqref{form611},  $\tilde{\Omega}_{\bs{j}}^{(4)}(\varepsilon^{2}I)-\Omega_{\bs{j}}^{(2)}(\varepsilon^{2}I)$ is independent of $\{I_{a_1},\cdots,I_{a_l}\}$ and $I_{\Delta_{\bs{j}}}$ if it exists. Similarly with the proof of \eqref{form65} in \Cref{le62}, one has
\begin{equation}
\label{form616}
\mu\big(|\tilde{\Omega}_{\bs{j}}^{(4)}(\varepsilon^{2}I)|\leq\gamma''\big)\leq\frac{4\gamma''\kappa_{\bs{j}}^{2s}}{\varepsilon^{2}\prod_{m\in\mb{N}_*}(1-e^{-A_1\sqrt{m}})}
\end{equation}
with $A_1=\frac{1}{2}(\sum_{m\in\mb{N}_*}m^{-\frac{3}{2}})^{-1}$.

On the other hand, there exist $P_{\bs{j}}(d), Q_{\bs{j}}(d)\in\mb{Z}[X]$ such that
\begin{align*}
\tilde{\Omega}_{\bs{j}}^{(4,4)}(\varepsilon^{2}I)=\sum_{\substack{d\leq N\\d\neq a_1,\cdots,a_l,\Delta_{\bs{j}}}}\varepsilon^{4}I_d^2\frac{P_{\bs{j}}(d)}{8Q_{\bs{j}}(d)},
\end{align*}
where if $\Delta_{\bs{j}}=0$, then
\begin{equation*}
P_{\bs{j}}(d)=\sum_{i=1}^{l}\delta_ia_i\prod_{k\neq i}(d^2-a_k^2),\quad Q_{\bs{j}}(d)=\prod_{k=1}^{l}(d^2-a_k^2),
\end{equation*}
and if $\Delta_{\bs{j}}>0$, then denote $(\delta_{l+1},a_{l+1}):=(-1,\Delta_{\bs{j}})$ and 
\begin{equation*}
P_{\bs{j}}(d)=\sum_{i=1}^{l+1}\delta_ia_i\prod_{k\neq i}(d^2-a_k^2),\quad Q_{\bs{j}}(d)=\prod_{k=1}^{l+1}(d^2-a_k^2).
\end{equation*}
In fact, the order of $P_{\bs{j}}(d)$ and $Q_{\bs{j}}(d)$ are at most $l$ and $l+1$ with respect to $d^2$, respectively. Then there exists a positive integer $\tilde{a}\in[1,2l+2]\backslash\{a_1,\cdots,a_l,\Delta_{\bs{j}}\}$ such that $P_{\bs{j}}(\tilde{a})\neq0$. Hence, one has $|P_{\bs{j}}(\tilde{a})|\geq1$, and by the fact $N\geq2(r+1)\geq2(l+1)$, one has $|Q_{\bs{j}}(\tilde{a})|\leq N^{2l+2}$. 
Thus, rewrite
$\tilde{\Omega}_{\bs{j}}^{(4)}(\varepsilon^{2}I)=d_{\tilde{a}}\varepsilon^{4}I_{\tilde{a}}^{2}+A_{\bs{j}}((\varepsilon^{2}I_{m})_{m\neq\tilde{a}})$, where
\begin{equation}
\label{form617}
|d_{\tilde{a}}|=\Big|\frac{P_{\bs{j}}(\tilde{a})}{8Q_{\bs{j}}(\tilde{a})}\Big|\geq\frac{1}{8N^{2l+2}}.
\end{equation}
%
Notice that
$$\big||d_{\tilde{a}}|\varepsilon^{4}I_{\tilde{a}}^{2}-|A_{\bs{j}}((\varepsilon^{2}I_{m})_{m\neq\tilde{a}})|\big|\leq|\tilde{\Omega}_{\bs{j}}^{(4)}(\varepsilon^{2}I)|\leq\gamma''.$$
When $|A_{\bs{j}}((\varepsilon^{2}I_{m})_{m\neq\tilde{a}})|<3\gamma''$, one has
$$x_{\tilde{a}}\leq\frac{2\tilde{a}^{2s}}{\varepsilon^{2}}\sqrt{\frac{\gamma''}{|d_{\tilde{a}}|}}.$$
 By \eqref{form69} and \eqref{form615}, one has
\begin{align}
\label{form618}
\mu\big(|\tilde{\Omega}_{\bs{j}}^{(4)}(\varepsilon^{2}I)|\leq\gamma''\big)
&\leq\lim_{M\to\infty}\frac{\prod_{m\neq\tilde{a},m\leq M}(\int_{0}^{+\infty} e^{-x_{m}}dx_m)\int_{0}^{\frac{2\tilde{a}^{2s}}{\varepsilon^{2}}\sqrt{\frac{\gamma''}{|d_{\tilde{a}}|}}}dx_{\tilde{a}}}{\prod_{m\leq M}\int_{0}^{A_1\sqrt{m}}e^{-x_m}dx_{m}}
\\\notag&\leq\frac{2\tilde{a}^{2s}}{\varepsilon^{2}\prod_{m\in\mb{N}_*}(1-e^{-A_1\sqrt{m}})}\sqrt{\frac{\gamma''}{|d_{\tilde{a}}|}}.
\end{align}
When $|A_{\bs{j}}((\varepsilon^{2}I_{m})_{m\neq\tilde{a}})|\geq3\gamma''$, one has
$$\frac{\tilde{a}^{2s}}{\varepsilon^{2}}\sqrt{\frac{|A_{\bs{j}}((\varepsilon^{2}I_{m})_{m\neq\tilde{a}})|-\gamma''}{|d_{\tilde{a}}|}}\leq x_{\tilde{a}}\leq\frac{\tilde{a}^{2s}}{\varepsilon^{2}}\sqrt{\frac{|A_{\bs{j}}((\varepsilon^{2}I_{m})_{m\neq\tilde{a}})|+\gamma''}{|d_{\tilde{a}}|}}.$$
 By \eqref{form69} and \eqref{form615}, one has
\begin{align}
\label{form619}
&\mu\big(|\tilde{\Omega}_{\bs{j}}^{(4)}(\varepsilon^{2}I)|\leq\gamma''\big)
\\\notag\leq&\lim_{M\to\infty}\frac{\prod_{m\neq\tilde{a},m\leq M}(\int_{0}^{+\infty} e^{-x_{m}}dx_{m})\int_{\frac{\tilde{a}^{2s}}{\varepsilon^{2}}\sqrt{\frac{|A_{\bs{j}}((\varepsilon^{2}I_{m})_{m\neq\tilde{a}})|-\gamma''}{|d_{\tilde{a}}|}}}^{\frac{\tilde{a}^{2s}}{\varepsilon^{2}}\sqrt{\frac{|A_{\bs{j}}((\varepsilon^{2}I_{m})_{m\neq\tilde{a}})|+\gamma''}{|d_{\tilde{a}}|}}}dx_{\tilde{a}}}{\prod_{m\leq M}\int_{0}^{A_1\sqrt{m}}e^{-x_m}dx_{m}}
\\\notag\leq&\frac{\tilde{a}^{2s}\big(\sqrt{|A_{\bs{j}}((I_{m})_{m\neq\tilde{a}})|+\gamma''}-\sqrt{|A_{\bs{j}}((I_{m})_{m\neq\tilde{a}})|-\gamma''}\big)}{\varepsilon^{2}\prod_{m\in\mb{N}_*}(1-e^{-A_1\sqrt{m}})\sqrt{|d_{\tilde{a}}|}}
\\\notag\leq&\frac{\tilde{a}^{2s}}{\varepsilon^{2}\prod_{m\in\mb{N}_*}(1-e^{-A_1\sqrt{m}})\sqrt{|d_{\tilde{a}}|}}\frac{2\gamma''}{\sqrt{|A_{\bs{j}}((I_{m})_{m\neq\tilde{a}})|-\gamma''}}
\\\notag\leq&\frac{\tilde{a}^{2s}}{\varepsilon^{2}\prod_{m\in\mb{N}_*}(1-e^{-A_1\sqrt{m}})}\sqrt{\frac{2\gamma''}{|d_{\tilde{a}}|}}.
\end{align}
By \eqref{form617}--\eqref{form619}, one has
\begin{align}
\label{form620}
\mu\big(|\tilde{\Omega}_{\bs{j}}^{(4)}(\varepsilon^{2}I)|\leq\gamma''\big)\leq&\frac{2\tilde{a}^{2s}}{\varepsilon^{2}\prod_{m\in\mb{N}_*}(1-e^{-A_1\sqrt{m}})}\sqrt{\frac{\gamma''}{|d_{\tilde{a}}|}}
\\\notag\leq&\frac{4\sqrt{2}(2l+2)^{2s}N^{l+1}\sqrt{\gamma''}}{\varepsilon^{2}\prod_{m\in\mb{N}_*}(1-e^{-A_1\sqrt{m}})}.
\end{align}
\indent To sum up, by the estimates \eqref{form616} and \eqref{form620}, for any $\bs{j}\in Irr(\mc{R}_{\leq N})$ with $\#\bs{j}=l$, one has
\begin{align}
\label{form621}
\mu\big(|\tilde{\Omega}_{\bs{j}}^{(4)}(\varepsilon^{2}I)|\leq\gamma''\big)\leq&\frac{4\min\{\gamma''\kappa_{\bs{j}}^{2s},\sqrt{2}(2l+2)^{2s}N^{l+1}\sqrt{\gamma''}\}}{\varepsilon^{2}\prod_{m\in\mb{N}_*}(1-e^{-A_1\sqrt{m}})}
\\\notag\leq&\frac{4\sqrt{2}(2l+2)^{2s}\gamma}{N^{l}\prod_{m\in\mb{N}_*}(1-e^{-A_1\sqrt{m}})}.
\end{align}
Hence, there exists a positive constant $\lambda_2$ depending on $r,s$ such that
\begin{align}
\label{form622}
\mu(\Theta^{(2)})
&\leq\sum_{l=2}^{r}\sum_{\substack{\bs{j}\in Irr(\mc{R}_{\leq N})\\\#\bs{j}=l}}\frac{4\sqrt{2}(2l+2)^{2s}\gamma}{N^{l}\prod_{m\in\mb{N}_*}(1-e^{-A_1\sqrt{m}})}
\\\notag&\leq\frac{2^{2s+\frac{5}{2}}\gamma}{\prod_{m\in\mb{N}_*}(1-e^{-A_1\sqrt{m}})}\sum_{l=2}^{r}(l+1)^{2s}
\\\notag&\leq\lambda_2\gamma,
\end{align}
which implies the estimate \eqref{form613}.
\end{proof}

\begin{proof}[Proof of $\Cref{le61}$]
Consider the set $\Theta$ defined by
\begin{align*}
\big\{|\Omega_{\bs{j}}^{(2)}(\varepsilon^{2}I)|>\tilde{\gamma}'\;\text{and}\;|\tilde{\Omega}_{\bs{j}}^{(4)}(\varepsilon^{2}I)|>\tilde{\gamma}''\;\text{for any}\;\bs{j}\in Irr(\mc{R}_{\leq N})\;\text{with}\;\#\bs{j}\leq r\big\},
\end{align*}
where
\begin{align*}
&\tilde{\gamma}':=\gamma\varepsilon^{2}\|z\|_{s-1}^{2}N^{-4l-2}\kappa_{\bs{j}}^{-2s},\\&\tilde{\gamma}'':=2\gamma\varepsilon^{2}\|z\|_{s-1}^{2}N^{-4l-2}\max\{\kappa_{\bs{j}}^{-2s},2\gamma\varepsilon^{2}\|z\|_{s-1}^{2}\}
\end{align*}
 with $l:=\#\bs{j}$.
By $\|z\|_{s-1}^2<\frac{1}{2}$, one has $\tilde{\gamma}'<\gamma$ and $\tilde{\gamma}''<\gamma''$. 
Hence, by \Cref{le62} and \Cref{le63}, there exists $\lambda$ depending on $r,s$ such that
\begin{equation}
\label{form623}
\mu(\Theta)\geq1-\lambda\gamma.
\end{equation}

Next, we will show that $\Theta\subseteq\mc{U}_{\gamma}^{N}$, that is to say, we need prove that when $|\tilde{\Omega}_{\bs{j}}^{(4)}(\varepsilon^{2}I)|>\tilde{\gamma}''$, one has 
$$|\Omega_{\bs{j}}^{(4)}(\varepsilon^{2}I)|>\tilde{\gamma}:=\gamma\varepsilon^{2}\|z\|_{s-1}^{2}N^{-4l-2}\max\{\kappa_{\bs{j}}^{-2s},\gamma\varepsilon^{2}\|z\|_{s-1}^{2}\}.$$ 
In view of \eqref{formomega4} and \eqref{form611},  one has
$$|\Omega_{\bs{j}}^{(4)}(\varepsilon^{2}I)-\tilde{\Omega}_{\bs{j}}^{(4)}(\varepsilon^{2}I)|\leq(l+1)\varepsilon^{4}\|z\|_{s-1}^{4}\kappa_{\bs{j}}^{-2s}.$$
Then by the fact $\frac{r+1}{2}N^{4r+2}\varepsilon^{2}\leq\gamma$, one has
$$|\Omega_{\bs{j}}^{(4)}(\varepsilon^{2}I)-\tilde{\Omega}_{\bs{j}}^{(4)}(\varepsilon^{2}I)|\leq\gamma \varepsilon^{2}\|z\|_{s-1}^{2}N^{-4l-2}\kappa_{\bs{j}}^{-2s}.$$
Thus
\begin{align*}
|\Omega_{\bs{j}}^{(4)}(\varepsilon^{2}I)|\geq&|\tilde{\Omega}_{\bs{j}}^{(4)}(\varepsilon^{2}I)|-|\Omega_{\bs{j}}^{(4)}(\varepsilon^{2}I)-\tilde{\Omega}_{\bs{j}}^{(4)}(\varepsilon^{2}I)|
\\>&\gamma\varepsilon^{2}\|z\|_{s-1}^{2}N^{-4l-2}\max\{\kappa_{\bs{j}}^{-2s},\gamma\varepsilon^{2}\|z\|_{s-1}^{2}\}.
\end{align*}
Hence, we finish the proof of \Cref{le61}.
\end{proof}

\subsection{Long time stability}
\label{sec62}
In this subsection, we will prove \Cref{th11}. 

Firstly, we go back to the original system \eqref{form11}.
In view of \eqref{form21}, one has
\begin{equation}
\label{form720}
\|u\|^2_{s+\frac{1}{2}}+\|v\|^2_{s-\frac{1}{2}}=\|\psi\|_s^2+\|\bar{\psi}\|_s^2=2\|\psi\|_s^2.
\end{equation}
In view of \eqref{form22} and \eqref{form23-12}, by the fact $Q(\psi,\overline{\psi})=\frac{1}{4}\|\psi+\bar{\psi}\|_{\frac{1}{2}}^2=\frac{1}{2}\|u\|_1^2$, one has
\begin{equation}
\label{form721}
\left(\begin{array}{c} \eta_a \\ \bar{\eta}_a  \end{array}\right)=
\frac{1}{\sqrt{1-\rho^2\big(\frac{1}{2}\|u\|_1^2\big)}}
\left(\begin{array}{cc}1 & \rho\big(\frac{1}{2}\|u\|_1^2\big) \\\rho\big(\frac{1}{2}\|u\|_1^2\big) & 1 \end{array}\right)
\left(\begin{array}{c} \psi_a \\ \bar{\psi}_a  \end{array}\right),
\end{equation}
\begin{equation}
\frac{d\tau}{dt}=\sqrt{1+2\varphi\Big(\frac{1}{2}\|u\|_1^2\Big)},
\end{equation}
where $\rho(x)=\frac{x}{1+x+\sqrt{1+2x}}$, and $\varphi$ is the inverse of the real map $x\mapsto x\sqrt{1+2x},\; x\geq0$.
%
%
Then assuming $\|u\|_{1}$ is properly small, we have the prior estimates
\begin{equation}
\label{form722}
\frac{2}{\sqrt{6}}\|\psi\|_s\leq\|\eta\|_s\leq\frac{\sqrt{6}}{2}\|\psi\|_s.
\end{equation}
and
\begin{equation}
\label{form730}
\frac{1}{2}t\leq\tau\leq2t.
\end{equation}
By \eqref{form24}, \eqref{form720} and \eqref{form722}, one has
\begin{equation}
\label{form723}
\frac{1}{\sqrt{3}}\big(\|u\|^2_{s+\frac{1}{2}}+\|v\|^2_{s-\frac{1}{2}}\big)^{\frac{1}{2}}\leq\|z\|_{s-1}=\|\eta\|_{s}\leq\frac{\sqrt{3}}{2}\big(\|u\|^2_{s+\frac{1}{2}}+\|v\|^2_{s-\frac{1}{2}}\big)^{\frac{1}{2}}.
\end{equation}

For any $\big(u(0,x),v(0,x)\big)\in B_{s}(\varepsilon)$, by \eqref{form723}, one has $\|z(0)\|_{s-1}<\frac{\sqrt{3}}{2}\varepsilon$.
Consider $N\geq2(r+1)$ and $\gamma\in(0,1)$ satisfying the condition
\begin{equation}
\label{form71}
\Big(\frac{2}{N}\Big)^{\frac{2(s-1)}{r}}\leq \varepsilon\leq\frac{\gamma^{13}}{(C_3N)^{1128(r-1)}}.
\end{equation}
Then the condition \eqref{form63} is satisfied, and thus by \Cref{le61}, there exists a positive constant $\lambda$ such that
\begin{equation}
\label{form724}
\mu\Big(\frac{\sqrt{3}}{2}\varepsilon z(0)\in\mc{U}_{\gamma}^{N}\Big)\geq1-\lambda\gamma.
\end{equation}
Next, we will show that for any $|\tau|\leq2\varepsilon^{-r}$, one has
\begin{equation}
\label{form726}
\|z(\tau)\|_{s-1}\leq\frac{2}{\sqrt{3}}\varepsilon,
\end{equation}
\begin{equation}
\label{form727}
\sup_{a\in\mb{N}_*}a^{2s-2}\big||z_{a}(\tau)|^2-|z_{a}(0)|^2\big|\leq\frac{2}{\sqrt{6}}\varepsilon^{3}.
\end{equation}

By \Cref{th31}, \Cref{th51} and \Cref{th52}, for $\varepsilon\leq \gamma^{13}(C_3N)^{-1128(r-1)}$, there exists a nearly identity coordinate transformation $\phi=\phi^{(1)}\circ\phi^{2}\circ\phi^{(3)}$ such that  $\bs{z'}=\phi^{-1}(\bs{z})$ satisfies the following estimates
\begin{align}
\label{form74}
&\sup_{\|z\|_{s-1}\leq \varepsilon}\|\bs{z}-\phi^{-1}(\bs{z})\|_{s-1}\leq\frac{3^{12}\times4^{r}N^{28}}{\gamma^{2}}\|z\|_{s-1}^{3},
\\\label{form75}
&\sup_{\|z'\|_{s-1}\leq \varepsilon}\|\bs{z}'-\phi(\bs{z}')\|_{s-1}\leq\frac{3^{12}\times4^{r}N^{28}}{\gamma^{2}}\|z'\|_{s-1}^{3},
\end{align}
 and then
\begin{equation}
\label{form76}
\|z'(0)\|_{s-1}\leq\|z(0)\|_{s-1}+\|z(0)-\phi^{-1}(z(0))\|_{s-1}\leq\sqrt{\frac{5}{6}}\varepsilon.
\end{equation}
Denote the escape time of $z'$ by $T:=\inf\{\tau\mid \|z'(\tau)\|_{s-1}=\varepsilon\}$. For any $|\tau|\leq T$, by \Cref{re41-12} and \Cref{th52}, one has
\begin{equation}
\label{form77}
\pa_{\tau}(\|z'(\tau)\|^2_{s-1})=\sum_{a\in\mb{N}_*}|a|^{2s-2}(\bar{z}'_a\pa_{\tau}z'_a+z'_a\pa_{\tau}\bar{z}'_a)\leq 2(F_1+F_2+F_3+F_4)
\end{equation}
with
\begin{align}
\label{form78}
F_1=&\sup_{\|z'\|_{s-1}\leq\varepsilon}\|R'''_{\geq2r+3}(\bs{z}')\|_{s-1}\|z'\|_{s-1},
\\\label{form79-1}
F_2=&\sup_{\|z'\|_{s-1}\leq\varepsilon}\Big\|(D\phi^{(3)})^{-1}(D\phi^{(2)})^{-1}Z_3^{>N}\circ\phi^{(2)}\circ\phi^{(3)}(\bs{z}')\Big\|_{s-1}\|z'\|_{s-1},
\\\label{form79}
F_3=&\sup_{\|z'\|_{s-1}\leq\varepsilon}\Big\|(D\phi^{(3)})^{-1}(D\phi^{(2)})^{-1}\Big(\sum_{l=2}^{r}K^{>N}_{2l+1}\Big)\circ\phi^{(2)}\circ\phi^{(3)}(\bs{z}')\Big\|_{s-1}\|z'\|_{s-1},
\\\label{form710}
F_4=&\sup_{\|z'\|_{s-1}\leq\varepsilon}\Big\|(D\phi^{(3)})^{-1}(D\phi^{(2)})^{-1}R'_{\geq2r+3}\circ\phi^{(2)}\circ\phi^{(3)}(\bs{z}')\Big\|_{s-1}\|z'\|_{s-1},
\end{align}
where $K^{>N}_{2l+1}$ is found in \eqref{form52}.

By \eqref{form587}, for any $|\tau|\leq T$ and $\varepsilon\leq\gamma^{13}(C_3N)^{-1128(r-1)}$, one has
\begin{equation}
\label{form711}
F_1\leq
\frac{4^{(r-3)(r+3)}(C_3N)^{564(r-1)^2}}{\gamma^{13r-17}}\|z'\|_{s-1}^{2r+4}\leq\varepsilon^{r+4}.
\end{equation}
In view of \eqref{form53} and \eqref{form585}, by \Cref{le30} and the coefficient estimate $\|Z_{3}\|_{\ell^{\infty}}\leq\frac{1}{4}$, for $\varepsilon^{r}\geq(\frac{2}{N})^{2(s-1)}$, one has
\begin{align}
\label{form712-12}
F_2&\leq\|(D\phi^{(3)})^{-1}(D\phi^{(2)})^{-1}\|_{\ms{P}_{s-1}\mapsto\ms{P}_{s-1}}\big\|Z^{>N}_{3}\circ\phi^{(2)}\circ\phi^{(3)}(z')\big\|_{s-1}\|z'\|_{s-1}
\\\notag&\leq2\frac{9}{N^{2(s-1)}}\|Z_{3}\|_{\ell^{\infty}}\big\|\phi^{(2)}\circ\phi^{(3)}(\bs{z}')\big\|_{s-1}^{3}\|z'\|_{s-1}
\\\notag&\leq\frac{36}{N^{2(s-1)}}\varepsilon^{4}
\\\notag&<36\varepsilon^{r+4};
\end{align}
by the coefficient estimate \eqref{form37}, one has
\begin{align}
\label{form712}
F_3\leq&\|(D\phi^{(3)})^{-1}(D\phi^{(2)})^{-1}\|_{\ms{P}_{s-1}\mapsto\ms{P}_{s-1}}\sum_{l=2}^{r}\big\|K^{>N}_{2l+1}\circ\phi^{(2)}\circ\phi^{(3)}(\bs{z}')\big\|_{s-1}\|z'\|_{s-1}
\\\notag\leq&2\sum_{l=2}^{r}\frac{3^{l+1}l^{2(s-1)}}{N^{2(s-1)}}\|K_{2l+1}\|_{\ell^{\infty}}\big\|\phi^{(2)}\circ\phi^{(3)}(\bs{z}')\big\|_{s-1}^{2l+1}\|z'\|_{s-1}
\\\notag\leq&\sum_{l=2}^{r}\frac{2^{2l+2}\times3^{l+1}l^{2(s-1)}}{N^{2(s-1)}}C_1^{l^2}\|z'\|_{s-1}^{2l+2}
\\\notag\leq&\frac{2^{5+2s}\times3^3C_1^4}{N^{2(s-1)}}\varepsilon^{6}
\\\notag\leq&3456C_1^4\varepsilon^{r+6};
\end{align}
and by the estimate \eqref{form38} and $\varepsilon\leq\gamma^{13}(C_3N)^{-1128(r-1)}$, one has
\begin{align}
\label{form713}
F_4&\leq\|(D\phi^{(3)})^{-1}(D\phi^{(2)})^{-1}\|_{\ms{P}_{s-1}\mapsto\ms{P}_{s-1}}\|R'_{\geq2r+3}\circ\phi^{(2)}\circ\phi^{(3)}(\bs{z}')\|_{s-1}\|z'\|_{s-1}
\\\notag&\leq2(3C_1)^{r(r+1)}\big\|\phi^{(2)}\circ\phi^{(3)}(\bs{z}')\big\|_{s-1}^{2r+3}\|z'\|_{s-1}
\\\notag&\leq\varepsilon^{r+4}.
\end{align}
Hence, by \eqref{form77}--\eqref{form713}, for any $|\tau|\leq T$, one has
\begin{align}
\label{form714}
\Big|\|z'(\tau)\|^2_{s-1}-\|z'(0)\|^2_{s-1}\Big|\leq&T\sup_{\tau\leq T}\pa_{\tau}(\|z'(\tau)\|^2_{s-1})
\\\notag\leq&2(1+36+3456C_1^4\varepsilon^{2}+1)T\varepsilon^{r+4}.
\end{align}
If $T\leq2\varepsilon^{-r}$, then by \eqref{form76} and \eqref{form714}, one has
\begin{align*}
\varepsilon^{2}=\|z'(T)\|_{s-1}^{2}\leq \|z'(0)\|_{s-1}^2+\Big|\|z'(T)\|^2_{s-1}-\|z'(0)\|^2_{s-1}\Big|
<\frac{5}{6}\varepsilon^{2}+\frac{1}{6}\varepsilon^{2},
\end{align*}
which is impossible. Hence, one has $T\geq2\varepsilon^{-r}$.
By \eqref{form75}, for any $\tau\leq2\varepsilon^{-r}$, one has
$$\|z(\tau)\|_{s-1}\leq\|z'(\tau)\|_{s-1}+\|z'(\tau)-\phi(z'(\tau))\|_{s-1}\leq\frac{2}{\sqrt{3}}\varepsilon.$$

 In addition, for any $a\in\mb{N}_{*}$, one has
\begin{align}
\label{form715}
&\big||z_{a}(\tau)|^2-|z_{a}(0)|^2\big|
\\\notag\leq&\big||z_{a}(\tau)|^2-|z'_{a}(\tau)|^2\big|+\big||z'_{a}(\tau)|^2-|z'_{a}(0)|^2\big|+\big||z'_{a}(0)|^2-|z_{a}(0)|^2\big|.
\end{align}
Similarly to \eqref{form77}--\eqref{form714} but $a^{2s-2}|z'_{a}(\tau)|^2$ instead of $\|z'\|_{s-1}$, for any $|\tau|\leq2\varepsilon^{-r}$, one has
\begin{equation}
\label{form716}
a^{2s-2}\big||z'_{a}(\tau)|^2-|z'_{a}(0)|^2\big|\leq4(1+36+3456C_1^4\varepsilon^2+1)\varepsilon^{4}.
\end{equation}
By \eqref{form74}, one has
\begin{align}
\label{form717}
|a|^{2(s-1)}\big||z_{a}(\tau)|^2-|z'_{a}(\tau)|^2\big|\leq&\big(\|z(\tau)\|_{s-1}+\|z'(\tau)\|_{s-1}\big)\|z(\tau)-z'(\tau)\|_{s-1}
\\\notag\leq&\big(\frac{2}{\sqrt{3}}+1\big)\frac{3^{12}\times4^{r}N^{28}}{\gamma^{2}}\varepsilon^{4}
\\\notag<&\frac{3^{13}\times4^{r}N^{28}}{\gamma^{2}}\varepsilon^{4}.
\end{align}
Then the estimate \eqref{form727} follows from \eqref{form715}--\eqref{form717} for  $\varepsilon\leq\gamma^{13}(C_3N)^{-1128(r-1)}$.

Go back to the original variables. For any $s\geq s_0=O(r^2)$, we take 
\begin{equation}
\label{form718}
\lambda\gamma=\varepsilon^{\frac{1}{14}} 
\end{equation} 
and in view of the condition \eqref{form71}, let the open set
\begin{equation}
\label{form719}
\mc{V}_{r,s}:=\bigcup_{0<\varepsilon\leq\varepsilon_{0}}\bigg(\Big(\bigcup_{C_3\big(\lambda^{13}\varepsilon^{\frac{1}{14}}\big)^{\frac{1}{1128(r-1)}}<\frac{1}{N}<2\varepsilon^{\frac{r}{2s}}}\mc{U}_{\gamma}^{N}\Big)\bigcap\Big(B_{s}(\varepsilon)\backslash \overline{B_{s}(\frac{\varepsilon}{2})}\Big)\bigg),
\end{equation}
where $\overline{B_{s}(\frac{\varepsilon}{2})}$ is the closure of $B_{s}(\frac{\varepsilon}{2})$.
Then by \eqref{form723} and \eqref{form724}, one has
\begin{equation}
\label{form725}
\mu\Big(\varepsilon(u(0), v(0))\in\mc{V}_{r,s}\Big)\geq1-\varepsilon^{\frac{1}{14}}.
\end{equation}
Hence, for $\big(u(0,x),v(0,x)\big)\in \mc{V}_{r,s}\bigcap B_{s}(\varepsilon)$, by \eqref{form730} \eqref{form723} and \eqref{form726}, for any 
$|t|\leq\varepsilon^{-r}$, one has
\begin{equation}
\label{form728}
\|u(t)\|^2_{s+\frac{1}{2}}+\|v(t)\|^2_{s-\frac{1}{2}}\leq3\|z(t)\|^2_{s-1}\leq4\varepsilon^2,
\end{equation}
which indeed implies the prior estimates \eqref{form722} and \eqref{form730}.
Notice that in view of \eqref{form721}, one has
\begin{equation}
\label{form731}
\frac{2}{\sqrt{6}}\big||\eta_a(t)|^2-|\eta_a(0)|^2\big|\leq\big||\psi_a(t)|^2-|\psi_a(0)|^2\big|\leq\frac{\sqrt{6}}{2}\big||\eta_a(t)|^2-|\eta_a(0)|^2\big|,
\end{equation}
where $\{\psi_a\}_{a\in\mb{N}_*}$ and $\{\eta_a\}_{a\in\mb{N}_*}$ are the sequences of Fourier coefficients of $\psi$ and $\eta$, respectively. Hence, in view of $I_a:=\frac{a|u_a|^{2}+a^{-1}|v_a|^2}{2}=|\psi_a|^2$ and $\eta_a=\frac{z_a}{a}$, by \eqref{form727} and \eqref{form731}, one has
\begin{equation}
\label{form729}
\sup_{a\in\mb{N}_*}a^{2s}|I_{a}(t)-I_{a}(0)|\leq\varepsilon^{3}.
\end{equation}
%

\section{The case of Gevrey and  analytic spaces}
\label{sec7}
\subsection{Notations and estimation of vector fields}
\label{sec71}
For $\rho\geq0$ and $0<\theta\leq1$, introduce the Hilbert space 
$$\ell_{\rho,\theta}^2:=\{z=\{z_a\}_{a\in\mb{N}_*}\in\mb{C}^{\mb{N}_*}\mid \|z\|_{\rho,\theta}^2:=\sum_{a\in \mb{N}_*}e^{2\rho a^{\theta}}|z_a|^2<+\infty\}.$$
Taking $\rho=0$,  it is the usual $\ell^2$ space.
Denote the phase space $\ms{P}_{\rho,\theta}:=\ell_{\rho,\theta}^2\times \ell_{\rho,\theta}^2$
and the open ball
$$B_{\rho,\theta}(\varepsilon):=\{\bs{z}:= (z,\bar{z})\in\ms{P}_{\rho,\theta}\mid \|\bs{z}\|^2_{\rho,\theta}:=\|z\|^2_{\rho,\theta}+\|\bar{z}\|^2_{\rho,\theta}<\varepsilon^2\}.$$
Then the convenient notation $\zeta=(\zeta_{j})_{j=(\delta,a)\in\mb{U}_{3}\times\mb{N}_*}$ in \eqref{form3-23-notation} has 
\begin{equation}
\label{form81}
\|\zeta\|_{\rho,\theta}:=\sum_{j\in\mb{U}_{3}\times\mb{N}_*}e^{2\rho|j|^{\theta}}|\zeta_{j}|=3\|z\|_{\rho,\theta}^2.
\end{equation}

Now, we estimate polynomial vector fields in the  space $\ell_{\rho,\theta}^2$.
\begin{lemma}
\label{le72}
For $\rho\geq0$, $0<\theta\leq1$ and $X(\bs{z})\in\ms{M}_{2l+1}$, we have
\begin{equation}
\label{form89}
\|X(\bs{z})\|_{\rho,\theta}<3^{l+1}\|X\|_{\ell^{\infty}}\|z\|_{\rho,\theta}\|z\|_{0,\theta}^{2l}.
\end{equation}
\end{lemma}
\begin{proof}
It is same as \Cref{le21-1} with the norm $\|\;\|_{\rho,\theta}$ in place of $\|\;\|_{s}$.
\end{proof}

Especially, we estimate the truncated remainder terms of  resonant polynomial vector fields in the following lemma.
\begin{lemma}
\label{le73}
Consider the vector field $K^{>N}(\bs{z})\in\ms{R}_{2l+1}$ with the $z_a$-component  
\begin{equation}
\label{form810}
K^{(z_a)}_{>N}(\bs{z})=
\left\{\begin{aligned}
z_a\sum_{\substack{\bs{j}\in\mc{R}_{l}^0\\j_{1}^{*}>N}}\tilde{K}^{(z_a,z_a)}_{\bs{j}}\zeta_{\bs{j}}+\bar{z}_a\sum_{\substack{\bs{j}\in\mc{R}_{l}^a\\j_{1}^{*}>N}}\tilde{K}^{(z_a,\bar{z}_a)}_{\bs{j}}\zeta_{\bs{j}},\;  &\text{for}\;a\leq N,
\\z_a\sum_{\substack{\bs{j}\in\mc{R}_{l}^0\\j_{1}^{*}>N}}\tilde{K}^{(z_a,z_a)}_{\bs{j}}\zeta_{\bs{j}}+\bar{z}_a\sum_{\bs{j}\in\mc{R}_{l}^a}\tilde{K}^{(z_a,\bar{z}_a)}_{\bs{j}}\zeta_{\bs{j}}, \quad  &\text{for}\; a>N.
\end{aligned}\right.
\end{equation}
For $\rho\geq0$ and $0<\theta\leq1$, one has
\begin{equation}
\label{form811}
\|K^{>N}(\bs{z})\|_{\rho,\theta}<\frac{3^{l+1}}{e^{2\rho(\frac{N}{l})^{\theta}}}\|K\|_{\ell^{\infty}}\|z\|_{\rho,\theta}^3\|z\|_{0,\theta}^{2l-2}.
\end{equation}
\end{lemma}
\begin{proof}
For $\bs{j}=(\delta_k,a_k)_{k=1}^l\in\mc{R}^a_{l}$ with $a>N$, one has $\sum_{k=1}^{l}\delta_ka_k=a>N$ and thus $j_{1}^{*}>\frac{N}{l}$. 
In view of \eqref{form810}, for any $\bs{j}\in\mc{R}^0_{l}$ or $\mc{R}^a_{l}$ in $K^{>N}(\bs{z})$, we have $j_{1}^{*}>\frac{N}{l}.$
Then by \eqref{form81} and the inequality 
$$\sum_{j>\frac{N}{l}}|\zeta_{j}|\leq  e^{-2\rho(\frac{N}{l})^{\theta}}\|\zeta\|_{\rho,\theta}=3e^{-2\rho(\frac{N}{l})^{\theta}}\|z\|_{\rho,\theta}^2,$$
 one has
\begin{align*}
\|K^{>N}(\bs{z})\|_{\rho,\theta}=&\sqrt{2}\Big(\sum_{a\in \mb{N}_*}e^{2\rho a^{\theta}}\big|K_{>N}^{(z_a)}\big|^2\Big)^{\frac{1}{2}}
\\\leq&\sqrt{2}\|K\|_{\ell^{\infty}}\Big(\sum_{a\in \mb{N}_*}e^{2\rho a^{\theta}}\big(|\bar{z}_a|+|z_{a}|\big)^2\Big)^{\frac{1}{2}}\sum_{\substack{\bs{j}\in(\mb{U}_{3}\times\mb{N}_*)^{l}\\j_{1}^{*}>\frac{N}{l}}}|\zeta_{\bs{j}}|
\\\leq&\sqrt{2}\|K\|_{\ell^{\infty}}\big(2\|z\|_{\rho,\theta}\big)\Big(3^{l}e^{-2\rho(\frac{N}{l})^{\theta}}\|z\|_{\rho,\theta}^2\|z\|_{0,\theta}^{2l-2}\Big)
\\<&\frac{3^{l+1}}{e^{2\rho(\frac{N}{l})^{\theta}}}\|z\|_{\rho,\theta}^3\|z\|_{0,\theta}^{2l-2}.
\end{align*}
which is the estimate \eqref{form811}.
\end{proof}

Next in view of the small divisors \eqref{formomega2} and \eqref{formomega4}, we give the new small divisor conditions with the weight $e^{2\rho j^{\theta}}$ in place of $j^{2s}$. Concretely, for $r\geq2$, $N\geq1$ and $\gamma>0$, we say that $z\in \ell_{\rho,\theta}^2$ belongs to the open set $\mc{U}_{\gamma}^{N}$, if for any $\bs{j}\in Irr(\mc{R}_{\leq N})$ with $\#\bs{j}=l\leq r$, one has
\begin{align}
\label{form82}
&|\Omega_{\bs{j}}^{(2)}(I)|>\gamma\|z\|_{\rho,\theta}^{2}N^{-4l-2}e^{-2\rho\kappa_{\bs{j}}^{\theta}},
\\\label{form83}
&|\Omega_{\bs{j}}^{(4)}(I)|>\gamma\|z\|_{\rho,\theta}^{2}N^{-4l-2}\max\{e^{-2\rho\kappa_{\bs{j}}^{\theta}},\gamma\|z\|_{\rho,\theta}^{2}\}.
\end{align}

Remark that the set $\mc{U}_{\gamma}^{N}$ is stable with respect to the action $I$. Concretely, For $z\in\mc{U}_{\gamma}^{N}$, if $z'\in\ell^2_{\rho,\theta}$ satisfies 
\begin{equation}
\label{form85}
\|z'\|_{\rho,\theta}\leq4\|z\|_{\rho,\theta}\quad\text{and}\quad
\sup_{a\leq N}e^{2\rho a^{\theta}}|I'_{a}-I_{a}|\leq\frac{\gamma^{2}\|z\|_{\rho,\theta}^{2}}{288(r+1)N^{4r+3}},
\end{equation}
then $z'\in\mc{U}_{\gamma/2}^{N}$.

Notice that in order to eliminate the resonant but non-integrable vector fields, we have set up a rational framework for the Sobolev space, where the control condition \eqref{form416} is 
$$\prod_{m=1}^{\#\bs{h}}\kappa_{\bs{h}_{m}}^{2s}\leq\prod_{m=1}^{\#\bs{j}}(j_{m}^*)^{2s}.$$
Similarly, we define the resonant rational vector fields in the space $\ell_{\rho,\theta}^2$.
In view of the small divisor conditions \eqref{form82} and \eqref{form83},  we use the weight $e^{2\rho j^{\theta}}$ of the space $\ell_{\rho,\theta}^2$ in place of $j^{2s}$ of the Sobolev space. Concretely,  instead of the condition \eqref{form416}, we use the following  condition
$$\prod_{m=1}^{\#\bs{h}}e^{2\rho\kappa_{\bs{h}_{m}}^{\theta}}\leq \prod_{m=1}^{\#\bs{j}}e^{2\rho(j_{m}^*)^{\theta}},$$
which is the control condition
\begin{equation}
\label{form84}
\sum_{m=1}^{\#\bs{h}}\kappa^{\theta}_{\bs{h}_{m}}\leq\sum_{m=1}^{\#\bs{j}}(j_{m}^*)^{\theta}.
\end{equation}
In the same way, the control condition \eqref{form84} is well kept in the commutator of rational vector fields.

In the following lemma, we estimate resonant rational vector fields in the space $\ell_{\rho,\theta}^2$.
\begin{lemma}
\label{le74}
Fix $\rho>0$, $0<\theta\leq1$,  $r\geq2$, $N\geq1$ and $\gamma\in(0,1)$. 
Being given $\Gamma_{l}, \Gamma'_{l}\in\mc{H}_{l,N}$ satisfying $\mc{J}_{\Gamma_l,\Gamma'_l}<+\infty$,
consider the rational vector field $Q_{\Gamma_{l},\Gamma'_{l}}\in\ms{H}_{l,N}$.
 For $z\in\mc{U}_{\gamma}^{N}$, one has
\begin{equation}
\label{form812}
\|Q_{\Gamma_l,\Gamma'_l}(\bs{z})\|_{\rho,\theta}\leq 6\sqrt{2}\mc{J}_{\Gamma_l,\Gamma'_l}^3\max_{(\alpha,\beta,\beta')\in\mc{F}_{l}}\frac{(12\alpha^2)^{\alpha-1}N^{(4\alpha+2)(\beta+\beta')}}{\gamma^{\beta+2\beta'}}\|Q_{\Gamma_{l},\Gamma'_{l}}\|_{\ell^{\infty}}\|z\|_{\rho,\theta}^{2l+1}.
\end{equation}
\end{lemma}
\begin{proof}
It is same as \Cref{le42} with the norm $\|\;\|_{\rho,\theta}$ in place of $\|\;\|_{s}$.
\end{proof}

\subsection{Normal formal theorem}
\label{sec72}
In view of \eqref{form21}, \eqref{form22} and \eqref{form24}, we have the change of variables
$$\mc{G}_{\rho,\theta,\frac{3}{2}}\times \mc{G}_{\rho,\theta,\frac{1}{2}}\to \ell_{\rho,\theta}^2\times\ell_{\rho,\theta}^2, \; (u,\pa_tu)\mapsto(z,\bar{z}):=\bs{z}.$$
Inserting this change into the original system \eqref{form11}, we obtain the system $\pa_{\tau}\bs{z}=X(\bs{z})$ with 
\begin{equation}
\label{form816}
X:=Z_1+\sum_{l=1}^{r}P_{2l+1}+R_{\geq2r+3},
\end{equation}
where $P_{2l+1}(\bs{z})\in\ms{M}^{{\rm rev}}_{2l+1}$ and the remainder term $R_{\geq2r+3}$ satisfy
$$\|P_{2l+1}\|_{\ell^{\infty}}\leq C_0^{l},$$
$$\|R_{\geq2r+3}(\bs{z})\|_{\rho,\theta}\leq C_0^r\|z\|_{\rho,\theta}^{2r+3}.$$ 
Remark that this is the system \eqref{form28} with the norm $\|\;\|_{\rho,\theta}$ in place of $\|\;\|_{s}$.

In the following, we will normalize the vector field \eqref{form816}, where the normal form process in the space $\ell_{\rho,\theta}^2$ is parallel to that in the Sobolev space. 
\begin{theorem}
\label{th71}
Fix $\rho>0$ and $0<\theta\leq1$. For $r\geq4$, $N\geq r$,  $0<\gamma<1$ and $0<\varepsilon\ll 1$ satisfying $\varepsilon\leq\frac{\gamma^{13}}{(C_3N)^{1128(r-1)}}$, there exists a nearly identity coordinate transformation 
$$\phi: \Big(\mc{U}^{N}_{\frac{\gamma}{2}}\times\mc{U}^{N}_{\frac{\gamma}{2}}\Big)\bigcap B_{\rho,\theta}(\frac{3}{2}\varepsilon)\to\Big(\mc{U}^{N}_{\frac{\gamma}{2^{r}}}\times\mc{U}^{N}_{\frac{\gamma}{2^{r}}}\Big)\bigcap B_{\rho,\theta}(4\varepsilon)$$ satisfying the estimates
\begin{align}
\label{form818}
&\sup_{\|z\|_{\rho,\theta}\leq \varepsilon}\|\bs{z}-\phi(\bs{z})\|_{\rho,\theta}\leq\frac{3^{12}\times4^{r}N^{28}}{\gamma^{2}}\|z\|_{\rho,\theta}^{3},
\\\label{form819}
&\sup_{\|z\|_{\rho,\theta}\leq \varepsilon}\|\phi^{-1}(\bs{z})-\bs{z}\|_{\rho,\theta}\leq\frac{3^{12}\times4^{r}N^{28}}{\gamma^{2}}\|z\|^{3}_{\rho,\theta},
\end{align} 
such that the vector field \eqref{form816} is transformed into
\begin{equation}
\label{form820}
X^{(1)}=Z_{\leq2r+1}+R'''_{\geq2r+3}+(D\phi')^{-1}\big(Z_3^{>N}+\sum_{l=2}^{r}K^{>N}_{2l+1}+R'_{\geq2r+3}\big)\circ\phi',
\end{equation}
where
\begin{enumerate}[(i)]
\item  the vector field $Z_{\leq2r+1}$ is a normal form of order at most $2r+1$ and has no effect on the actions;
\item the remainder term $R^{(1)}_{\geq2r+3}$ satisfies the estimate
  \begin{equation}
  \label{form821}
  \|R'''_{\geq2r+3}(\bs{z}')\|_{\rho,\theta}\leq\frac{4^{(r-3)(r+3)}(C_3N)^{564(r-1)^2}}{\gamma^{13r-17}}\|z'\|_{\rho,\theta}^{2r+3};
  \end{equation}
  \item the vector field $Z_3^{>N}$ is a cubic integrable truncated remainder term with the coefficient estimate  $\|Z^{>N}_{3}\|_{\ell^{\infty}}\leq\frac{1}{4}$;
   \item  $K_{2l+1}^{>N}\in\ms{R}^{{\rm rev}}_{2l+1}$ is the truncated remainder term with the coefficient estimate
     \begin{equation}
    \label{form822}
    \|K^{>N}_{2l+1}\|_{\ell^{\infty}}\leq C_1^{l^2};
     \end{equation}
 \item the remainder term $R'_{\geq2r+3}$ satisfies the estimate
    \begin{equation}
    \label{form823}
    \|R'_{\geq2r+3}(\bs{z}')\|_{\rho,\theta}<(3C_1)^{r(r+1)}\|z'\|_{\rho,\theta}^{2r+3}.
     \end{equation}
  \item  the transformation $\phi'$ is a nearly identity coordinate transformation and satisfies the estimate
    \begin{equation}
    \label{form824}
    \|D\phi'\|_{\ms{P}_{\rho,\theta}\mapsto\ms{P}_{\rho,\theta}}\leq2,
    \end{equation}
and the same estimate is fulfilled by the inverse transformation.
\end{enumerate}
\end{theorem}
\begin{proof}
It follows from \Cref{th31}, \Cref{th51} and \Cref{th52} but using the space $\ell_{\rho,\theta}^2$ in place of the Sobolev space. Concretely, by \Cref{le72} and \Cref{le74},  the vector field estimates in the space $\ell_{\rho,\theta}^2$ is same as that in the Sobolev space with the norm $\|\;\|_{\rho,\theta}$ in place of $\|\;\|_{s}$. Then we have
\begin{align}
\label{formth71-1}
&\sup_{\bs{z}\in B_{\rho,\theta}(3\varepsilon)}\|\bs{z}-\phi^{(1)}(\bs{z})\|_{\rho,\theta}<9C_1\|z\|_{\rho,\theta}^{3},
\\\label{formth71-2}
&\sup_{\bs{z}\in B_{\rho,\theta}(2\varepsilon)}\|\bs{z}-\phi^{(2)}(\bs{z})\|_{\rho,\theta}\leq\frac{3^{11}\times4^{r}N^{28}}{\gamma^{2}}\|z\|_{\rho,\theta}^{3},
\\\label{formth71-3}
&\sup_{\bs{z}\in B_{\rho,\theta}(\frac{3}{2}\varepsilon)}\|\bs{z}-\phi^{(3)}(\bs{z})\|_{\rho,\theta}\leq\frac{2^{9(r-2)}(C_3N)^{2256}}{\gamma^{9}}\|z\|_{\rho,\theta}^{5},
\end{align}
and their inverse transformations satisfy the same estimates.
Thus for $\varepsilon\leq\frac{\gamma^{13}}{(C_3N)^{1128(r-1)}}$, the transformation $\phi=\phi^{(1)}\circ\phi^{2}\circ\phi^{(3)}$ satisfies the estimate \eqref{form818} and \eqref{form819}.
In view of  \eqref{form584}, we have the transformed vector field \eqref{form820} with the normal form 
$$Z_{\leq2r+1}=Z_1+Z_3^{\leq N}+Z_{5}^{\leq N}+\sum_{l=3}^{r}Z_{\Gamma_{l},\Gamma'_{l}}+\sum_{l=3}^{r}\tilde{Z}_{\tilde{\Gamma}_{l}}+\sum_{l=3}^{r-1}\tilde{Z}_{\tilde{\tilde{\Gamma}}_{l}}$$
and the transformation $\phi'=\phi^{(2)}\circ\phi^{(3)}$.
Moreover, the coefficient estimate \eqref{form822} follows from the estimate \eqref{form37}; by \eqref{form38} and \Cref{le72}, the remainder term $R'_{\geq2r+3}$ satisfies the estimate \eqref{form823};  and by \eqref{formth71-2}, \eqref{formth71-3}, the transformation $\phi'$ satisfies the estimate \eqref{form824}.
\end{proof}

\subsection{Proof of \Cref{th12}}
\label{sec73}
Introduce the Gaussian measure
\begin{equation}
\label{form86}
d\mu=\frac{e^{-\sum_{m\in\mb{N}_{*}}e^{2\rho m^{\theta}}m^{2}|z_m|^{2}}dzd\bar{z}}{\int_{\sum_{m\in\mb{N}_{*}}e^{2\rho m^{\theta}}|z_m|^{2}\leq\frac{1}{2}}e^{-\sum_{m\in\mb{N}_{*}}e^{2\rho m^{\theta}}m^{2}|z_m|^{2}}dzd\bar{z}},
\end{equation}
and then we have the following measure estimate.
\begin{lemma}
\label{le71}
Fix $\rho>0$ and $0<\theta\leq1$. For $r\geq2$, $N\geq2(r+1)$ and $\gamma>0$, if
\begin{equation}
\label{form87}
\varepsilon^{2}\leq\frac{2\gamma}{(r+1)N^{4r+2}},
\end{equation}
then there exists a positive constant $\lambda$ independent of $r,N,\gamma,\varepsilon$ such that
\begin{equation}
\label{form88}
\mu(\varepsilon z\in\mc{U}_{\gamma}^{N})\geq1-re^{2\rho(2r+2)^{\theta}}\lambda\gamma.
\end{equation}
\end{lemma}
\begin{proof}
Define the set
\begin{align*}
\Theta:=\big\{|\Omega_{\bs{j}}^{(2)}(\varepsilon^{2}I)|>\gamma'\;\text{and}\;|\tilde{\Omega}_{\bs{j}}^{(4)}(\varepsilon^{2}I)|>\gamma''\;\text{for any}\;\bs{j}\in Irr(\mc{R}_{\leq N})\;\text{with}\;\#\bs{j}\leq r\big\},
\end{align*}
where
\begin{align*}
&\gamma':=\gamma\varepsilon^{2}\|z\|_{\rho,\theta}^{2}N^{-4l-2}e^{-2\rho\kappa_{\bs{j}}^{\theta}},
\\&\gamma'':=2\gamma\varepsilon^{2}\|z\|_{\rho,\theta}^{2}N^{-4l-2}\max\{e^{-2\rho\kappa_{\bs{j}}^{\theta}},2\gamma\varepsilon^{2}\|z\|_{\rho,\theta}^{2}\}
\end{align*}
with $l:=\#\bs{j}$.
According to the proof of \Cref{le61}, we have $\Theta\subseteq\mc{U}_{\gamma}^{N}$.
Thus we only need to prove
\begin{equation}
\label{formle824}
\mu(\Theta)\geq1-re^{2\rho(2r+2)^{\theta}}\lambda\gamma.
\end{equation}

Consider the complementary set of $\Theta$:
\begin{equation}
\label{formle825}
\Theta^{c}=\Theta^{(1)}\cup\Theta^{(2)},
\end{equation}
 where 
\begin{align}
\label{formle826}
&\Theta^{(1)}=\{\exists \;\bs{j}\in Irr(\mc{R}_{\leq N})\;\text{with}\;\#\bs{j}\leq r\;\text{such that}\;|\Omega_{\bs{j}}^{(2)}(\varepsilon^{2}I)|\leq\gamma'\},
\\\label{formle827}
&\Theta^{(2)}=\{\exists \;\bs{j}\in Irr(\mc{R}_{\leq N})\;\text{with}\;\#\bs{j}\leq r\;\text{such that}\;|\tilde{\Omega}_{\bs{j}}^{(4)}(\varepsilon^{2}I)|\leq\gamma''\}.
\end{align}
Now we estimate the measure of $\Theta^{(1)}$ and $\Theta^{(2)}$, respectively.

\noindent\textbf{Estimation of $\Theta^{(1)}$.} By $\|z\|_{\rho,\theta}^2<\frac{1}{2}$ and $\bs{j}\in Irr(\mc{R}_{\leq N})$, one has
$$\mu(|\Omega_{\bs{j}}^{(2)}(\varepsilon^{2}I)|\leq\gamma')\leq\mu\big(|\Omega_{\bs{j}}^{(2)}(\varepsilon^{2}I)|\leq\frac{1}{2}\gamma\varepsilon^{2}N^{-4l-2}e^{-2\rho\kappa_{\bs{j}}^{\theta}}\big).$$
By \Cref{le62} with the weight $e^{2\rho j^{\theta}}$ in place of $j^{2s}$,  there exists a positive constant $\lambda_1$ independent of $r,N,\gamma,\varepsilon$ such that
\begin{align}
\label{formle828}
\mu(\Theta^{(1)})\leq\lambda_1\gamma.
\end{align}

\noindent\textbf{Estimation of $\Theta^{(2)}$.} By $\|z\|_{\rho,\theta}^2<\frac{1}{2}$ and $\bs{j}\in Irr(\mc{R}_{\leq N})$, one has
\begin{equation}
\label{formle829}
\mu(|\tilde{\Omega}_{\bs{j}}^{(4)}(\varepsilon^{2}I)|\leq\gamma'')\leq\mu\big(|\tilde{\Omega}_{\bs{j}}^{(4)}(\varepsilon^{2}I)|\leq\tilde{\gamma}''\big)
\end{equation}
with $\tilde{\gamma}''=\gamma\varepsilon^{2}N^{-4l-2}\max\{e^{-2\rho\kappa_{\bs{j}}^{\theta}},\gamma\varepsilon^{2}\}$.
In view of the proof of \Cref{le63}, by 
\eqref{form616} and \eqref{form620} with the weight $e^{2\rho j^{\theta}}$ in place of $j^{2s}$, we obtain
 \begin{equation}
\label{formle830}
\mu\big(|\tilde{\Omega}_{\bs{j}}^{(4)}(\varepsilon^{2}I)|\leq\tilde{\gamma}''\big)\leq\frac{4\tilde{\gamma}''e^{2\rho\kappa_{\bs{j}}^{\theta}}}{\varepsilon^{2}\prod_{m\in\mb{N}_*}(1-e^{-A_1\sqrt{m}})},
\end{equation}
\begin{align}
\label{formle831}
\mu\big(|\tilde{\Omega}_{\bs{j}}^{(4)}(\varepsilon^{2}I)|\leq\tilde{\gamma}''\big)\leq\frac{4\sqrt{2}e^{2\rho(2l+2)^{\theta}}N^{l+1}\sqrt{\tilde{\gamma}''}}{\varepsilon^{2}\prod_{m\in\mb{N}_*}(1-e^{-A_1\sqrt{m}})},
\end{align}
where $A_1=\frac{1}{2}(\sum_{m\in\mb{N}_*}m^{-\frac{3}{2}})^{-1}$.
By the estimates \eqref{formle830} and \eqref{formle831}, we have
\begin{align}
\label{formle832}
\mu\big(|\tilde{\Omega}_{\bs{j}}^{(4)}(\varepsilon^{2}I)|\leq\tilde{\gamma}''\big)\leq&\frac{4\min\{\tilde{\gamma}''e^{2\rho\kappa_{\bs{j}}^{\theta}},\sqrt{2}e^{2\rho(2l+2)^{\theta}}N^{l+1}\sqrt{\tilde{\gamma}''}\}}{\varepsilon^{2}\prod_{m\in\mb{N}_*}(1-e^{-A_1\sqrt{m}})}
\\\notag\leq&\frac{4\sqrt{2}e^{2\rho(2l+2)^{\theta}}\gamma}{N^{l}\prod_{m\in\mb{N}_*}(1-e^{-A_1\sqrt{m}})}.
\end{align}
In view of \eqref{formle827}, by \eqref{formle829} and \eqref{formle832}, there exists a positive constant $\lambda_2$ independent of $r,N,\gamma,\varepsilon$ such that
\begin{align}
\label{formle833}
\mu(\Theta^{(2)})
&\leq\sum_{l=2}^{r}\sum_{\substack{\bs{j}\in Irr(\mc{R}_{\leq N})\\\#\bs{j}=l}}\frac{4\sqrt{2}e^{2\rho(2l+2)^{\theta}}\gamma}{N^{l}\prod_{m\in\mb{N}_*}(1-e^{-A_1\sqrt{m}})}
\\\notag&\leq\frac{4\sqrt{2}\gamma}{\prod_{m\in\mb{N}_*}(1-e^{-A_1\sqrt{m}})}\sum_{l=2}^{r}e^{2\rho(2l+2)^{\theta}}
\\\notag&\leq re^{2\rho(2r+2)^{\theta}}\lambda_2\gamma.
\end{align}
To sum up, in view of \eqref{formle825}, by \eqref{formle828} and \eqref{formle833}, one has
\begin{align*}
\mu(\Theta^{c})\leq\mu(\Theta^{(1)})+\mu(\Theta^{(2)})\leq\lambda_1\gamma+re^{2\rho(2r+2)^{\theta}}\lambda_2\gamma\leq re^{2\rho(2r+2)^{\theta}}(\lambda_1+\lambda_2)\gamma,
\end{align*}
 which implies that the measure estimate \eqref{formle824} holds with taking $\lambda=\lambda_1+\lambda_2$.
\end{proof}

Finally, we prove \Cref{th12}.
In view of the change of variables \eqref{form21}, \eqref{form22} and \eqref{form24}, we have the transformation $(u,v)\mapsto(z,\bar{z}):=\bs{z}$, and if $\|u\|_{1}$ is properly small, we have the prior estimates
\begin{equation}
\label{form813}
\frac{1}{\sqrt{3}}\big(\|u\|^2_{\rho,\theta,\frac{3}{2}}+\|v\|^2_{\rho,\theta,\frac{1}{2}}\big)^{\frac{1}{2}}\leq\|z\|_{\rho,\theta}\leq\frac{\sqrt{3}}{2}\big(\|u\|^2_{\rho,\theta,\frac{3}{2}}+\|v\|^2_{\rho,\theta,\frac{1}{2}}\big)^{\frac{1}{2}}
\end{equation}
\begin{equation}
\label{form814}
\frac{1}{2}t\leq\tau\leq2t.
\end{equation}
Thus for any $\big(u(0,x),v(0,x)\big)\in B_{\rho,\theta}(\varepsilon)$, one has 
\begin{equation}
\label{form815}
\|z(0)\|_{\rho,\theta}<\frac{\sqrt{3}}{2}\varepsilon.
\end{equation}

Let  $N=|\ln\varepsilon|^{1+\frac{2}{\theta}}$, $r=[\frac{|\ln\varepsilon|}{15800\ln N}]$,  $\gamma=\varepsilon^{\frac{1}{14}}$
 and the open set
\begin{equation}
\label{form817}
\mc{V}_{\rho,\theta}:=\bigcup_{0<\varepsilon\leq\varepsilon_{0}}\Big(\mc{U}_{\gamma}^{N}\bigcap\big(B_{\rho,\theta}(\varepsilon)\backslash \overline{B_{\rho,\theta}(\frac{\varepsilon}{2})}\big)\Big),
\end{equation}
where $\overline{B_{\rho,\theta}(\frac{\varepsilon}{2})}$ is the closure of $B_{\rho,\theta}(\frac{\varepsilon}{2})$.
In view of \Cref{le71}, for any $0<\varepsilon\leq\varepsilon_0$ with the enough small $\varepsilon_0$, we have
\begin{equation}
\label{form819-1}
\varepsilon^{2}\leq\frac{2\gamma}{(r+1)N^{4r+2}}\quad\text{and}\quad re^{2\rho(2r+2)^{\theta}}\lambda<\varepsilon^{-\frac{1}{210}},
\end{equation}
and thus
\begin{equation}
\label{form818-1}
\mu(\frac{\sqrt{3}}{2}\varepsilon z(0)\in\mc{U}_{\gamma}^{N})\geq1-re^{2\rho(2r+2)^{\theta}}\lambda\gamma>1-\varepsilon^{\frac{1}{15}},
\end{equation}
which implies the measure estimate \eqref{form113}.

For any $0<\varepsilon\leq\varepsilon_0$ with the enough small $\varepsilon_0$, we have 
\begin{equation}
\label{form3-28}
\varepsilon\leq\frac{\gamma^{13}}{(C_3N)^{1128(r-1)}}.
\end{equation}
 Then by \Cref{th71}, there exists a nearly identity coordinate transformation such that $\phi(\bs{z}')=\bs{z}$. 
By \eqref{form818} and \eqref{form815}, one has
\begin{equation}
\label{form827}
\|z'(0)\|_{\rho,\theta}\leq\|z(0)\|_{\rho,\theta}+\|z(0)-z'(0)\|_{\rho,\theta}\leq\sqrt{\frac{5}{6}}\varepsilon.
\end{equation}
Denote the escape time of $z'$ by $T:=\inf\{\tau\mid \|z'(\tau)\|_{\rho,\theta}=\varepsilon\}$. For any $|\tau|\leq T$, by \eqref{form821}--\eqref{form824} and \Cref{le73}, one has
\begin{align}
\label{form828}
&\Big|\|z'(\tau)\|^2_{\rho,\theta}-\|z'(0)\|^2_{\rho,\theta}\Big|
\\\notag\leq&T\sup_{\tau\leq T}\pa_{\tau}(\|z'(\tau)\|^2_{\rho,\theta})
\\\notag\leq&2T\sup_{\|z'\|_{\rho,\theta}\leq\varepsilon}\Big(\|R'''_{\geq2r+3}(\bs{z}')\|_{\rho,\theta}\|z'\|_{\rho,\theta}
\\\notag&\qquad\quad+\|(D\phi')^{-1}\|_{\ms{P}_{\rho,\theta}\mapsto\ms{P}_{\rho,\theta}}\big\|\big(Z_3^{>N}+\sum_{l=2}^{r}K^{>N}_{2l+1}+R'_{\geq2r+3}\big)\circ\phi'(\bs{z}')\big\|_{\rho,\theta}\|z'\|_{\rho,\theta}\Big)
\\\notag\leq&2T\Big(\frac{4^{(r-3)(r+3)}
(C_3N)^{564(r-1)^2}}{\gamma^{13r-17}}\varepsilon^{2r+4}
\\\notag&\qquad\quad+\frac{36}{e^{2\rho(\frac{N}{r})^{\theta}}}\varepsilon^{4}+\sum_{l=2}^{r}\frac{3^{l+1}C_1^{2l}}{e^{2\rho(\frac{N}{l})^{\theta}}}(2\varepsilon)^{2l+2}+2^{2r+4}(3C_1)^{r(r+1)}\varepsilon^{2r+4}\Big)
\end{align}
For any $0<\varepsilon\leq\varepsilon_0$ with the enough small $\varepsilon_0$, we have
\begin{equation}
\label{form829}
e^{-2\rho(\frac{N}{r})^{\theta}}\leq\varepsilon^{\frac{|\ln\varepsilon|}{15800(1+\frac{2}{\theta})\ln|\ln\varepsilon|}},
\end{equation}
\begin{equation}
\label{form830}
\Big(\frac{4^{(r-3)(r+3)}
(C_3N)^{564(r-1)^2}}{\gamma^{13r-17}}+2^{2r+4}(3C_1)^{r(r+1)}\Big)\varepsilon^{2r+4}\leq\varepsilon^{4+\frac{|\ln\varepsilon|}{15800(1+\frac{2}{\theta})\ln|\ln\varepsilon|}}.
\end{equation}
Hence, by \eqref{form828}--\eqref{form830}, one has
\begin{equation}
\label{form831}
\Big|\|z'(\tau)\|^2_{\rho,\theta}-\|z'(0)\|^2_{\rho,\theta}\Big|\lesssim T\varepsilon^{4+\frac{|\ln\varepsilon|}{15800(1+\frac{2}{\theta})\ln|\ln\varepsilon|}}.
\end{equation}
If $T\leq2\varepsilon^{-\frac{|\ln\varepsilon|}{15800(1+\frac{2}{\theta})\ln|\ln\varepsilon|}}$, then by \eqref{form827} and \eqref{form831}, one has
\begin{align*}
\varepsilon^{2}=\|z'(T)\|_{\rho,\theta}^{2}\leq \|z'(0)\|_{\rho,\theta}^2+\Big|\|z'(T)\|^2_{\rho,\theta}-\|z'(0)\|^2_{\rho,\theta}\Big|
<\frac{5}{6}\varepsilon^{2}+\frac{1}{6}\varepsilon^{2},
\end{align*}
which is impossible. Hence, one has 
$$T\geq2\varepsilon^{-\frac{|\ln\varepsilon|}{15800(1+\frac{2}{\theta})\ln|\ln\varepsilon|}}.$$
By \eqref{form819}, for any $|\tau|\leq2\varepsilon^{-\frac{|\ln\varepsilon|}{15800(1+\frac{2}{\theta})\ln|\ln\varepsilon|}}$, one has
\begin{equation}
\label{form832-1}
\|z(\tau)\|_{\rho,\theta}\leq\|z'(\tau)\|_{\rho,\theta}+\|z'(\tau)-z(\tau)\|_{\rho,\theta}\leq\frac{2}{\sqrt{3}}\varepsilon.
\end{equation}

 In addition, for any $a\in\mb{N}_{*}$, one has
\begin{align}
\label{form832}
&\big||z_{a}(\tau)|^2-|z_{a}(0)|^2\big|
\\\notag\leq&\big||z_{a}(\tau)|^2-|z'_{a}(\tau)|^2\big|+\big||z'_{a}(\tau)|^2-|z'_{a}(0)|^2\big|+\big||z'_{a}(0)|^2-|z_{a}(0)|^2\big|.
\end{align}
Similarly to the estimates \eqref{form828}--\eqref{form831} but $e^{2\rho a^{\theta}}|z'_{a}(\tau)|^2$ instead of $\|z'\|_{\rho,\theta}$, for any $|\tau|\leq2\varepsilon^{-\frac{|\ln\varepsilon|}{15800(1+\frac{2}{\theta})\ln|\ln\varepsilon|}}$, one has
\begin{equation}
\label{form833}
e^{2\rho a^{\theta}}\big||z'_{a}(\tau)|^2-|z'_{a}(0)|^2\big|\lesssim\varepsilon^{4}.
\end{equation}
By \eqref{form819}, one has
\begin{align}
\label{form834}
e^{2\rho a^{\theta}}\big||z_{a}(\tau)|^2-|z'_{a}(\tau)|^2\big|
\leq&\big(\|z(\tau)\|_{\rho,\theta}+\|z'(\tau)\|_{\rho,\theta}\big)\|z(\tau)-z'(\tau)\|_{\rho,\theta}
\\\notag\leq&\big(\frac{2}{\sqrt{3}}+1\big)\frac{3^{12}\times4^{r}N^{28}}{\gamma^{2}}\varepsilon^{4}
\\\notag<&\frac{3^{13}\times4^{r}N^{28}}{\gamma^{2}}\varepsilon^{4}.
\end{align}
Then by \eqref{form832}--\eqref{form834}, one has
\begin{equation}
\label{form836}
\sup_{a\in\mb{N}_*}e^{2\rho a^{\theta}}\big||z_{a}(\tau)|^2-|z_{a}(0)|^2\big|\leq\frac{2}{\sqrt{6}}\varepsilon^{3}.
\end{equation}

Go back to the original variables. For $\big(u(0,x),v(0,x)\big)\in \mc{V}_{\rho,\theta}\bigcap B_{\rho,\theta}(\varepsilon)$, by \eqref{form813} and \eqref{form832-1}, for any 
$|t|\leq\varepsilon^{-\frac{|\ln\varepsilon|}{15800(1+\frac{2}{\theta})\ln|\ln\varepsilon|}}$, one has
\begin{equation}
\label{form728'}
\|u(t)\|^2_{\rho,\theta,\frac{3}{2}}+\|v(t)\|^2_{\rho,\theta,\frac{1}{2}}\leq3\|z(t)\|^2_{\rho,\theta}\leq4\varepsilon^2,
\end{equation}
which indeed implies the prior estimates \eqref{form813} and \eqref{form814}.
Notice that 
\begin{equation}
\label{form838}
\frac{2}{\sqrt{6}a^2}\big||z_a(t)|^2-|z_a(0)|^2\big|\leq|I_a(t)-I_a(0)|\leq\frac{\sqrt{6}}{2a^2}\big||z_a(t)|^2-|z_a(0)|^2\big|
\end{equation}
with the action $I_a:=\frac{a|u_a|^{2}+a^{-1}|v_a|^2}{2}$. Thus by \eqref{form836} and \eqref{form838}, one has
\begin{equation}
\label{form729'}
\sup_{a\in\mb{N}_*}e^{2\rho a^{\theta}}a^{2}|I_{a}(t)-I_{a}(0)|\leq\varepsilon^{3}.
\end{equation}





\end{document}